\documentclass[10pt,a4paper]{amsart}
\usepackage[utf8]{inputenc}
\usepackage{amssymb, amsmath, amscd, amsthm}
\usepackage{dsfont}

\usepackage{color}
\usepackage[font=small]{caption}
\usepackage{subcaption}
\usepackage{tikz}
\usetikzlibrary{shapes.geometric}
\usetikzlibrary{decorations.markings}
\usetikzlibrary{decorations.pathreplacing}
\usetikzlibrary{calc}
\usetikzlibrary{positioning}
\usetikzlibrary{arrows}
\usepackage{tikz-cd}
\usepackage[]{hyperref}

\theoremstyle{definition}
\newtheorem{definition}{Definition}
\newtheorem{proposition}[definition]{Proposition}
\newtheorem{theorem}[definition]{Theorem}
\newtheorem{lemma}[definition]{Lemma}
\newtheorem{corollary}[definition]{Corollary}
\newtheorem{remark}[definition]{Remark}

\newtheorem{example}{Example}

\numberwithin{equation}{section}
\numberwithin{definition}{section}
\numberwithin{figure}{section}

\renewcommand\Re{\operatorname{Re}}
\renewcommand\Im{\operatorname{Im}}

\newcommand{\Spin}{\text{Spin}}
\newcommand{\GL}{ {\widetilde{GL}^+}}
\newcommand{\id}{\text{id}}

\newcommand{\Cb}{\mathbb{C}}

\tikzstyle{product}=[draw, circle, fill, inner sep=0.05cm]
\tikzstyle{naka}=[draw, ellipse, inner sep=0.05cm]
\allowdisplaybreaks

\begin{document}

\title[State sum construction of 2d TQFTs on spin surfaces]{State sum construction of two-dimensional\\ topological quantum field theories\\ on spin surfaces}

\begin{abstract}
We provide a combinatorial model for spin surfaces. Given a triangulation of an oriented surface, a spin structure is encoded by assigning to each triangle a preferred edge, and to each edge an orientation and a sign, subject to certain admissibility conditions. The behaviour of this data under Pachner moves is then used to define a state sum topological field theory on spin surfaces. The algebraic data is a $\Delta$-separable Frobenius algebra whose Nakayama automorphism is an involution. We find that a simple extra condition on the algebra guarantees that the amplitude is zero unless the combinatorial data satisfies the admissibility condition required for the reconstruction of the spin structure.
\end{abstract}

\date{}

\author{Sebastian Novak}
\address{Sebastian Novak\\Fachbereich Mathematik\\Universit\"at Hamburg\\Bundesstra\ss e 55\\20146 Hamburg\\Germany}
\email{sebastian.novak@mailbox.org}

\author{Ingo Runkel}
\address{Ingo Runkel\\Fachbereich Mathematik\\Universit\"at Hamburg\\Bundesstra\ss e 55\\20146 Hamburg\\Germany}
\email{ingo.runkel@uni-hamburg.de}

\maketitle

\newpage

\tableofcontents

\newpage

\section{Introduction}

One way to construct a two-dimensional topological field theory on oriented surfaces is via the so-called lattice topological field theory -- or state sum -- approach \cite{fukuma1994lattice, bachas1993topological}. There, one triangulates the surface, assigns an invariant to the triangulated surface and finally proves that this invariant is independent of the chosen triangulation. The algebraic data in terms of which this invariant is defined is a $\Delta$-separable symmetric Frobenius algebra in a symmetric monoidal category (for example, the category of complex vector spaces); the invariant itself is a morphism in that category (in the example, a linear map) \cite{lauda2007state}. The invariant is defined for surfaces with non-empty boundary and is compatible with gluing together boundary components.

\medskip

The aim of the present paper is to provide an analogous construction for two-dimensional topological field theory on surfaces with spin structure. We start by giving a combinatorial model of triangulated spin surfaces which is well suited for the state sum description (Section \ref{sec:spintriangulations}). We determine the behaviour of the combinatorial data under change of triangulation (and under other local moves). This allows us to conclude that, given a $\Delta$-separable Frobenius algebra whose Nakayama automorphism squares to the identity (again all taken in some fixed symmetric monoidal category -- see Section \ref{sec:2dlatticeTFT}), one can assign a morphism to a triangulated spin surface in such a way that (i)
this morphism is independent of the triangulation (and other choices), and
(ii)
the resulting invariant is compatible with gluing of boundary components.

Topological field theories on spin surfaces and combinatorial models for spin structures have already been investigated in a number of places:
\begin{itemize}\setlength{\leftskip}{-1em}
\item[-]
In papers by Kuperberg, and by Cimasoni and Reshetikhin, a relation between dimer models and spin structures is proved \cite{kuperberg1998exploration, cimasoni2007dimers, cimasoni2008dimers}. The construction starts from a dimer configuration (a perfect matching) on a surface graph together with a Kasteleyn orientation of the edges. From this the authors construct a vector field with a finite number of singular points of even winding number to which they can assign a class in $H^1(P_{SO},\mathbb{Z}_2)$. We believe, however, that for the purpose of the state sum construction, our approach of gluing a spin surface out of ``spin triangles'' (see below) is more directly applicable than the approach via dimer models.
\item[-]
In the description of open-closed two-dimensional topological field theory by Moore and Segal also the spin case is covered \cite{moore2006d}. Moore and Segal use generators and relations of the bordism category to deduce the algebraic structure on the state spaces. The state sum construction we carry out produces an example of this, and we show how our results compare to theirs in Remarks \ref{rem:N-involution-on-Z} and \ref{rem:MS-euler-char-matches}.
\item[-]
In Lurie's approach to fully extended topological field theories \cite{lurie2009classification}, framed topological field theories correspond to fully dualisable objects in a suitable higher symmetric monoidal category. Since state sum constructions of topological field theories are in a sense  as local as possible, there should be a direct relation between the two. However, to our knowledge, this has not yet been made precise in general (see, however, \cite{davidovich2011state} for special cases). What one can do is to compare the algebraic objects produced by the two approaches. We do this in Remark \ref{rem:compare-Lurie}. 
\item[-]
In the recent paper \cite{budney2013combinatorial}, Budney provides a combinatorial model for spin structures on 
manifolds of arbitrary dimension. When restricted to the two-dimensional case, this model is very similar to ours. Indeed, this model seems equally well suited to investigate state sum constructions and it would be interesting to apply it in higher dimensions. Since the method by which we arrive at our combinatorial model is quite different from Budney's, we think it is still useful to present our construction in detail.
\item[-]
In the paper \cite{barrett2013spin} by Barrett and Tavares, state sum models for spin surfaces were investigated independently from the present paper. There, spin structures are characterised via an embedding of the surface into $\mathbb{R}^3$ along which the canonical spin structure on $\mathbb{R}^3$ is pulled back. While the geometric setting in \cite{barrett2013spin} is different from our approach, on the algebraic side there is some overlap of results, which we briefly address in Remark \ref{rem:BT-BCP-relation}.
\end{itemize}

\subsection{Summary of results}

\subsubsection{Geometric part} 

We work with spin structures on compact smooth oriented manifolds with parametrised boundaries. The surfaces are not equipped with a metric, so that for us, a spin structure is a fibre-wise non-trivial double cover of the oriented frame bundle (which is a principal $GL_2^+$-bundle), see Section \ref{sec:spin-no-metric}. To combinatorially describe the spin structure on a spin surface $\Sigma$, we proceed as follows.
\begin{itemize}\setlength{\leftskip}{-1em}
\item[-]
Denote the surface without spin structure by $\underline\Sigma$. Choose a triangulation of $\underline\Sigma$. Decorate each edge with an orientation and each triangle with a preferred edge. We refer to these two choices as a {\em marking}, see Section \ref{sec:marking-comb-surf}.
\item[-]
Pick a ``standard triangle'' $\underline\Delta \subset \Cb$ and equip it with a spin structure (which is unique up to isomorphism). Denote the resulting spin surface by $\Delta$. The triangulation of $\underline\Sigma$ includes as part of its data a smooth map $\chi_\sigma : \underline\Delta \to \underline\Sigma$ for each triangle $\sigma$. Choose a spin lift $\tilde\chi_\sigma : \Delta \to \Sigma$. There are two such lifts for each triangle $\sigma$, see again Section \ref{sec:marking-comb-surf}.
\item[-]
Consider two adjacent triangles $\sigma_L$ and $\sigma_R$ in $\underline\Sigma$ and let $e$ be their common edge. The labels $L$/$R$ are such that $\sigma_L$ is to the left of $e$ and $\sigma_R$ is to the right with respect to the orientation chosen on $e$. The differential $d(\chi_L^{-1}\circ \chi_R)_p$ for a point $p$ of the corresponding edge of $\underline\Delta$ induces a transformation on the frame bundle. Choose a ``standard lift'' of all such transformations to the spin surface $\Delta$. (Actually, a good deal of effort goes into the definition of the standard lift, see Section \ref{sec:index_inner}.) The standard lift is then either equal to $\tilde\chi_L^{-1}\circ \tilde\chi_R|_p$, or it differs by a minus sign. Define the {\em edge sign} $s(e)$ as the difference between these lifts, see Section \ref{sec:index_inner}.
\item[-]
A similar treatment is applied to edges on the parametrised boundaries. We refer to Section \ref{sec:edge-sign-bnd} for details. Here we just remark that boundary components are parametrised by small annuli in $\Cb$, and such an annulus can carry two non-isomorphic spin structures: a Neveu-Schwarz type spin structure (NS-type) and a Ramond-type spin structure (R-type). Accordingly, when treating parametrised boundaries below, we always have to distinguish these two cases.
\end{itemize}
The combinatorial data for a spin surface $\Sigma$ thus consists of a triangulation, a marking and the edge signs. 

Let us call an assignment $e \mapsto s(e)$ of signs to edges {\em admissible} if around each
vertex $v$ we have $\prod_e s(e) = -1$, where the product is over all edges $e$ with $v \in \partial e$. Here we assumed that the marking is such that all edges are oriented towards the vertex and none of the edges touching $v$ are the preferred edge of a triangle (otherwise the rule is slightly more complicated, see Corollary \ref{lem:vertex.rule.inner}). For a boundary vertex there is a condition for NS-type boundaries and for R-type boundaries, see Lemma \ref{lem:vertex.rule.boundary}. 

By construction, a spin surface $\Sigma$ produces admissible edge signs $s_\Sigma$. Conversely, any assignment of edge signs $s$ defines a spin structure on $\underline\Sigma$ minus its vertices. This spin structure extends to the vertices if and only if $s$ is admissible. Let us denote the resulting spin surface by $S(\underline\Sigma,s)$. 
Given a triangulated surface $\underline\Sigma$, denote by $\mathcal{M}_{\underline\Sigma}$ the set of all pairs $(d,s)$ where $d$ is a marking on $\underline\Sigma$ and $s$ are admissible edge signs for $(\underline\Sigma,d)$. Keeping the triangulation fixed but changing the choice of marking or the choice of spin lifts $\tilde\chi_\sigma$ modifies the edge signs in a specific way, introducing an equivalence relation $\sim_\text{fix}$ on $\mathcal{M}_{\underline\Sigma}$. The equivalence classes classify spin structures on $\underline\Sigma$:

\medskip\noindent
{\bf Theorem \ref{thm:parametrisation-of-spin}.}
The assignment $(d,s) \mapsto S\big(\,(\underline\Sigma,d)\,,\,s\big)$ induces a bijection
\begin{equation*}
    \mathcal{M}_{\underline{\Sigma}} / \sim_\text{fix} ~ \longrightarrow ~ (\text{spin structures on $\underline\Sigma$ up to isom.\ of spin str.}) \ .
\end{equation*}

\medskip
The behaviour of the edge signs under change of marking, of the spin lifts $\tilde\chi_\sigma$, and of the triangulation via 2-2 and 3-1 Pachner moves is worked out in Sections \ref{sec:moves-notriang} and \ref{sec:moves-pachner}. The resulting rules dictate the properties an algebraic structure needs to have in order to define an invariant of the spin surface.

\subsubsection{Algebraic part} 

	Fix an additive symmetric monoidal category $\mathcal{S}$. 
Choose an object $A \in \mathcal{S}$ and three morphisms: $t: A \otimes A \otimes A \to \mathbf{1}$ and $c_{+1}, c_{-1} : \mathbf{1} \to A \otimes A$. Given a spin surface $\Sigma$ and a marked triangulation thereof together with the edge signs computed as above, one can produce a morphism 
$$
	T_\text{triang}(\Sigma) : \mathbf{1} \longrightarrow (A^{\otimes 3})^{\otimes B}\ ,
$$
where $B$ is the number of boundary components of $\Sigma$, see Section \ref{sec:prelim-graph}. The definition is simply $T_\text{triang} = \big(\bigotimes_\sigma t\big) \circ \tau \circ \big(\bigotimes_e c_{s(e)}\big)$, where the product is over all triangles $\sigma$ and over all edges $e$. The permutation $\tau$ is determined by the connectivity of edges and triangles, by the orientation of the edges, and by the preferred edge assigned to each triangle.
We get three copies of $A$ for each boundary component because the triangulation is required to have precisely three edges on each boundary component.

If one imposes that the data $t,c_{\pm1}$ satisfies the invariance conditions coming from the changes of marking, spin lifts and triangulation mentioned above, the morphism $T_\text{triang}$ in turn depends only on the spin surface (with parametrised boundaries) $\Sigma$, see Section \ref{sec:localmoves}.

The conditions that $t,c_{\pm1}$ have to satisfy can be phrased in a more standard way under two additional assumptions: 1) The copairing $c_{-1}$ should be non-degenerate -- this allows to define an associative product on $A$ in terms of $t$ and $c_{-1}$. 2) This product should have a unit. Under these assumptions, the conditions on $t,c_{\pm1}$ can be rephrased as the requirement that $A$ carries the structure of a 
\begin{quote}
\em
$\Delta$-separable Frobenius algebra whose Nakayama automorphism squares to the identity.
\end{quote}
The definitions of these terms are given in Section \ref{sec:ana-alg-str}. In brief, a Frobenius algebra $A$ is an algebra and a coalgebra, and the two structures are related by the Frobenius condition (i.e.\ the coproduct is a bimodule map). $A$ is $\Delta$-separable if the product $\mu$ is left-inverse to the coproduct $\Delta$, i.e.\ $\mu \circ \Delta = \id_A$. Each Frobenius algebra in a symmetric monoidal category is equipped with an automorphism, called the Nakayama automorphism, defined in terms of the pairing and copairing of $A$ and the symmetric braiding of $\mathcal{S}$.

Let us write $T_A$ for a version of $T_\text{triang}$ where we used the non-degenerate pairing of $A$ to exchange source and target object, that is,
$T_A(\Sigma) :   (A^{\otimes 3})^{\otimes B}  \to  \mathbf{1}$. This change is made to correspond to the physical notion of a correlator as a multilinear map to the complex numbers -- mathematically the other variant works just as well.
Our main result is:

\medskip\noindent
{\bf Theorem \ref{thm:TA-triang-indep}.}
Let $A$ be a Frobenius algebra in a symmetric monoidal category $\mathcal{S}$, such that $A$ is $\Delta$-separable and has a Nakayama automorphism which is an involution. 
Then $T_A(\Sigma)$ is independent of the choice of spin triangulation of the spin surface $\Sigma$ and $T_A(\Sigma) = T_A(\Sigma')$ for isomorphic spin surfaces $\Sigma$ and $\Sigma'$.

\medskip
Denote the counit of $A$ by $\varepsilon : A \to \mathbf{1}$. The non-degenerate invariant pairing on $A$ is  given by $b = \varepsilon \circ \mu : A \otimes A \to \mathbf{1}$. Let $c_{U,V} : U \otimes V \to V \otimes U$ be the symmetric braiding of $\mathcal{S}$. A Frobenius algebra is symmetric if $b \circ c_{A,A} = b$. A symmetric Frobenius algebra satisfies $N = \id_A$ (and vice versa). In particular, a $\Delta$-separable symmetric Frobenius algebra $A$ provides an example for Theorem \ref{thm:TA-triang-indep}. However, if $N = \id_A$, then $T_A(\Sigma)$ is actually independent of the spin structure. This is maybe not too surprising as in this case $A$ is the datum needed to define a state sum topological field theory on oriented surfaces without spin structure, see \cite{fukuma1994lattice, bachas1993topological, lauda2007state}.

\medskip
 
Let us turn back to the general case of $N \circ N = \id_A$ but $A$ is not necessarily symmetric. In this case it may happen that $\mu \circ (N \otimes \id_A) \circ \Delta = 0$, or, in other words, $N * \id_A = 0$ where $*$ is the convolution product on $\mathrm{End}(A)$. This condition is interesting: given a marked triangulation of a surface together with an assignment of edge signs (admissible or not), one can define $T_A$ by the same rules as before, though in general it will now depend on the triangulation. But if $N * \id = 0$, one can show that $T_A$ gives zero unless the edge signs are admissible (Section \ref{sec:N*id=0}). In this sense, algebras $A$ as in Theorem \ref{thm:TA-triang-indep} which in addition satisfy $N * \id = 0$ automatically enforce the admissibility condition.

The invariant $T_A$ is compatible with gluing of surfaces as described in Proposition \ref{prop:T-gluing} and so defines a topological field theory for spin surfaces. To make the connection to the usual functorial formulation of a 2-1 topological field theory, we have to compute the state spaces assigned to boundary circles of NS- and of R-type. These are defined as the images of  idempotents on $A \otimes A \otimes A$ obtained from cylinders with spin structure, see Sections \ref{sec:TFT.cylinder} and \ref{sec:cylinder-projections}. These idempotents factor through $A$, and assuming that they split, we obtain the state spaces as subobjects
$$
	Z^{NS} , Z^R \subset A  \ .
$$
The corresponding idempotents on $A$ are easy to describe:
$$
	P^{NS} = \mu \circ c_{A,A} \circ (N \otimes \id_A) \circ \Delta \quad , \qquad 
	P^{R} = \mu \circ c_{A,A} \circ \Delta \ .
$$
From this it is, for example, straightforward to check that $Z^{NS}$ is the centre of $A$ (Lemma \ref{lem:ZNS-is-centre-of-A}).
The projectors $P^{NS}$ and $P^R$ also appear in \cite{Brunner:2013ota} in the study of ``generalised twisted sectors'', see Remark \ref{rem:BT-BCP-relation}.

The Nakayama automorphism $N$ commutes with $P^{NS/R}$. This means it can be restricted to an automorphism of $Z^{NS}$ and $Z^R$ which, by abuse of notation, we also call $N$. One checks that on $Z^{NS}$, $N$ acts as the identity while on $Z^R$, $N$ defines an involution. Geometrically, this results from the fact that there is a unique spin structure on cylinder with NS-type boundary components and two non-isomorphic spin structures on a cylinder with R-type boundary components (Remark \ref{rem:N-involution-on-Z}).
Let us write 
$$
	Z = Z^{NS} \oplus Z^R 
$$
for the total state space. The algebra $A$ induces the structure of a $\mathbb{Z}_2$-graded Frobenius algebra on $Z$. 
We will write the group operation of the grading group $\mathbb{Z}_2$ multiplicatively, so that the two degrees are $\pm1$. Then $Z^{NS}$ has degree $+1$ and $Z^R$ has degree $-1$. Write $\mu_Z$ for the multiplication on $Z$ and denote the embeddings of the two summands by $\iota^Z_+ : Z^{NS} \to Z$ and $\iota^Z_- : Z^R \to Z$. The algebra $Z$ is graded-commutative in the sense that, for $\alpha,\beta \in \{\pm1\}$,
$$
	\mu_Z \circ c_{Z,Z} \circ (\iota^Z_\alpha \otimes \iota^Z_\beta) = \mu_Z \circ (N_{\gamma} \otimes \id_Z) \circ  (\iota^Z_\alpha \otimes \iota^Z_\beta) \ ,
$$
where $\gamma=-1$ if $\alpha=\beta=-1$ and $\gamma=1$ else, and $N_{+1} = \id$ and $N_{-1} = N$ (Proposition \ref{prop:Z-is-graded-Frobalg}).

\medskip

In \cite[Sect.\,2.6]{moore2006d}, Moore and Segal consider the case that $\mathcal{S}$ is the category of super vector spaces. They require that the involution $N$ on $Z$ coincides with the parity involution of a super vector space. $Z^{NS}$ then has to be purely even, and one easily checks that in this case $Z$ is commutative when considered as an algebra in $\mathbf{Vect}$ instead of $\mathbf{SVect}$ (i.e.\ $\mu_Z \circ c_{Z,Z}^{\mathbf{Vect}} = \mu_Z$), in agreement with \cite{moore2006d}. Moore and Segal find that $Z$ should satisfy an extra condition, namely that the NS- and R-type ``Euler characters'' agree. This condition also holds in our case, see Remark \ref{rem:MS-euler-char-matches} for details.

\subsection{Outlook}\label{sec:outlook}

There are at least three directions which would be very interesting to pursue further. 

\medskip

The first direction, and in fact the original motivation for starting this research, is to try to use this formalism to formulate two-dimensional quantum field theories on spin surfaces in terms of theories on oriented surfaces with defect lines. Let us expand a bit on this point. A two-dimensional quantum field theory with defects is defined on surfaces decorated by one-dimensional submanifolds, so-called (topological) defect lines. These defect lines are equipped with an orientation and a label determining the defect condition. See \cite{davydov2011field} for an introduction to field theories with defects. Defects can for example be used to describe (generalised) orbifolds, see \cite{frohlich2009defect, carqueville2012orbifold}. 
In this construction, to evaluate a correlator in the orbifold theory one considers the correlator of the same surface in the original theory, but equipped with a network of defect lines. 
	The defect lines and junctions must be such that the result is independent of the choice of triangulation. 
One can organise two-dimensional quantum field theories with defect lines into a bicategory. The endomorphism category of a given theory is then a monoidal category. 
Much by the same argument as used in the state-sum construction of a two-dimensional topological field theory on oriented surfaces, one finds that the defect and junctions have to come from a $\Delta$-separable symmetric Frobenius algebra in that endomorphism category (which is not braided, so the symmetry condition is formulated slightly differently), see \cite{frohlich2009defect, carqueville2012orbifold} for details. 
	An analogous construction works for quantum field theories on spin surfaces. 
Namely, if the category of topological endo-defects of a theory on oriented surfaces contains an algebra as in Theorem \ref{thm:TA-triang-indep} (with appropriate definition of the Nakayama automorphism), this 
	allows one
to define a quantum field theory on spin surfaces. A benefit of this construction would be that the description of correlators of two-dimensional rational conformal field theory on oriented surfaces via three-dimensional topological field theory developed in \cite{fuchs2002tft,fuchs2005tft, fjelstad2008uniqueness, frohlich2009defect, fjelstad2012rcft} can be brought to bear also on two-dimensional conformal field theory on spin surfaces.

\medskip

Secondly, the condition $N * \id=0$ discussed above deserves further study. It indicates that the sum over spin structures can be implemented by a local statistical model in the following sense: Fix a marked triangulation on a surface $\underline\Sigma$. Think of the sign $s(e)$ attached to an edge $e$ as a statistical variable and consider the sum over all edge sign configurations. If $A$ as in Theorem \ref{thm:TA-triang-indep} satisfies $N * \id=0$, only the admissible edge sign configurations can contribute. Thus the statistical sum over all edge signs implements the sum over spin structures (with some over-counting depending on the triangulation). This is to some extend comparable to the relation between dimer models and spin structures discussed in \cite{cimasoni2007dimers, cimasoni2008dimers}. There, for a {\em fixed} dimer configuration on a surface graph, the different Kasteleyn orientations (up to an equivalence relation) precisely describe the possible spin structures on the surface \cite[Cor.\,3]{cimasoni2007dimers}. A Kasteleyn orientation has to satisfy a rule for each face of the graph, and one could try to build a model which sums over all edge orientations but assigns non-zero Boltzmann weight only to Kasteleyn orientations. It would be interesting to investigate the condition $N*\id=0$ and the connection to dimer models further.

\medskip

Thirdly, it would be desirably to have state sum descriptions available also for higher dimensional spin topological field theories. It should be possible to construct such theories starting, for example, from the combinatorial model developed by Budney in \cite{budney2013combinatorial}. To do so, one would have to understand how the three-dimensional Pachner moves affect the combinatorial data in \cite{budney2013combinatorial}, and to deduce the algebraic structure necessary to obtain an invariant. It is known that spherical fusion categories provide examples of state sum constructions of three-dimensional topological field theories for oriented three-manifolds \cite{turaev1992state, barrett1996invariants}. Thus, as a guiding principle, such categories have to be examples of the algebraic structure defining spin-topological field theories (but those examples would -- as in the two-dimensional case -- not be sensitive to the spin structure). The result could then be compared to existing (non-state sum) constructions of spin topological field theories in three dimensions, notably \cite{blanchet1996topological, beliakova1998spin}.

\bigskip\noindent
{\bf Acknowledgements:}
	Some of the results reported in this work were presented by SN at the {\em Workshop on Field Theories with Defects} (University of Hamburg, 13.2--15.2.2013) and at the {\em Simons Center Summer Workshop 2013 ``Defects''} (Simons Center, Stony Brook, 22.7--16.8.2013).
We gratefully acknowledge helpful discussions with
	Bruce Bartlett,
	Alexander Barvels,
	Tilman Bauer,
	David B\"ucher,
	Nils Carqueville,
	Alexei Davydov,
	Chris Douglas,
	Malte Dyckmanns,
	J\"urgen Fuchs,
	Andre Henriques
	and
	Christoph Sachse,
    as well as the suggestions made by the referee.
SN is supported by the DFG funded Research Training Group 1670 ``Mathematics inspired by string theory and quantum field theory''. IR is supported in part by the DFG funded Collaborative Research Center 676 ``Particles, Strings, and the Early Universe''.

\section{Spin structures on surfaces}

A \emph{surface} is an oriented, smooth, two-dimensional real manifold, possibly with boundary. Maps between surfaces are smooth and orientation-preserving. We identify the complex plane $\Cb$ with $\mathbb{R}^2$ and sometimes use complex coordinates, but maps between subsets of $\Cb$ need not be holomorphic.
\subsection{Spin structures} \label{sec:spin-no-metric}

Let $GL^+_n$ be the group of orientation preserving linear automorphisms of $\mathbb{R}^n$.
The inclusion $i:SO_n\to GL^+_n$ is a homotopy equivalence by the $QR$-decomposition. By covering theory we then get a commutative diagram of Lie groups,

\begin{equation}
\begin{tikzcd}[row sep=tiny]
        & \Spin_n \ar{dd}{\tilde\imath} \ar{r}{p_{SO}} & SO_n \ar{dd}{i}\\
        \mathbb{Z}_2 \ar{ur} \ar{dr} && \\
        & \GL_n \ar{r}{p_{GL}}  & GL^+_n ,
    \end{tikzcd}
    \label{eq:GL+covering}
\end{equation}
such that $\GL_n$ is the nontrivial twofold covering of $GL_n$.
This leads us to the following

\begin{definition} \label{eq:spin_structure_wm}
Let $\Sigma$ be surface. Let $\zeta :F_{GL^+} \to \Sigma$ be its bundle of oriented frames. A \emph{spin structure} on $\Sigma$ is a pair $(\eta ,p)$, consisting of a $\GL_2$-principal bundle $\eta :P_\GL \to M$ and a map $p:P_\GL \to F_{GL^+}$ such that the following diagram commutes:
\begin{equation}
    \begin{tikzcd}[row sep=tiny]
        P_\GL \times \GL_2 \ar{rr}{R_{\GL_2}} \ar{dd}{p\times p_{GL}} && F_\GL \ar{dd}{p}\ar{dr}{\eta}&\\
        &&& M\\
        F_{GL^+}\times GL^+_2 \ar{rr}{R_{GL^+_2}} &&F_{GL^+} \ar{ur}{\zeta}
    &.
\end{tikzcd}
\end{equation}
Here $p_{GL}:\GL_2\to GL^+_2$ is the covering homomorphism and $R_{\GL_2}$, $R_{GL^+_2}$ denote the right actions.
An \emph{isomorphism of spin structures} is a map of principal bundles $f:P_\GL \to P'_\GL$ such that $p'\circ f = p$.
\end{definition}

We remind the reader of the usual classification result for spin structures (see e.g.\ \cite{johnson1980} or 
	\cite[Thm.\,II.1.7]{lawson1989spin}):
\begin{proposition}
Isomorphism classes of spin structures on a given oriented surface $\Sigma$ are in one-to-one correspondence to $H^1(\Sigma,\mathbb{Z}_2)$.
\label{prop:class.spinstructures}
\end{proposition}

After these preliminaries, we can introduce the main geometrical object of this paper:

\begin{definition}
A \emph{spin surface} $\Sigma$ is a surface $\underline\Sigma$, together with a spin structure $P_\GL(\Sigma)$.
\end{definition}

\begin{definition}\label{def:morp-spin-surf}
A \emph{morphism} (or \emph{map}) \emph{of spin surfaces} is given by a map of $\GL_2$ bundles $\tilde{f}:P_\GL(\Sigma) \to P_\GL(\Sigma')$
such that the diagram
\begin{equation}
\begin{tikzcd}
P_\GL(\Sigma) \ar{r}{\tilde{f}} \ar{d}{p}  & P_\GL(\Sigma') \ar{d}{p'} \\
F_{GL^+}(\Sigma) \ar{r}{df_*} \ar{d} & F_{GL^+}(\Sigma') \ar{d}  \\
\underline{\Sigma} \ar{r}{f} & \underline{\Sigma}'
\end{tikzcd}
\end{equation}
commutes. Here $f: \underline{\Sigma} \to \underline{\Sigma}'$ denotes the underlying map of surfaces, and $df_*$ the map induced on the bundle of oriented frames by the derivative of $f$.
\end{definition}

We will sometimes write $\tilde f: \Sigma \to\Sigma'$ for a map between spin surfaces as an abbreviation of $\tilde f : P_\GL(\Sigma) \to P_\GL(\Sigma')$.

\begin{remark} \label{rem:map=iso}
	Note that an isomorphism of spin structures as in Definition \ref{eq:spin_structure_wm} is required to be the identity on the underlying surface, while a map of spin surfaces may even relate spin structures with different underlying surfaces. In fact,
to give a map $\tilde f: \Sigma \to \Sigma'$ between spin surfaces is the same as to give an isomorphism of spin structures from $P_\GL(\Sigma)$ to the pullback spin structure $\tilde{f}^* P_\GL(\Sigma') = F_{GL^+} \times_{F'_{GL^+}} P_{\GL}(\Sigma')$. 
\end{remark}

\medskip

Let $\Sigma$ be a spin surface. By right multiplication with the nontrivial element of the kernel of $p_{GL}:\GL_2 \to GL^+_2$, we obtain a natural involution $\omega :\Sigma\to \Sigma$, the \emph{leaf exchange automorphism}.

\subsection{A construction of $\GL_2$}

We give a construction of $\GL_2$, parallel to the construction of the metaplectic group in \cite[Sect.\,I.1.8]{lion1980weil}.
$GL^+_2$ acts on the complex upper half plane 
$\mathbb{H} = \{ z \,|\, \mathrm{Im}(z)>0 \}$ as
\begin{equation}
    g.z=\begin{pmatrix}a & b\\c&d\end{pmatrix}.z = \frac{az+b}{cz+d}\ .
    \label{eq:GL2+action}
\end{equation}
The denominator will be denoted by $j_g(z)=cz+d$; 
	it satisfies $j_{g_1g_2}(z) = j_{g_1}(g_2.z) \, j_{g_2}(z)$.
We define $\GL_2$ as 
\begin{equation}
    \GL_2 := \big\{ (g,\varepsilon ) \,\big|\, g \in GL^+_2,\, \varepsilon :\mathbb{H} \to \Cb \; \text{holomorphic s.t.}\; \varepsilon (z)^2 = j_g(z) \big\} \ .
    \label{eq:GL-def}
\end{equation}
Composition and inverse are given by
\begin{align}
        (g_1,\varepsilon_1) \circ (g_2,\varepsilon_2) &:= (g_1g_2, \varepsilon ) ~~ , ~~ && \text{with }~~ \varepsilon (z) = \varepsilon_1(g_2.z)\, \varepsilon_2(z) \ ,
        \\
        (g,\varepsilon )^{-1} &:= (g^{-1},\tilde{\varepsilon}) ~~ , ~~ &&  \text{with } ~~ \tilde{\varepsilon}(z) = \varepsilon (g^{-1}.z)^{-1} \ . \nonumber
    \end{align}
The unit is $e = (\mathds{1},1)$.
The map $p_{GL}: \GL_2 \to GL^+_2$ is given by
\begin{equation}
p_{GL}:(g,\varepsilon ) \mapsto g \ .
\label{eq:GL-projection}
\end{equation}
Notice that for an element $(g,\varepsilon )\in \GL_2$, the function $\varepsilon $ is uniquely determined by giving its value at a single point $z\in \mathbb{H}$, e.g.\ at $z=i$.

In the following it will be useful to describe scaled rotations in $GL_2^+$, as well as their preimage in $\GL$, via elements of $\mathbb{C}^\times$. To do so, we define the injective group homomorphisms
\begin{align} \label{eq:i-delta-C*-def}
i_{\mathbb{C}^\times} : \mathbb{C}^\times &\to GL_2^+
~~ &, \qquad
\delta : \mathbb{C}^\times &\to \GL
\\
z &\mapsto \big(\begin{smallmatrix} \Re z & -\Im z \\ \Im z & \Re z \end{smallmatrix}\big) 
&
z &\mapsto \big(\,i_{\mathbb{C}^\times}(z^2) \,,\, \varepsilon_z \,\big)\quad \text{s.t.} \quad \varepsilon_z(i) = z \ .
\nonumber
\end{align}
We will identify $SO_2$ with the unit circle, $SO_2 = \{z\in \Cb \,|\,|z|=1\} \subset \mathbb{C}^\times$. Then, for example, the map $\delta$ lifts the inclusion map of $SO_2$ into $GL_2^+$: the diagram
\begin{equation}
    \begin{tikzcd}
        SO_2 \ar{r}{\delta} \ar{d}{p_{SO}} & \GL_2 \ar{d}{p_{GL}} \\
        SO_2 \ar{r}{i_{\mathbb{C}^\times}}& GL_2^+
    \end{tikzcd}
    \label{eq:GL-Spin-pullback square}
\end{equation}
commutes. We will from now on suppress the inclusion $i_{\mathbb{C}^\times}$ in our formulas, but we will write out $\delta$ to avoid confusion.

\begin{remark}
The exact sequence 
\begin{equation}
    \begin{tikzcd}
        1 \ar{r} & \mathbb{Z}_2 \ar{r}{\delta} & \GL_2 \ar{r}{p_{GL}} & GL_2^+ \ar{r} &1
    \end{tikzcd}
\end{equation}
is a central extension of groups and the centre of $\GL_2$ is the preimage of the centre of $GL_2^+$ under $p_{GL}$.
In particular, the elements in the image of $\delta$ which are in the centre are
\begin{equation} \label{eq:delta-cap-Z}
    \delta(SO_2) \cap Z(\GL_2) = \{ \delta(\pm1), \delta(\pm i) \} \ .
    \end{equation}
\end{remark}

\subsection{$QR$-decomposition for $\GL_2$}

We start from the $QR$-decomposition of an invertible matrix into an orthogonal and an upper triangular part. If we require that the upper triangular matrix have nonnegative diagonal entries, the $QR$-decomposition is unique, and we get smooth maps
\begin{equation}
 \mathbf{Q}: GL_2^+ \to SO_2 \quad , \qquad
    \mathbf{R}: GL_2^+ \to T_2 \ ,
\end{equation}
such that $\mathbf{Q}(g)\, \mathbf{R}(g) = g$ for $g \in GL_2^+$.
Here, $T_2$ is the space of upper triangular matrices with positive diagonal entries.
Notice, however, that $\mathbf{Q}$ and $\mathbf{R}$ are not group homomorphisms. 

We now describe a related decomposition for $\GL_2$. The preimage $p_{GL}^{-1}\left( T_2 \right)$ has two connected components.
Let $\tilde{T}_2$ be the connected component of the identity, 
\begin{equation}
  \tilde T_2 = \big\{ \, \big(
        \big(\begin{smallmatrix}
             a & b \\ 0 & c 
        \end{smallmatrix}\big), \sqrt{c} \big) \,\big|\, a,c>0
        \big\} \ .
\end{equation}
Note that $\tilde T_2$ is a subgroup of $\GL_2$. The second component in $p_{GL}^{-1}\left( T_2 \right)$ is given by replacing $\sqrt{c}$ by $-\sqrt{c}$ in the above expression.

\begin{lemma} \label{lem:QR-tilde}
The $QR$-decomposition lifts to $\GL_2$, i.e.\ there are unique smooth maps
\begin{equation}
 \mathbf{\tilde{Q}}: \GL_2 \to \delta(SO_2) \quad , \qquad
\mathbf{\tilde{R}}: \GL_2 \to \tilde{T}_2 \ ,
\end{equation}
such that $\mathbf{\tilde{Q}}(g) \, \mathbf{\tilde{R}}(g) = g$ for all $g\in \GL_2$, and such that 
\begin{equation} \label{eq:tQ-tR-conditions}
        p_{GL}\circ\mathbf{\tilde{Q}} = \mathbf{Q}\circ p_{GL} \quad , \qquad
        p_{GL}\circ\mathbf{\tilde{R}} = \mathbf{R}\circ p_{GL} \ .
\end{equation}
\end{lemma}

\begin{proof}
Consider the lift of the inclusion map $i:T_2\to GL_2^+$ given by
\begin{equation}
	\tilde\imath:T_2 \to \GL_2 \quad, \quad
	\big(\begin{smallmatrix}
         a & b \\ 0 & c 
        \end{smallmatrix}\big) \mapsto	
    \big(\big(\begin{smallmatrix}
         a & b \\ 0 & c 
    \end{smallmatrix}\big), \sqrt{c} \big) 
\end{equation}
and define 
$\mathbf{\tilde{R}} := \tilde\imath \circ \mathbf{R} \circ p_{GL}$. Using that $p_{GL} \circ \tilde\imath = \id$ and that $\tilde\imath$ is a group homomorphism, one checks that $p_{GL}(g\, \mathbf{\tilde{R}}(g)^{-1}) =\mathbf{Q}(p_{GL}(g)) \in SO(2)$ for all $g\in \GL_2$. Hence we can define $\mathbf{\tilde{Q}}(g) := g\, \mathbf{\tilde{R}}(g)^{-1} \in \delta(SO(2))$.
It is then immediate that $\mathbf{\tilde{Q}}(g) \, \mathbf{\tilde{R}}(g) = g$ and that \eqref{eq:tQ-tR-conditions} holds.

Next we turn to the uniqueness of the $QR$-decomposition in $\GL_2$. Suppose that $g \in \GL_2$ has been written as $g = \delta(q)\cdot t$ with $q \in SO_2$ and $t \in \tilde T_2$. Applying $p_{GL}$ and using uniqueness of the $QR$-decomposition of $GL_2^+$, we see that either $\delta(q) = \mathbf{\tilde{Q}}(g)$ and $t = \mathbf{\tilde{R}}(g)$, or else $\delta(q) = \mathbf{\tilde{Q}}(g) \delta(-1)$ and $t = \mathbf{\tilde{R}}(g)  \delta(-1)$. But $\mathbf{\tilde{R}}(g)  \delta(-1) \notin \tilde T_2$ and so only the first possibility is realised.
\end{proof}
    
We will refer to the decomposition $g = \mathbf{\tilde{Q}}(g) \, \mathbf{\tilde{R}}(g)$ of an element $g$ of $\GL$ as $\widetilde{QR}$-decomposition.
    
\medskip
   
Upper triangular matrices preserve the standard flag in $\mathbb{R}^2$. The $QR$-decomposition can thus be used to study how a given linear map acts on these subspaces. 
	For example, $g \in GL_2^+$ lies in $T_2$ if and only if it preserves the subspace $\Cb e_1$, and in this case $\mathbf{R}(g)=g$ and consequently $\mathbf{Q}(g)=\mathds{1}$.
Later we need to look at rotated bases and thus need a rotated $QR$-decomposition.
For $\alpha \in SO_2$ define the maps
    \begin{align}
            \mathbf{Q}_\alpha  &: GL_2^+ \to SO_2 \ , &
            \mathbf{\tilde{Q}}_\alpha  &: \GL_2 \to \delta(SO_2) \ , 
            \\
            \mathbf{R}_\alpha  &: GL_2^+ \to T_2\ , &
            \mathbf{\tilde{R}}_\alpha  &: \GL_2 \to \tilde T_2
            \nonumber
    \end{align}
    as, for $g\in GL_2^+$ and $\tilde{g} \in \GL_2$,
    \begin{align} \label{eq:QR_alpha-def}
            \mathbf{Q}_\alpha (g) &= \alpha \mathbf{Q}(\alpha^{-1}g\alpha )\alpha^{-1} 
            	= \mathbf{Q}(\alpha^{-1}g\alpha )
            \ , &
            \mathbf{R}_\alpha (g) &= \alpha \mathbf{R}(\alpha^{-1}g\alpha )\alpha^{-1} 
            \ ,\\
            \mathbf{\tilde{Q}}_\alpha (\tilde{g}) &= \tilde{\alpha}\mathbf{\tilde Q}(\tilde{\alpha}^{-1}\tilde{g}\tilde{\alpha})\tilde{\alpha}^{-1} 
            	= \mathbf{\tilde Q}(\tilde{\alpha}^{-1}\tilde{g}\tilde{\alpha})
            \ , &
            \mathbf{\tilde{R}}_\alpha (\tilde{g}) &= \tilde{\alpha}\mathbf{\tilde R}(\tilde{\alpha}^{-1}\tilde{g}\tilde{\alpha})\tilde{\alpha}^{-1} \ .
            \nonumber
    \end{align}
    Here $\tilde{\alpha}\in \GL_2$ is a lift of $\alpha $, and since $\delta (-1)$ is in the centre of $\GL_2$, the definition does not depend on the choice of $\tilde{\alpha}$.

Clearly, we still have $g = \mathbf{Q}_\alpha(g) \mathbf{R}_\alpha(g)$ and $\tilde g = \mathbf{\tilde Q}_\alpha(\tilde g) \mathbf{\tilde R}_\alpha(\tilde g)$. Furthermore, if $g$ leaves the subspace $\Cb \alpha e_1$ invariant, then $\mathbf{R}_\alpha(g) = g$ (since $\alpha^{-1} g \alpha$ leaves $\Cb e_1$ invariant) and hence also $\mathbf{Q}_\alpha(g) = \mathds{1}$.

\begin{lemma} 
    \begin{enumerate}
        \item Let $g, h\in GL_2^+$ and $\alpha , \beta \in SO_2$. Then 
            \begin{align}
                \mathbf{Q}_\alpha (\beta g) &= \beta \, \mathbf{Q}_\alpha (g),\\
                \mathbf{Q}_\alpha (g) &= \mathbf{Q}_\alpha (g \, \mathbf{Q}_\alpha (h)^{-1}h) = \mathbf{Q}_\alpha (g \, h\mathbf{Q}_\alpha (h^{-1})).
                \nonumber
            \end{align}
        \item Let $g, h\in \GL_2$ and $\alpha , \beta \in SO_2$. Then
            \begin{align} \label{eq:lem.qr.5}
                \mathbf{\tilde{Q}}_\alpha (\delta(\beta) g) &= \delta(\beta)\mathbf{\tilde{Q}}_\alpha (g),\\
                \mathbf{\tilde{Q}}_\alpha (g) &= \mathbf{\tilde{Q}}_\alpha (g \, \mathbf{\tilde{Q}}_\alpha (h)^{-1}h) = \mathbf{\tilde{Q}}_\alpha (g \, h\mathbf{\tilde{Q}}_\alpha (h^{-1})) \ .
                \nonumber
            \end{align}
    \end{enumerate}
\label{lem:QR-rules}
\end{lemma}

    \begin{proof}
It suffices to show part 2. Part 1 then follows by applying $p_{GL}$. 
Furthermore, the case for general $\alpha$ follows straightforwardly once we verified the claims for $\alpha = 1$. Let thus $g, h\in \GL_2$ for
  $\alpha=1$.

For the first equality in \eqref{eq:lem.qr.5}, compose $g = \mathbf{\tilde{Q}}(g)\,\mathbf{\tilde R}(g)$ with $\delta(\beta)$ to get $\delta(\beta) g = q r$, with $q = \delta(\beta) \mathbf{\tilde{Q}}(g)$ and $r = \mathbf{\tilde R}(g)$. From the uniqueness of the $\widetilde{QR}$-decomposition in Lemma \ref{lem:QR-tilde}, it follows that $\mathbf{\tilde{Q}}(\delta(\beta) g) =  \delta(\beta) \mathbf{\tilde{Q}}(g)$.
For the second equality, start with
        \begin{equation}
            \mathbf{\tilde{R}}(h) = \mathbf{\tilde{Q}}(h)^{-1}h \ .
            \label{eq:lem.qr.proof.R}
        \end{equation}
Multiplying both sides with $g$ gives
$
\mathbf{\tilde{Q}}(g) \mathbf{\tilde{R}}(g) \mathbf{\tilde{R}}(h)
= 
g\mathbf{\tilde{Q}}(h)^{-1}h 
$.
        Since $\tilde{T}_2$ is a subgroup and by uniqueness of the $\widetilde{QR}$-decomposition, the second equality in \eqref{eq:lem.qr.5} follows.
        Now invert equation \eqref{eq:lem.qr.proof.R} and replace $h$ with its inverse:
        \begin{equation}
            \mathbf{\tilde{R}}(h^{-1})^{-1} = h\,\mathbf{\tilde{Q}}(h^{-1}) \ .
            \label{eq:lem.qr.proof.Rinv}
        \end{equation}
        Multiplying by $g$ gives
$
\mathbf{\tilde{Q}}(g) \mathbf{\tilde{R}}(g) \mathbf{\tilde{R}}(h^{-1})^{-1}
= 
gh \mathbf{\tilde{Q}}(h^{-1})
$.
For the same reason as above, this shows the third equality in \eqref{eq:lem.qr.5}.
    \end{proof}

\subsection{Example: Two spin structures on $\Cb^\times$}
\label{sec:ss_on_C}
By Proposition \ref{prop:class.spinstructures}, there are two  isomorphism classes of spin structures on $\Cb^\times$. Two spin surfaces $\Cb^{NS}$ and $\Cb^R$ which represent these classes can be described as follows.

We let $\Cb^{NS}$ be the spin surface with the trivial spin structure on $\Cb^\times$. 
As a spin bundle it is given by the trivial principal bundle 
	$P_\GL(\Cb^{NS})=\Cb^\times \times \GL_2$. 
The right action of $\GL_2$ is given by right multiplication on the second component. The projection to the frame bundle is
\begin{equation}
    \begin{split}
    p^{NS}:P_\GL(\Cb^{NS}) &\to \Cb^\times \times GL_2^+\\
    (z,g)&\mapsto (z,p_{GL}(g))\ .
    \end{split}
    \label{eq:p_NS-def}
\end{equation}
The correspondence between oriented frames and elements of $GL_2^+$ is by taking the two basis vectors as the two column vectors of the $2{\times}2$-matrix.
For $\Cb^R$ we again take the manifold $P_\GL(\Cb^{R}) = \Cb^\times \times \GL_2$ with $\GL_2$-action by right multiplication. The difference to $\Cb^{NS}$ lies in the projection to the frame bundle, which for $\Cb^R$ is
\begin{equation}
    \begin{split}
        p^{R} : P_\GL(\Cb^{R}) &\to \Cb^\times \times GL^+_2 \\
         (z,g) &\mapsto \big(z, z\, p_{GL}(g)\big) \ .
    \end{split}
    \label{eq:SC-proj}
\end{equation}
(Recall that we do not write out the embedding $i_{\mathbb{C}^\times} : \mathbb{C}^\times \to GL_2^+$ from \eqref{eq:i-delta-C*-def}.)
	One quickly checks that $p^R(z,g)p_{GL}(h) = p^R(z,gh)$.
We have thus defined two spin surfaces $\Cb^{NS}, \Cb^{R}$ with underlying surface $\underline{\Cb^{NS}} = \underline{\Cb^{R}} = \Cb^\times$. 

A simple path based argument shows that these are indeed non-isomorphic. 
We lift the loop $\hat\zeta : [0,2\pi ] \to \Cb^\times \times GL^+_2$, $\hat\zeta(t) = (e^{it}, e^{it})$
along $p^{NS}$ and $p^R$ to obtain 
$\tilde{\zeta}^{NS/R}: [0,2\pi ] \to \Cb^{NS/R}$ with $\tilde{\zeta}^{NS}(t) = (e^{it}, \delta(e^{it/2}))$ and $\tilde{\zeta}^{R}(t) = (e^{it}, \delta (1))$.
We observe that $\tilde{\zeta}^R$ is closed while $\tilde{\zeta}^{NS}$ is not. Thus $\Cb^{NS}$ and $\Cb^{R}$ are different. Of these, $\Cb^{NS}$ extends to the whole of $\Cb$.
We will therefore use the notation $\Cb^{NS}$ for the unpunctured complex plane with the (unique up to isomorphism) spin structure $\Cb^{NS} = \Cb \times \GL_2$.

\subsection{Lifting properties of maps}
\label{sec:map.lifting}
\begin{lemma}
Let $\Sigma,\Sigma'$ be two spin surfaces and $f: \underline\Sigma \to \underline\Sigma'$ a map between the underlying surfaces. Suppose that $\underline\Sigma$ is contractible. Then there exist precisely two maps $\tilde f_{1,2} : \Sigma \to \Sigma'$ of spin surfaces with underlying map $f$; these are related by $\tilde{f}_1 = \tilde{f}_2 \circ \omega$, where $\omega$ is the leaf exchange automorphism.
    \label{lem:Smap.lift}
\end{lemma}
    \begin{proof}
By Remark \ref{rem:map=iso}, to give a map $\tilde f : \Sigma \to \Sigma'$ is equivalent to giving an isomorphism $P_\GL(\Sigma) \to f^* P_\GL(\Sigma')$ of spin structures. Such an isomorphism exists, since $\underline\Sigma$ is contractible and so by Proposition \ref{prop:class.spinstructures} there is only one isomorphism class of spin structures on $\underline\Sigma$. Finally, any two such lifts are either equal or related by $\omega$ since $\underline\Sigma$ is in particular connected.
    \end{proof}

\begin{lemma}
    Let $\tilde{f}:\Sigma \to \Sigma'$ be a morphism of spin surfaces with underlying map $f$. Let $H:[0,1]\times \underline\Sigma \to \underline\Sigma'$ be a smooth homotopy, i.e.\ $H$ is continuous and $H_t$ is smooth for all $t\in [0,1]$. Assume $H_0=f$. Then there is a unique lift $\tilde{H}:[0,1]\times \Sigma \to \Sigma'$ such that $\tilde{H}_0=\tilde{f}$ and such that $\tilde H_t$ is a map of spin surfaces for each $t$.
    \label{lem:Smap.lift2}
\end{lemma}

    \begin{proof}
        Taking derivatives of $H$ at fixed times $t$, we obtain a lift of $H$ to the bundle of oriented frames. The result then follows from the homotopy lifting property of $p:P_\GL(\Sigma') \to F_{GL^+}(\Sigma')$.
    \end{proof}
    
In Section \ref{sec:ss_on_C} we saw that a spin structure on $\mathbb{C}^\times$ can be extended to $\mathbb{C}$ iff the path $\hat\zeta$ does not have a closed lift in the spin bundle. We now extend this argument to spin structures on arbitrary surfaces.
Let $\underline{\Sigma}$ be a surface. We denote by $\pi_T:F_{GL^+}\left( \underline{\Sigma} \right)\to T\underline{\Sigma}$ the projection that picks the first vector of a frame. 
A {\em (smooth) simple closed curve} is a closed path, that is a smooth embedding when considered as a map of $S^1$ into the surface.
Such a curve $\gamma:[0,1]\to \underline{\Sigma}$ induces a closed curve $d\gamma:[0,1] \to T\underline{\Sigma}$ by taking the derivative.
The curve $d\gamma$ always lifts along $\pi_T$ by completing the frame. (The derivative $d\gamma$ is non-zero everywhere by definition since $\gamma$ is an embedding.) Any two such lifts will be homotopic, as they only differ by right multiplication with a curve in $T_2$.

\begin{lemma}
    Let $\underline{\Sigma}$ be a surface
	and 
$v\in\underline{\Sigma}$, together with a spin structure on $\underline{\Sigma}\setminus\{v\}$.
    Let $\gamma:[0,1]\to \underline{\Sigma}$ be a contractible smooth simple closed curve encircling $v$ (i.e. on $\underline{\Sigma}\setminus\{v\}$ it is no longer contractible). Then for any $\hat{\gamma}:[0,1]\to F_{GL^+}\left( \underline{\Sigma} \right)$ with $\pi_T( \hat{\gamma} )=d\gamma$ the following are equivalent:
    \begin{enumerate}
        \item The spin structure on $\underline{\Sigma}\setminus \{v\}$ extends to $\underline{\Sigma}$.
        \item $\hat{\gamma}$ does not have a spin lift.
    \end{enumerate}
    \label{lem:ssc}
\end{lemma}
\begin{proof}
    By the Jordan-Sch\"onflies theorem the image of the curve $\gamma$ bounds a disk. We can therefore find a chart $\psi:U\to \Sigma$, $U\subset\mathbb{C}$ open, in which $\text{Im}(\gamma)$ bounds the unit disk and $\psi^{-1}(v)=0$. Then $\psi^{-1}\circ \gamma$ is isotopic to either $\zeta:[0,2\pi] \to \mathbb{C}$, $t \mapsto e^{it}$, or its reverse, which has the same lifting properties. We can then assume w.l.o.g. that $\hat{\gamma}$ is homotopic to $\hat{\zeta}$ from Section \ref{sec:ss_on_C} which lifts to the spin bundle iff the spin structure does not extend to $\mathbb{C}$.
\end{proof}

\subsection{Surfaces with parametrised boundary} 

We first define a set of collars around 
	$S^1 \subset \Cb$:
\begin{align}     \label{eq:S1Nbhd}
    \mathcal{A} &:= \left\{ A_{r,R} \subset \Cb : A_{r,R} = \{z\in \Cb: 
    r<|z|<R \}; 
	0<    
    r<1<R \right\} \ ,\\
    \mathcal{A}_{\geq 1} &:= \left\{ A_{r,R} \cap  \{z\in \Cb: \lvert z \rvert \geq 1 \} : A_{r,R} \in \mathcal{A} \right\} \ .
    \nonumber
\end{align}

\begin{definition}    \label{def:surf.parameterised.boundary}
    A \emph{surface with parametrised boundary} is a compact surface $\Sigma$ together with smooth 
    	orientation preserving
embeddings $\varphi_i: U_i \to \Sigma$, $i=1,\dots,B$, where $U_i \in \mathcal{A}_{\geq 1}$
	 and $B$ is the number of connected components of the boundary $\partial\Sigma$ of $\Sigma$. We require that $\bigcup_{i=1}^B \varphi_i(\partial U_i) = \partial \Sigma$ and that the images $\varphi_i(U_i)$, $i=1,\dots,B$ are pairwise disjoint. A \emph{diffeomorphism between surfaces with parametrised boundary} is a diffeomorphism between the surfaces compatible with the germs of the boundary embeddings. 
\end{definition}	
	
Such a parametrisation in particular induces a linear order on the boundary components which will be used later. 
Unless otherwise indicated, in the following ``surface'' will stand for ``surface with parametrised boundary''.

	Boundary components of surfaces can be glued using the parametrisation.
To do this in a unique way we fix the gluing diffeomorphism
\begin{equation}
\label{eq:Cx-inversion}
    s:\Cb^\times \to \Cb^\times \qquad , \quad
    z\mapsto z^{-1} \ .
\end{equation}

\begin{definition}[Glueing of parametrised surfaces]
Let $(\Sigma,(\varphi_i)_i)$ be a surface and $(i,j)$, $i\neq j$ be a pair of boundary components. The glued surface $\Sigma_{i\#j}$ is obtained by identifying points via the diffeomorphism $\varphi_i \circ s \circ \varphi_j^{-1}|_{\partial \Sigma_j}:\partial \Sigma_j \to \partial \Sigma_i$. 
After possibly restricting the maps $\varphi_i$ and $\varphi_j$ to smaller collars, also denoted by $U_i$ and $U_j$, we get an embedding $\varphi_{i,j}:\mathcal{A}\ni U_i\cup s(U_j) \to \Sigma_{i\#j}$  given by
\begin{equation}
    \varphi_{i,j}(z) = 
    \begin{cases}
        &\varphi_i(z),\; \text{if} \; \lvert z\rvert \geq 1,\\
        &\varphi_j(s^{-1}(z)),\; \text{if} \; \lvert z \rvert < 1.
    \end{cases}
\end{equation}
The differentiable structure on the glued surface $\Sigma_{i\#j}$ is the one compatible with the differentiable structure on $\Sigma\backslash(\partial \Sigma_i\cup \partial \Sigma_j)$ and the differentiable structure induced by $\varphi_{i,j}$ on its image.
\label{def:surface.gluing}
\end{definition}

The definition of the glued surface is symmetric, $\Sigma_{i\#j} = \Sigma_{j\#i}$.

\subsection{Spin-surfaces with parametrised boundary}

As above we first define sets $\mathcal{A}^{NS}$ and $\mathcal{A}^{R}$ of spin collars for the $NS$ and $R$ spin structures on the annulus,
\begin{equation}
    \mathcal{A}^{NS/R} := \big\{ \, \Cb^{NS/R}|_{\underline U} \,\big|\, \underline{U} \in \mathcal{A}\, \big\}
    \quad , \qquad
    \mathcal{A}^{NS/R}_{\geq 1} :=\big\{ \, \Cb^{NS/R}|_{\underline U} \,\big|\, \underline{U} \in  \mathcal{A}_{\geq 1} \big\} \ .
\end{equation}

\begin{definition}    \label{def:ssurf.param.boundary}
    A \emph{spin surface with parametrised boundary} is a compact spin surface $\Sigma$ together with a collection $(\tilde{\varphi}_i)_{i=1,\dots,B}$, of spin embeddings
\begin{equation}
\tilde{\varphi}_i : U_i \to \Sigma
\end{equation}
with $U_i \in \mathcal{A}^{NS}_{\geq 1}\sqcup \mathcal{A}^{R}_{\geq 1}$ (disjoint union), and such that the tuple $(\underline{\Sigma},(\varphi_i)_{i=1,\dots,B})$ of underlying surface and parametrisation is a surface (with parametrised boundary).
We call a boundary component $i$ of NS-type if $U_i\in \mathcal{A}^{NS}$ and of R-type  if $U_i\in \mathcal{A}^{R}$.
\end{definition}

Spin structures on manifolds with boundary are defined in \cite{blanchet1996topological} from a homotopy-theoretic viewpoint; in \cite{randal2010homology} $r$-spin surfaces with boundaries and their mapping class groups are treated.

As for surfaces, in the following we will write ``spin surface'' for ``spin surface with parametrised boundary'' unless stated otherwise.

\medskip

By taking its derivative the diffeomorphism $s$ from \eqref{eq:Cx-inversion} induces a map 
	$ds_*:\Cb^\times \times GL^+_2 \to \Cb^\times \times GL^+_2$,
\begin{equation}
    ds_* : (z,g) \mapsto \left( \frac{1}{z}, 
	\frac{-1}{z^{2}} \, g 
    \right).
\end{equation}
The map $ds_*$ has two lifts $\tilde{s}^{\pm}_{NS/R}: \Cb^{NS/R}\to \Cb^{NS/R}$, which we specify via 
\begin{equation}
\tilde s^\varepsilon_{NS}(1,\delta (1))=(1,\delta(-i \varepsilon)) \quad , \qquad
\tilde s^\varepsilon_{R}(1,\delta (1))=(1,\delta(i \varepsilon)) \ .
\label{eq:spin.gluing.map}
\end{equation}
The reason to choose different looking conventions for $\tilde s_{NS}$ and $\tilde s_R$ is a simplification in Lemma \ref{lem:index.gluing} below.
We will omit the label $R$ or $NS$ if it is clear from the context which of the maps is used.

\begin{definition}
Let $(\Sigma,(\tilde{\varphi}_i)_i)$ be a spin surface. We call a triple $(i,j,\varepsilon )$, with $i,j$ distinct boundary components and $\varepsilon \in \{\pm\}$, {\em spin gluing data}. Here the boundaries $i$ and $j$ have to be either both of NS-type, or both of R-type. 
We define the {\em glued spin surface}
	$\Sigma_{i\#j}^\varepsilon$ 
by identifying points along the boundary via the homeomorphism $\tilde{\varphi}_i \circ \tilde s^\varepsilon  \circ \tilde{\varphi}_j^{-1}\vert_{\partial \Sigma_j}$, analogous to Definition \ref{def:surface.gluing}, and use the maps $\tilde{\varphi}_i$ and $\tilde{\varphi}_j\circ (\tilde s^{\varepsilon})^{-1}$ to define the differential structure and the spin structure. The bundle projection and right action commute with the gluing maps, and thus are defined on $\Sigma^\varepsilon_{i\#j}$ in the obvious way.
\label{def:glue.spin}
\end{definition}

	Since $(\tilde s^{\varepsilon})^{-1} = \tilde s^{-\varepsilon}$, the gluing operation is not symmetric, but instead satisfies 
\begin{equation}
	\Sigma_{i\#j}^\varepsilon = \Sigma_{j\#i}^{-\varepsilon} \ .
\end{equation}

\section{Spin Triangulations}\label{sec:spintriangulations}

\subsection{Smooth triangulations with boundary}\label{sec:smooth-triang}
\begin{figure}[tb]
    \begin{tikzpicture}
        \node[name=A, regular polygon, regular polygon sides=3, minimum size=4cm, rotate=30] at (0,0){};
        \begin{scope}[decoration={
            markings,
            mark=at position 0.5 with {\arrow{stealth}}}
        ] 
        \foreach \N/\M in {1/2,2/3,3/1} \draw[postaction=decorate] (A.corner \N) -- (A.corner \M);
        \end{scope}
        \draw[->] (30:0.5cm) arc (30:320:0.5cm);
    \end{tikzpicture}
\caption{Orientation convention for a triangle in a combinatorial surface. The circular arrow gives the order on the vertices defining the orientation. The arrows on the edges give the edge orientation induced by the orientation of the triangle. When comparing orientation of simplices to orientations of surfaces, our convention is that the above orientation matches that of the paper plane, thought of as $\mathbb{R}^2$ with its standard orientation.}
\label{fig:triangle-orientation}
\end{figure}
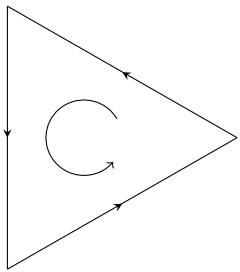

Below, we will make use of combinatorial surfaces and smooth triangulations.
Let us mention that a combinatorial surface is a simplicial complex $\mathcal{C}$ such that the polytope $|\mathcal{C}|$ is homeomorphic to a two-manifold, and that a smooth triangulation of a surface $\Sigma$ is a homeomorphism $|\mathcal{C}| \to \Sigma$ which is a smooth embedding when restricted to a simplex. We take our combinatorial surfaces to be oriented (see \cite[Sect.\,5]{lee2000} for a definition) and our orientation convention is given in Figure \ref{fig:triangle-orientation}.

\medskip

The standard triangle $\underline\Delta$ is the convex hull of the vertices $\{ 1, e^{\frac{2\pi i}{3}}  , e^{\frac{4\pi i}{3}}\} \subset \Cb$. We consider it as a simplicial complex with the usual simplices.

\begin{definition}\label{def:comb-spin-surf}
A \emph{combinatorial surface with parametrised boundary} (or {\em combinatorial surface} for short) is a combinatorial surface $\mathcal{C}$ together with simplicial
	embeddings
$f_i:\partial \underline\Delta \to \partial \mathcal{C}$, where $i$ runs from $1$ to the number of boundary components, such that the boundary $\partial\mathcal{C}$ is the disjoint union of all $f_i(\partial \underline\Delta)$.
	The $f_i$ have to be orientation reversing in the sense that the induced orientation on an edge of $\underline\Delta$ is mapped to the opposite orientation of the boundary edge in $\partial\mathcal{C}$ as induced by the adjacent triangle.
\end{definition}
	Via this definition we impose in particular that each boundary component of $\mathcal{C}$ consists of precisely three edges and three vertices. 
	The orientation convention is such that when using $f_i$ to glue the triangle into $\mathcal{C}$ one obtains an oriented simplicial complex.
	
\medskip

Triangulated surfaces can be glued. To formulate the gluing procedure, we need the map $s_{\mathcal{C}}:\partial \underline\Delta\to \partial \underline\Delta$,
	$z \mapsto \bar z$. It acts on vertices as
\begin{equation}
        s_{\mathcal{C}}:\partial \underline\Delta \to \partial \underline\Delta ~~,~~~
            1 \mapsto 1 ~,~~
            e^{\frac{2\pi i}{3}} \mapsto e^{\frac{4\pi i}{3}} ~,~~
            e^{\frac{4\pi i}{3}} \mapsto e^{\frac{2\pi i}{3}} \ .
    \label{eq:combinatorial.gluing.map}
\end{equation}
Let $i\neq j$ label two boundary components of a combinatorial surface $\mathcal{C}$. We can glue the surface as an abstract simplicial complex by identifying simplices along the map 
	$f_i\circ s_{\mathcal{C}} \circ f_j^{-1}$. 
If we obtain a simplicial complex this way,\footnote{
A simple example to illustrate the necessity of this condition is as follows: take $\mathcal{C}$ to be the disjoint union of two standard triangles $\underline\Delta$ with boundaries parametrised by $s_{\mathcal{C}}$. The two boundary components cannot be glued since the result would not be a simplicial complex, i.e.\ $(1,2)$ is not simplicial gluing data.

However, for a given combinatorial surface and arbitrary gluing data $(i,j)$, it is always possible to choose a subdivision of $\mathcal{C}$, fixing the boundary triangulation, such that $(i,j)$ becomes simplicial gluing data.

} 
we call $(i,j)$ {\em simplicial gluing data}. The resulting simplicial complex is denoted as $\mathcal{C}_{i\#j}$ and it is again a combinatorial surface.

\medskip

To triangulate surfaces with parametrised boundary we first define a canonical triangulation of the unit circle $S^1$, 
    \begin{equation}
    \varphi_S:
    \lvert\partial \underline\Delta\rvert \to S^1 \quad , \quad
        z \mapsto \frac{z}{\lvert z \rvert} \ .
    \end{equation}

\begin{definition} \label{def:triang-surf}
    A \emph{triangulated surface with parametrised boundary} (or {\em triangulated surface} for short) is a tuple $((\mathcal{C},f_i),\varphi ,(\Sigma ,\varphi_i))$, where $(\mathcal{C},f_i)$ is a combinatorial surface, $(\Sigma ,\varphi_i)$ is a surface, and 
$\varphi : |\mathcal{C}| \to \Sigma$ 
is a triangulation such that $\varphi \circ f_i = \varphi_i \circ \varphi_S$.
\end{definition}

Since $s_{\mathcal{C}}(z) = \bar{z}$, the diagram
   \begin{equation} \label{eq:sC-s-diagram}
   \begin{tikzcd}
       \lvert \partial \underline\Delta\rvert  \ar{r}{\varphi_S} \ar{d}{\lvert s_{\mathcal{C}}\rvert} & S^1 \ar{d}{s} \\
       \lvert \partial \underline\Delta\rvert \ar{r}{\varphi_S} & S^1
   \end{tikzcd}    
   \end{equation}
   commutes. This allows us to make the following

\begin{definition} \label{def:glue-triang-surf}
    Let $\Sigma=((\mathcal{C},f_i),\varphi ,(\Sigma ,\varphi_i))$ be a triangulated surface and $(i,j)$ be simplicial gluing data. The {\em glued triangulated surface} is
\begin{equation}
	\Sigma_{i\#j}:=\big( (\mathcal{C}_{i\#j},\hat{f}_k),\varphi_{i\#j},(\Sigma_{i\#j},\hat{\varphi}_k) \big) \ , 
\end{equation}
with $\hat{f}_k$, $\hat{\varphi}_k$ being the remaining boundary parametrisations and $\varphi_{i\#j}$ the quotient of the original triangulating map $\varphi$.
\end{definition}

\subsection{Markings on combinatorial surfaces}\label{sec:marking-comb-surf}

The combinatorial description of spin surfaces requires some extra data. 
Let $\mathcal{C}$ be a combinatorial surface. The first piece of data is an orientation on the edges of $\mathcal{C}$. We encode this by choosing for each edge $e\in \mathcal{C}_1$ a vertex $d_0^1(e)$ on the boundary of $e$. This determines a second map $e \mapsto d_1^1(e)$ by picking the other boundary vertex at each edge. We think of an edge as being oriented from $d_0^1(e)$ to $d_1^1(e)$, see Figure \ref{fig:label_convention}. 

\begin{figure}[tb]
    \centering
        \begin{tikzpicture}
            \draw[-] (0cm,0cm) -- +(2cm,0cm);
            \draw[->,>=stealth] (0cm,0cm) -- +(1cm,0cm);
            \node[above] at (1cm,0cm) {$e$};
            \node[left] at (0,0) {$d_0^1(e)$};
            \node[right] at (2cm,0cm) {$d_1^1(e)$};
        \end{tikzpicture}
        \caption{An edge $e$ with orientation from $d_0^1(e)$ to $d_1^1(e)$.}
    \label{fig:label_convention}
\end{figure}

The second piece of data is a ``starting edge'' for each triangle in $\mathcal{C}$, that is, for each $\sigma \in \mathcal{C}_2$ we choose an edge $d_0^2(\sigma )$ of $\sigma$. This induces two further maps $\sigma \mapsto d_1^2(\sigma)$ and $\sigma \mapsto d^2_2(\sigma)$ by choosing the next and next-to-next edge counterclockwise. For the standard triangle $\underline\Delta$ this is illustrated in Figure \ref{fig:st.triangle}, which also gives our numbering convention for the edges of $\underline\Delta$.

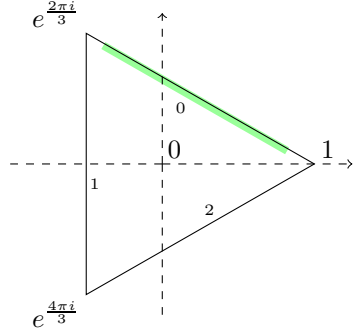
\begin{figure}[tb]
    \begin{tikzpicture}[]
        \node[name=A, regular polygon, regular polygon sides=3, minimum size=4cm, rotate=30]{};
        \draw[green!40, line width=3pt] ($(A.corner 3)!0.9!(A.corner 1)$)++(-2pt,0pt) -- ($($(A.corner 1)!0.9!(A.corner 3)$)+(-2pt,0pt)$);
        \node[name=A, regular polygon, regular polygon sides=3, draw,minimum size=4cm, rotate=30]{};
        \node[above left=-2pt] at (A.corner 1) {$e^{\frac{2\pi i}{3}}$};
        \node[below left=-2pt] at (A.corner 2) {$e^{\frac{4\pi i}{3}}$};
        \node[above right=-2pt] at (A.corner 3) {$1$};
        \node[above right=-2pt] at (A.center) {$0$};
        \node[above right=2pt and -2pt] at (A.side 2) {\tiny{$2$}};
        \node[below left=-2pt and 2pt] at (A.side 3) {\tiny{$0$}};
        \node[below right=2pt and -2pt] at (A.side 1) {\tiny{$1$}};
        \draw[->, dashed] (A.center) -- +(2.5cm,0cm);
        \draw[->, dashed] (A.center) -- +(0cm,2cm);
        \draw[dashed] (A.center) -- +(-2cm,0cm);
        \draw[dashed] (A.center) -- +(0cm,-2cm);
    \end{tikzpicture}
    \caption{Standard triangle $\underline\Delta$; the small numbers $0$, $1$, $2$ indicate the numbering of the edges. The first edge has also been marked with a fat green line.}
        \label{fig:st.triangle}
\end{figure}
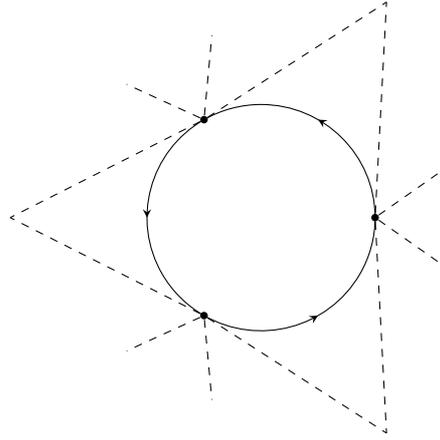
\begin{figure}[tb]
 \begin{tikzpicture}
        \foreach \N in {0,1,2} {
            \node[name=c\N, fill, inner sep=1pt, circle] at (\N*120:1.5cm){};
            \coordinate (t\N) at ($(\N*120+60:3.3cm)$);
            \coordinate (a\N) at ($(\N*120+15:2.5cm)$);
            \coordinate (b\N) at ($(\N*120-15:2.5cm)$);
            }
        \foreach \N/\M in {0/2,1/0,2/1} {
            \draw[dashed] (c\N) -- (t\N);
            \draw[dashed] (c\N) -- (t\M);
            \draw[dashed] (c\N) -- (a\N);
            \draw[dashed] (c\N) -- (b\N);
            }
        \begin{scope}[decoration={
            markings,
            mark=at position 0.5 with {\arrow{stealth}}}
        ] 
        \draw[postaction=decorate] (0:1.5cm) arc (0:120:1.5cm);
        \draw[postaction=decorate] (120:1.5cm) arc (120:240:1.5cm);
        \draw[postaction=decorate] (240:1.5cm) arc (240:360:1.5cm);
        \end{scope}
    \end{tikzpicture}
\caption{Marking of boundary edges: The marking has to be such that the orientation induced on the edge as in Figure \ref{fig:label_convention} agrees with the orientation induced by the parametrisation maps, see Definition \ref{def:comb-spin-surf}. Equivalently, the orientation of boundary edges induced by the marking is opposite to that induced by the adjacent triangle via Figure \ref{fig:triangle-orientation}.}
\label{fig:marking-boundary-edge}
\end{figure}

To summarise:

\begin{definition} \label{def:marked-triang-surf}
    A \emph{marked combinatorial surface with parametrised boundary} (or \emph{marked combinatorial surface} for short) is a combinatorial surface $\mathcal{C}$ together with maps $d^1_0:\mathcal{C}_1 \to \mathcal{C}_0$ and $d^2_0:\mathcal{C}_2 \to \mathcal{C}_1$ such that:
    \begin{itemize}
        \item $d_0^1(e) \in \partial(e)$ for all $e\in \mathcal{C}_1$,
        \item $d_0^2(\sigma ) \in \partial(\sigma )$ for all $\sigma \in \mathcal{C}_2$.
    \end{itemize}
    In addition, for boundary edges $e\in (\partial \mathcal{C})_1$ we require that they are directed 
    in accordance with the boundary orientation as imposed by the parametrisation maps, see Figure \ref{fig:marking-boundary-edge}.
    A \emph{marked triangulated surface (with parametrised boundary)} is a triangulated surface together with a marking on its combinatorial surface.
\end{definition}

\begin{definition}[Glueing of markings] \label{def:glue-mark}
    Let $(\mathcal{C},f_i)$ be a marked combinatorial surface with parametrised boundary, $(i,j)$ simplicial gluing data and $\mathcal{C}_{i\#j}$ the glued surface. For $\sigma \in \mathcal{C}$ we denote by $[\sigma ]\in \mathcal{C}_{i\#j}$ the image of $\sigma $ under the quotient map.
    The marking on $\mathcal{C}_{i\#j}$ is defined as follows:
    \begin{itemize}
        \item For $e\in \mathcal{C}_1\setminus \left( \mathrm{im}(f_i) \cup \mathrm{im}(f_j) \right)$: $d_0^1([e])=[d_0^1(e)]$.
        \item For $e\in (\mathrm{im}(f_j))_1$: $d_0^1([e])=[d_0^1(e)]$.
        \item For $e\in (\mathrm{im}(f_i))_1$: $d_0^1([e])=[d_0^1\left( (f_j\circ s_\mathcal{C}^{-1} \circ f_i^{-1})(e)\right)]$.
        \item For $\sigma \in \mathcal{C}_2$: $d_0^2([\sigma ])=[d_0^2(\sigma )]$.
    \end{itemize}
\end{definition}

In a marked simplicial surface $\mathcal{C}$, for each face $\sigma \in \mathcal{C}_2$ there is a unique orientation preserving affine linear isomorphism 
	$\check\chi_\sigma : \underline\Delta  \to \sigma$ 
which maps the marked edge of the standard triangle to $d^2_0(\sigma )$. Consequently, in a marked triangulated surface $(\mathcal{C},\varphi ,\Sigma)$ there is a canonical smooth embedding $\chi_\sigma : \underline\Delta   \to \Sigma$, $\chi_\sigma := \varphi \circ \check\chi_\sigma $ for each $\sigma \in \mathcal{C}_2$.

\medskip

Recall from Section \ref{sec:ss_on_C} that $\Cb^{NS}$ is the (unpunctured) complex plane with spin structure $\Cb^{NS} = \Cb \times \GL_2$. We define $\Delta$ to be the triangle $\underline\Delta$ with spin structure $\Cb^{NS}|_{\underline\Delta}$. The spin structure on $\underline\Delta$ is unique up to isomorphism.
We can now introduce another concept central for this paper.

\begin{definition}\label{def:spin-triang-surf-def}
    A \emph{spin triangulated surface (with parametrised boundary)} $\Sigma$ is a spin surface $\Sigma$ and a marked triangulated surface $((\mathcal{C},f_i),\varphi ,(\underline{\Sigma},\varphi_i))$ together with a choice of spin lift $\tilde{\chi}_\sigma :\Delta  \to \Sigma$ of the map 
	$\chi_\sigma : \underline\Delta \to \underline\Sigma$ 
for every face $\sigma \in \mathcal{C}_1$. 
\end{definition}

    Since simplices are connected and simply connected, a spin lift of $\chi_\sigma $ always exists and is uniquely determined by giving its value at one point. For every face of the triangulation there are two possible choices for the spin lift of the characteristic map.

\subsection{Edge signs for inner edges}
\label{sec:index_inner}

Given a spin triangulated surface, our next aim is to give a combinatorial description of the spin structure. This will be achieved by assigning signs to the edges of the triangulation. The definition of these signs and the description of their behaviour under changes of the triangulation and under gluing will be the main input into the algebraic treatment of lattice spin topological field theory in Section \ref{sec:2dlatticeTFT}.

\medskip

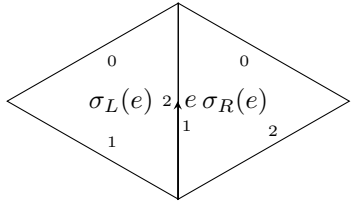
\begin{figure}[tb]
    \begin{tikzpicture}[]
        \node[name=A, regular polygon, regular polygon sides=3, draw,minimum size=3cm, rotate=30] at (0.75cm,0cm){};
        \node[name=B, regular polygon, regular polygon sides=3, draw,minimum size=3cm, rotate=90] at (-0.75cm,0cm){};
        \draw[->, >=stealth] (A.corner 2) -- (A.side 1);
        \node[] at (A.center) {$\sigma_R(e)$};
        \node[] at (B.center) {$\sigma_L(e)$};
        \node[above right=2pt and -2pt] at (A.side 2) {\tiny{$2$}};
        \node[below left=-2pt and 2pt] at (A.side 3) {\tiny{$0$}};
        \node[below right=4pt and -2pt] at (A.side 1) {\tiny{$1$}};
        \node[left=-1pt] at (B.side 2) {\tiny{$2$}};
        \node[below right=-2pt and 2pt] at (B.side 3) {\tiny{$0$}};
        \node[above right=-2pt and 2pt] at (B.side 1) {\tiny{$1$}};
        \node[right=-1pt] at (A.side 1) {$e$};
    \end{tikzpicture}
    \caption{An inner edge with left and right adjacent faces. In the configuration above $k_L=2$ and $k_R=1$.}
        \label{fig:edge_indexing_1}
\end{figure}

    Let $\mathcal{C}$ be a marked combinatorial surface. 
Recall from Figure \ref{fig:label_convention}  that the boundary maps $d^1_0, d^1_1$ give each edge $e\in \mathcal{C}_1$ a (1-)orientation. For any inner edge $e$, we denote by $\sigma_L(e)$ the adjacent face that induces this orientation on $e$, and by $\sigma_R(e)$ the face that induces the opposite orientation, see Figure \ref{fig:edge_indexing_1}. 
Furthermore let $k_L(e)$ and $k_R(e)$ be such that $d^2_{k_L(e)}(\sigma_L(e))=e=d^2_{k_R(e)}(\sigma_R(e))$. We say $e$ is the $k_L(e)$'th edge of $\sigma_L(e)$ and the $k_R(e)$'th edge of $\sigma_R(e)$.
If the edge $e$ is clear from the context we will often drop the argument in $\sigma_{L/R}$ and $k_{L/R}$.

Now fix an edge $e$ in a marked triangulated surface $(\mathcal{C},\varphi ,\Sigma)$. 
Let $p$ be a point on the $k_R(e)$'th edge of $\Delta $. 
Then the derivative $d(\chi_{\sigma_L}^{-1}\circ \chi_{\sigma_R})_p\in GL_2^+$ rotates a tangent vector in the direction of the edge by $e^{2\pi i(k_L/3-k_R/3+1/2)}$. This can be written in terms of the $QR$-decomposition of $GL_2^+$ as 
\begin{equation} \label{eq:Qalpha-angle}
    \mathbf{Q}_\alpha (  d(\chi_{\sigma_L}^{-1}\circ \chi_{\sigma_R})_p ) = e^{2\pi i(k_L/3-k_R/3+1/2)}
    \quad , ~~ \text{ where } ~~ \alpha =e^{2\pi i\left( \frac{k_R}{3}+ \frac{5}{12} \right)} \ .
\end{equation}
The constant $\frac{2 \pi \, 5}{12} = 150^\circ$ is the angle the edge labelled 0 in the standard triangle forms with the real axis (Figure \ref{fig:st.triangle}). Thus $\alpha$ is the angle between the edge labelled $k_R$ and the real axis. To avoid having to write out the uninteresting constant angle, we abbreviate
\begin{equation}\label{eq:alpha_0-def}	
	\alpha_0:= e^{2\pi i\,\frac{5}{12}} \ ,
\end{equation}
such that $\alpha =e^{ \frac{2\pi ik_R}{3}} \alpha_0$.

For a spin map $\tilde{f}:\Cb^{NS} \to \Cb^{NS}$ and a point $p \in \Cb$ we denote by $g_p(\tilde{f})\in \GL_2$ the element such that 
\begin{equation} \label{eq:g_p-def}
	\tilde{f}(p,g)=(f(p),g_p(\tilde{f})\cdot g) \ .
\end{equation}

\begin{definition}[Edge signs for inner edges]
Let $e$ be an inner edge of a spin triangulated surface and let $p$ be a point on the $k_R(e)$'th edge of $\Delta$.
The {\em edge sign} $s(e) \in \{\pm 1\}$ for the edge $e$ is defined via
    \begin{equation}\label{eq:se-inner_edge}
        \delta(s(e)) = \tilde{\mathbf{Q}}_{\alpha}\left( g_p(\tilde{\chi}_{\sigma_L}^{-1} \circ \tilde{\chi}_{\sigma_R})\right)\cdot\delta (e^{-\pi i(k_L/3-k_R/3+1/2)}) \ ,
    \end{equation}
where  $\alpha =e^{ \frac{2\pi ik_R}{3}} \alpha_0$.   
\label{def:se-inner_edge}
\end{definition}

By continuity of $\mathbf{\tilde Q}_{\alpha}$, $s(e)$ does not depend on the choice of the point $p$ so that it makes sense not to include $p$ in the notation. Because $\mathbf{\tilde Q}(\dots)$ and $\delta(\dots)$ commute in $\GL_2$, and by employing Lemma \ref{lem:QR-rules}, we may also write
    \begin{equation}\label{eq:se-inner_edge-2}
        \delta(s(e)) = \tilde{\mathbf{Q}}_{\alpha}\left( \delta (e^{-\pi i(k_L/3-k_R/3+1/2)}) \, g_p(\tilde{\chi}_{\sigma_L}^{-1} \circ \tilde{\chi}_{\sigma_R})\right) \ .
    \end{equation}

\subsection{Edge signs for boundary edges}\label{sec:edge-sign-bnd}

    Let $\mathcal{C}$ be a marked combinatorial surface. Let $e\in \mathcal{C}_1$ be a boundary edge. Therefore it will have only a single adjacent face, which -- by convention -- is on the right side and denoted as $\sigma_R(e)$ as in Section \ref{sec:index_inner}. The edge $e$ is then the $k_R(e)$'th edge of $\sigma_{R}(e)$. We define the index $k_L(e)$ by the condition that under the boundary parametrisation $f_i:\partial \Delta \to \mathcal{C}$, $e$ is the $k_{L}$'th edge of $\Delta$.

The point $z= \chi_{\sigma_R}^{-1} \circ\varphi_i (e^{2\pi i (\frac{k_L}{3} + \frac{1}{6})})$ lies on the $k_R(e)$'th edge of $\Delta $. 
The derivative $d(\varphi_i^{-1}\circ \chi_{\sigma_R})_z \in GL_2^+$ rotates a tangent vector in the direction of the edge by an angle $e^{2\pi i \left( \frac{k_L}{3}-\frac{k_R}{3} + \frac{1}{2} \right)}$.
In terms of the $QR$-decomposition this can be written as
\begin{equation}
    \mathbf{Q}_\alpha (d(\varphi_i^{-1}\circ \chi_{\sigma_R})_z) = e^{2\pi i \left(\frac{k_L}{3}-\frac{k_R}{3} +\frac{1}{2} \right)} \ ,
\end{equation}
where $\alpha =e^{\frac{2\pi ik_R}{3}}\alpha_0$. We can now state the definition of edge signs for boundary edges. The definition needs some justification which will be provided in the lemma following the definition.

\begin{definition}[Edge signs for boundary edges]
Let $e$ be a boundary edge of a spin triangulated surface, and let $z$ and $\alpha $ be as above. Depending on whether the boundary is of NS or R type, the {\em edge sign} 
	 $s(e) \in \{\pm 1\}$
is defined via
    \begin{equation}
        \delta(s(e)) = \tilde{\mathbf Q}_\alpha \big( g_z( \tilde{\varphi}_i^{-1} \circ \tilde{\chi}_{\sigma_R} ) \big) \cdot 
\begin{cases}
\text{$NS$ type} :&         
\delta\big(e^{-\pi i ( \frac{k_L}{3}-\frac{k_R}{3} +\frac{1}{2} )}\big)
\\
\text{$R$ type} :&
 \delta \big( e^{-\pi i ( -\frac{k_R}{3} - \frac{1}{6}  + \frac{1}{2} )} \big) \ .
\end{cases}
\label{eq:bnd-edge-sign-def}
\end{equation}
\label{def:se-boundary_edge}
\end{definition}

\begin{lemma}
The right hand side of \eqref{eq:bnd-edge-sign-def} defines an element in $\{ \delta (\pm 1)\} \subset \GL_2$.
\end{lemma}
    \begin{proof}
        If the boundary is of NS type, the argument is as in the case of an inner edge. The case of an R type boundary is slightly more subtle. 
Pick a small neighbourhood $U$ of $z$ and an appropriate neighbourhood $V$ of $e^{2\pi i (\frac{k_L}{3} + \frac{1}{6})}$.
      Then the following diagram commutes:
        \begin{equation}
            \begin{tikzcd}[column sep=huge]
                U\times \GL_2 \ar{r}{\tilde{\varphi}_i^{-1}\circ \tilde{\chi}_{\sigma_R}} \ar{d}{p^{NS}} & V\times \GL_2 \ar{d}{p^R}\\
                U\times GL_2^+ \ar{r}{d(\varphi_i^{-1}\circ \chi_{\sigma_R})} & V\times  GL_2^+
            \end{tikzcd}.
            \label{eq:Rboundindex}
        \end{equation}
Applying this to $(z,e)$ and inserting $p_{NS}$, $p_R$ from \eqref{eq:p_NS-def}, \eqref{eq:SC-proj} gives 
        \begin{equation}
            e^{2\pi i \left(\frac{k_L}{3} + \frac{1}{6}\right)} p_{GL}\big( g_z( \tilde{\varphi}_i^{-1}\circ \tilde{\chi}_{\sigma_R} ) \big) = d(\varphi_i^{-1}\circ \chi_{\sigma_R})_z  \ .
        \end{equation}
If one applies $\mathbf{Q}_\alpha$ to both sides and uses Lemma \ref{lem:QR-rules} one arrives at 
\begin{align}
e^{2\pi i \left(\frac{k_L}{3} + \frac{1}{6}\right)} \mathbf{Q}_\alpha\big( 
	p_{GL}\big(
g_z( \tilde{\varphi}_i^{-1}\circ \tilde{\chi}_{\sigma_R} ) 
	\big)
\big) = e^{2\pi i \left(\frac{k_L}{3}-\frac{k_R}{3} +\frac{1}{2} \right)} \ .
\end{align}
This is equivalent to $p_{GL}\Big( \delta\big(e^{\pi i ( \frac{k_R}{3} + \frac{1}{6} - \frac{1}{2} )}\big) \mathbf{\tilde{Q}}_\alpha  \left(g_z\left( \tilde{\varphi}_i^{-1}\circ \tilde{\chi}_{\sigma_R} \right)  \right) \Big) = 1$.
    \end{proof}
    
We have defined all geometric and combinatorial ingredients needed for our combinatorial model of spin structures:
\begin{itemize} \itemsep .5em

\item {\em surface} $(\Sigma,\varphi_i)$ -- Definition \ref{def:surf.parameterised.boundary}, \\ 
$\Sigma$ -- compact surface; $\varphi_i$ -- boundary parametrisation.

\item {\em marked triangulated surface} $((\mathcal{C},f_i,d^1_0,d^2_0), \varphi, (\Sigma,\varphi_i))$ -- Definition \ref{def:marked-triang-surf},
\\
$\mathcal{C}$ -- combinatorial surface; $f_i$ -- boundary parametrisation of $\mathcal{C}$; $d^1_0$, $d^2_0$ -- marking; $\varphi : |\mathcal{C}| \to \Sigma$ -- smooth triangulation; $(\Sigma,\varphi_i)$ -- surface.
\item {\em spin surface} $(\Sigma,\tilde\varphi_i)$ -- Definition \ref{def:ssurf.param.boundary},
\\
$\Sigma$: compact spin surface; $\tilde\varphi_i$: boundary parametrisation via spin maps.

\item {\em spin triangulated surface} $((\mathcal{C},f_i,d^1_0,d^2_0), (\varphi, \tilde\chi_\sigma) ,(\Sigma,\tilde\varphi_i))$ -- Definition \ref{def:spin-triang-surf-def},
\\
$(\Sigma,\tilde\varphi_i)$ -- spin surface;
$((\mathcal{C},f_i,d^1_0,d^2_0),\varphi ,(\underline{\Sigma},\varphi_i))$ -- marked triangulated surface; 
$\tilde\chi_\sigma$ -- spin lift of $\chi_\sigma : \Delta \to \underline\Sigma$.

\item {\em edge signs} $s(e)$ -- Definitions \ref{def:se-inner_edge} and \ref{eq:bnd-edge-sign-def}.
\end{itemize}

\subsection{Behaviour of edge signs under gluing}
\begin{lemma} \label{lem:index.gluing}
    Let $\Sigma$ be a spin triangulated surface and $(i,j,\varepsilon )$ 
	be spin gluing data such that $(i,j)$ is simplicial gluing data. Let $e_i, e_j$ be edges on the $i$'th and $j$'th boundary respectively which are to be glued together, i.e.\ $e_i=\left( f_i \circ s_{\mathcal{C}} \circ f_j^{-1} \right)(e_j)$. 
    Let $e=[e_i]=[e_j]$ be the glued edge. 
    Then $s(e)= \varepsilon \, s(e_i)s(e_j)$.
\end{lemma}

\begin{proof}
We introduce the following abbreviations: $\sigma_i=\sigma_R(e_i)$ and $\sigma_j=\sigma_R(e_j)$, as well as
\begin{itemize}
    \item $k_L=k_L(e)=k_R(e_i)$,
    \item $k_i=k_L(e_i)$,
    \item $k_R=k_R(e)=k_R(e_j)$,
    \item $k_j=k_L(e_j)$.
\end{itemize}

\noindent
$NS$ {\em type boundary}:
We want to get $\tilde{s}_{NS}^\varepsilon$ explicitly for $z\in S^1$. 
By \eqref{eq:spin.gluing.map}, $\tilde{s}_{NS}^{\varepsilon}(1,e)=(1,\delta(-i \varepsilon))$.
Let $\alpha \in \mathbb{R}$ and $g_\alpha \in \GL_2$ such that $\tilde{s}_{NS}^{\varepsilon}( e^{i\alpha}, e ) = ( e^{-i\alpha} , g_\alpha  )$.
Evaluating the identity $p^{NS} \circ \tilde{s}^\varepsilon = ds_* \circ p^{NS}$ at the point $(e^{i\alpha}, e)$ thus gives $p_{GL}(g_\alpha) = -e^{-2i\alpha}$. Since $g_0 = \delta(-\varepsilon i)$, by continuity of $g_\alpha$ in $\alpha$ we get $g_\alpha = \delta(-\varepsilon  i e^{-i\alpha})$ and thus $\tilde{s}_{NS}^{\varepsilon}( e^{i\alpha}, e )=(e^{-i\alpha},\delta(-\varepsilon  i e^{-i\alpha}))$.

Let $z$ be a point on the $k_R$'th edge of $\Delta$ and 
\begin{equation}
        z_1 = (\varphi_j^{-1} \circ \chi_{\sigma_j})(z) \quad , \qquad
        z_2 = (s\circ \varphi_j^{-1} \circ \chi_{\sigma_j})(z) \  .
\end{equation}
We have  
\begin{align}\label{eq:sign-glue-aux1}
    g_z\left( \tilde{\chi}_{[\sigma_i]}^{-1} \circ \tilde{\chi}_{[\sigma_j]} \right)
    &= g_z \!\left( \tilde{\chi}_{\sigma_i}^{-1} \circ \tilde{\varphi}_i \circ \tilde{s}^\varepsilon  \circ \tilde{\varphi}_j^{-1} \circ \tilde{\chi}_{\sigma_j} \right) 
    \\
    &= g_{z_{2}}\!\left( \tilde{\chi}_{\sigma_i}^{-1} \circ \tilde{\varphi}_i \right) \, g_{z_1}\!\left( \tilde{s}^\varepsilon \right) \, g_z\!\left( \tilde{\varphi}_j^{-1} \circ \tilde{\chi}_{\sigma_j} \right) \ .
    \nonumber
\end{align}
We pick $z$ such that $z_1=e^{2\pi i\left( \frac{k_j}{3} + \frac{1}{6} \right)}$, so that in particular
\begin{equation} \label{eq:sign-glue-aux2}
	g_{z_1}\!\left( \tilde{s}^\varepsilon \right) = \delta\big({-}\varepsilon  i e^{-2\pi i\left( \frac{k_j}{3} + \frac{1}{6} \right)}\big) \ .
\end{equation}
Let $\alpha =e^{\frac{2\pi ik_R}{3}}\alpha_0$.
\begin{align}
\label{eq:signs-gluing-aux1}
        \delta(s(e))
        &\stackrel{\text{\eqref{eq:se-inner_edge-2}}}= 
        \tilde{\mathbf{Q}}_\alpha  \Big( \delta\big(e^{-\pi i( \frac{k_L}{3} - \frac{k_R}{3} + \frac{1}{2} )}\big) g_z\big( \tilde{\chi}_{[\sigma_i]}^{-1} \circ \tilde{\chi}_{[\sigma_j]} \big) \Big) 
        \\
        &\stackrel{\text{\eqref{eq:sign-glue-aux1}}}= \tilde{\mathbf{Q}}_\alpha  \Big( \delta\big(e^{-\pi i( \frac{k_L}{3} - \frac{k_R}{3} + \frac{1}{2} )}\big) g_{z_{2}}\!\left( \tilde{\chi}_{\sigma_i}^{-1} \circ \tilde{\varphi}_i \right) g_{z_1}\!\left( \tilde{s}^\varepsilon \right) g_z\!\left( \tilde{\varphi}_j^{-1} \circ \tilde{\chi}_{\sigma_j} \right) \Big)
        \nonumber \\
        &= \delta(s(e_j)) \, \tilde{\mathbf{Q}}_\alpha\Big( \delta\big(e^{-\pi i( \frac{k_L}{3} - \frac{k_R}{3} + \frac{1}{2} )} \big) \, g_{z_{2}}\!\left( \tilde{\chi}_{\sigma_i}^{-1} \circ \tilde{\varphi}_i \right) \, g_{z_1}\left( \tilde{s}^\varepsilon\right) 
\nonumber \\
& \hspace{3em} 
        \circ \delta\big(e^{\pi i( \frac{k_j}{3}-\frac{k_R}{3}+\frac12 )}  s(e_j)^{-1}  e^{-\pi i( \frac{k_j}{3}-\frac{k_R}{3} + \frac12 )} \big) \, g_z\!\left( \tilde{\varphi}_j^{-1} \circ \tilde{\chi}_{\sigma_j} \right) \Big) \ .
        \nonumber 
\end{align}
	In the last step we used the first equality in \eqref{eq:lem.qr.5} and that $\delta(s(e_j))$ is central in $\GL$, see \eqref{eq:delta-cap-Z}.
Now insert the identity
\begin{equation}
\label{eq:signs-gluing-aux2}
	\delta(s(e_j))^{-1} \stackrel{\text{\eqref{eq:bnd-edge-sign-def}}}= \Big(
{\mathbf{\tilde Q}}_\alpha \big( g_z( \tilde{\varphi}_j^{-1} \circ \tilde{\chi}_{\sigma_j} ) \big) 
\delta\big(e^{-\pi i ( \frac{k_j}{3}-\frac{k_R}{3} +\frac{1}{2} )}\big) \Big)^{-1}
\end{equation}
at the second occurrence of $s(e_j)$. One can apply Lemma \ref{lem:QR-rules} to remove the combination ${\mathbf{\tilde Q}}_\alpha \big( g_z(-) \big)^{-1} g_z(-)$. This leads to 
\begin{align}
\text{Eqn.\,\eqref{eq:signs-gluing-aux1}}
        &=
        \delta(s(e_j))  \tilde{\mathbf{Q}}_\alpha  \Big( \delta\big(e^{-\pi i( \frac{k_L}{3} - \frac{k_R}{3} + \frac{1}{2} )} \big) g_{z_{2}}\!\left( \tilde{\chi}_{\sigma_i}^{-1} \circ \tilde{\varphi}_i \right) g_{z_1}\!\left( \tilde{s}^\varepsilon \right) \delta\big(e^{\pi i( \frac{k_j}{3}-\frac{k_R}{3} +\frac12)} \big) \Big) 
        \nonumber \\
        &\stackrel{\text{\eqref{eq:sign-glue-aux2}}}= \delta(-s(e_j)) \tilde{\mathbf{Q}}_\alpha  \Big( \delta\big(e^{-\pi i( \frac{k_L}{3} - \frac{k_R}{3} + \frac{1}{2} )} \big)g_{z_{2}}\!\left( \tilde{\chi}_{\sigma_i}^{-1} \circ \tilde{\varphi}_i \right)   \delta\big( \varepsilon e^{\pi i( \frac{1}{2} - \frac{1}{3} -\frac{k_j}{3}-\frac{k_R}{3} + \frac12)} \big) \Big)
        \nonumber \\
        &= \delta(-s(e_i)s(e_j)\varepsilon)  \tilde{\mathbf{Q}}_\alpha  \Big( \delta\big( e^{-\pi i( \frac{k_L}{3} - \frac{k_R}{3}  )} \big) g_{z_{2}}\!\left( \tilde{\chi}_{\sigma_i}^{-1} \circ \tilde{\varphi}_i \right) 
\nonumber \\
& \hspace{7em} 
        \circ \delta\big( e^{\pi i( \frac{k_i}{3} - \frac{k_L}{3} )}s(e_i) e^{\pi i( \frac{k_L}{3} -\frac{k_i}{3} - \frac{k_j}{3} -\frac{1}{3} -\frac{k_R}{3} + \frac12)} \big) \Big)
        \nonumber \\
        &\stackrel{k_i+k_j=2}{=} \delta(s(e_i)s(e_j)\varepsilon)  \tilde{\mathbf{Q}}_\alpha  \Big( \delta\big(e^{-\pi i( \frac{k_L}{3} - \frac{k_R}{3}  )} \big) g_{z_{2}}\!\left( \tilde{\chi}_{\sigma_i}^{-1} \circ \tilde{\varphi}_i \right) 
\nonumber \\
& \hspace{7em} 
        \circ \delta\big( e^{\pi i( \frac{k_i}{3} - \frac{k_L}{3} )}s(e_i)^{-1} e^{\pi i( \frac{k_L}{3} -\frac{k_R}{3} + \frac12)} \big) \Big)
        \nonumber \\
        &\stackrel{\alpha '=\alpha_0e^{ 2\pi ik_L/3}}{=}
        \delta(s(e_i)s(e_j)\varepsilon)  \tilde{\mathbf{Q}}_{\alpha'} \Big( g_{z_{2}}\!\left( \tilde{\chi}_{\sigma_i}^{-1} \circ \tilde{\varphi}_i \right) \delta\big(e^{\pi i( \frac{k_i}{3} - \frac{k_L}{3} + \frac12)}s(e_i) \big)\Big)
       \nonumber \\
 &= \delta(s(e_i)s(e_j)\varepsilon)\ .
  \nonumber
\end{align}
In the last step a replacement analogous to \eqref{eq:signs-gluing-aux2} was used for $s(e_i)$.

\medskip
\noindent
$R$ {\em type boundary}:
The condition for an $R$-type boundary is slightly simpler: we need to determine $\tilde{s}^\varepsilon_{R}$ for $z\in S^1$ from the defining condition $\tilde{s}_{R}^{\varepsilon}(1,e)=(1,\delta(i \varepsilon))$. In analogy to the $NS$-case, using $p_R$ from \eqref{eq:SC-proj} to evaluate $p^{R} \circ \tilde{s}^\varepsilon = ds\circ p^{R}$ gives $g_\alpha = \varepsilon i$, i.e.\ $g_\alpha$ is independent of $\alpha$.

We pick $z$ and $\alpha $ as before. Then
\begin{align}
        \delta(s(e)) &=  \tilde{\mathbf{Q}}_\alpha \Big( \delta\big(e^{-\pi i( \frac{k_L}{3} - \frac{k_R}{3} + \frac{1}{2} )}\big) g_z\big( \tilde{\chi}_{[\sigma_i]}^{-1} \circ \tilde{\chi}_{[\sigma_j]} \big) \Big)
        \\
        &= \delta( s(e_j) ) \tilde{\mathbf{Q}}_\alpha \Big( \delta\big(e^{-\pi i( \frac{k_L}{3} - \frac{k_R}{3} + \frac{1}{2} )}\big) g_{z_{2}}\big( \tilde{\chi}_{\sigma_i}^{-1} \circ \tilde{\varphi}_i \big) g_{z_1}\left( \tilde{s}^\varepsilon \right)
        \nonumber \\
        &\hspace{7em}
        \delta\big(  e^{\pi i( -\frac{k_R}{3} - \frac{1}{6}+ \frac{1}{2} )}s(e_j) e^{-\pi i( -\frac{k_R}{3} - \frac{1}{6}+ \frac{1}{2} )} \big) g_z\big( \tilde{\varphi}_j^{-1} \circ \tilde{\chi}_{\sigma_j} \big) \Big)
\nonumber        \\
        &\stackrel{\text{Lem.\,\ref{lem:QR-rules}}}{=} \delta( s(e_j) )\tilde{\mathbf{Q}}_\alpha \Big( \delta\big(e^{-\pi i( \frac{k_L}{3} - \frac{k_R}{3} + \frac{1}{2} )}\big) g_{z_{2}}\big( \tilde{\chi}_{\sigma_i}^{-1} \circ \tilde{\varphi}_i \big) g_{z_1}\left( \tilde{s}^\varepsilon\right) \delta\big(  e^{\pi i( -\frac{k_R}{3} - \frac{1}{6}+ \frac{1}{2} )}\big)\Big)
\nonumber        \\
        &= \delta(  s(e_j)\varepsilon )  \tilde{\mathbf{Q}}_\alpha   \Big( \delta\big(e^{-\pi i( \frac{k_L}{3} - \frac{k_R}{3} )}\big) g_{z_{2}}\big( \tilde{\chi}_{\sigma_i}^{-1} \circ \tilde{\varphi}_i \big)  \delta\big(  e^{\pi i( -\frac{k_R}{3} - \frac{1}{6}+ \frac{1}{2} )}\big)\Big)
\nonumber        \\
        &\stackrel{\alpha '=\alpha_0e^{2\pi i k_L/3 }}{=}
        \delta(  s(e_i)s(e_j)\varepsilon )  \tilde{\mathbf{Q}}_{\alpha '}\Big(  g_{z_{2}}\big( \tilde{\chi}_{\sigma_i}^{-1} \circ \tilde{\varphi}_i \big)  \delta\big(s(e_i)  e^{\pi i( -\frac{k_L}{3} - \frac{1}{6}+ \frac{1}{2} )}\big)\Big)
\nonumber        \\
        &\stackrel{\text{Lem.\,\ref{lem:QR-rules}}}{=}\delta(  \varepsilon  s(e_i)s(e_j)) \ .
\nonumber       
\end{align}
\end{proof}

In the notation of Lemma \ref{lem:index.gluing}: recall from Definition \ref{def:glue-mark} that the edge $e$ carries the orientation induced by $e_j$, which is opposite to the one induced by $e_i$. The rule $s(e) = \varepsilon \, s(e_i)s(e_j)$ is compatible with the observation $\Sigma^\varepsilon_{i \# j} = \Sigma^{-\varepsilon}_{j \# i}$ made for glued spin surfaces below Definition \ref{def:glue.spin}. Namely, in $\Sigma^{-\varepsilon}_{j \# i}$ the edge $e$ has opposite orientation and opposite sign, which is one of the local moves discussed in the next section.

\subsection{Moves leaving the triangulation invariant}\label{sec:moves-notriang}

In this section and Section \ref{sec:moves-pachner} below we investigate how the edge signs behave under change of marking and triangulation.
For the calculations below we fix a lift $\tilde\alpha_0$ of the angle $\alpha_0$ defined in \eqref{eq:alpha_0-def}. We choose
\begin{equation}\label{eq:tildealpha_0-def}	
	\tilde\alpha_0:= e^{\pi i\,\frac{5}{12}} \ ,
\end{equation}
which satisfies $p_{GL}(\delta(\tilde\alpha_0)) = \alpha_0$, as required. 
We will also denote with $[n]_3 \in \{0,1,2\}$ the representative of $n$ mod $3$. 

The following lemma details how the edge signs behave under change of marking, see Figure \ref{fig:local-moves} for an illustration.

\begin{figure}[tb]
    1)\raisebox{-1.7cm}{\begin{tikzpicture}
        \node[name=A, regular polygon, regular polygon sides=3, draw,minimum size=2cm] at (0cm,0cm){};
        \node[name=A, regular polygon, regular polygon sides=3, draw,minimum size=2cm] at (0cm,0cm){};
        \node[name=s0,right] at (A.side 3) {$s_0$};
        \node[name=s1,left] at (A.side 1) {$s_1$};
        \node[name=s2,below] at (A.side 2) {$s_2$};
        \draw[<->] (1.2cm,0cm) -- (1.8cm,0cm) ;
        \node[name=A, regular polygon, regular polygon sides=3, draw,minimum size=2cm] at (3cm,0cm){};
        \node[name=A, regular polygon, regular polygon sides=3, draw,minimum size=2cm] at (3cm,0cm){};
        \node[name=s0,right] at (A.side 3) {$-s_0$};
        \node[name=s1,left] at (A.side 1) {$-s_1$};
        \node[name=s2,below] at (A.side 2) {$-s_2$};
    \end{tikzpicture}}\\
    2)\raisebox{-1.3cm}{\begin{tikzpicture}
        \node[name=A1, regular polygon, regular polygon sides=3, minimum size=1.5cm, rotate=30] at (0.375cm,0cm){};
        \node[name=B1, regular polygon, regular polygon sides=3, minimum size=1.5cm, rotate=90] at (-0.375cm,0cm){};
        \draw (A1.corner 2) -- (A1.corner 3);
        \draw (A1.corner 3) -- (A1.corner 1);
        \draw (B1.corner 1) -- (B1.corner 2);
        \draw (B1.corner 3) -- (B1.corner 1);
        \node[right] at (A1.side 1) {$s$};
        \draw[<->] (0.375cm+1cm,0cm) -- (3cm+-0.375cm-1cm,0cm);
        \node[name=A2, regular polygon, regular polygon sides=3, minimum size=1.5cm, rotate=30] at (3cm+0.375cm,0cm){};
        \node[name=B2, regular polygon, regular polygon sides=3, minimum size=1.5cm, rotate=90] at (3cm+-0.375cm,0cm){};
        \draw (A2.corner 2) -- (A2.corner 3);
        \draw (A2.corner 3) -- (A2.corner 1);
        \draw (B2.corner 1) -- (B2.corner 2);
        \draw (B2.corner 3) -- (B2.corner 1);
        \node[right] at (A2.side 1) {$-s$};
        \begin{scope}[decoration={
            markings,
            mark=at position 0.5 with {\arrow{stealth}}}
        ] 
        \draw[postaction=decorate] (A1.corner 1) -- (A1.corner 2);
        \draw[postaction=decorate] (A2.corner 2) -- (A2.corner 1);
        \end{scope}
    \end{tikzpicture}}\\
    3)\raisebox{-1.7cm}{\begin{tikzpicture}
        \node[name=A, regular polygon, regular polygon sides=3, draw,minimum size=2cm] at (0cm,0cm){};
        \draw[green!40, line width=4pt] ($(A.corner 3)!0.9!(A.corner 1)$)++(-2pt,0pt) -- ($($(A.corner 1)!0.9!(A.corner 3)$)+(-2pt,0pt)$);
        \node[name=A, regular polygon, regular polygon sides=3, draw,minimum size=2cm] at (0cm,0cm){};
        \node[name=s0,right] at (A.side 3) {$s_0$};
        \node[name=s1,left] at (A.side 1) {$s_1$};
        \node[name=s2,below] at (A.side 2) {$s_2$};
        \draw[<->] (1.2cm,0cm) -- (1.8cm,0cm) ;
        \node[name=A, regular polygon, regular polygon sides=3, draw,minimum size=2cm] at (3cm,0cm){};
        \draw[green!40, line width=4pt] ($(A.corner 1)!0.9!(A.corner 2)$)++(2pt,0pt) -- ($($(A.corner 2)!0.9!(A.corner 1)$)+(2pt,0pt)$);
        \node[name=A, regular polygon, regular polygon sides=3, draw,minimum size=2cm] at (3cm,0cm){};
        \node[name=s0,right] at (A.side 3) {$-s_0$};
        \node[name=s1,left] at (A.side 1) {$s_1$};
        \node[name=s2,below] at (A.side 2) {$s_2$};
    \end{tikzpicture}}
    \caption{The moves 1--3 of Lemma \ref{lem:index.marking}. Where the marking is not explicitly specified it is arbitrary but fixed.}
    \label{fig:local-moves}
\end{figure}
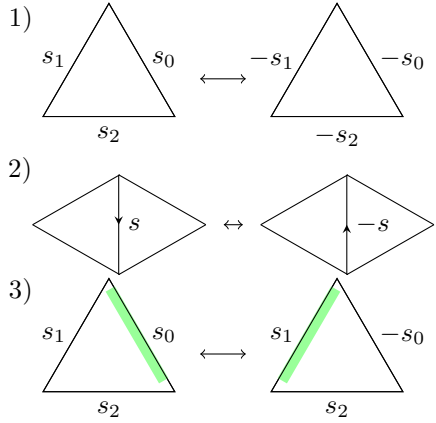

\begin{lemma}
    Let $\Sigma = ((\mathcal{C},f_i), (\varphi, \tilde\chi_\sigma) ,(\Sigma,\tilde\varphi_i))$ be a spin  triangulated surface and let $s(e)$ be the corresponding edge signs.
    \begin{enumerate}
        \item Let $\sigma $ be a face of $\mathcal{C}$. Change $\tilde{\chi}_\sigma $ by precomposing with the leaf exchange automorphism, i.e.
            \begin{equation*}
	        \tilde{\chi}_{\text{new},\sigma} = \tilde{\chi}_{\text{old},\sigma} \circ \omega \ .
            \end{equation*}
            Then $s(e)\mapsto -s(e)$ for all edges $e\in \partial(\sigma )$.
        \item Let $e$ be an inner edge of $\mathcal{C}$. Exchange
        the marking $d_0^1(e)$, i.e.
            \begin{equation*}
            d_{\text{new},0}^1(e) = d_{\text{old},1}^1(e) \ .
            \end{equation*}
            Then $s(e) \mapsto -s(e)$.
        \item Let $\sigma $ be a face of $\mathcal{C}$. Change the marking $d_0^2(\sigma )$ by picking the next edge counterclockwise, i.e.
            \begin{equation*}
                d_{\text{new},0}^2(\sigma) = 	d_{\text{old},1}^2(\sigma) \ .
            \end{equation*}
            Then there is a choice for the $\tilde{\chi}_{\text{new},\sigma}$ such that only the edge signs on the previously marked edge of $\sigma$ changes. In more detail, the only change is
            \begin{equation*}
                s_\text{new}(d_{\text{new},2}^2(\sigma ))=s_\text{new}(d_{\text{old},0}^2(\sigma ))=-s_\text{old}(d_{\text{old},0}^2(\sigma )) \ . 
            \end{equation*}
            (Here, $s_\text{new}$ is evaluated with respect to the lifts $\tilde{\chi}_{\text{new},\sigma}$.)
    \end{enumerate}     
    \label{lem:index.marking}
\end{lemma}

    \begin{proof} 
For a given edge $e$ we will abbreviate $k_L = k_L(e)$, $\sigma_L = \sigma_L(e)$, etc.  Primed quantities indicate the new choice of data / the resulting edge signs.

\medskip\noindent    
    (1)
If $e$ is inner, then composition with the leaf exchange automorphism just multiplies  $g_z(\tilde{\chi}_{\sigma_L}^{-1} \circ \tilde{\chi}_{\sigma_R})$ with $\delta (-1)$, which can be pulled out of the $QR$-decomposition since it is central. If $e$ is a boundary edge, then the same reasoning applies for $g_z\!\left( \tilde{\varphi}_i^{-1} \circ \tilde{\chi}_{\sigma_R} \right)$.

\medskip\noindent
(2)  Writing out $\mathbf{\tilde Q}_\alpha $ in formula \eqref{eq:se-inner_edge-2} for the edge sign $s(e)$ in terms of \eqref{eq:QR_alpha-def}, we get
                \begin{equation}
                    \delta(s(e)) = \tilde{\mathbf{Q}} \left( \delta\big( e^{-\pi i( \frac{k_L}{3}  + \frac{1}{2}  )}\tilde{\alpha}_0^{-1} \big) g_p\!\left( \tilde{\chi}_{\sigma_L}^{-1} \circ \tilde{\chi}_{\sigma_R}\right) \delta\big( e^{\pi i   \frac{k_R}{3}}\tilde{\alpha}_0  \big) \right).
                \end{equation}
Note that since the two possible lifts of $\alpha_0$ differ by $\delta(\pm1)$, which is central in $\GL$, the value $\delta(s(e))$ is actually independent of the choice of $\tilde\alpha_0$ we made in \eqref{eq:tildealpha_0-def}. The same applies to the calculations below.

                After changing the direction of $e$ the new edge sign is given by exchanging $L \leftrightarrow R$:
                \begin{equation} \label{eq:edge-sign-local-move-aux1}
                    \delta(s'(e)) = \tilde{\mathbf{Q}} \left( \delta\big( e^{-\pi i( \frac{k_R}{3}  + \frac{1}{2}  )}\tilde{\alpha}_0^{-1} \big) g_q\!\left( \tilde{\chi}_{\sigma_R}^{-1} \circ \tilde{\chi}_{\sigma_L}\right) \delta\big( e^{\pi i   \frac{k_L}{3}}\tilde{\alpha}_0  \big) \right) \ ,
                \end{equation}
	where $q =  \chi_{\sigma_L}^{-1} \circ \chi_{\sigma_R}(p)$.
                The product $s(e)s'(e)$ is then determined by
                \begin{equation}
                    \begin{split}
                        \delta(s(e)s'(e)) &= \tilde{\mathbf{Q}}\left( \delta\big( s'(e) e^{-\pi i(  \frac{k_L}{3}  + \frac{1}{2}  )} \tilde{\alpha}_0^{-1} \big) g_p\!\left( \tilde{\chi}_{\sigma_L}^{-1} \circ \tilde{\chi}_{\sigma_R}\right) \delta\big( e^{\pi i   \frac{k_R}{3} } \tilde{\alpha}_0  \big) \right) \\
                        &= \tilde{\mathbf{Q}}\left( \delta\big( e^{\pi i}  
                        s'(e) \big)
                        \left( \delta\big( e^{-\pi i ( \frac{k_R}{3}  + \frac{1}{2}  )}\tilde{\alpha}_0^{-1} \big) g_q\!\left( \tilde{\chi}_{\sigma_R}^{-1} \circ \tilde{\chi}_{\sigma_L}\right) \delta\big( e^{\pi i  \frac{k_L}{3}}\tilde{\alpha}_0 \big) \right)^{-1}  \right) \\
                    \end{split}
                \end{equation}
where we used that $g_p = g_q^{-1}$ and that $\delta(e^{\pi i /2})$ is central in $\GL_2$, see \eqref{eq:delta-cap-Z}. 
Substituting \eqref{eq:edge-sign-local-move-aux1} for $s(e')$ in the above expression gives
                \begin{equation}
                    \delta(s(e)s'(e)) = \tilde{\mathbf{Q}}\left( \delta\big(e^{\pi i}\big) \right) = \delta(-1) \ .
                \end{equation}

\medskip\noindent
(3) Rotating the marked edge of $\sigma $ counterclockwise means replacing $\chi_\sigma $ with $\chi_\sigma ' := \chi_\sigma  \circ (z\mapsto e^{2\pi i/3}z)$. We then choose 
$\tilde{\chi}'_\sigma := \tilde{\chi}_\sigma  \circ ( (z,g) \mapsto (e^{2\pi i/3}z,\delta(e^{\pi i/3})g))$.
Let $e$ be the $k$-th edge of $\sigma $ before the change of marking. It becomes the $k'$-th edge of $\sigma $, with $k'=[k-1]_3$.
We will first treat the case that $e$ is an inner edge and that $\sigma$ is to the right of $e$. 
Let $\alpha =e^{2\pi i \frac{k}{3} }\alpha_0$ and $\alpha '=e^{2\pi i \frac{k'}{3}  }\alpha_0$. Then 
                \begin{align}
                        \delta\big(s'(e)\big) &= \mathbf{\tilde{Q}}_{\alpha '}\left(\delta\Big(e^{-\pi i\left(\frac{k_L}{3}-\frac{k'}{3} + \frac{1}{2}\right)}\Big)g(\tilde{\chi}_{\sigma_L(e)}^{-1} \circ \tilde{\chi}_{\sigma}')\right)\\
                        &= \mathbf{\tilde{Q}}_{\alpha}\left(\delta\Big(e^{-\pi i\left(\frac{k_L}{3}-\frac{k'}{3}  - \frac{1}{3} + \frac{1}{2}\right)}\Big)g(\tilde{\chi}_{\sigma_L(e)}^{-1} \circ \tilde{\chi}_{\sigma}')\delta\!\left( e^{-\frac{\pi i}{3}} \right)\right)\nonumber\\
                        &= \delta\Big( e^{-\pi i\left( \frac{k-1}{3} - \frac{k'}{3} \right)} \Big)\mathbf{\tilde{Q}}_{\alpha}\left(\delta \Big(e^{-\pi i\left(\frac{k_L}{3}-\frac{k}{3} + \frac{1}{2}\right)}\Big)g(\tilde{\chi}_{\sigma_L(e)}^{-1} \circ \tilde{\chi}_{\sigma})\right)                    \nonumber\\
                    &= \delta\Big( e^{-\pi i\left( \frac{k-1}{3} - \frac{k'}{3} \right)} \Big) \delta\big(s(e)\big) \ . 
                    \nonumber
                \end{align}
If $k=0$ then $k-1-k'=-3$ and thus $s'(e)=-s(e)$. Otherwise $k-1-k'=0$ and $s'(e)=s(e)$. 

If $e$ is a boundary edge, then in the above calculation replace $\tilde{\chi}_{\sigma_L(e)}$ with $\tilde{\varphi}_i$ and use the phase as stated in \eqref{eq:bnd-edge-sign-def}.

If $e$ is an inner edge such that $\sigma$ is to the left of $e$, we can change the edge orientation by move (2), apply the above argument, and then change the edge orientation back, again by move (2). The sign changes in $s(e)$ form the two flips of the edge orientation cancel.
    \end{proof}

\subsection{Lifting properties of simple closed curves}\label{sec:lifting-prop}
In Section \ref{sec:map.lifting} we examined the lifting properties of smooth simple closed curves to the spin bundle. We will now determine how the lifting behaviour depends on the edge signs. 

To treat inner and boundary edges on the same footing, we use the boundary parametrisation $\varphi_i : U_i \to \Sigma$ to enlarge the surface $\Sigma$ to a new surface $\Sigma^+$ obtained by gluing on little collars: pick any $0<r<1$, define the open sets $U_i^+$ as $U_i \cup \{ z \in \mathbb{C} \,|\, r{<}|z| {\le} 1 \,\}$ and set
\begin{equation}
\Sigma^+ = \Sigma \sqcup U_1^+ \sqcup \cdots \sqcup U_B^+  / \sim \ ,
\end{equation}
where $\sim$ identifies $\varphi_i(z) \in \Sigma$ with $z \in U_i$. 
It is easy to see that the following are equivalent: a) the structure of a spin surface (with parametrised boundary) on $\Sigma$, and b) a spin structure on $\Sigma^+$ which on $U_i^+$ is equal to $\Cb^{NS/R}|_{U_i^+}$, depending on the type of the $i$'th boundary component of $\Sigma$. We will use description b).

We will have need for spin structures on triangulated surfaces minus their vertices.

\begin{definition}
Let $(\mathcal{C},\Sigma,\varphi)$ be a triangulated surface. A \emph{punctured spin structure} on $(\mathcal{C},\Sigma,\varphi)$ is a spin structure on $\Sigma^+\setminus \varphi(\mathcal{C}_0)$ which on the glued-on collars 
\begin{equation}
    \left(U_i^+\cap\{z\in \mathbb{C} \vert \lvert z \rvert\leq 1\}\right)\setminus \{1,e^{2\pi i/3},e^{4\pi i/3}\}
\end{equation} is equal to the restriction of $\Cb^{NS/R}$.
A \emph{punctured spin triangulated surface} is a marked triangulated surface together with a punctured spin structure and a choice of spin lift $\tilde{\chi}_\sigma:\Delta\setminus\Delta_0\to \Sigma$ for each triangle $\sigma\in \mathcal{C}_2$.
The boundary parametrisation maps $\varphi_i:U_i \to \Sigma$ extend naturally to embeddings $U_i^+\to \Sigma^+$ and for a punctured spin triangulated surface we get a spin lift for $\varphi_i|_{\{\lvert z \rvert \leq 1\}}$.
By slight abuse of notation we call this spin lift $\tilde{\varphi}_i$.
\end{definition}

We will see in Section \ref{sse:spin_reconstruct} how to construct a punctured spin structure on a marked triangulated surface with an arbitrary assignment of edge signs.
On the other hand one can define edge signs for punctured spin triangulated surfaces in the same way as for spin triangulated surfaces since Definitions \ref{def:se-inner_edge} and \ref{def:se-boundary_edge} do not rely on the extendibility of the spin structure.
In the following let $(\mathcal{C},\tilde\varphi,\Sigma)$ be a punctured spin triangulated surface with edge signs
	$s(e)$.

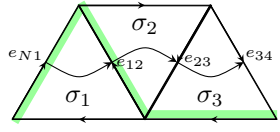
\begin{figure}[tb]
    \centering
    \begin{tikzpicture}
        \node[name=A1, regular polygon, regular polygon sides=3, draw=none, minimum size=2cm] at (0cm,0cm){};
        \node[name=A2, regular polygon, regular polygon sides=3, draw=none, minimum size=2cm, rotate=180] at (0.87cm, 0.5cm){};
        \node[name=A3, regular polygon, regular polygon sides=3, draw=none, minimum size=2cm] at (1.73cm, 0cm){};
        \node[below] at (A1.center) {$\sigma_{1}$};
        \node[above] at (A2.center) {$\sigma_{2}$};
        \node[below] at (A3.center) {$\sigma_{3}$};
        \draw[green!40, line width=3pt] (A1.corner 1)++(1pt,-1pt) -- ($(A1.corner 2)+(1pt,-1pt)$);
        \draw[green!40, line width=3pt] (A2.corner 3)++(1pt,1pt) -- ($(A2.corner 1)+(1pt,1pt)$);
        \draw[green!40, line width=3pt] (A3.corner 2)++(0pt,2pt) -- ($(A3.corner 3)+(0pt,2pt)$);
        \draw[->,>=stealth, very thin] (A1.corner 2) -- (A1.side 1);
        \draw[->,>=stealth, very thin] (A1.corner 3) -- (A1.side 2);
        \draw[->,>=stealth, very thin] (A1.corner 3) -- (A1.side 3);
        \draw[->,>=stealth, very thin] (A2.corner 3) -- (A2.side 2);
        \draw[->,>=stealth, very thin] (A3.corner 3) -- (A3.side 2);
        \draw[->,>=stealth, very thin] (A3.corner 1) -- (A3.side 3);
        \draw[->,>=stealth, very thin] (A3.corner 1) -- (A3.side 1);
        \draw[->,>=stealth] (A1.side 1) ..  controls (A1.center) .. (A1.side 3);
        \draw[->,>=stealth] (A1.side 3) ..  controls (A2.center) .. (A3.side 1);
        \draw[->,>=stealth] (A3.side 1) ..  controls (A3.center) .. (A3.side 3);
        \foreach \Na/\Nb in {1/2,2/3,3/1} {
            \foreach \N in {1,2,3} \draw[] (A\N.corner \Na) -- (A\N.corner \Nb);
        }
        \node[name=A1, regular polygon, regular polygon sides=3, draw, minimum size=2cm] at (0cm,0cm){};
        \node[name=A2, regular polygon, regular polygon sides=3, draw, minimum size=2cm, rotate=180] at (0.87cm, 0.5cm){};
        \node[name=A3, regular polygon, regular polygon sides=3, draw, minimum size=2cm] at (1.73cm, 0cm){};
        \node[above left=-4pt] at (A1.side 1) {\tiny{$e_{N1}$}};
        \node[above right=-6pt and -2pt] at (A1.side 3) {\tiny{$e_{12}$}};
        \node[below right=-6pt and -1pt] at (A3.side 1) {\tiny{$e_{23}$}};
        \node[above right=-4pt] at (A3.side 3) {\tiny{$e_{34}$}};
    \end{tikzpicture}
    \caption{A part of a configuration of $N$ adjacent triangles. The first three triangles are drawn. For that configuration we have $k_1=0$, $k_2=0$ and $k_3=2$, as well as $\mu_{N1}=-1$, $\mu_{12}=-1$, $\mu_{23}=1$, $\mu_{34}=1$ and $\eta_1=-1$, $\eta_2=1$, $\eta_3=-1$.}
    \label{fig:triang.path.conf}
\end{figure}

\subsubsection{Paths transversing inner edges}\label{sec:path_argument_inner}

To describe the paths we want to lift to the spin bundle explicitly, we need a little preliminary setup.
Let $(\sigma_i)_{i\in \mathbb{Z}_N}$, $\sigma_i\in \mathcal{C}_2$ be a sequence of $N$ distinct triangles such that each two consecutive triangles $\sigma_i$, $\sigma_{i+1}$ intersect in an edge $e_{i,i+1}$.
We describe the marking on $(\sigma_i)$ and $(e_{i,i+1})$ explicitly:
\begin{itemize}
    \item $k_i\in \{0,1,2\}$ is the position of the edge $e_{i-1,i}$ in $\sigma_i$, 
    $d^2_{k_i}(\sigma_i)=e_{i-1,i}$. 
In other words, the path enters the triangle $\sigma_i$ through the edge $d^2_{k_i}(\sigma_i)$.
    \item $\eta_i=\pm 1$ describes the position of the edge $e_{i,i+1}$ relative to the edge $e_{i-1,i}$ in $\sigma_i$, $d^2_{[k_i+\eta_i]_3}(\sigma_i)=e_{i,i+1}$. In other words, the path exits through the edge $d^2_{[k_i+\eta_i]_3}(\sigma_i)$.
    \item $\mu_{i,i+1}=\pm 1$ describes the direction of the edge $e_{i,i+1}$: $\mu_{i,i+1}=1$ if $\sigma_i$ is to the right of $e_{i,i+1}$ and $\mu_{i,i+1}=-1$ otherwise.
\end{itemize}
An example of such a configuration is shown in Figure \ref{fig:triang.path.conf}. 
We will in the following abbreviate
\begin{equation}\label{eq:sii+1-def}
	s_{i,i+1} = s(e_{i,i+1}) \ .
\end{equation}

\begin{lemma}\label{lem:path_lifting_inner_generic}
    Given a configuration as above let $\gamma:[0,1]\to \underline{\Sigma}$ be a smooth simple closed curve that can be written as the composition of paths $\gamma_0,\dots,\gamma_{N-1}$,
    \begin{equation}
        \gamma=\gamma_0\star \gamma_1 \star \dots \star \gamma_{N-1}.
    \end{equation}
 Here,
	$\gamma$ is a path from $\gamma_0(0)$ to $\gamma_{N-1}(1)$ and the individual components
$\gamma_i:[a_i,b_i]\to \Sigma$
are smooth paths such that:
    \begin{itemize}
        \item $\text{Im}\, \gamma_i \subset \text{Im}\, \varphi(\sigma_i)$.
        \item $\gamma_i(a_i)\in \text{Im}\, \varphi(e_{i-1,i})$.
        \item $\gamma_i(b_i) \in \text{Im}\, \varphi(e_{i,i+1})$.
    \end{itemize}
Let $\hat{\gamma}:[0,1]\to F_{GL^+}(\underline{\Sigma})$ be a lift of $d\gamma$ along 
	the projection $\pi_T:F_{GL^+}\left( \underline{\Sigma} \right)\to T\underline{\Sigma}$
	mapping a frame to its first component vector. 
Then $\hat{\gamma}$ has a spin lift if and only if
\begin{equation}
    \prod_{i=1}^N s_{i,i+1} \mu_{i,i+1} e^{\pi i\left( \frac{k_i + \eta_i - [k_i+\eta_i]_3}{3} + \frac{1-\eta_i}{2} \right)} = 1. 
    \label{eq:lemma_path_lifting_inner_generic}
\end{equation}
\end{lemma}

    \begin{figure}[tb]
    \centering
    \raisebox{4em}{a)}
    \begin{tikzpicture}
        \node[name=A1, regular polygon, regular polygon sides=3, draw, minimum size=2cm] at (0cm,0cm){};
        \node[name=A2, regular polygon, regular polygon sides=3, draw, minimum size=2cm, rotate=180] at (0.87cm, 0.5cm){};
        \node[below] at (A1.center) {$\sigma_{1}$};
        \node[below] at (A2.center) {$\sigma_{2}$};
        \begin{scope}[decoration={
            markings,
            mark=at position 0.5 with {\arrow{stealth}}}
        ] 
        \draw[postaction=decorate] (A1.side 1) .. controls (A1.center) and ($(A1.center)+(0.3cm,-0.3cm)$) ..  (A1.side 3);
        \draw[postaction=decorate] (A1.side 3) .. controls ($(A2.center)+(0.3cm,0.3cm)$) .. (A2.side 2);
        \end{scope}
    \end{tikzpicture}
    \hspace{4em}
    \raisebox{4em}{b)}
    \begin{tikzpicture}
        \node[name=A1, regular polygon, regular polygon sides=3, draw, minimum size=2cm] at (0cm,0cm){};
        \node[name=A2, regular polygon, regular polygon sides=3, draw, minimum size=2cm, rotate=180] at (0.87cm, 0.5cm){};
        \node[below] at (A1.center) {$\sigma_{1}$};
        \node[below] at ($(A2.center)!0.4!(A2.corner 2)$) {$\sigma_{2}$};
        \begin{scope}[decoration={
            markings,
            mark=at position 0.5 with {\arrow{stealth}}}
        ] 
        \draw[postaction=decorate] (A1.side 1) .. controls (A1.center) and ($(A1.corner 1)!0.2!(A1.corner 3)$) ..  (A1.side 3);
        \draw[postaction=decorate] (A1.side 3) .. controls ($(A1.corner 3)!0.2!(A1.corner 1)$) and (A2.center) .. (A2.side 2);
        \end{scope}
    \end{tikzpicture}
    \caption{a) The paths $\chi_i\circ \zeta_i$ can intersect an edge in an arbitrary angle and thus the composition may not be differentiable. b) The modified paths intersects 
the middle edge tangentially and thus up to a reparametrisation a smooth composition is possible.}
    \label{fig:triang.path_smooth}
\end{figure}
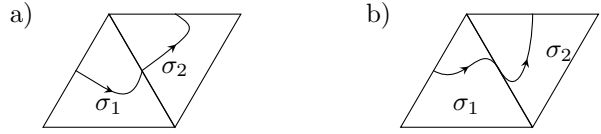

\begin{proof}
    Any two curves $\gamma,\gamma':[0,1]\to \Sigma$ that satisfy the assumptions of Lemma \ref{lem:path_lifting_inner_generic} are isotopic: We can find an isotopy triangle by triangle. 
    By the reasoning in Section \ref{sec:map.lifting} the corresponding paths in the frame bundle are homotopic. 
    It is then sufficient to examine the lifting properties of one such curve.

\medskip\noindent    
$\bullet$\,\emph{Explicit construction of a suitable curve in the frame bundle}.
    We start by defining a simple closed curve $\zeta:[0,1]\to \underline{\Sigma}$. Let
    \begin{equation}\label{eq:triangle.path.segment}
        \begin{split}
        \zeta_i:[0,1]&\to \underline\Delta\\
        t&\mapsto z_i - \eta_i r_0\alpha_0 e^{2\pi i\left( -\frac{\eta_i t}{6} + \frac{k_i}{3} \right)} \ ,
        \end{split}
    \end{equation}
where $\alpha_0 = e^{2\pi i\,\frac{5}{12}}$ as in \eqref{eq:alpha_0-def}. The centre point $z_i=e^{\frac{2\pi i}{3}\left( k_i + \frac{1}{2}(1+\eta_i) \right)}$ is the preimage under $\check\chi_{\sigma_i}$ of the vertex $e_{i-1,i}$ and $e_{i,i+1}$ intersect in. The radius $r_0=\frac{\sqrt{3}}{2}$ is chosen such that $\zeta_i$ starts and ends on the midpoints of the edges. This ensures that
    \begin{equation}
        \left( \chi_{\sigma_i} \circ \zeta_i \right)(1)= \left( \chi_{\sigma_{i+1}} \circ \zeta_{i+1} \right)(0),
    \end{equation}
    and thus we obtain the simple closed curve
    \begin{equation}
        \zeta := \left( \chi_{\sigma_0} \circ \zeta_0 \right) \star \left( \chi_{\sigma_1} \circ \zeta_1 \right) \star \dots \star \left( \chi_{\sigma_{N-1}} \circ \zeta_{N-1} \right).
    \end{equation}
    The curve $\zeta$ may not be differentiable at the edges, even if reparametrisation (with non-vanishing velocity)
is taken into account. We fix this as depicted in Figure \ref{fig:triang.path_smooth}:
    We change $\zeta_i$ slightly at the edges, such that it intersects the edge tangentially. By reparametrisation of the paths, the composition can then be made a smooth simple closed curve. The differentials of these paths then have a lift to the frame bundle which are (up to a homotopy that leaves the base path fixed) composable. We describe these 
    lifts
    in the limit of changing the initial paths $\zeta_i$ minimally. The change at the edges can then be described by rotations in the frame bundle:
    Let
\begin{align}
            \hat{\zeta}_i &: [0,1] \to \underline\Delta \times GL_2^+ 
            ~~,~~~&
            t&\mapsto 
            \left( \zeta_i(t), \frac{2\pi i}{6} r_0 \alpha_0 e^{2\pi i\left( -\frac{\eta_i t}{6} + \frac{k_i}{3} \right)} \right) \ ,
            \nonumber \\
            \hat{\zeta}_i^L &:[0,1] \to \underline\Delta \times GL_2^+ 
            ~~,~~~&
            t& \mapsto \left( \zeta_i(0), \frac{2\pi i }{6} r_0 \alpha_0 e^{2\pi i\left( \frac{t-1}{4} + \frac{k_i}{3} \right)} \right) \ ,
	\\
            \hat{\zeta}_i^R&:[0,1] \to \underline\Delta \times GL_2^+ 
            ~~,~~~&
            t&\mapsto \left( \zeta_i(1), \frac{2\pi i }{6} r_0 \alpha_0 e^{2\pi i \left(- \frac{t}{4} - \frac{\eta_i}{6} + \frac{k_i}{3}  \right)} \right).
            \nonumber
    \end{align}
    The paths $\hat{\zeta}_i$ are lifts of the unmodified paths $d\zeta_i$ and 
	$\hat\zeta_i^L, \hat\zeta_i^R$ 
represent the left/right rotation. Let
    \begin{equation}
        \hat{\zeta}_i^0:= \hat{\zeta}_i^L \star \hat{\zeta}_i \star \hat{\zeta}_i^R.
    \end{equation}
    We next verify explicitly that the paths 
    $(\chi_{\sigma_i})_* \circ \hat{\zeta}_i^0$ 
are composable up to homotopy:
    Let $\alpha =e^{2\pi i\frac{k_i+\eta_i}{3}}\alpha_0$. Using an $\alpha$-rotated $QR$-decomposition, see \eqref{eq:QR_alpha-def}, we obtain that
    \begin{equation}
        d\left( \chi_{\sigma_i+1}^{-1} \circ \chi_{\sigma_i} \right)_*
        \vert_{\zeta_i(1)} = e^{2\pi i\left( \frac{k_{i+1}}{3} - \frac{k_i+\eta_i}{3} + \frac{1}{2} \right)}\alpha t_i\alpha^{-1},
        \label{eq:qr.pathcompose}
    \end{equation}
    with $t_i \in \mathbf{T}_2$ an upper triangular matrix. Therefore
    \begin{equation}
        \begin{split}
            &\left( d\chi_{\sigma_{i+1}}^{-1}\circ d\chi_{\sigma_i} \right) \hat{\zeta}_i^R\left( 1 \right)  \\
            &= \left( \chi_{\sigma_{i+1}^{-1}}\circ \chi_{\sigma_i} \left( \zeta_i\left( 1 \right) \right) 
            ~,~
            \frac{2\pi}{6} e^{2\pi i\left( \frac{k_{i+1}}{3} - \frac{k_i+\eta_i}{3} + \frac{1}{2} \right)} \alpha t_i\alpha^{-1} r_0 \alpha_0 e^{2 \pi i\left( - \frac{\eta_i}{6} + \frac{k_i}{3} \right)} \right) \\
            &= \left( \zeta_{i+1}(0) \,,\, \frac{2\pi}{6}r_0 e^{2\pi i\left( \frac{k_{i+1}}{3} + \frac{1}{2} - \frac{\eta_i}{2} \right)} \alpha_0 t_i \right) \\
            &= \hat{\zeta}_{i+1}^L(0).t_i.
        \end{split}
        \label{eq:pathcompose.triang}
    \end{equation}
    Here in the first step we used equation \eqref{eq:qr.pathcompose}, in the second step we use that $e^{-\pi i\eta_i}$ is in the centre of $GL_2^+$. In the last step we use the right action of $GL_2^+$ on $F_{GL^+}(\underline\Delta )$.
    The desired homotopy is then given by right action with a path from $t_i$ to the identity matrix.
    Using these homotopies and composing we finally obtain a closed curve
    \begin{equation}
        \hat{\zeta}:[0,1]\to F_{GL^+}(\zeta)
    \end{equation}
    that is homotopic to a lift of the original curve $d\gamma$ to $F_{GL^+}(\zeta)$.

\medskip\noindent    
$\bullet$\,\emph{Lifting properties of the curve $\hat{\zeta}$}.
    Next we determine how this curve lifts to the spin bundle of $\Sigma$. We first pick spin lifts of the paths $\hat{\zeta}_i$, $\hat{\zeta}_i^L$ and $\hat{\zeta}_i^R$: Let $\tilde{\zeta}_i, \tilde{\zeta}_i^L, \tilde{\zeta}_i^R:[0,1] \to P_{\GL}(\Sigma)$ be given by
    \begin{align}
        \tilde{\zeta}_i(t) &= \left(\zeta_i(t) \,,\, \delta\!\left( \sqrt{\tfrac{2\pi}{6}} \, \tilde{r}_0 \tilde{\alpha}_0e^{\pi i \left(-\frac{\eta_it}{6} + \frac{1}{4} + \frac{ k_i}{3}\right)}\right)\right)\label{eq:path.triangle.spinlift}\\
        \tilde{\zeta}_i^L(t) &= \left(\zeta_i(0) \,,\, \delta\!\left( \sqrt{\tfrac{2\pi}{6}} \, \tilde{r}_0\tilde{\alpha}_0e^{\pi i\left(\frac{t}{4} + \frac{k_i}{3}\right)}\right)\right)\label{eq:path.leftrotate.spinlift}\\
        \tilde{\zeta}_i^R(t)&= \left(\zeta_i(1) \,,\, \delta\!\left( \sqrt{\tfrac{2\pi}{6}} \, \tilde{r}_0\tilde{\alpha}_0e^{\pi i\left(\frac{1-t}{4} - \frac{\eta_i}{6} + \frac{k_i}{3}\right)}\right)\right),\label{eq:path.rightrotate.spinlift}
    \end{align}
    with $\tilde{r}_0$ a lift of $r_0$ to $\Spin_2$ and $\tilde\alpha_0$ as defined in \eqref{eq:tildealpha_0-def}.
    These are chosen such that the compositions $\tilde{\zeta}_i^L\star \tilde{\zeta}_i \star \tilde{\zeta}_i^R$ exist.
    We write the spin transition functions as
    \begin{equation}
        g_{\zeta_i(1)}\left( \tilde{\chi}_{\sigma_{i+1}}^{-1} \circ \tilde{\chi}_{\sigma_i} \right) = \delta\!\left( s_{i,i+1}\mu_{i,i+1} e^{\pi i\left( \frac{k_{i+1}}{3} - \frac{[k_i + \eta_i]_3}{3} + \frac{1}{2} \right)} \right)
    \tilde{\alpha} \tilde{t}_i\tilde{\alpha}^{-1},
    \end{equation}
    with $\delta(\tilde{\alpha})$ being a spin lift of $\alpha =\alpha_0e^{2\pi i\left( \frac{k_i+\eta_i}{3} \right)}$, and $\tilde{t}_i\in \tilde{\mathbf{T}}_2$ the lift of $t_i$. Using this, we compute
    \begin{equation}
        \begin{split}
            &\left(\tilde{\chi}_{\sigma_{i+1}}^{-1}\circ \tilde{\chi}_{\sigma_i}\right)\left( \tilde{\zeta}_i^R (1) \right)  \\
            &= \left( \zeta_{i+1}(0)\,,\, \delta\!\left( \sqrt{\tfrac{2\pi}{6}}\,\tilde{r}_0 s_{i,i+1}\mu_{i,i+1} e^{\pi i\left( \frac{k_{i+1}-[k_i+\eta_i]_3}{3} + \frac{1}{2} \right)}\tilde{\alpha} \right)  \tilde{t}_i \delta\!\left( \tilde{\alpha}^{-1} \tilde{\alpha}_0 e^{\pi i\left( -\frac{\eta_i}{6}+ \frac{k_i}{3} \right)}  \right) \right)\\
            &=\left( \zeta_{i+1}(0) \,,\, \delta\!\left(\sqrt{\tfrac{2\pi}{6}}\,\tilde{r}_0 e^{\pi i\left( \frac{k_{i+1}}{3} \right)} s_{i,i+1}\,\mu_{i,i+1}\, e^{\pi i\left( \frac{k_i+\eta_i-[k_i+\eta_i]_3}{3} + \frac{1-\eta_i}{2} \right)} \tilde{\alpha}_0 \right)  \tilde{t}_i \right) \\
            &= \tilde{\zeta}_{i+1}^L(0).(\delta(\omega_{i,i+1})\tilde{t}_i) \ ,
        \end{split}
        \label{eq:spinpath.inner.compose.sign}
    \end{equation}
where in the second step we used that
$\delta\left(e^{-\pi i \eta_i/2}\right)$ is central, see \eqref{eq:delta-cap-Z}.
The sign $\omega_{i,i+1}$ is given by
    \begin{equation}
        \omega_{i,i+1} = s_{i,i+1}\mu_{i,i+1} e^{\pi i\left( \frac{k_{i}+\eta_i}{3} - \frac{[k_i + \eta_i]_3}{3} + \frac{1}{2}  - \frac{\eta_i}{2}\right)} \ .
    \label{eq:sign.path}
\end{equation}
    The homotopies we used to remove the $t_i$ have unique lifts, so whether $\hat{\zeta}$ has a closed spin lift depends only on the signs $\omega_{i,i+1}$.
    Multiplying these together this gives a total sign
    \begin{equation}
        \varepsilon = \prod_{i\in \mathbb{Z}_N} \omega_{i,i+1},
    \end{equation}
    i.e.\ the spin lift of the curve is closed iff $\varepsilon =1$.
\end{proof}

We now evaluate Lemma \ref{lem:path_lifting_inner_generic} for circular paths around vertices. 

\begin{corollary} \label{lem:vertex.rule.inner}
Let $v$ be an inner vertex and $(\sigma_i)_{i=0,\dots,N-1}$ 
be the triangles containing $v$,
ordered counterclockwise starting with an arbitrary triangle $\sigma_0$. Let $D$ be the number of triangles $\sigma_i$ such that $k_i=0$ 
	(i.e.\ the path enters $\sigma_i$ through the marked edge).
Let $K$ be the number of edges $e_{i,i+1}$ pointing away from $v$. Then the spin structure on $\Sigma$ extends to $v$ if and only if
    \begin{equation}
        \prod_{i \in \mathbb{Z}_N} s_{i,i+1} = (-1)^{D+K+1}.
    \label{eq:vertex.rule.inner}
    \end{equation}
\end{corollary}

\begin{proof}
    We pick a differentiable simple closed curve $\gamma$ around the vertex $v$ as in Figure \ref{fig:triang.path.vertex} that fulfils the decomposition conditions of Lemma \ref{lem:path_lifting_inner_generic}. 
    It is now a simple counting problem to reformulate equation \eqref{eq:lemma_path_lifting_inner_generic}.
    We go term by term through the factors in equation \eqref{eq:sign.path}, the definition of $\omega_{i,i+1}$:
    Since we chose the sequence of triangles counterclockwise we have $\eta_i=-1$ for all 
	$i=0,\dots,N-1$
    and thus get a minus sign from the $\frac{1-\eta_i}{2}$ term for each edge. 
    The sign $\mu_{i,i+1}$ is $-1$ if the edge points to $v$. We get an extra minus sign if $k_{i}+\eta_i \neq [k_i + \eta_i]_3$ which happens if $k_i=0$ or equivalently $\chi_{\sigma_i}(1)=v$.
    Collecting these together, we see that the curve has a closed lift if and only if
\begin{equation}
    \prod_{i\in \mathbb{Z}_N}  s_{i,i+1} (-1)^{D+K}=1.
\end{equation}
Since the curve $\gamma$ also fulfils the conditions of Lemma \ref{lem:ssc} the spin structure can be extended if and only if this curve does not admit a closed lift to the spin bundle.
\end{proof}

\begin{figure}[tb]
    \centering
    \raisebox{7em}{a)}~~~
    \begin{tikzpicture}
        \node[name=A1, regular polygon, regular polygon sides=7,minimum size=3cm] at (0cm,0cm){};
        \foreach \N/\M in {1/2,2/3,3/4,7/1} \draw (A1.corner \N) -- (A1.corner \M);
        \foreach \N/\M in {4/5,5/6,6/7} \draw[dashed] (A1.corner \N) -- (A1.corner \M);
        \foreach \N in {1,2,3,4,7} \draw (A1.corner \N) -- (A1.center);
        \foreach \N in {5,6} \draw[dashed] (A1.corner \N) -- (A1.center);
        \foreach \N/\M/\K in {1/2/0,2/3/1,3/4/2} \node at ($0.33*(A1.corner \N)+0.33*(A1.corner \M)+0.33*(A1.center)$){$\sigma_\K$};
        \node at ($0.33*(A1.corner 7)+0.33*(A1.corner 1)+0.33*(A1.center)+(2pt,0pt)$){$\sigma_{N-1}$};
 
        \node[below right] at (A1.center) {\tiny{$v$}};
        \draw [postaction={decorate, decoration={markings, mark=between positions 0 and 1 step 0.34 with {\arrow{stealth}}}}] (0,0) circle (0.5cm);
    \end{tikzpicture}
    \hspace{6em}
    \raisebox{7em}{b)}~~~
    \begin{tikzpicture}
        \node[name=A2, regular polygon, regular polygon sides=7, minimum size=2.9cm] at (0cm,0cm){};
        \foreach \N/\M in {1/2,2/3,3/4,7/1} \draw[green!40, line width=3pt] (A2.corner \N) -- (A2.corner \M);
        \foreach \N/\M in {4/5,5/6,6/7} \draw[dashed, green!40, line width=3pt] (A2.corner \N) -- (A2.corner \M);
        \node[name=A1, regular polygon, regular polygon sides=7, minimum size=3cm] at (0cm,0cm){};
        \foreach \N/\M in {1/2,2/3,3/4,7/1} \draw (A1.corner \N) -- (A1.corner \M);
        \foreach \N/\M in {4/5,5/6,6/7} \draw[dashed] (A1.corner \N) -- (A1.corner \M);
        \begin{scope}[decoration={markings, mark=at position 0.5 with {\arrow{stealth}}}]
            \foreach \N in {1,2,3,4,7} \draw[postaction=decorate] (A1.corner \N) -- (A1.center);
            \foreach \N in {5,6} \draw[dashed, postaction=decorate] (A1.corner \N) -- (A1.center);
        \end{scope}
        \foreach \N/\M/\K in {1/2/0,2/3/1,3/4/2} \node at ($0.33*(A1.corner \N)+0.33*(A1.corner \M)+0.33*(A1.center)$){$\sigma_\K$};
        \node at ($0.33*(A1.corner 7)+0.33*(A1.corner 1)+0.33*(A1.center)+(2pt,0pt)$){$\sigma_{N-1}$};
        \node[below right] at (A1.center) {\tiny{$v$}};
        \draw [postaction={decorate, decoration={markings, mark=between positions 0 and 1 step 0.34 with {\arrow{stealth}}}}] (0,0) circle (0.5cm);
    \end{tikzpicture}
    \caption{a) Simple closed curve around a vertex $v$; the triangles containing $v$ are ordered counter-clockwise. b) A possible marking of the triangles; for this choice, the statement of Corollary \ref{lem:vertex.rule.inner} becomes particularly easy: the spin structure extends iff $\prod_{i=0}^{N-1} s_{i,i+1} = -1$.}
   \label{fig:triang.path.vertex}
\end{figure}
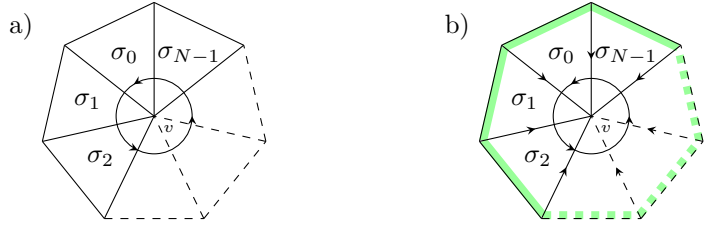

\subsubsection{Paths at the boundary}

We aim to get a rule similar to corollary \ref{lem:vertex.rule.inner} for boundary vertices. For $v$ a boundary vertex we label the surrounding edges and triangles as in Figure \ref{fig:boundary.vertex.labelling}. 
	As in \eqref{eq:sii+1-def} let $s_{i,i+1}=s(e_{i,i+1})$ 
for all $i=0,\dots,N$, with $s_{N,N+1}\equiv s_{N,0}$.

\begin{figure}[tb]
    \centering
    \begin{tikzpicture}
        \node[name=Dc, regular polygon, regular polygon sides=3, minimum size=2cm, rotate=30] at (0cm,0cm){}; 
        \node[name=Cc, circle, minimum size=2cm, draw] at (0cm,0cm){};
        \node[name=Cdash, circle, minimum size=0.7cm, draw, dashed] at (0cm,0cm){};
        \node[name=A, regular polygon, regular polygon sides=7, minimum size=4cm, rotate=12.9] at (Dc.corner 3){}; 
        \foreach \N in {1,2,3} \node[inner sep=0.02cm, fill, circle] at (Dc.corner \N){};
        \foreach \N in {1,4,5} \draw (A.corner \N) -- (Dc.corner 3);
        \foreach \N in {6,7} \draw[dashed] (A.corner \N) -- (Dc.corner 3);
        \draw (A.corner 1) -- (Dc.corner 1);
        \draw (A.corner 4) -- (Dc.corner 2);
        \foreach \N/\M in {4/5} \draw (A.corner \N) -- (A.corner \M);
        \foreach \N/\M in {5/6,6/7,7/1} \draw[dashed] (A.corner \N) -- (A.corner \M);
        \node at (0,0) {\small{$\partial\Sigma_i$}};
        \node[above right=-2pt and 0pt] at (Dc.corner 3) {\tiny{$v$}};
        \begin{scope}[decoration={markings, mark=at position 0.5 with {\arrow{stealth}}}]
            \draw[postaction=decorate] (Dc.corner 3) ++(-120:0.8cm) arc (-120:120:0.8cm);
            \draw[postaction=decorate] (Dc.corner 3) ++(120:0.8cm) .. controls ($(Dc.corner 3)+(130:0.8cm)$) and ($(Dc.corner 3)+(-0.1cm,0.1cm)$) .. ($(Dc.corner 3)+(-0.1cm,0cm)$);
            \draw[postaction=decorate] ($(Dc.corner 3)+(-0.1cm,0cm)$) .. controls ($(Dc.corner 3)+(-0.1cm,-0.1cm)$) and ($(Dc.corner 3)+(-130:0.8cm)$) .. ($(Dc.corner 3)+(-120:0.8cm)$);
        \end{scope}
        \node[above] at ($0.33*(A.corner 1)+0.33*(A.corner 2)+0.33*(A.center)+(1pt,0pt)$) {\small{$\sigma_{N}$}};
        \node[below] at ($0.33*(A.corner 3)+0.33*(A.corner 4)+0.33*(A.center)+(1pt,0pt)$) {\small{$\sigma_{1}$}};
        \node[below=-2pt] at ($0.33*(A.corner 4)+0.33*(A.corner 5)+0.33*(A.center)+(1pt,0pt)$) {\small{$\sigma_{2}$}};
        \node[above] at ($(Dc.corner 1)!0.5!(Dc.corner 3)$) {\tiny{$e_{N0}$}};
        \node[below] at ($(Dc.corner 2)!0.5!(Dc.corner 3)$) {\tiny{$e_{01}$}};
        \node[right=-2pt] at ($(A.corner 4)!0.5!(A.center)$) {\tiny{$e_{12}$}};
        \node[right=-2pt] at ($(A.corner 5)!0.5!(A.center)$) {\tiny{$e_{23}$}};
    \end{tikzpicture}
    \caption{The labelling of triangles and edges around a boundary vertex $v$. The triangles $(\sigma_i)_{i=1,\dots,N}$ are ordered counterclockwise, starting at the boundary. Edges are given by $e_{i,i+1}=\sigma_i\cap \sigma_{i+1}$ for $i=1,\dots,N-1$. $e_{0,1}$ and $e_{N,0}$ are boundary edges of $\sigma_1$ and $\sigma_N$ respectively. Also shown is a sketch of the simple closed curve we use.}
   \label{fig:boundary.vertex.labelling}
\end{figure}
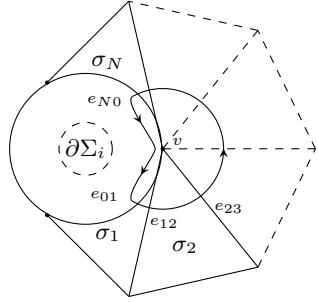

\begin{lemma}\label{lem:vertex.rule.boundary}
    For a vertex $v$ on the boundary component $i$ let $D_R$ be the number of triangles $\sigma_i$ with $\chi_{\sigma_i}(1)=v$ or equivalently $d^2_0(\sigma_i)=e_{i-1,i}$. If $\varphi_i(1)=v$ let $D_{NS}=D_R+1$ and let $D_{NS}=D_R$ otherwise. Let $K$ be the number of edges $e_{i,i+1}$ pointing away from $v$, counting the boundary edges. The spin structure on $\Sigma^+$ extends to $v$ if and only if
    \begin{equation}\label{eq:boundary-vertex-rule}
        \prod_{i=0}^N s_{i,i+1} = (-1)^{D_{R/NS}+K+1}
    \end{equation}
    for a $R/NS$-type boundary respectively.
\end{lemma}

\begin{proof}
    We first describe the marking on the relevant edges and triangles more explicitly:
    \begin{itemize}
        \item $k_i\in \{0,1,2\}$ is fixed by $e_{i-1,i}=d^2_{k_i}(\sigma_i)$ for $i=1,\dots,N$ as in Section \ref{sec:path_argument_inner}.
        \item $k_0\in \{0,1,2\}$ is the position of $e_{N,0}$ on the boundary under the boundary inclusion map $f_i$,
        		i.e.\ $f_i(e_{k_0}) = e_{N,0}$, where $e_0, e_1, e_2$ are the edges of $\underline\Delta$.
        \item $\eta_i=-1$ for $i=1,\dots,N$ which agrees with the definition in Section \ref{sec:path_argument_inner} since we ordered the triangles counterclockwise.
        \item $\mu_{i,i+1}=\pm 1$
            for $i=1,\dots,N-1$ 
            are as defined in Section \ref{sec:path_argument_inner}.
    \end{itemize}
    As was the case for inner vertices, we determine the lifting behaviour of a specific curve from the frame bundle to the spin bundle via the edge signs.
We then relate it to the known spin lifting behaviour of contractible simple closed curves.
    
\medskip\noindent    
$\bullet$\,\emph{Construction of a suitable curve in the frame bundle.} 
    As in the proof of Lemma \ref{lem:path_lifting_inner_generic} we first construct a simple closed curve around the vertex $v$.
    For the boundary part we need a new path segment. 
    Let 
    \begin{equation}
        \begin{split}
        \zeta_0:[0,1] &\to U_i^+\\
        t&\mapsto e^{2\pi i\left( - \frac{t}{3} + \frac{1}{6} + \frac{k_0}{3} \right)}.
        \end{split}
    \end{equation}
    The path $\zeta_0$ now intersects $\varphi_i^{-1}(v)$, the preimage of the vertex $v$. To avoid this we exchange it by an isotopic path in the unit disc that still lies in $U_i^+$. Its explicit form is not relevant as we are only interested in lifting properties of the path segments.
    The segments $\zeta_i:[0,1]\to \underline\Delta$ for $i=1,\dots,N$ are as in equation \eqref{eq:triangle.path.segment}. The composition
    \begin{equation}
        \zeta:= \left( \varphi_i\circ \zeta_0 \right) \star \left( \chi_{\sigma_1}\circ \zeta_1 \right) \star \dots \star \left( \chi_{\sigma_N}\circ \zeta_N \right).
    \end{equation}
    is then a simple closed curve. By the same procedure as in the proof of Lemma \ref{lem:path_lifting_inner_generic} (see also Figure \ref{fig:triang.path_smooth}) we obtain a smooth simple closed curve. Its velocity curve has a lift to the frame bundle and we describe a curve homotopic to this lift explicitly by implementing the kinks as rotations in the frame bundle.
    Let $\hat{\zeta}_i$, $\hat{\zeta}_i^0$,  $\hat{\zeta}_i^{L/R}$ ($i = 1,\dots,N-1$) and $\hat{\zeta}_N^{L}$ be as in the proof of Lemma \ref{lem:path_lifting_inner_generic}. Let
    \begin{equation}
            \hat{\zeta}_N^{R\prime}:[0,1] \to \underline\Delta \times GL_2^+ 
            \quad , \quad
            t \mapsto \left( \zeta_N(1), \frac{2\pi i}{6} r_0\alpha_0 e^{2\pi i \left( \frac{t}{4} - \frac{\eta_N}{6} + \frac{k_N}{3} \right)} \right)
    \end{equation}
    and set 
	$\hat{\zeta}_N^{0\prime}:= \hat{\zeta}_N^L\star \hat{\zeta}_N\star \hat{\zeta}_N^{R\prime}$.
    For the boundary part we let 
    \begin{equation}
            \hat{\zeta}_0:[0,1] \to U_i^+ \times GL_2^+
            \quad , \quad
            t \mapsto \left( \zeta_0(t), - \frac{2\pi i}{3} e^{2\pi i\left( -\frac{t}{3}+ \frac{1}{6} + \frac{k_0}{3} \right)} \right) \ .
    \end{equation}
    The isotopy we used to make $\zeta_0$ avoid $\varphi_i^{-1}(v)$ lifts to the frame-bundle, so we can change $\hat{\zeta}_0$ in the same way.
    
    By \eqref{eq:pathcompose.triang} we already know that the paths 
	$(d\chi_{\sigma_i})_*\hat{\zeta}_i$ and $(d\chi_{\sigma_{i+1}})_*\hat{\zeta}_{i+1}$ 
are composable up to right action by an upper triangular matrix for $i=1,\dots,N-1$. We proceed to show that the same holds for $(d\chi_{\sigma_N})_*\hat{\zeta}_N'$ and $(d\varphi_i)_*\hat{\zeta}_0$, as well as for $(d\varphi_i)_*\hat{\zeta}_0$ and 
	$(d\chi_{\sigma_1})_*\hat{\zeta}_1$.

Let $\alpha =\alpha_0 e^{2\pi i\left( \frac{k_N+\eta_N}{3} \right)}$.
\begin{equation}
    \begin{split}
    &d\left( \varphi_i^{-1}\circ \chi_{\sigma_N} \right)_*
    \left( \hat{\zeta}_N^{R\prime} \left( 1 \right) \right)  \\
    &=\left( \left( \varphi_i^{-1}\circ \chi_{\sigma_N} \right)(\zeta_N\left( 1 \right)), 
\frac{2\pi}{6}e^{2\pi i\left( \frac{k_0}{3}- \frac{k_N+\eta_N}{3} + \frac{1}{2} \right)} \alpha  t_N \alpha^{-1}r_0 \alpha_0 i^2 e^{2\pi i\left( -\frac{\eta_N}{6} + \frac{k_N}{3} \right)} \right)\\
    &=\left( \zeta_0(0), \frac{2\pi}{6}r_0 e^{2\pi i\left( \frac{k_0}{3} + \frac{1}{2} + \frac{1}{4} + \frac{1}{6} \right)} t_N \right)\\
    &= \left( \zeta_0(0), -\frac{2 \pi i}{3}
 e^{2\pi i\left( \frac{k_0}{3} + \frac{1}{6} \right)} \right).\frac{r_0t_N}{2} = \hat{\zeta}_0(0).\frac{r_0t_N}{2}.
    \end{split}
\end{equation}
Here we used the explicit value of $\alpha_0= e^{2\pi i\left( \frac{1}{4} + \frac{1}{6} \right)}$ in the
	second
step, and introduced a matrix $t_N\in \mathbf{T}_2$. 
Now let $\alpha =\alpha_0e^{2\pi i \frac{k_1}{3} }$.
\begin{equation}
    \begin{split}
        &d\left( \varphi_i^{-1} \circ \chi_{\sigma_1} \right)_* \left( \hat{\zeta}_1^L(0) \right) =\\
        &= \left( \left( \varphi_i^{-1} \circ \chi_{\sigma_1} \right) \left( \zeta_1(0) \right), \frac{2\pi}{6}e^{2\pi i\left(
			\frac{k_0-1}{3}
         - \frac{k_1}{3} + \frac{1}{2} \right)} \alpha t_0 \alpha^{-1} r_0 \alpha_0 e^{2\pi i\frac{k_1}{3}} \right)\\
        &= \left( \zeta_0\left( 1 \right), -\frac{2\pi i}{3} e^{2\pi i\left( \frac{k_0-1}{3} + \frac{1}{6} \right)} \right).\frac{r_0t_0}{2} = \hat{\zeta}_0\left( 1 \right).\frac{r_0t_0}{2}.
    \end{split}
\end{equation}
Here we used again the value of $\alpha_0$, and introduced a matrix $t_0\in \mathbf{T}_2$. As in Section \ref{sec:path_argument_inner} we pick paths in $\mathbf{T}_2$ from $t_i$, $i=0,\dots,N$ to the identity matrix, as well as from $r_0/2$ to $1$,
and use these to obtain a closed curve 
\begin{equation}
    \hat{\zeta}:[0,1]\to F_{GL^+}(\Sigma ).
\end{equation}

\noindent    
$\bullet$\,\emph{Lifting $\hat{\zeta}$ to the spin bundle}.
In order to give an explicit lift of $\hat{\zeta}$ to the spin bundle of $\Sigma $, we pick spin lifts of $\hat{\zeta}_i, \hat{\zeta}_i^L$ and $\hat{\zeta}_i^R$ as in equations \eqref{eq:path.triangle.spinlift}--\eqref{eq:path.rightrotate.spinlift}, and of $\hat{\zeta}_N^{R\prime}$ and $\hat{\zeta}_0$ as follows:
\begin{equation}
    \tilde{\zeta}_N^{R\prime}(t) = \left( \zeta_N\left( 1 \right) \,,\, \delta\!\left( \sqrt{\tfrac{2\pi}{6}} \, \tilde{r}_0\tilde{\alpha}_0 e^{\pi i\left( \frac{t}{4} + \frac{1}{4}
		-\frac{\eta_N}{6} 
     + \frac{k_N}{3} \right)} \right) \right) \ .
\end{equation}
The spin lift of $\hat{\zeta}_0$ depends on the given spin structure on the boundary component $i$, see Section \ref{sec:ss_on_C}. We choose
\begin{equation}
    \tilde{\zeta}_0^{NS}(t) = \left( \zeta_0(t)\,,\, \delta\!\left( \sqrt{\tfrac{2\pi}{3}} \, e^{\pi i\left( -\frac{t}{3} + \frac{1}{6} + \frac{k_0}{3} - \frac{1}{4} \right)} \right) \right)
\end{equation}
and 
\begin{equation}
    \tilde{\zeta}_0^R(t) = \left( \zeta_0(t)\,,\, \delta\!\left(  \sqrt{\tfrac{2\pi}{3}}e^{-\frac{\pi i}{4}} \right) \right) \ ,
\end{equation}
respectively.
By \eqref{eq:spinpath.inner.compose.sign}, the composition of the spin lifts of $\left(\tilde{\chi}_{\sigma_i}\right)_*\tilde{\zeta}_i$ yields signs $\omega_{i,i+1}$ for $i=1,\dots,N-1$.
For the boundary path $\hat{\zeta}_0$ we have to distinguish between $NS$ and $R$-type spin structure on the boundary.

\medskip
\noindent
{\em NS--type boundary:}
We write the spin transition function at $e_{N,0}$ as
\begin{equation}
    g_{\zeta_N\left( 1\right)}\left( \tilde{\varphi}_i^{-1} \circ \tilde{\chi}_{\sigma_N} \right) = \delta\!\left( s_{N,0} e^{\pi i\left( \frac{k_0}{3} - \frac{[k_N-1]_3}{3} + \frac{1}{2} \right)} \right) \delta(\tilde{\alpha}) \tilde{t}_N \delta(\tilde{\alpha})^{-1},
\end{equation}
with $\delta(\tilde{\alpha})$ being a spin lift of $\alpha = \alpha_0e^{2\pi i\left( \frac{k_N-1}{3} \right)}$ and $\tilde{t}_N\in \tilde{\mathbf{T}}_2$  the lift of $t_N$. Then
\begin{equation}
    \begin{split}
        &\left( \tilde{\varphi}_i^{-1} \circ \tilde{\chi}_{\sigma_N} \right) \left(  \tilde{\zeta}_N^{R\prime} \left( 1 \right) \right) =\\
        &= \left( \zeta_0(0) \,,\, \delta\!\left(\sqrt{\tfrac{2\pi}{6}}\,s_{N,0} e^{\pi i\left( \frac{k_0}{3} - \frac{[k_N-1]_3}{3} + \frac{1}{2} \right)}  \tilde{\alpha} \right) \tilde{t}_N  \delta\!\left( \tilde{\alpha}^{-1} \tilde{r}_0\tilde{\alpha}_0 e^{\pi i\left( \frac{1}{2} + \frac{1}{6} + \frac{k_N}{3} \right)} \right) \right)\\
        &= \left( \zeta_0(0) \,,\, \delta\!\left( \sqrt{\tfrac{2\pi}{3}}\, e^{\pi i\left( \frac{k_0}{3} + \frac{1}{6} -\frac{1}{4} \right)} \right) \delta\!\left( \frac{\omega_{N,0} \tilde{r}_0}{\sqrt{2}} \right) \tilde{t}_N \right)
        = \tilde{\zeta}_0^{NS}(0)\, .\,  \delta\!\left( \frac{\omega_{N,0} \tilde{r}_0}{\sqrt{2}} \right) \tilde{t}_N  \ .
    \end{split}
\end{equation}
Here $\omega_{N,0}$ is a sign,
\begin{equation}
    \omega_{N,0} = s_{N,0} \, e^{\pi i\left( \frac{k_N-1}{3} - \frac{[k_N-1]_3}{3} \right)} \ .
\end{equation}
We write the spin transition function at $e_{0,1}$ as
\begin{equation}
    g_{\zeta_1(0)}\left( \tilde{\varphi}_i^{-1} \circ \tilde{\chi}_{\sigma_1} \right) = \delta\!\left(s_{0,1} \,  e^{\pi i\left( \frac{[k_0-1]_3}{3} - \frac{k_1}{3} + \frac{1}{2} \right)} \right) \delta(\tilde{\alpha}) \tilde{t_0} \delta(\tilde{\alpha})^{-1}
\end{equation}
with $\delta(\tilde{\alpha})$ a spin lift of $\alpha =\alpha_0e^{2\pi i\frac{k_1}{3}}$ and $\tilde{t}_0\in \tilde{\mathbf{T}}_2$ the spin lift of $t_0$.
Then 
\begin{equation}
    \begin{split}
        &\left( \tilde{\varphi}_i^{-1}\circ \tilde{\chi}_{\sigma_1}\right)  \left(  \tilde{\zeta}_1^L(0) \right) \\
        &=\left( \zeta_0\left( 1 \right) \,,\, \delta\!\left(\sqrt{\tfrac{2\pi}{6}}\,s_{0,1}  e^{\pi i\left( \frac{[k_0-1]_3}{3} - \frac{k_1}{3} + \frac{1}{2} \right)} \tilde{\alpha} \right)  \tilde{t}_0  \, \delta\!\left( \tilde{\alpha}^{-1} \tilde{r}_0 \tilde{\alpha}_0 e^{\pi i\frac{k_1}{3}} \right) \right)\\
        &= \left( \zeta_0\left( 1 \right)\,,\, \delta\!\left( \sqrt{\tfrac{2\pi}{3}}\, e^{\pi i\left( \frac{k_0-1}{3} + \frac{1}{6} -\frac{1}{4} \right)}  \frac{\omega_{0,1} \tilde{r}_0 }{\sqrt{2}} \right)\tilde{t}_0\right)  \\
        &=  \tilde{\zeta}_0(1) \,.\, \delta\!\left(  \frac{\omega_{0,1} \tilde{r}_0 }{\sqrt{2}} \right)\tilde{t}_0 \ .
    \end{split}
\end{equation}
Here $\omega_{0,1}$ is a sign,
\begin{equation}
    \omega_{0,1}= - s_{0,1} \, e^{\pi i\left( \frac{[k_0-1]_3}{3} - \frac{k_0-1}{3} \right)} \ .
\end{equation}
Let
\begin{equation}
    \varepsilon =\left( \prod_{i=1}^{N-1} \omega_{i,i+1} \right) \omega_{N,0}\omega_{0,1}.
    \label{eq:total.sign.ns.path}
\end{equation}
A spin lift of $\hat{\zeta}$ is closed iff $\varepsilon =1$.
Using the numbers $D_{NS}$ and $K$ defined above and counting the signs as in the proof of Lemma \ref{lem:vertex.rule.boundary} we can reformulate equation \eqref{eq:total.sign.ns.path} as
\begin{equation}
    \varepsilon = \left( \prod_{i=0}^{N} s_{i,i+1}  \right) (-1)^{D_{NS}+K}.
\end{equation}

\medskip
\noindent
{\em R--type boundary:}
We write the spin transition function at $e_{N,0}$ as
\begin{equation}
    g_{\zeta_N\left(  1 \right)}\left( \tilde{\varphi}_i^{-1} \circ \tilde{\chi}_{\sigma_N} \right) = \delta\!\left( s_{N,0}  \, e^{\pi i\left( -\frac{[k_N-1]_3}{3} - \frac{1}{6} + \frac{1}{2} \right)}\right) \delta(\tilde{\alpha}) \tilde{t}_N \delta(\tilde{\alpha})^{-1} \ ,
\end{equation}
with $\delta(\tilde{\alpha})$ a spin lift of $\alpha =\alpha_0e^{2\pi i \frac{k_N-1}{3} }$ and $\tilde{t}_N \in \tilde{\mathbf{T}}_2$ the spin lift of $t_N$.
Then 
\begin{equation}
    \begin{split}
       & \left( \tilde{\varphi}^{-1}\circ \tilde{\chi}_{\sigma_N} \right) \left(  \tilde{\zeta}_N^{R\prime} \left(  1 \right) \right) = \\
       &= \left( \zeta_0(0) \,,\, \delta\!\left(\sqrt{\tfrac{2\pi}{6}}\,s_{N,0}  e^{\pi i\left( -\frac{[k_N-1]_3}{3} - \frac{1}{6} + \frac{1}{2} \right)} \tilde{\alpha} \right) \tilde{t}_N \delta\!\left( \tilde{\alpha}^{-1} \tilde{r}_0\tilde{\alpha}_0  e^{\pi i\left( \frac{1}{2} + \frac{1}{6} + \frac{k_N}{3} \right)} \right) \right) \\
       &= \left( \zeta_0(0)\,,\, \delta\!\left(\sqrt{\tfrac{2\pi}{3}} \, e^{-\frac{\pi i}{4}} \frac{\omega_{N,0}  \tilde{r}_0 }{\sqrt{2}} \right)  \tilde{t}_N \right) = \tilde{\zeta}_0(0)\,.\, \delta\!\left( \frac{\omega_{N,0}  \tilde{r}_0 }{\sqrt{2}} \right)  \tilde{t}_N  \ .
    \end{split}
\end{equation}
Here 
\begin{equation}
    \omega_{N,0} = s_{N,0} \, e^{\pi i\left( \frac{k_N-1}{3} - \frac{[k_N-1]_3}{3} \right)} \ .
\end{equation}
We write the spin transition function at $e_{0,1}$ as
\begin{equation}
    g_{\zeta_1(0)}\left( \tilde{\varphi}_i^{-1}\circ \tilde{\chi}_{\sigma_1} \right) = s_{0,1} \delta\!\left( e^{\pi i\left( -\frac{k_1}{3} - \frac{1}{6} + \frac{1}{2} \right)} \right)\delta(\tilde{\alpha}) \tilde{t}_N \delta(\tilde{\alpha})^{-1} \ ,
\end{equation}
with $\delta(\tilde{\alpha})$ a spin lift of $\alpha =\alpha_0e^{2\pi i \frac{k_1}{3}}$ and $\tilde{t}_0\in \tilde{\mathbf{T}_2}$ the spin lift of $t_0$. Then
\begin{equation}
    \begin{split}
        &\left( \tilde{\varphi}_i^{-1}\circ \tilde{\chi}_{\sigma_1} \right) \left( \tilde{\zeta}_1^L(0) \right) =\\
        &=\left( \zeta_0\left( 1 \right)\,,\, \delta\!\left(\sqrt{\tfrac{2\pi}{6}}\,s_{0,1}  e^{\pi i\left( -\frac{k_1}{3} - \frac{1}{6} + \frac{1}{2} \right)} \tilde{\alpha}\right)  \tilde{t}_0 \delta\!\left( \tilde{\alpha}^{-1} \tilde{r}_0 \tilde{\alpha}_0 e^{\frac{\pi ik_1}{3}} \right)\right)\\
        &= \left( \zeta_0\left(1  \right) \,,\, \delta\!\left(\sqrt{\tfrac{2\pi}{3}}\, e^{-\frac{\pi i}{4}}  \frac{\omega_{0,1} \tilde{r}_0}{\sqrt{2}} \right) \tilde{t}_0 \right) 
        = \tilde{\zeta}_0^R\left(  1 \right)\,.\,\delta\!\left(\frac{\omega_{0,1} \tilde{r}_0}{\sqrt{2}} \right) \tilde{t}_0 
    \end{split}
\end{equation}
with $\omega_{0,1}= -s_{0,1}$. Let
\begin{equation}
    \varepsilon = \left(\prod_{i=1}^{N-1}\omega_{i,i+1}\right) \omega_{N,0} \omega_{0,1}.
\end{equation}
A spin lift of $\hat{\zeta}$ is closed iff $\varepsilon =1$, and we can rewrite $\varepsilon $ as
\begin{equation}
    \varepsilon = \left( \prod_{i=0}^{N} s_{i,i+1} \right) (-1)^{K+D_{R}}.
\end{equation}

Since $\hat{\zeta}$ is homotopic to a lift of the derivative of a differentiable simple closed curve we can apply Lemma \ref{lem:ssc}. We then get that the spin structure can be extended if and only if the lift of $\hat{\zeta}$ is not closed, i.e. $\varepsilon=-1$.
\end{proof}

\subsection{Spin structures and admissible edge signs}
\label{sse:spin_reconstruct}
For this section, let us fix a marked triangulated surface 
$$
	\Sigma = ((\mathcal{C},f_i,d^1_0,d^2_0), (\varphi,\chi), (\Sigma,\varphi_i)) \ .
$$

\begin{definition}\label{def:admissible-signs}
An edge sign assignment $s : \mathcal{C}_1 \to \{\pm1 \}$ is called {\em admissible} if
\begin{enumerate}
\item condition \eqref{eq:vertex.rule.inner} is satisfied at each inner vertex,
\item for each boundary component, all three vertices on that boundary component satisfy \eqref{eq:boundary-vertex-rule} for $NS$-type boundary conditions, or all three vertices satisfy \eqref{eq:boundary-vertex-rule} for $R$-type boundary conditions.
\end{enumerate}
Depending on the situation in (2), we call a boundary component of {\em of $NS$-type} or {\em of $R$-type}.
\end{definition}

Given admissible edge signs $s$,  we are going to construct a spin structure on $\Sigma$ in two steps. First, we use the edge signs to construct the spin structure on $\Sigma \setminus \{ \text{vertices}\}$ -- this part works without conditions on the edge signs $s$. Then we use the admissibility condition to extend the spin structure to the vertices of the triangulation. We will denote the resulting spin structure as
\begin{equation} \label{eq:S(Sig,s)-def}
	S(\Sigma,s) \ .
\end{equation}
By construction, $S(\Sigma,s)$ is a spin triangulated surface. 

\medskip

We now give the detailed construction.
For each face $\sigma \in \mathcal{C}_2$, pick a smooth extension $\chi_\sigma^+$ of $\chi_\sigma$ to some ($\sigma$-dependent) open neighbourhood $\underline\Delta^{+,\sigma}$ of the standard triangle $\underline\Delta$. 

\medskip\noindent
{\em Open cover with only trivial triple intersections:}  Let $\mathcal{I} = \mathcal{C}_2 \cup \{1,\dots,B\}$. Recall from Section \ref{sec:lifting-prop} the construction of $\Sigma^+$ and 
that $\varphi(\mathcal{C}_0) \subset \Sigma^+$ is the set of images of the vertices under the triangulation map $\varphi$. We will construct an open cover $(V_\alpha)_{\alpha \in \mathcal{I}}$ of $\Sigma^+ \setminus \varphi(\mathcal{C}_0)$ such that non-empty intersections $V_\alpha \cap V_\beta$ are contractible, and such that $V_\alpha \cap V_\beta \cap V_\gamma = \emptyset$ whenever $\alpha,\beta,\gamma$ are pairwise distinct.

Around the image in $\Sigma^+$ of each edge $e \in \mathcal{C}_1$ minus its endpoints choose a contractible open neighbourhood $W_e$ not containing any vertices. By shrinking $W_e$ if necessary, we may assume that 
\begin{itemize}
\item $W_e \cap W_{e'} = \emptyset$ for $e \neq e'$, 
\item for each face $\sigma$ and each edge $e$ on the boundary of $\sigma$, $W_e$ is contained in the image of the extended map $\chi_\sigma^+$,
\item for each boundary edge $e$ on the $i$'th boundary component, $W_e$ is contained in $U^+_i$ (considered as a subset of the quotient surface $\Sigma^+$).
\end{itemize}
For $\sigma \in \mathcal{C}_2$ we set 
\begin{equation}
	V_\sigma = \underline{\mathring\Delta} \cup \bigcup_{e \in \partial(\sigma)} (\chi^+_\sigma)^{-1}(W_e) \ .
\end{equation}
Here, $\underline{\mathring\Delta}$ denotes the interior of $\underline\Delta$. For $i \in \{1,\dots,B\}$ we set
\begin{equation}
	V_i = \{ z \in \Cb | r<|z|<1 \} \cup \bigcup_{e \in \mathrm{im}(f_i)} \varphi_i^{-1}(W_e) \ .
\end{equation}
	We will identify the $V_\sigma$ and $V_i$ with their images in $\Sigma^+$. The $V_\sigma$ and $V_i$ then
give a cover with the desired properties.

\medskip\noindent
{\em Spin structure on the collared surface less the vertices:}
We define a spin structure on $\Sigma^+ \setminus \varphi(\mathcal{C}_0)$
via the atlas $(V_\alpha)_{\alpha \in \mathcal{I}}$. For $\sigma \in \mathcal{C}_2$ fix $\tilde V_\sigma = \Cb^{NS}|_{V_\sigma}$. For $i \in \{1,\dots,B\}$ we take $\tilde V_i = \Cb^{NS/R}|_{V_i}$ depending on whether the $i$'th boundary is of $NS$ or $R$ type.

For an inner edge $e \in \mathcal{C}_1$ fix the spin lift of the transition function 
\begin{equation}
f_e : (\chi^+_{\sigma_R})^{-1}(W_e) \to (\chi^+_{\sigma_L})^{-1}(W_e)
\end{equation}
to be $\tilde f_e(z,g) = (f_e(z),g_z g)$ where $g_z \in \GL_2$ is uniquely determined by the requirement that (i) $p_{GL}(g_z) = (df_e)_z$, and that (ii) on a point $p$ on the preimage of the edge $e$ we have
\begin{equation}
        \delta(s(e)) = \tilde{\mathbf{Q}}_{\alpha}(g_p) \cdot\delta (e^{-\pi i(k_L/3-k_R/3+1/2)}) \ ,
\end{equation}
i.e.\ the rule in \eqref{eq:se-inner_edge} is satisfied.

For a boundary edge $e$ on the $i$'th boundary component we fix the spin lift of the transition function 
$f_e : (\chi^+_{\sigma_R})^{-1}(W_e) \to \varphi_i^{-1}(W_e)$
to be $\tilde f_e(z,g) = (f_e(z),g_z g)$, where now $g_z$ is characterised as follows. 
Let $p= \chi_{\sigma_R}^{-1} \circ\varphi_i (e^{2\pi i (\frac{k_L}{3} + \frac{1}{6})})$ as in Section \ref{sec:edge-sign-bnd}. The transformation $g_z \in \GL_2$ is uniquely determined by $p_{GL}(g_z) = (df_e)_z$ and by demanding that at the point $p$ we have
    \begin{equation}
        \delta(s(e)) = \tilde{\mathbf Q}_\alpha(  g_p ) \cdot 
\begin{cases}
\text{$NS$ type} :&         
\delta\big(e^{-\pi i ( \frac{k_L}{3}-\frac{k_R}{3} +\frac{1}{2} )}\big)
\\
\text{$R$ type} :&
 \delta \big( e^{-\pi i ( -\frac{k_R}{3} - \frac{1}{6}  + \frac{1}{2} )} \big) \ .
\end{cases}
\end{equation}
This is the rule stated in \eqref{eq:bnd-edge-sign-def}.

Since there are no non-trivial triple overlaps, there is no cocycle condition on the spin transition functions. Hence the above assignment defines a spin structure on $\Sigma^+ \setminus \varphi(\mathcal{C}_0)$.

\medskip\noindent
{\em Extending the spin structure to the entire collared surface:}
If the extension of the spin structure to the vertices $\varphi(\mathcal{C}_0)$ exists, it is unique. The conditions for extendibility are stated in Corollary \ref{lem:vertex.rule.inner} and Lemma \ref{lem:vertex.rule.boundary}.
Since we assumed that $s$ it admissible, we do indeed obtain a spin structure on $\Sigma$.

\medskip

Let now 
$$
	\Lambda = ((\mathcal{C},f_i,d^1_0,d^2_0), (\varphi, \tilde\chi_\sigma) ,(\Lambda,\tilde\varphi_i))
$$
be a spin triangulated surface. Denote the underlying marked triangulated surface by $\underline\Lambda$. Let $s_{\Lambda}$ be the edge signs for $\Lambda$ from Definitions \ref{def:se-inner_edge} and \ref{eq:bnd-edge-sign-def}.

\begin{theorem}\label{thm:reconstruction}
Let $\Lambda$ be a spin triangulated surface. The spin structures $\Lambda$ and $S(\underline\Lambda,s_\Lambda)$ on $\underline\Lambda$ are isomorphic.
\end{theorem}

\begin{proof}
This is evident from the explicit construction of $S(\underline\Lambda,s_\Lambda)$. In terms of the atlas $(V_\alpha)_{\alpha \in \mathcal{I}}$ considered above, the isomorphism of spin structures $S(\underline\Lambda,s_\Lambda) \to \Lambda$ is simply given by $\tilde\chi^+_\sigma$ on $V_\sigma$ ($\sigma \in \mathcal{C}_2$) and by the identity on $V_i$ ($i \in \{1,\dots,B\}$).
\end{proof}

Isomorphism classes of spin structures on triangulated surfaces can be parametrised by equivalence classes of markings and admissible edge signs. In more detail, fix a triangulated surface $\Sigma=((\mathcal{C},f_i),\varphi ,(\Sigma ,\varphi_i))$ (without marking). Consider the set $\mathcal{M}_\Sigma := \{ ((d^1_0,d^2_0),s) \}$ of pairs of markings $(d^1_0,d^2_0)$ on $\Sigma$ and admissible edge signs $s$ on the resulting marked triangulated surface.  The moves described in Lemma \ref{lem:index.marking} leave the underlying triangulation fixed and just operate on markings and edge signs. They generate an equivalence relation on $\mathcal{M}_\Sigma$ which we denote by $\sim_\text{fix}$.

\begin{theorem}\label{thm:parametrisation-of-spin}
The assignment $((d^1_0,d^2_0),s) \mapsto S\big(\,\text{($\Sigma$ with $d^1_0,d^2_0$})\,,\,s\big)$ induces a bijection
\begin{equation}\label{eq:equiv-mark-spinstr-iso}
    \mathcal{M}_{\underline{\Sigma}} / \sim_\text{fix} ~ \longrightarrow ~ (\text{spin structures on $\Sigma$ up to isom.\ of spin str.}) \ .
\end{equation}
\end{theorem}

\begin{proof}
{\em Well-definedness on equivalence classes:} Let $(d,s)$ and $(d',s')$ be two pairs in $\mathcal{M}_\Sigma$ linked by a move from Lemma \ref{lem:index.marking}. Let $\Lambda$ be the spin triangulated surface $S( (\Sigma,d),s)$, and let $\Lambda'$ be the new spin triangulated surface resulting from the move as in  Lemma \ref{lem:index.marking}. Then $\Lambda$ and $\Lambda'$ have the same underlying spin surface and differ only in marking and choice of spin lifts, in particular $\underline\Lambda = (\Sigma,d)$ and 
	$\underline\Lambda' = (\Sigma,d')$. 
By construction, $s_{\Lambda'} = s'$, and so by Theorem \ref{thm:reconstruction} we have $\Lambda' \cong S( (\Sigma,d'),s')$.

\smallskip\noindent
{\em Surjectivity:} Immediate from Theorem \ref{thm:reconstruction}.

\smallskip\noindent
{\em Injectivity:} Let $(d,s), (d',s') \in \mathcal{M}_\Sigma$ be such that $S( (\Sigma,d),s)$ and $S( (\Sigma,d'),s')$ are isomorphic spin structures on $\Sigma$. 
	The moves in Lemma \ref{lem:index.marking} relate any two markings on on $\Sigma$, and so
the equivalence class of $(d',s')$ contains elements with marking $d$, i.e.\ there is $s''$ such that $(d',s') \sim_\text{fix} (d,s'')$. 
Let $\Lambda := S( (\Sigma,d),s)$ and $\Lambda'' :=S( (\Sigma,d),s'')$ be the corrsponding spin triangulated surfaces.
	By well-definedness of \eqref{eq:equiv-mark-spinstr-iso},
there this is an isomorphism $f : \Lambda \to \Lambda''$ of spin structures. Write $\tilde\chi_\sigma$ and $\tilde\chi_\sigma''$ for the spin lifts in $\Lambda$ and $\Lambda''$, respectively, of the embedding $\chi_\sigma$ (which is the same for $\Lambda$ and $\Lambda''$ as their marking agrees). Then $f \circ \tilde\chi_\sigma$ and $\tilde\chi''_\sigma$ are either equal or they differ by a sheet exchange. Applying sheet exchanges where necessary and changing the edge signs $s''$ according to Lemma \ref{lem:index.marking} produces a new pair $(d,s''')$ in the equivalence class of $(d',s')$ with the property that $\Lambda''' :=S( (\Sigma,d),s''')$ has isomorphic spin structure to $\Lambda$, and that the isomorphism $f$ can be chosen such that $f \circ \tilde\chi_\sigma = \tilde\chi'''_\sigma$ for all triangles. By definition of the edge signs, it then follows that $s=s'''$. Thus $(d,s) \sim_\text{fix} (d,s''') \sim_\text{fix} (d',s')$.
\end{proof}

We stress that the above theorem classifies spin structures up to isomorphism of spin structures as in Definition \ref{eq:spin_structure_wm}, not up to isomorphism of spin surfaces as in Definition \ref{def:morp-spin-surf}. For example, if $\Sigma$ is a torus with empty boundary, $\mathcal{M}_\Sigma / \sim_\text{fix}$ has four elements, explicit representatives of which will be given in Section \ref{sec:TFT-spin-torus}.

\subsection{Pachner moves}\label{sec:moves-pachner}

Recall that any two finite combinatorial manifolds that are PL-homeomorphic, can be transformed into each other by a finite sequence of Pachner moves \cite{pachner1991pl}.
	In two dimensions, there are the 2-2 and the 3-1 Pachner move (and its inverse).
We want to examine the effect of these moves on the edge signs. 

A 2d-Pachner move on a complex $\mathcal{C}$ changes at most three adjacent triangles. 
We say two spin triangulated surfaces $(\mathcal{C},\varphi,\Sigma)$ and $(\mathcal{C}',\varphi',\Sigma')$ {\em are related by a Pachner move}, if the underlying complexes $\mathcal{C}$ and $\mathcal{C}'$ are related by a Pachner move, and if the spin lifts $\tilde{\chi}$ and markings are not affected away from the triangles changed by the Pachner move.

\begin{figure}[tb]
    \begin{tikzpicture}[]
        \node[name=A, regular polygon, regular polygon sides=4, draw,minimum size=3cm]{};
        \node[above right=-2pt] at (A.corner 1) {$v_2$};
        \node[above left=-2pt] at (A.corner 2) {$v_3$};
        \node[below left=-2pt] at (A.corner 3) {$v_0$};
        \node[below right=-2pt] at (A.corner 4) {$v_1$};
        \node[above] at (A.side 1) {$\sigma_A$};
        \node[left=-3pt] at (A.side 2) {$\sigma_B$};
        \node[below] at (A.side 3) {$\sigma_C$};
        \node[right=-3pt] at (A.side 4) {$\sigma_D$};
        \node[below right=8pt] at (A.corner 2) {$\sigma_1$};
        \node[above left=8pt] at (A.corner 4) {$\sigma_2$};
        \draw[green!40, line width=5pt] (A.corner 3) -- ($(A.corner 1)$);
        \draw[] (A.corner 3) -- (A.corner 1);
        \draw[->,>=stealth] (A.corner 3) -- (A.center);
        \foreach \Na/\Nb in {1/2,2/3,3/4,4/1}
        {\draw[->,>=stealth] (A.corner \Na) -- (A.side \Na);
        \draw (A.corner \Na) -- (A.corner \Nb);}
        \foreach \N/\Npos/\Na in {1/below/A,2/right/B,3/above/C,4/left/D}
        \node[\Npos=-1pt] at (A.side \N) {\scriptsize{${s_{\Na}}$}};
        \node[above left=-2pt] at (A.center) {\scriptsize{$s $}};
        \node[regular polygon, regular polygon sides=8, draw,minimum size=3cm, dotted, rotate=22.5]{};
        \node[name=B, regular polygon, regular polygon sides=4, draw,minimum size=3cm] at (5cm,0cm){};
        \node[above right=-2pt] at (B.corner 1) {$v_2$};
        \node[above left=-2pt] at (B.corner 2) {$v_3$};
        \node[below left=-2pt] at (B.corner 3) {$v_0$};
        \node[below right=-2pt] at (B.corner 4) {$v_1$};
        \node[above] at (B.side 1) {$\sigma_A$};
        \node[left=-3pt] at (B.side 2) {$\sigma_B$};
        \node[below] at (B.side 3) {$\sigma_C$};
        \node[right=-3pt] at (B.side 4) {$\sigma_D$};
        \node[above right=8pt] at (B.corner 3) {$\sigma_3$};
        \node[below left=8pt] at (B.corner 1) {$\sigma_4$};
        \draw[green!40, line width=5pt] (B.corner 4) -- ($(B.corner 2)$);
        \draw[] (B.corner 4) -- (B.corner 2);
        \draw[->,>=stealth] (B.corner 4) -- (B.center);
        \foreach \Na/\Nb in {1/2,2/3,3/4,4/1}
        {\draw[->,>=stealth] (B.corner \Na) -- (B.side \Na);
        \draw (B.corner \Na) -- (B.corner \Nb);}
        \foreach \N/\Npos/\Na in {1/below/A,2/right/B,3/above/C,4/left/D}
        \node[\Npos=-2pt] at (B.side \N) {\scriptsize{${s_{\Na}'}$}};
        \node[above right=-2pt] at (B.center) {\scriptsize{${s '}$}};
        \node[regular polygon, regular polygon sides=8, draw,minimum size=3cm, dotted, rotate=22.5] at (5cm,0cm){};
        \draw[<->] (A.east)++(1cm,0cm) -- ($(B.west)+(-1cm,0)$);
    \end{tikzpicture}
    \caption{Configuration of markings and labels for the Pachner 2-2 move. $v_0,\dots,v_3$ are vertices in $\mathcal{C}$ and in $\mathcal{C}'$.
	$\sigma_A,\dots,\sigma_D$ are faces in $\mathcal{C}$ and $\mathcal{C}'$, while
 $\sigma_1,\sigma_2$ belong to $\mathcal{C}$ and $\sigma_3,\sigma_4$ belong to $\mathcal{C}'$.
$s_A,\dots,s_D$ and $s $ as well as the primed version are the edge signs of the respective edges. The orientation of the edges is as indicated by the arrows. The green lines indicate the marked edges of the triangles, 
     e.g.\ $\{v_0,v_2\}$ 
for $\sigma_1$.}
    \label{fig:Pachner2.2}
\end{figure}
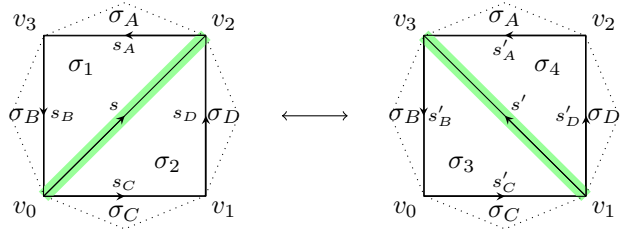

\begin{proposition} \label{thm:pachner2-2}
    Let $(C,\varphi,\Sigma)$ and $(C',\varphi',\Sigma)$ be spin triangulated surfaces related by a Pachner 2-2 move such that the configuration of markings on the affected subcomplex is as in Figure \ref{fig:Pachner2.2}.
    If the spin lifts on $\chi_{\sigma_3}$ and $\chi_{\sigma_4}$ are such that $s'=s$ and $s_A'=s_A$, then the remaining edge signs are related as
    \begin{equation} \label{eq:pachner2-2}
        s_B'=-s\,s_B ~~,\quad
        s_C'=-s_C ~~,\quad
        s_D'=-s\,s_D \ .
    \end{equation}
\end{proposition}

The other choices can be obtained by composing $\tilde{\chi}_3$ or $\tilde{\chi}_4$ with the leaf exchange automorphism; the corresponding edge signs are given by Lemma \ref{lem:index.marking}.

    \begin{proof}
    First assume all vertices $v_0,\dots,v_3$ are inner. We use Corollary \ref{lem:vertex.rule.inner} at each of those vertices and count the difference in the 
	numbers $K$ and $D$. 
For $v_1$ we have
    \begin{equation}
        s_C s_D = (-1)^{\Delta D+ \Delta K} s_C's 's_D'.
    \end{equation}
    Simple counting yields $\Delta K=1$ and $\Delta D=1$. Thus
    \begin{equation}
        s_Cs_D=s_C's 's_D'.
    \end{equation}
    The similar counting argument for the other vertices yields
\begin{align}
    v_2~&:~~ s_A \,s\, s_D = - s_A'\,s_D' \nonumber\\ 
    v_3~&:~~ s_A\,s_B = -s_A'\,s\, 's_B' \\
    v_0~&:~~ s_B\,s\,  s_C = s_B'\,s_C' \nonumber
\end{align}
	Assuming $s=s'$ and $s_A' = s_A$, the above set of equations has \eqref{eq:pachner2-2} as unique solution.

By Lemma \ref{lem:vertex.rule.boundary} for vertices on the boundary we have the same dependence on the numbers $K$ and $D_{NS}$ or $D_R$. Thus the result holds for these cases, too.
    \end{proof}

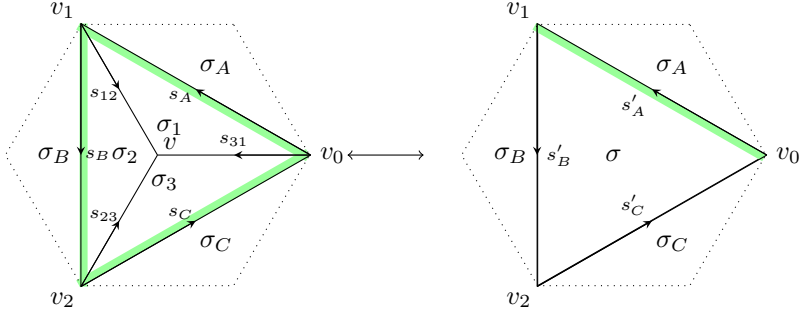
\begin{figure}[tb]
    \begin{tikzpicture}[]
        \node[name=A, regular polygon, regular polygon sides=3, draw,minimum size=4cm, rotate=30]{};
        \draw[green!40, line width=3pt] (A.corner 1)+(1pt,0pt) -- ($(A.corner 2)+(1pt,0pt)$);
        \draw[green!40, line width=3pt] (A.corner 2)+(-1pt,1pt) -- ($(A.corner 3)+(-1pt,1pt)$);
        \draw[green!40, line width=3pt] (A.corner 3)+(-1pt,-1pt) -- ($(A.corner 1)+(-1pt,-1pt)$);
        \node[above left=-2pt] at (A.corner 1) {$v_1$};
        \node[below left=-2pt] at (A.corner 2) {$v_2$};
        \node[right] at (A.corner 3) {$v_0$};
        \node[above right=-2pt] at (A.center) {$v$};
        \node[above right=2pt and -2pt] at (A.side 3) {$\sigma_A$};
        \node[left] at (A.side 1) {$\sigma_B$};
        \node[below right=2pt and -2pt] at (A.side 2) {$\sigma_C$};
        \foreach \Na/\Nb in {1/2,2/3,3/1}
        {\draw[->,>=stealth] (A.corner \Na) -- (A.side \Na);
        \draw (A.corner \Na) -- (A.corner \Nb);
        \draw (A.corner \Na) -- (A.center);
        \draw[->,>=stealth](A.corner \Na) -- ($(A.corner \Na)!0.5!(A.center)$);
        }
        \node[below left=-5pt] at ($(A.corner 1)!0.5!(A.center)$) {\scriptsize{${s_{12}}$}};
        \node[above left=-5pt] at ($(A.corner 2)!0.5!(A.center)$) {\scriptsize{${s_{23}}$}};
        \node[above] at ($(A.corner 3)!0.5!(A.center)$) {\scriptsize{${s_{31}}$}};
        \foreach \N/\Npos/\Na/\No in {1/right/B/-2pt,2/above left/C/-4pt,3/below left/A/-4pt}{
            \node[\Npos=\No] at (A.side \N) {\scriptsize{${s_{\Na}}$}};
        }
        \node[right=8pt] at (A.side 1) {$\sigma_2$};
        \node[above left=8pt and 3pt] at (A.side 2) {$\sigma_3$};
        \node[below left=8pt and 1pt] at (A.side 3) {$\sigma_1$};
        \node[regular polygon, regular polygon sides=6, draw,minimum size=4cm, dotted]{};
        \node[name=B, regular polygon, regular polygon sides=3, draw,minimum size=4cm, rotate=30] at (6cm,0cm){};
        \draw[green!40, line width=3pt] (B.corner 3)+(-1pt,-1pt) -- ($(B.corner 1)+(-1pt,-1pt)$);
        \node[above left=-2pt] at (B.corner 1) {$v_1$};
        \node[below left=-2pt] at (B.corner 2) {$v_2$};
        \node[right] at (B.corner 3) {$v_0$};
        \node[above right=2pt and -2pt] at (B.side 3) {$\sigma_A$};
        \node[left] at (B.side 1) {$\sigma_B$};
        \node[below right=2pt and -2pt] at (B.side 2) {$\sigma_C$};
        \foreach \Na/\Nb in {1/2,2/3,3/1}
        {\draw[->,>=stealth] (B.corner \Na) -- (B.side \Na);
        \draw (B.corner \Na) -- (B.corner \Nb);
        }
        \foreach \N/\Npos/\Na/\No in {1/right/B/,2/above left/C/-2pt,3/below left/A/-2pt}{
            \node[\Npos=\No] at (B.side \N) {\scriptsize{${s '_{\Na}}$}};
        }
        \node[] at (B.center) {$\sigma $};
        \node[regular polygon, regular polygon sides=6, draw,minimum size=4cm, dotted] at (6cm,0cm) {};
        \draw[<->] (A.center)++(2.5cm,0cm) -- ($(B.center)+(-2.5cm,0cm)$);
    \end{tikzpicture}
    \caption{Configuration of markings and labels for the Pachner 3-1 move and its inverse. $v_0,v_1,v_2$ are vertices in $\mathcal{C}$ and $\mathcal{C}'$, and $v$ is a vertex in $\mathcal{C}$. $\sigma_1,\sigma_2,\sigma_3$ are faces of $\mathcal{C}$, $\sigma$ is a face of $\mathcal{C}'$, and $\sigma_A,\sigma_B,\sigma_C$ are faces in $\mathcal{C}$ and $\mathcal{C}'$. $s_{12},s_{23},s_{31}$ and $s_A,s_B,s_C$ as well as the primed version are the edge signs of the respective edges. The orientation of the edges is as indicated by the arrows. The green lines indicate the marked edges of the triangles.}
    \label{fig:Pachner1.3}
\end{figure}

\begin{proposition}\label{thm:pachner3-1}
    Let $(C,\varphi,\Sigma)$ and $(C',\varphi',\Sigma)$ be related by a Pachner 3-1 move. Let the configuration of markings on the affected subcomplex be as in Figure \ref{fig:Pachner1.3}.
    If the spin lift of $\chi_\sigma$ is such that $s_A'=s_A$, then the remaining edge signs are related by
    \begin{equation}\label{eq:pachner3-1}
        s_B' = s_{12}s_B ~~ , \quad
        s_C' = -s_{31}s_C~~ , \quad
        s_{12}s_{23}s_{31}= -1 \ ,
    \end{equation}
	where $s_{12}$, $s_{23}$, $s_{31}$ are arbitrary, subject to the last condition.
\end{proposition}
\begin{proof}   
    We will assume $v_0,v_1,v_2$ are inner. The argument in case some vertices belong to the boundary is the same.
As in Proposition \ref{thm:pachner2-2}, we use Corollary \ref{lem:vertex.rule.inner} at each of those vertices to get the relations
    \begin{align}
	    v~&:~~ s_{12}s_{23} s_{31} = -1 \nonumber \\
        v_0~&:~~ s_A s_{31} s_C = - s_A' s_C'\\
        v_1~&:~~ s_A s_{12} s_B = s_A' s_B' \nonumber \\
        v_2~&:~~ s_Bs_{23} s_C = s_B' s_C'\nonumber 
\end{align}
We can choose 
	$\tilde{\chi}_\sigma$ 
such that $s_A'= s_A$. Then the unique solution to the equations for $v_0,v_1,v_2$ is $s_B' = s_{12} s_B$ and $s_C'= -s_{31} s_C$.
\end{proof}

\section{Application: Two-dimensional lattice spin TFT}\label{sec:2dlatticeTFT}

\subsection{Preliminaries about Graphs}\label{sec:prelim-graph}

We first recall some graph theoretic notions used in \cite{joyal1991geometry}. 
All graphs are finite.
A \emph{graph with boundary}, in the following just {\em graph}, $\Gamma =(\Gamma ,\partial \Gamma )$ is a graph $\Gamma $ together with a set $\partial \Gamma $ of univalent vertices. Elements of $\partial \Gamma $ are called \emph{outer} vertices, and vertices not in $\partial \Gamma $ are called \emph{inner} vertices.
	A graph $\Gamma$ is {\em directed} if each edge is equipped with an orientation.
For a directed graph and a vertex $v$, the set of ingoing edges is denoted as $\text{in}(v)$ and that of outgoing edges as $\text{out}(v)$.
A \emph{polarised graph} is a directed graph together with a choice of linear order on $\text{in}(v)$ and $\text{out}(v)$ for each inner vertex $v$. 
A \emph{progressive graph} is a directed graph with no (oriented) circuits. The \emph{domain} $\text{dom}(\Gamma )$ (resp.\ $\emph{codomain}$ $\text{cod}(\Gamma )$) of a progressive graph is the union of $\text{out}(v)$ (resp.\ $\text{in}(v)$) over all $v\in \partial \Gamma$. An \emph{anchored progressive graph} $\Gamma $ is a progressive graph together with linear orders on both $\text{dom}(\Gamma )$ and $\text{cod}(\Gamma )$.

\medskip

\begin{figure}[tb]
\begin{center}
\raisebox{7em}{a)} \hspace{1em}
\begin{tikzpicture}
        \node[name=A, regular polygon, regular polygon sides=3, draw,minimum size=3cm, dashed]{};
        \draw (A.side 1) -- (A.center);
        \draw (A.side 2) -- (A.center);
        \draw (A.side 3) -- (A.center);
        \node[above right=3pt] at (A.center) {\tiny{$0$}};
        \node[above left=3pt] at (A.center) {\tiny{$1$}};
        \node[below right=3pt and -2pt] at (A.center) {\tiny{$2$}};
        \node[above=4pt] at (A.center) {$\sigma $};
        \node[above right] at (A.side 3) {$d^2_0(\sigma )$};
        \node[above left] at (A.side 1) {$d^2_1(\sigma )$};
        \node[below] at (A.side 2) {$d^2_2(\sigma )$};
    \end{tikzpicture}
\hspace{5em}
\raisebox{7em}{b)}\hspace{1em}
\raisebox{1em}{
\begin{tikzpicture}
    \begin{scope}[decoration={
        markings,
        mark=at position 0.5 with {\arrow{stealth}}}
    ] 
    \draw[dashed, postaction=decorate] (0cm,1cm) -- (0cm,2cm);
    \draw[dashed, postaction=decorate] (0cm,0cm) -- (0cm,1cm);
    \draw (-1cm,1cm) -- (0cm,1cm);
    \draw (1cm,1cm) -- (0cm,1cm);
    \end{scope}
    \node[] at (0cm,1cm) {$\times$};
    \node[above left=-2pt and 4pt] at (0cm,1cm) {\tiny{$0$}};
    \node[above right=-2pt and 4pt] at (0cm,1cm) {\tiny{$1$}};
    \end{tikzpicture}}
\end{center}         
    \caption{The polarisation of the graph $\Gamma$. (a) A single triangle $\sigma $ (in dashed lines) together with the dual graph. The linear order of the edges at the vertex in the centre is indicated by small numbers $0$, $1$, $2$ which corresponds to the linear order given on the edges of $\sigma$ by the maps $d^2_0$, $d^2_1$ and $d^2_2$. (b) A single edge of the triangulation as a dashed line, together with its orientation. The dual edge is drawn solid, and in its centre a new vertex has been placed. The linear order on the edges of that vertex is indicated by the numbers $0$ and $1$.}
    \label{fig:graph.order}
\end{figure}
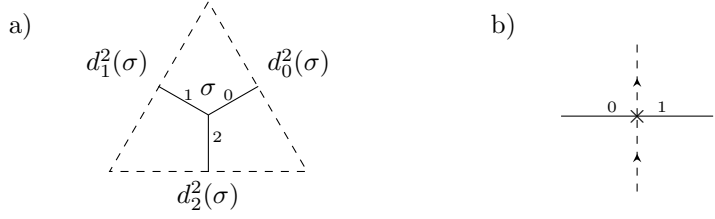

Starting from a marked combinatorial surface $\mathcal{C}$ (Definition \ref{def:marked-triang-surf}) we produce a progressive polarised graph $\Gamma(\mathcal{C})$ as follows:
\begin{enumerate}
    \item Take the 1-skeleton of the Poincar\'e dual of $\mathcal{C}$. This yields a graph $\Gamma '$ with only univalent and trivalent vertices. 
	(The univalent vertices sit at the end of edges dual to edges of $\mathcal{C}$ that lie on the boundary.)
    Let $\partial \Gamma $ be the set of all univalent vertices of $\Gamma '$.
    \item Put an additional vertex on each 
	edge.
The resulting graph $(\Gamma ,\partial \Gamma )$ has bi- and trivalent inner vertices, with each inner edge bounding one bi- and one trivalent vertex.
	(An edge is inner if none of its bounding vertices is in $\partial \Gamma$.)
	Note that every edge bounds exactly one bivalent vertex.
    \item Turn the graph $\Gamma $ into a directed graph by orienting each edge
	away from the bivalent vertex.
    \item For every trivalent vertex $v$, order the set $\text{in}(v)$ as depicted in Figure \ref{fig:graph.order}\,a), using the boundary maps $d^2$ of $\mathcal{C}$. 
	(The set $\text{out}(v)$ is empty by the orientation choice in 3.)
    \item For every bivalent vertex $v$, order the set $\text{out}(v)$ as depicted in Figure \ref{fig:graph.order}\,b), using the boundary maps $d^1$ of $\mathcal{C}$.
\end{enumerate}
	Since every edge starts at a bivalent vertex, the graph $\Gamma(\mathcal{C})$ constructed above has no circuits, so that it is indeed a polarised progressive graph. By the same argument, $\text{dom}(\Gamma(\mathcal{C}))=\emptyset$.

\medskip

To anchor this graph $\Gamma (\mathcal{C})$, we have to give an ordering on 
	$\text{cod}(\Gamma )$. 
We first order the elements of $\text{cod}(\Gamma )$ by the number of the boundary component of $\mathcal{C}$ they start from. 
We then order the three edges for each boundary 
	component according to the order on $\underline\Delta$ (see Figure \ref{fig:st.triangle}), transported to the $i$'th boundary via the parametrisation map $f_i:\partial \underline\Delta  \to \partial \mathcal{C}_i$.

\medskip

Let $\mathcal{S}$ be a symmetric monoidal category. 
	We will assume $\mathcal{S}$ to be strict monoidal in order to simplify notation. (But we will nonetheless think of $\mathbf{Vect}$ and $\mathbf{SVect}$ as examples, leaving it to the reader to add associators in the relevant places.)
The symmetric braiding will be denoted as 
\begin{equation}
	\sigma_{U,V} : U \otimes V \to V \otimes U \ .
\end{equation}
We will briefly sketch how to pass from an anchored polarised progressive graph $\Gamma$ to a morphism in $\mathcal{S}$, for details see \cite[Ch.\,2]{joyal1991geometry}. Fix a {\em valuation} on $\Gamma$, that is, to each edge of $\Gamma$ assign an object of $\mathcal{S}$, and to each vertex $v$ a morphism in $\mathcal{S}$ compatible with the objects and linear order on $\text{in}(v)$ and $\text{out}(v)$. Since the graph is progressive, one can fix a order on the vertices of $\Gamma$ such that for $v \ge v'$, there are no edges directed from $v$ to $v'$ (``all edges go up''). 
The graph $\Gamma$ together with the valuation is called a {\em diagram in} $\mathcal{S}$.

Now compose all the morphisms assigned to the vertices in the chosen order, using the symmetric structure of $\mathcal{S}$ to match the in-- and outgoing objects of the morphisms as dictated by the edges and tensoring with identity morphisms where necessary. This results in a morphism in $\mathcal{S}$ from the tensor product of the objects assigned to $\text{dom}(\Gamma)$ in the chosen order to the corresponding product for $\text{cod}(\Gamma)$. The resulting morphism is called the {\em value} of the diagram. By \cite[Cor.\,2.3]{joyal1991geometry}, the value of a diagram is independent of the ordering chosen on the vertices \cite[Ch.\,2]{joyal1991geometry}.\footnote{
    Strictly speaking at some point we have to take the geometrical realisation of the graph, since \cite{joyal1991geometry} deals with topological graphs. Different choices of the geometrical realisation lead however to isomorphic diagrams, which due to  \cite[Cor.\,2.3]{joyal1991geometry} represent identical morphisms.
	}

\medskip

We proceed to define a valuation for anchored polarised progressive graphs of the form $\Gamma(\mathcal{C})$ as described above. Choose $A\in \text{Ob}(\mathcal{S})$ and morphisms 
\begin{equation} \label{eq:morph-for-val}
    \begin{split}
        c_1&:\mathbf{1}_\mathcal{S} \to A\otimes A \ , \\
        c_{-1}&: \mathbf{1}_\mathcal{S} \to A\otimes A \ ,\\
        t&:A\otimes A\otimes A \to \mathbf{1}_\mathcal{S} \ ,
    \end{split}
\end{equation}
in $\mathcal{S}$. 

Let $\Lambda = ((\mathcal{C},f_i,d^1_0,d^2_0), (\varphi, \tilde\chi_\sigma) ,(\Sigma,\tilde\varphi_i))$ be a spin triangulated surface (Definition \ref{def:spin-triang-surf-def}). To each bivalent vertex in $\Gamma (\mathcal{C})$ corresponds an (inner or boundary) edge $e$ in the triangulation of $\Sigma$, with an edge sign $s(e)$. To every such bivalent vertex assign the map $c_{s(e)}$.
To every trivalent vertex assign the map $t$.
This defines a valuation on $\Gamma(\mathcal{C})$ and therefore a morphism in $\mathcal{S}$. 
Let $B$ be the number of boundary components of $\Sigma$. Then $|\text{cod}(\Gamma)| = 3B$ while $\text{dom}(\Gamma)$ is empty. We denote the value of the diagram by
\begin{equation} \label{eq:Ttriang-def}
	T_\text{triang}(\Lambda) : \mathbf{1}_\mathcal{S} \longrightarrow (A^{\otimes 3})^{\otimes B}   \ .
\end{equation}

\subsection{Local moves}\label{sec:localmoves}

The local moves from Section \ref{sec:moves-notriang} and \ref{sec:moves-pachner} relate spin triangulations and consequently the corresponding diagrams. From these moves we will derive a sufficient set of relations on the maps $c_1$, $c_{-1}$ and $t$ such that the resulting morphism $T_\text{triang}$ is invariant under the moves.
Some of these conditions are easiest presented in the standard graphical notation for morphisms in a tensor category, see e.g.\ \cite{baki}.
This is basically the language of diagrams where the vertices are drawn as boxes labelled by the morphism assigned by the valuation. The separation into $\text{in}(v)$ and $\text{out}(v)$ as well as the linear order is encoded by how the lines attach to the boxes. For example, a morphism $f : A \otimes B \otimes C \to D \otimes E$ is drawn as:
\begin{equation}\label{eq:morph-and-edge-order-convention}
\raisebox{-5em}{
    \begin{tikzpicture}
        \node[name=f,draw, minimum width=1.5cm] at (0,0) {$f$};
        \coordinate (fA) at ($(f.south)!0.66!(f.south west)$);
        \coordinate (fC) at ($(f.south)!0.66!(f.south east)$);
        \coordinate (fD) at ($(f.north)!0.5!(f.north west)$);
        \coordinate (fE) at ($(f.north)!0.5!(f.north east)$);
        \coordinate (A) at (-0.7cm,-1.5cm);
        \coordinate (B) at (0cm,-1.5cm);
        \coordinate (C) at (0.7cm,-1.5cm);
        \coordinate (D) at (-0.5cm,1.5cm);
        \coordinate (E) at (0.5cm,1.5cm);
        \coordinate (su) at (0cm,0.3cm);
        \coordinate (sd) at (0cm,-0.3cm);
        \draw (A) .. controls ($(A)+(su)$) and ($(fA)+(sd)$) .. (fA);
        \draw (B) -- (f);
        \draw (C) .. controls ($(C)+(su)$) and ($(fC)+(sd)$) .. (fC);
        \draw (fD) .. controls ($(fD)+(su)$) and ($(D)+(sd)$) .. (D);
        \draw (fE) .. controls ($(fE)+(su)$) and ($(E)+(sd)$) .. (E);
        \node[right] at (A) {$A$};
        \node[right] at (B) {$B$};
        \node[right] at (C) {$C$};
        \node[right] at (D) {$D$};
        \node[right] at (E) {$E$};
        \node[below left] at (fA) {\tiny{$0$}};
        \node[below left] at (f.south) {\tiny{$1$}};
        \node[below left] at (fC) {\tiny{$2$}};
        \node[above left] at (fD) {\tiny{$0$}};
        \node[above left] at (fE) {\tiny{$1$}};
        \draw [decorate,decoration={brace,amplitude=5},xshift=-4pt,yshift=0pt]
(-1cm,-1.5cm) -- (-1cm,-0.5cm) node [black,midway,xshift=-0.6cm] {in};
        \draw [decorate,decoration={brace,amplitude=5},xshift=-4pt,yshift=0pt]
(-1cm,0.5cm) -- (-1cm,1.5cm) node [black,midway,xshift=-0.6cm] {out};
    \end{tikzpicture}}
\end{equation}

\begin{enumerate}  \setlength{\leftskip}{-1.5em}
\item {\em Edge orientation change.}
By Lemma \ref{lem:index.marking}\,(2), replacing the orientation of a single edge $e$ corresponds to a change of the edge sign $s(e)$. This corresponds to exchanging $c_1$ and $c_{-1}$ on the corresponding bivalent vertex $v$, together with a change of linear order on the outgoing edges. 
We thus require that
\begin{equation}
    c_1 = \sigma_{A,A} \circ c_{-1} \ .
    \label{eq:comparison.of.copairings}
\end{equation}

\item {\em Leaf exchange automorphism on a single triangle.}
By Lemma \ref{lem:index.marking}\,(1), exchanging the spin lift of the characteristic map $\chi_\sigma $ for a single triangle $\sigma $ corresponds to inverting the edge signs on its three bounding edges $e_0, e_1, e_2$. We thus require that
\begin{equation}
    \raisebox{-3em}{
    \begin{tikzpicture}
        \node[name=t, draw, minimum width=1.5cm] at (0cm,2cm) {$t$};
        \coordinate (tl) at ($(t.south west)!0.2!(t.south east)$);
        \coordinate (tr) at ($(t.south west)!0.8!(t.south east)$);
        \node[name=c0, draw] at (-1.5cm,0cm) {$c_{s(e_0)}$};
        \coordinate (c0l) at ($(c0.north west)!0.2!(c0.north east)$);
        \coordinate (c0r) at ($(c0.north west)!0.8!(c0.north east)$);
        \node[name=c1, draw] at (0cm,0cm) {$c_{s(e_1)}$};
        \coordinate (c1l) at ($(c1.north west)!0.2!(c1.north east)$);
        \coordinate (c1r) at ($(c1.north west)!0.8!(c1.north east)$);
        \node[name=c2, draw] at (1.5cm,0cm) {$c_{s(e_2)}$};
        \coordinate (c2l) at ($(c2.north west)!0.2!(c2.north east)$);
        \coordinate (c2r) at ($(c2.north west)!0.8!(c2.north east)$);
        \coordinate (su) at (0cm,0.5cm);
        \coordinate (sd) at (0cm,-0.5cm);
        \coordinate (out0) at (2cm,2.5cm);
        \coordinate (out1) at (2.5cm,2.5cm);
        \coordinate (out2) at (3cm,2.5cm);
        \draw (c0l) .. controls ($(c0l)+(su)$) and ($(tl)+(sd)$) .. (tl);
        \draw (c0r) .. controls ($(c0r)+(su)$) and ($(out0)+(sd)$) .. (out0);
        \draw (c1l) .. controls ($(c1l)+(su)$) and ($(t.south)+(sd)$) .. (t.south);
        \draw (c1r) .. controls ($(c1r)+(su)$) and ($(out1)+(sd)$) .. (out1);
        \draw (c2l) .. controls ($(c2l)+(su)$) and ($(tr)+(sd)$) .. (tr);
        \draw (c2r) .. controls ($(c2r)+(su)$) and ($(out2)+(sd)$) .. (out2);
    \end{tikzpicture}}
    =
    \raisebox{-3em}{
    \begin{tikzpicture}
        \node[name=t, draw, minimum width=1.5cm] at (0cm,2cm) {$t$};
        \coordinate (tl) at ($(t.south west)!0.2!(t.south east)$);
        \coordinate (tr) at ($(t.south west)!0.8!(t.south east)$);
        \node[name=c0, draw] at (-1.5cm,0cm) {$c_{-s(e_0)}$};
        \coordinate (c0l) at ($(c0.north west)!0.2!(c0.north east)$);
        \coordinate (c0r) at ($(c0.north west)!0.8!(c0.north east)$);
        \node[name=c1, draw] at (0cm,0cm) {$c_{-s(e_1)}$};
        \coordinate (c1l) at ($(c1.north west)!0.2!(c1.north east)$);
        \coordinate (c1r) at ($(c1.north west)!0.8!(c1.north east)$);
        \node[name=c2, draw] at (1.5cm,0cm) {$c_{-s(e_2)}$};
        \coordinate (c2l) at ($(c2.north west)!0.2!(c2.north east)$);
        \coordinate (c2r) at ($(c2.north west)!0.8!(c2.north east)$);
        \coordinate (su) at (0cm,0.5cm);
        \coordinate (sd) at (0cm,-0.5cm);
        \coordinate (out0) at (2cm,2.5cm);
        \coordinate (out1) at (2.5cm,2.5cm);
        \coordinate (out2) at (3cm,2.5cm);
        \draw (c0l) .. controls ($(c0l)+(su)$) and ($(tl)+(sd)$) .. (tl);
        \draw (c0r) .. controls ($(c0r)+(su)$) and ($(out0)+(sd)$) .. (out0);
        \draw (c1l) .. controls ($(c1l)+(su)$) and ($(t.south)+(sd)$) .. (t.south);
        \draw (c1r) .. controls ($(c1r)+(su)$) and ($(out1)+(sd)$) .. (out1);
        \draw (c2l) .. controls ($(c2l)+(su)$) and ($(tr)+(sd)$) .. (tr);
        \draw (c2r) .. controls ($(c2r)+(su)$) and ($(out2)+(sd)$) .. (out2);
    \end{tikzpicture}}
    \label{eq:alg.leaf}
\end{equation}
Written out explicitly, this identity reads
\begin{equation}
	f(s(e_0),s(e_1),s(e_2)) = f(-s(e_0),-s(e_1),-s(e_2)) \ ,
\end{equation}
where
\begin{align}
	f(\alpha,\beta,\gamma) 
	&=
	\big( t \otimes \id_{A \otimes A \otimes A} \big)
	\circ
	\big( \id_{A \otimes A} \otimes c_{A,A} \otimes \id_{A \otimes A} \big) 
	\\
	& \qquad \circ
	\big( \id_A \otimes c_{A,A} \otimes c_{A,A} \otimes \id_A \big) 
	\circ
	\big( c_\alpha \otimes c_\beta \otimes c_\gamma \big)  \ .
	\nonumber 
\end{align}

\item {\em Cyclic permutation of boundary edges for a single triangle.}
By Lemma \ref{lem:index.marking}\,(3), the value of the diagram will not change under such a cyclic permutation if we require the identity
\begin{equation}
    \raisebox{-3em}{
    \begin{tikzpicture}
        \node[name=t, draw, minimum width=1.5cm] at (0cm,2cm) {$t$};
        \coordinate (tl) at ($(t.south west)!0.2!(t.south east)$);
        \coordinate (tr) at ($(t.south west)!0.8!(t.south east)$);
        \node[name=c0, draw] at (-1.5cm,0cm) {$c_{s(e_0)}$};
        \coordinate (c0l) at ($(c0.north west)!0.2!(c0.north east)$);
        \coordinate (c0r) at ($(c0.north west)!0.8!(c0.north east)$);
        \node[name=c1, draw] at (0cm,0cm) {$c_{s(e_1)}$};
        \coordinate (c1l) at ($(c1.north west)!0.2!(c1.north east)$);
        \coordinate (c1r) at ($(c1.north west)!0.8!(c1.north east)$);
        \node[name=c2, draw] at (1.5cm,0cm) {$c_{s(e_2)}$};
        \coordinate (c2l) at ($(c2.north west)!0.2!(c2.north east)$);
        \coordinate (c2r) at ($(c2.north west)!0.8!(c2.north east)$);
        \coordinate (su) at (0cm,0.5cm);
        \coordinate (sd) at (0cm,-0.5cm);
        \coordinate (out0) at (2cm,2.5cm);
        \coordinate (out1) at (2.5cm,2.5cm);
        \coordinate (out2) at (3cm,2.5cm);
        \draw (c0l) .. controls ($(c0l)+(su)$) and ($(tl)+(sd)$) .. (tl);
        \draw (c0r) .. controls ($(c0r)+(su)$) and ($(out0)+(sd)$) .. (out0);
        \draw (c1l) .. controls ($(c1l)+(su)$) and ($(t.south)+(sd)$) .. (t.south);
        \draw (c1r) .. controls ($(c1r)+(su)$) and ($(out1)+(sd)$) .. (out1);
        \draw (c2l) .. controls ($(c2l)+(su)$) and ($(tr)+(sd)$) .. (tr);
        \draw (c2r) .. controls ($(c2r)+(su)$) and ($(out2)+(sd)$) .. (out2);
    \end{tikzpicture}}
    =
    \raisebox{-3em}{
    \begin{tikzpicture}
        \node[name=t, draw, minimum width=1.5cm] at (0cm,2cm) {$t$};
        \coordinate (tl) at ($(t.south west)!0.2!(t.south east)$);
        \coordinate (tr) at ($(t.south west)!0.8!(t.south east)$);
        \node[name=c0, draw] at (-1.5cm,0cm) {$c_{-s(e_0)}$};
        \coordinate (c0l) at ($(c0.north west)!0.2!(c0.north east)$);
        \coordinate (c0r) at ($(c0.north west)!0.8!(c0.north east)$);
        \node[name=c1, draw] at (0cm,0cm) {$c_{s(e_1)}$};
        \coordinate (c1l) at ($(c1.north west)!0.2!(c1.north east)$);
        \coordinate (c1r) at ($(c1.north west)!0.8!(c1.north east)$);
        \node[name=c2, draw] at (1.5cm,0cm) {$c_{s(e_2)}$};
        \coordinate (c2l) at ($(c2.north west)!0.2!(c2.north east)$);
        \coordinate (c2r) at ($(c2.north west)!0.8!(c2.north east)$);
        \coordinate (su) at (0cm,0.5cm);
        \coordinate (sd) at (0cm,-0.5cm);
        \coordinate (out0) at (2cm,2.5cm);
        \coordinate (out1) at (2.5cm,2.5cm);
        \coordinate (out2) at (3cm,2.5cm);
        \draw (c0l) .. controls ($(c0l)+(su)$) and ($(tr)+(sd)$) .. (tr);
        \draw (c0r) .. controls ($(c0r)+(su)$) and ($(out0)+(sd)$) .. (out0);
        \draw (c1l) .. controls ($(c1l)+(su)$) and ($(tl)+(sd)$) .. (tl);
        \draw (c1r) .. controls ($(c1r)+(su)$) and ($(out1)+(sd)$) .. (out1);
        \draw (c2l) .. controls ($(c2l)+(su)$) and ($(t.south)+(sd)$) .. (t.south);
        \draw (c2r) .. controls ($(c2r)+(su)$) and ($(out2)+(sd)$) .. (out2);
    \end{tikzpicture}}
    \label{eq:alg.cyclic}
\end{equation}

\item {\em Pachner 2-2 move.} 
By Proposition \ref{thm:pachner2-2}, a sufficient condition for invariance is
\begin{equation}
    \begin{split}
    \begin{tikzpicture}
        \node[name=t1, draw, minimum width=1.5cm] at (0cm,2cm) {$t$};
        \coordinate (t1l) at ($(t1.south west)!0.2!(t1.south east)$);
        \coordinate (t1r) at ($(t1.south west)!0.8!(t1.south east)$);
        \node[name=t2, draw, minimum width=1.5cm] at (2.5cm,2cm) {$t$};
        \coordinate (t2l) at ($(t2.south west)!0.2!(t2.south east)$);
        \coordinate (t2r) at ($(t2.south west)!0.8!(t2.south east)$);
        \node[name=cA, draw] at (-1cm,0cm) {$c_{s_A}$};
        \coordinate (cAl) at ($(cA.north west)!0.2!(cA.north east)$);
        \coordinate (cAr) at ($(cA.north west)!0.8!(cA.north east)$);
        \node[name=cB, draw] at (0cm,0cm) {$c_{s_B}$};
        \coordinate (cBl) at ($(cB.north west)!0.2!(cB.north east)$);
        \coordinate (cBr) at ($(cB.north west)!0.8!(cB.north east)$);
        \node[name=c, draw] at (1cm,0cm) {$c_{s}$};
        \coordinate (cl) at ($(c.north west)!0.2!(c.north east)$);
        \coordinate (cr) at ($(c.north west)!0.8!(c.north east)$);
        \node[name=cC, draw] at (2cm,0cm) {$c_{s_C}$};
        \coordinate (cCl) at ($(cC.north west)!0.2!(cC.north east)$);
        \coordinate (cCr) at ($(cC.north west)!0.8!(cC.north east)$);
        \node[name=cD, draw] at (3cm,0cm) {$c_{s_D}$};
        \coordinate (cDl) at ($(cD.north west)!0.2!(cD.north east)$);
        \coordinate (cDr) at ($(cD.north west)!0.8!(cD.north east)$);
        \coordinate (su) at (0cm,0.5cm);
        \coordinate (sd) at (0cm,-0.5cm);
        \coordinate (outA) at (4cm,2.5cm);
        \coordinate (outB) at (4.5cm,2.5cm);
        \coordinate (outC) at (5cm,2.5cm);
        \coordinate (outD) at (5.5cm,2.5cm);
        \draw (cAl) .. controls ($(cAl)+(su)$) and ($(t1.south)+(sd)$) .. (t1.south);
        \draw (cAr) .. controls ($(cAr)+(su)$) and ($(outA)+2.5*(sd)$) .. (outA);
        \draw (cBl) .. controls ($(cBl)+(su)$) and ($(t1r)+(sd)$) .. (t1r);
        \draw (cBr) .. controls ($(cBr)+(su)$) and ($(outB)+2*(sd)$) .. (outB);
        \draw (cCl) .. controls ($(cCl)+(su)$) and ($(t2.south)+(sd)$) .. (t2.south);
        \draw (cCr) .. controls ($(cCr)+(su)$) and ($(outC)+(sd)$) .. (outC);
        \draw (cDl) .. controls ($(cDl)+(su)$) and ($(t2r)+(sd)$) .. (t2r);
        \draw (cDr) .. controls ($(cDr)+(su)$) and ($(outD)+(sd)$) .. (outD);
        \draw (cl) .. controls ($(cl)+(su)$) and ($(t1l)+(sd)$) .. (t1l);
        \draw (cr) .. controls ($(cr)+(su)$) and ($(t2l)+(sd)$) .. (t2l);
    \end{tikzpicture}\\
    =  \raisebox{-3em}{ \begin{tikzpicture}
        \node[name=t1, draw, minimum width=1.5cm] at (0cm,2cm) {$t$};
        \coordinate (t1l) at ($(t1.south west)!0.2!(t1.south east)$);
        \coordinate (t1r) at ($(t1.south west)!0.8!(t1.south east)$);
        \node[name=t2, draw, minimum width=1.5cm] at (2.5cm,2cm) {$t$};
        \coordinate (t2l) at ($(t2.south west)!0.2!(t2.south east)$);
        \coordinate (t2r) at ($(t2.south west)!0.8!(t2.south east)$);
        \node[name=cA, draw] at (-1.2cm,0cm) {$c_{s_A}$};
        \coordinate (cAl) at ($(cA.north west)!0.2!(cA.north east)$);
        \coordinate (cAr) at ($(cA.north west)!0.8!(cA.north east)$);
        \node[name=cB, draw] at (0cm,0cm) {$c_{-ss_B}$};
        \coordinate (cBl) at ($(cB.north west)!0.2!(cB.north east)$);
        \coordinate (cBr) at ($(cB.north west)!0.8!(cB.north east)$);
        \node[name=c, draw] at (1cm,0cm) {$c_{s}$};
        \coordinate (cl) at ($(c.north west)!0.2!(c.north east)$);
        \coordinate (cr) at ($(c.north west)!0.8!(c.north east)$);
        \node[name=cC, draw] at (2cm,0cm) {$c_{-s_C}$};
        \coordinate (cCl) at ($(cC.north west)!0.2!(cC.north east)$);
        \coordinate (cCr) at ($(cC.north west)!0.8!(cC.north east)$);
        \node[name=cD, draw] at (3.3cm,0cm) {$c_{-ss_D}$};
        \coordinate (cDl) at ($(cD.north west)!0.2!(cD.north east)$);
        \coordinate (cDr) at ($(cD.north west)!0.8!(cD.north east)$);
        \coordinate (su) at (0cm,0.5cm);
        \coordinate (sd) at (0cm,-0.5cm);
        \coordinate (outA) at (4cm,2.5cm);
        \coordinate (outB) at (4.5cm,2.5cm);
        \coordinate (outC) at (5cm,2.5cm);
        \coordinate (outD) at (5.5cm,2.5cm);
        \draw (cBl) .. controls ($(cBl)+(su)$) and ($(t1.south)+(sd)$) .. (t1.south);
        \draw (cAr) .. controls ($(cAr)+(su)$) and ($(outA)+2.5*(sd)$) .. (outA);
        \draw (cCl) .. controls ($(cCl)+(su)$) and ($(t1r)+(sd)$) .. (t1r);
        \draw (cBr) .. controls ($(cBr)+(su)$) and ($(outB)+2*(sd)$) .. (outB);
        \draw (cDl) .. controls ($(cDl)+(su)$) and ($(t2.south)+(sd)$) .. (t2.south);
        \draw (cCr) .. controls ($(cCr)+(su)$) and ($(outC)+(sd)$) .. (outC);
        \draw (cAl) .. controls ($(cAl)+(su)$) and ($(t2r)+(sd)$) .. (t2r);
        \draw (cDr) .. controls ($(cDr)+(su)$) and ($(outD)+(sd)$) .. (outD);
        \draw (cl) .. controls ($(cl)+(su)$) and ($(t1l)+(sd)$) .. (t1l);
        \draw (cr) .. controls ($(cr)+(su)$) and ($(t2l)+(sd)$) .. (t2l);
    \end{tikzpicture}}
    \end{split}
    \label{eq:alg.pachner2-2}
\end{equation}

\item {\em Pachner 3-1 move and its inverse.}
By Proposition \ref{thm:pachner3-1}, a sufficient condition for invariance is
\begin{equation}\label{eq:alg.pachner3-1}
    \begin{split}
    \begin{tikzpicture}
        \foreach \N in {1,2,3} {
        \node[name=t\N, draw, minimum width=1.5cm, xshift=2*\N cm] at (0cm,2cm) {$t$};
        \coordinate (tl\N) at ($(t\N.south west)!0.2!(t\N.south east)$);
        \coordinate (tr\N) at ($(t\N.south west)!0.8!(t\N.south east)$);
        \coordinate (tc\N) at (t\N.south);
    }
        \foreach \N/\M in {1/A,2/12,3/B,4/31,5/C,6/23}{
            \node[name=c\N, draw, xshift=\N cm] at (0cm,0cm) {$c_{s_{\M}}$};
        \coordinate (cl\N) at ($(c\N.north west)!0.3!(c\N.north east)$);
        \coordinate (cr\N) at ($(c\N.north west)!0.7!(c\N.north east)$);
    }
        \foreach \N in {1,2,3}{
            \coordinate[xshift=0.5*\N cm] (out\N) at (7cm,2.5cm);
    }
        \coordinate (su) at (0cm,0.5cm);
        \coordinate (sd) at (0cm,-0.5cm);
        \foreach \N/\M in {1/1,2/3,3/5} {
            \draw (cl\M) .. controls ($(cl\M)+(su)$) and ($(tl\N)+(sd)$) .. (tl\N);
            \draw (cr\M) .. controls ($(cr\M)+(su)$) and ($(out\N)+4*(sd)$) .. (out\N);
        }
        \draw (cl2) .. controls ($(cl2)+(su)$) and ($(tc1)+(sd)$) .. (tc1);
        \draw (cr2) .. controls ($(cr2)+(su)$) and ($(tr2)+(sd)$) .. (tr2);
        \draw (cl4) .. controls ($(cl4)+(su)$) and ($(tc3)+(sd)$) .. (tc3);
        \draw (cr4) .. controls ($(cr4)+(su)$) and ($(tr1)+(sd)$) .. (tr1);
        \draw (cl6) .. controls ($(cl6)+(su)$) and ($(tc2)+(sd)$) .. (tc2);
        \draw (cr6) .. controls ($(cr6)+(su)$) and ($(tr3)+(sd)$) .. (tr3);
    \end{tikzpicture}\\
    = \qquad
    \raisebox{-1cm}{
    \begin{tikzpicture}
        \node[name=t, draw, minimum width=1.5cm] at (0cm,2cm) {$t$};
        \coordinate (tl) at ($(t.south west)!0.2!(t.south east)$);
        \coordinate (tr) at ($(t.south west)!0.8!(t.south east)$);
        \node[name=c0, draw] at (-1.5cm,0cm) {$c_{s_A}$};
        \coordinate (c0l) at ($(c0.north west)!0.3!(c0.north east)$);
        \coordinate (c0r) at ($(c0.north west)!0.7!(c0.north east)$);
        \node[name=c1, draw] at (0cm,0cm) {$c_{s_{12}s_B}$};
        \coordinate (c1l) at ($(c1.north west)!0.3!(c1.north east)$);
        \coordinate (c1r) at ($(c1.north west)!0.7!(c1.north east)$);
        \node[name=c2, draw] at (1.5cm,0cm) {$c_{-s_{31}s_C}$};
        \coordinate (c2l) at ($(c2.north west)!0.3!(c2.north east)$);
        \coordinate (c2r) at ($(c2.north west)!0.7!(c2.north east)$);
        \coordinate (su) at (0cm,0.5cm);
        \coordinate (sd) at (0cm,-0.5cm);
        \coordinate (out0) at (2cm,2.5cm);
        \coordinate (out1) at (2.5cm,2.5cm);
        \coordinate (out2) at (3cm,2.5cm);
        \draw (c0l) .. controls ($(c0l)+(su)$) and ($(tl)+(sd)$) .. (tl);
        \draw (c0r) .. controls ($(c0r)+(su)$) and ($(out0)+(sd)$) .. (out0);
        \draw (c1l) .. controls ($(c1l)+(su)$) and ($(t.south)+(sd)$) .. (t.south);
        \draw (c1r) .. controls ($(c1r)+(su)$) and ($(out1)+(sd)$) .. (out1);
        \draw (c2l) .. controls ($(c2l)+(su)$) and ($(tr)+(sd)$) .. (tr);
        \draw (c2r) .. controls ($(c2r)+(su)$) and ($(out2)+(sd)$) .. (out2);
    \end{tikzpicture}}
    \end{split}
\end{equation}
whenever $s_{12}s_{23}s_{31}=-1$.
\end{enumerate}

\begin{proposition}    \label{prop:invariance}
Let $A$ and $t,c_{\pm1}$ satisfy relations 1--5 above.
Let $\Lambda_1$ and $\Lambda_2$ be two spin triangulated surfaces with the same underlying spin surface $\Sigma$. Then $T_\text{triang}(\Lambda_1) = T_\text{triang}(\Lambda_2)$.
\end{proposition}

\begin{proof}
For $\alpha=1,2$, 
let $\Lambda_\alpha = ((\mathcal{C}_{[\alpha]},f_{[\alpha]i},d^1_{[\alpha]0},d^2_{[\alpha]0}), (\varphi_{[\alpha]}, \tilde\chi_{[\alpha]\sigma}) ,(\Sigma,\tilde\varphi_{i}))$.
To prove the assertion, we will modify $\Lambda_2$ in several steps, each one leaving $T_\text{triang}$ invariant.

By \cite[Cor.\,10.13]{munkres1966elementary} we can approximate $\varphi_{[2]}$ by a triangulation $\varphi_{[2a]} :|\mathcal{C}_{[2a]}|\to \underline\Sigma$ such that
	$\varphi_{[1]}^{-1}\circ \varphi_{[2a]}$ is piecewise-linear and $\mathcal{C}_{[2a]}$ is a subdivision of $\mathcal{C}_{[2]}$. We construct a spin triangulated surface $\Lambda_{2a}$ from $\Lambda_2$ in two steps. First we pass from $\mathcal{C}_{[2]}$ to $\mathcal{C}_{[2a]}$ via a series of Pachner 2-2, 3-1 and 1-3 moves, and carry out the spin lifts of this sequence of moves as chosen in Section \ref{sec:moves-pachner} on the spin triangulation $\Lambda_a$. 
Relations 4 and 5 guarantee that this does not change $T_\text{triang}$.
Then carry out the small deformation from the resulting map $\varphi : |\mathcal{C}_{[2a]}| \to \underline\Sigma$ to $\varphi_{[2a]}$, along with a lift of the deformation to the spin lifts $\tilde\chi$. 
This does not affect the combinatorial data, and hence not $T_\text{triang}$. Altogether,	
\begin{equation}
T_\text{triang}(\Lambda_2)
=
T_\text{triang}(\Lambda_{2a}) \ .
\end{equation}
Since the triangulations $\varphi_{[1]}$ and $\varphi_{[2a]}$ now have a common subdivision, we can pass from $\varphi_{[2a]}$ to $\varphi_{[1]}$ by a sequence of Pachner moves. Let $\Lambda_{2b}$ be the spin triangulation resulting from the spin lift of this sequence, so that again by Relations 4 and 5 we have
\begin{equation}
T_\text{triang}(\Lambda_{2a})
=
T_\text{triang}(\Lambda_{2b}) \ .
\end{equation}
At this point we have $\mathcal{C}_{[1]} = \mathcal{C}_{[2b]}$, $\varphi_{[1]} = \varphi_{[2b]}$ and $f_{[1]i} = f_{[2b]i}$, so that $\Lambda_1$ and $\Lambda_{2b}$ differ at most in the marking $d^2_0,d^1_0$ and in the choice of spin lifts $\tilde\chi$.

Denote by $s_1$, resp.\ $s_{2b}$ the edge signs resulting from $\Lambda_1$, resp.\ $\Lambda_{2b}$. Since $\Lambda_1$ and $\Lambda_{2b}$ are spin triangulations of the same spin surface $\Sigma$, by Theorem \ref{thm:reconstruction} we have $S(\underline \Lambda_1,s_1) \cong S(\underline \Lambda_{2b},s_{2b})$ as spin structures. Since the underlying triangulations already agree, by Theorem \ref{thm:parametrisation-of-spin} we must have $(d_{[1]},s_1) \sim_\text{fix} (d_{[2b]},s_{2b})$, where $d_{[1]} := (d_{[1]0}^1,d_{[1]0}^2)$ and dito for $d_{[2b]}$. By Relations 1--3, $T_\text{triang}$ is constant on equivalence classes for $\sim_\text{fix}$. Thus finally
\begin{equation}
T_\text{triang}(\Lambda_{2b})
=
T_\text{triang}(\Lambda_{1}) \ .
\end{equation}
\end{proof}

\subsection{Analysis of the algebraic structure}\label{sec:ana-alg-str}

\begin{figure}[tb]
    \begin{tabular}[h]{lc@{\hspace{3em}}lc@{\hspace{3em}}lc@{\hspace{3em}}lc}
        $b=$&
        \raisebox{-0.4cm}{
        \begin{tikzpicture}
            \node[name=b,circle,fill,inner sep=0.05cm] at (1cm,0.5cm){};
            \coordinate (in1) at (0.5cm,-0.5cm);
            \coordinate (in2) at (1.5cm,-0.5cm);
            \draw (in1) .. controls ($(in1)+(0cm,0.5cm)$) .. (b);
            \draw (in2) .. controls ($(in2)+(0cm,0.5cm)$) .. (b);
        \end{tikzpicture}}&
        $c_{-1}=$&
        \raisebox{-0.4cm}{
        \begin{tikzpicture}
            \node[name=b,circle,fill,inner sep=0.05cm] at (1cm,-0.5cm){};
            \coordinate (in1) at (0.5cm,0.5cm);
            \coordinate (in2) at (1.5cm,0.5cm);
            \draw (in1) .. controls ($(in1)+(0cm,-0.5cm)$) .. (b);
            \draw (in2) .. controls ($(in2)+(0cm,-0.5cm)$) .. (b);
        \end{tikzpicture}}&
        $N=$&
        \raisebox{-0.4cm}{
    \begin{tikzpicture}
    \node[name=N,circle,draw,inner sep=0.05cm] at (0.7cm,0cm){};
    \coordinate (in) at (0.7cm,-0.5cm);
    \coordinate (out) at (0.7cm,0.5cm);
    \draw (in) -- (N);
    \draw (N) -- (out);
\end{tikzpicture}}&
$N_\varepsilon=$&
        \raisebox{-0.4cm}{
\begin{tikzpicture}
    \node[name=N,circle,draw,inner sep=0.05cm] at (1cm,0cm){$\varepsilon$};
    \coordinate (in) at (1cm,-0.5cm);
    \coordinate (out) at (1cm,0.5cm);
    \draw (in) -- (N);
    \draw (N) -- (out);
\end{tikzpicture}}\\[1cm]
$\mu=$&
        \raisebox{-0.45cm}{
\begin{tikzpicture}
    \node[name=p,circle,fill,inner sep=0.05cm] at (1cm,0cm){};
    \coordinate (in1) at (0.5cm,-0.5cm);
    \coordinate (in2) at (1.5cm,-0.5cm);
    \coordinate (out) at (1cm,0.5cm);
    \draw (in1) -- (p);
    \draw (in2) -- (p);
    \draw (p) -- (out);
\end{tikzpicture}}&
$\eta=$&
\begin{tikzpicture}
    \node[name=N,circle,draw,inner sep=0.05cm] at (0.7cm,0cm){};
    \coordinate (out) at (0.7cm,0.5cm);
    \draw (N) -- (out);
\end{tikzpicture}&
$\Delta=$&
        \raisebox{-0.45cm}{
\begin{tikzpicture}
    \node[name=p,circle,fill,inner sep=0.05cm] at (1cm,0cm){};
    \coordinate (out1) at (0.5cm,0.5cm);
    \coordinate (out2) at (1.5cm,0.5cm);
    \coordinate (in) at (1cm,-0.5cm);
    \draw (out1) -- (p);
    \draw (out2) -- (p);
    \draw (p) -- (in);
\end{tikzpicture}}&
$\varepsilon=$&
        \raisebox{-0.45cm}{
\begin{tikzpicture}
    \node[name=N,circle,draw,inner sep=0.05cm] at (0.7cm,0cm){};
    \coordinate (in) at (0.7cm,-0.5cm);
    \draw (N) -- (in);
\end{tikzpicture}}
\end{tabular}
\caption{Abbreviated graphical notation for frequently used morphisms. 
	The pairing $b$ and copairing $c_{-1}$ give a duality on $A$, see \eqref{eq:nondegen}. 
The Nakayama automorphism $N$ is defined in \eqref{eq:Nak-definition} and $N_\varepsilon$ is defined in \eqref{eq:N_eps-def}. The product $\mu$ is given in \eqref{eq:A-product-def}, the existence of the unit $\eta$ as assumed in Assumption 2, the coproduct $\Delta$ and counit $\varepsilon$ are defined in \eqref{eq:coproduct-counit}.}
    \label{fig:graph-symbol}
\end{figure}
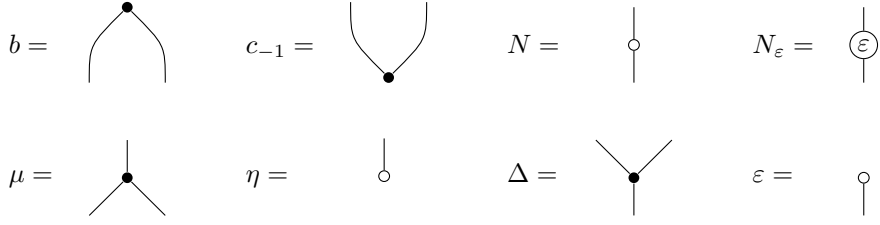

In the previous section we described a set of relations between the morphisms $c_1$, $c_{-1}$ and $t$, which guarantee invariance of the 
	$T_\text{triang}$
under the local moves. 
To further analyse these relations we will make two additional assumptions. The first one is:
\begin{quote}
{\bf{Assumption 1:}} The copairing $c_{-1}$ is nondegenerate, i.e.\ there is a map $b:A\otimes A \to \mathbf{1}$ such that
\begin{equation}\label{eq:nondegen}
    (b\otimes \id_A)\circ (\id_A \otimes c_{-1}) = \id_A = (\id_A\otimes b) \circ (c_{-1}\otimes \id_A) \ .
\end{equation}
\end{quote}
By \eqref{eq:comparison.of.copairings}, Assumption 1 in particular implies that also $c_{1}$ is nondegenerate. The map $b$ in Assumption 1 is unique, and we may use it to define
    \begin{equation} \label{eq:Nak-definition}
        N = (b\otimes \id_A) \circ (\id_A\otimes \sigma_{A,A}) \circ (\id_A\otimes c_{-1})  ~:~ A \longrightarrow A \ .
    \end{equation}
In graphical notation, this reads
\begin{equation} \label{eq:Nak-definition-pic}
    \raisebox{-0.8cm}{
    \begin{tikzpicture}
        \coordinate (in) at (0cm,0cm);
        \coordinate (out) at (0cm,2cm);
        \node[name=N,circle, draw, inner sep=0.05cm] at (0cm,1cm){};
        \draw (in) -- (N);
        \draw (N) -- (out);
    \end{tikzpicture}}
    ~=~
    \raisebox{-0.8cm}{
    \begin{tikzpicture}
        \coordinate (in) at (0.3cm,0cm);
        \coordinate (out) at (1.3cm,2cm);
        \node[name=b, circle, fill, inner sep=0.05cm] at (0.5cm,1.5cm){};
        \node[name=c, circle, fill, inner sep=0.05cm] at (1cm,0.5cm){};
        \coordinate (bl) at (0.3cm,1cm);
        \coordinate (br) at (1cm,1cm);
        \coordinate (cl) at (0.5cm,1cm);
        \coordinate (cr) at (1.5cm,1cm);
        \draw (in) .. controls (bl) .. (b);
        \draw (b) .. controls (br) and (cr) .. (c);
        \draw (c) .. controls (cl) and ($(out)+(0cm,-1cm)$) .. (out);
    \end{tikzpicture}} \ ,
\end{equation} 
where we have started to use shorthand graphical symbols for some morphisms that will appear frequently. These abbreviations are collected in Figure \ref{fig:graph-symbol}.

The map $N$ is invertible, with inverse given by
    \begin{equation} \label{eq:Nak-inv}
        N^{-1} = (\id_A \otimes b) \circ (\sigma_{A,A}\otimes \id_A) \circ (c_{-1}\otimes \id_A) \ .
    \end{equation}
One of the two computations to verify this is a follows:
\begin{equation}
	N^{-1} \circ N = 
    \raisebox{-1.5cm}{
    \begin{tikzpicture}
        \coordinate (in1) at (0.2cm,0cm);
        \coordinate (out1) at (1.3cm,1.5cm);
        \node[name=b1, circle, fill, inner sep=0.05cm] at (0.5cm,1.5cm){};
        \node[name=c1, circle, fill, inner sep=0.05cm] at (1cm,0.5cm){};
        \coordinate (bl1) at ($(b1)+(-0.3cm,-0.3cm)$);
        \coordinate (br1) at ($(b1)+(0.3cm,-0.3cm)$);
        \coordinate (cl1) at ($(c1)+(-0.3cm,0.3cm)$);
        \coordinate (cr1) at ($(c1)+(0.3cm,0.3cm)$);
        \draw (in1) .. controls (bl1) .. (b1);
        \draw (b1) .. controls (br1) and (cr1) .. (c1);
        \draw (c1) .. controls (cl1) and ($(out1)+(0cm,-0.5cm)$) .. (out1);
        \coordinate (in2) at (out1);
        \node[name=b2, circle, fill, inner sep=0.05cm] at ($(c1)+(0cm,2.5cm)$){};
        \node[name=c2, circle, fill, inner sep=0.05cm] at ($(b1)+(0cm,0.5cm)$){};
        \coordinate (bl2) at ($(b2)+(-0.3cm,-0.3cm)$);
        \coordinate (br2) at ($(b2)+(0.3cm,-0.3cm)$);
        \coordinate (cl2) at ($(c2)+(-0.3cm,0.3cm)$);
        \coordinate (cr2) at ($(c2)+(0.3cm,0.3cm)$);
        \coordinate (out2) at ($(in1)+(0cm,3.5cm)$);
        \draw (in2) .. controls (br2) .. (b2);
        \draw (b2) .. controls (bl2) and (cl2) .. (c2);
        \draw (c2) .. controls (cr2) and ($(out2)+(0cm,-0.5cm)$) .. (out2);
    \end{tikzpicture}}\stackrel{\text{deform}}{=}
    \raisebox{-1.5cm}{
    \begin{tikzpicture}
        \coordinate (in1) at (0.2cm,0cm);
        \coordinate (out1) at (1.3cm,1.5cm);
        \node[name=b1, circle, fill, inner sep=0.05cm] at (1cm,2cm){};
        \node[name=c1, circle, fill, inner sep=0.05cm] at (2cm,0.5cm){};
        \coordinate (bl1) at ($(b1)+(-0.3cm,-0.3cm)$);
        \coordinate (br1) at ($(b1)+(0.3cm,-0.3cm)$);
        \coordinate (cl1) at ($(c1)+(-0.3cm,0.3cm)$);
        \coordinate (cr1) at ($(c1)+(0.3cm,0.3cm)$);
        \draw (in1) .. controls ($(in1)+(0cm,0.5cm)$) and (bl1) .. (b1);
        \draw (b1) .. controls (br1) and (cr1) .. (c1);
        \node[name=c2, circle, fill, inner sep=0.05cm] at ($(c1)+(-1cm,0cm)$){};
        \node[name=b2, circle, fill, inner sep=0.05cm] at ($(c2)+(0cm,1cm)$){};
        \coordinate (bl2) at ($(b2)+(-0.3cm,-0.3cm)$);
        \coordinate (br2) at ($(b2)+(0.3cm,-0.3cm)$);
        \coordinate (cl2) at ($(c2)+(-0.3cm,0.3cm)$);
        \coordinate (cr2) at ($(c2)+(0.3cm,0.3cm)$);
        \coordinate (out2) at ($(in1)+(2cm,3.5cm)$);
        \draw (c1) .. controls (cl1) and (br2) .. (b2);
        \draw (b2) .. controls (bl2) and (cl2) .. (c2);
        \draw (c2) .. controls (cr2) and ($(out2)+(0cm,-1cm)$) .. (out2);
    \end{tikzpicture}}\stackrel{\eqref{eq:nondegen}}{=}
    \raisebox{-1.5cm}{
    \begin{tikzpicture}
        \coordinate (in1) at (0.2cm,0cm);
        \coordinate (out1) at (1.3cm,1.5cm);
        \node[name=b1, circle, fill, inner sep=0.05cm] at (0.5cm,2cm){};
        \coordinate (bl1) at ($(b1)+(-0.3cm,-0.3cm)$);
        \coordinate (br1) at ($(b1)+(0.3cm,-0.3cm)$);
        \draw (in1) .. controls ($(in1)+(0cm,0.5cm)$) and (bl1) .. (b1);
        \node[name=c2, circle, fill, inner sep=0.05cm] at (1cm,0.5cm){};
        \coordinate (cl2) at ($(c2)+(-0.3cm,0.3cm)$);
        \coordinate (cr2) at ($(c2)+(0.3cm,0.3cm)$);
        \coordinate (out2) at (1.3cm,3.5cm);
        \draw (b1) .. controls (br1) and (cl2) .. (c2);
        \draw (c2) .. controls (cr2) and ($(out2)+(0cm,-1cm)$) .. (out2);
    \end{tikzpicture}}\stackrel{\eqref{eq:nondegen}}{=}\id_A \ .
\end{equation}    
The map $N$ has an additional important property.
\begin{lemma}
    The map $N^2$ acts trivially on the map $t$:
    \begin{equation}
        t \circ (N^2 \otimes \id_{A \otimes A}) = t \circ (\id_A \otimes N^2 \otimes \id_A)= t \circ (\id_{A\otimes A} \otimes N^2) = t\ .
    \end{equation}
    \label{lem:naka.t}
\end{lemma}
\begin{proof}
   We verify the first equality. The other cases then follow from the cyclic property of $t$ stated in \eqref{eq:alg.cyclic}. We have:
    \begin{align}
    \raisebox{-0.8cm}{\begin{tikzpicture}
        \coordinate (su) at (0cm,0.5cm);
        \coordinate (sd) at (0cm,-0.5cm);
        \node[name=t, draw, minimum width=1.5cm, minimum height=1.4em] at (0cm,2cm) {$t$};
            \coordinate (t1) at ($(t.south west)!0.33!(t.south)$);
            \coordinate (t2) at (t.south);
            \coordinate (t3) at ($(t.south east)!0.33!(t.south)$);
        \node[name=N, draw, minimum height=1.4em] at ($(t1)+(0cm,-1cm)$){$N^2$};
        \foreach \N in {1,2,3} \coordinate (in\N) at ($(t\N)+(0cm,-1.7cm)$);
        \foreach \N/\M in {
            in1/N,
            N/t1,
            in2/t2,
            in3/t3} \draw (\N) .. controls ($(\N)+(su)$) and ($(\M)+(sd)$) .. (\M);
    \end{tikzpicture}}
    &\stackrel{\text{(1)}}{=}
    \raisebox{-0.8cm}{\begin{tikzpicture}
        \coordinate (su) at (0cm,0.3cm);
        \coordinate (sd) at (0cm,-0.3cm);
        \node[name=t, draw, minimum width=1.5cm, minimum height=1.4em] at (0cm,2cm) {$t$};
            \coordinate (t1) at ($(t.south west)!0.33!(t.south)$);
            \coordinate (t2) at (t.south);
            \coordinate (t3) at ($(t.south east)!0.33!(t.south)$);
        \node[name=b1, draw, minimum width=0.8cm, minimum height=1.4em] at (-1.5cm,2cm) {$b$};
        \node[name=b2, draw, minimum width=0.8cm, minimum height=1.4em] at (1.5cm,2cm) {$b$};
        \node[name=b3, draw, minimum width=0.8cm, minimum height=1.4em] at (2.75cm,2cm) {$b$};
        \foreach \N in {1,2,3} {
            \coordinate (b\N1) at ($(b\N.south west)!0.25!(b\N.south east)$);
            \coordinate (b\N2) at ($(b\N.south west)!0.75!(b\N.south east)$);
        }
        \node[name=c1, draw, minimum width=0.8cm, minimum height=1.4em] at ($0.5*(b12)+0.5*(t1)+(0cm,-1cm)$) {$c_{1}$};
        \node[name=c2, draw, minimum width=0.8cm, minimum height=1.4em] at ($(t2)+(0.3cm,-1cm)$){$c_{-1}$};
        \node[name=c3, draw, minimum width=0.8cm, minimum height=1.4em] at ($(t3)+(1cm,-1cm)$){$c_{-1}$};
        \foreach \N in {1,2,3} {
            \coordinate (c\N1) at ($(c\N.north west)!0.25!(c\N.north east)$);
            \coordinate (c\N2) at ($(c\N.north west)!0.75!(c\N.north east)$);
        }
        \node[name=N, draw, minimum height=1.4em] at ($(b11)+(-0.3cm,-1cm)$){$N$};
        \coordinate (in1) at ($(b11)+(-0.3cm,-1.7cm)$);
        \coordinate (in2) at ($(b31)+(0cm,-1.7cm)$);
        \coordinate (in3) at ($(b32)+(0.3cm,-1.7cm)$);
        \foreach \N/\M in {
            in1/N,
            N/b11,
            c11/b12,
            c12/t1,
            c21/t2,
            c22/b21,
            c31/t3,
            c32/b31,
            in2/b22,
            in3/b32}
            \draw (\N) .. controls ($(\N)+(su)$) and ($(\M)+(sd)$) .. (\M);
    \end{tikzpicture}}
    \\&\stackrel{\text{(2)}}{=}
    \raisebox{-0.8cm}{\begin{tikzpicture}
        \coordinate (su) at (0cm,0.3cm);
        \coordinate (sd) at (0cm,-0.3cm);
        \node[name=t, draw, minimum width=1.5cm, minimum height=1.4em] at (0cm,2cm) {$t$};
            \coordinate (t1) at ($(t.south west)!0.33!(t.south)$);
            \coordinate (t2) at (t.south);
            \coordinate (t3) at ($(t.south east)!0.33!(t.south)$);
        \node[name=b1, draw, minimum width=0.8cm, minimum height=1.4em] at (-1.5cm,2cm) {$b$};
        \node[name=b2, draw, minimum width=0.8cm, minimum height=1.4em] at (1.5cm,2cm) {$b$};
        \node[name=b3, draw, minimum width=0.8cm, minimum height=1.4em] at (2.75cm,2cm) {$b$};
        \foreach \N in {1,2,3} {
            \coordinate (b\N1) at ($(b\N.south west)!0.25!(b\N.south east)$);
            \coordinate (b\N2) at ($(b\N.south west)!0.75!(b\N.south east)$);
        }
        \node[name=c1, draw, minimum width=0.8cm, minimum height=1.4em] at ($0.5*(b12)+0.5*(t1)+(0cm,-1cm)$) {$c_{-1}$};
        \node[name=c2, draw, minimum width=0.8cm, minimum height=1.4em] at ($(t2)+(0.3cm,-1cm)$){$c_{1}$};
        \node[name=c3, draw, minimum width=0.8cm, minimum height=1.4em] at ($(t3)+(1cm,-1cm)$){$c_{1}$};
        \foreach \N in {1,2,3} {
            \coordinate (c\N1) at ($(c\N.north west)!0.25!(c\N.north east)$);
            \coordinate (c\N2) at ($(c\N.north west)!0.75!(c\N.north east)$);
        }
        \node[name=N, draw, minimum height=1.4em] at ($(b11)+(-0.3cm,-1cm)$){$N$};
        \coordinate (in1) at ($(b11)+(-0.3cm,-1.7cm)$);
        \coordinate (in2) at ($(b31)+(0cm,-1.7cm)$);
        \coordinate (in3) at ($(b32)+(0.3cm,-1.7cm)$);
        \foreach \N/\M in {
            in1/N,
            N/b11,
            c11/b12,
            c12/t1,
            c21/t2,
            c22/b21,
            c31/t3,
            c32/b31,
            in2/b22,
            in3/b32}
            \draw (\N) .. controls ($(\N)+(su)$) and ($(\M)+(sd)$) .. (\M);
    \end{tikzpicture}}
    \nonumber\\&\stackrel{\text{(3)}}{=}
    \raisebox{-0.8cm}{\begin{tikzpicture}
        \coordinate (su) at (0cm,0.3cm);
        \coordinate (sd) at (0cm,-0.3cm);
        \node[name=t, draw, minimum width=1.5cm, minimum height=1.4em] at (0cm,2cm) {$t$};
            \coordinate (t1) at ($(t.south west)!0.33!(t.south)$);
            \coordinate (t2) at (t.south);
            \coordinate (t3) at ($(t.south east)!0.33!(t.south)$);
        \node[name=b1, draw, minimum width=0.8cm, minimum height=1.4em] at (-1.5cm,2cm) {$b$};
        \node[name=b2, draw, minimum width=0.8cm, minimum height=1.4em] at (1.5cm,2cm) {$b$};
        \node[name=b3, draw, minimum width=0.8cm, minimum height=1.4em] at (2.75cm,2cm) {$b$};
        \foreach \N in {1,2,3} {
            \coordinate (b\N1) at ($(b\N.south west)!0.25!(b\N.south east)$);
            \coordinate (b\N2) at ($(b\N.south west)!0.75!(b\N.south east)$);
        }
        \node[name=c1, draw, minimum width=0.8cm, minimum height=1.4em] at ($0.5*(b12)+0.5*(t1)+(0cm,-1cm)$) {$c_{1}$};
        \node[name=c2, draw, minimum width=0.8cm, minimum height=1.4em] at ($(t2)+(0.3cm,-1cm)$){$c_{1}$};
        \node[name=c3, draw, minimum width=0.8cm, minimum height=1.4em] at ($(t3)+(1cm,-1cm)$){$c_{1}$};
        \foreach \N in {1,2,3} {
            \coordinate (c\N1) at ($(c\N.north west)!0.25!(c\N.north east)$);
            \coordinate (c\N2) at ($(c\N.north west)!0.75!(c\N.north east)$);
        }
        \coordinate (in1) at ($(b11)+(0cm,-1.7cm)$);
        \coordinate (in2) at ($(b31)+(0cm,-1.7cm)$);
        \coordinate (in3) at ($(b32)+(0.3cm,-1.7cm)$);
        \foreach \N/\M in {
            in1/b11,
            c11/b12,
            c12/t1,
            c21/t2,
            c22/b21,
            c31/t3,
            c32/b31,
            in2/b22,
            in3/b32}
            \draw (\N) .. controls ($(\N)+(su)$) and ($(\M)+(sd)$) .. (\M);
    \end{tikzpicture}}
    \stackrel{\text{(4)}}{=}
    \raisebox{-0.8cm}{\begin{tikzpicture}
        \coordinate (su) at (0cm,0.5cm);
        \coordinate (sd) at (0cm,-0.5cm);
        \node[name=t, draw, minimum width=1.5cm, minimum height=1.4em] at (0cm,2cm) {$t$};
            \coordinate (t1) at ($(t.south west)!0.33!(t.south)$);
            \coordinate (t2) at (t.south);
            \coordinate (t3) at ($(t.south east)!0.33!(t.south)$);
        \foreach \N in {1,2,3} \coordinate (in\N) at ($(t\N)+(0cm,-1.7cm)$);
        \foreach \N/\M in {
            in1/t1,
            in2/t2,
            in3/t3} \draw (\N) .. controls ($(\N)+(su)$) and ($(\M)+(sd)$) .. (\M);
    \end{tikzpicture}}
    \nonumber
    \end{align}
In step 1 we used the non-degeneracy of $c_{-1}$ (Assumption 1) to insert two pairs $b$, $c_{-1}$. We also replaced one of the $N$'s by its definition in \eqref{eq:Nak-definition} and used the edge orientation change \eqref{eq:comparison.of.copairings} to trade $c_{-1}$ for $c_1$. Step 2 is the leaf exchange \eqref{eq:alg.leaf}. In step 3 the leftmost $b$, $c_{-1}$ pair is replaced by $\id_A$ using Assumption 1 and the remaining $N$ is replaced by its definition in terms of $c_1$ as in step 1. Finally, in step 4 all $c_1$ are converted to $c_{-1}$ via a leaf exchange \eqref{eq:alg.leaf} and then the three pairs $b,c_{-1}$ are cancelled via Assumption 1.
\end{proof}

In particular, precomposing $t$ with $N$ is identical to precomposing with $N^{-1}$.
Below, we will frequently use of the notation
\begin{equation} \label{eq:N_eps-def}
	N_\varepsilon = \begin{cases} 
		\id_A &: \varepsilon = +1 \\
		N &: \varepsilon = -1 
	\end{cases}
	\quad .
\end{equation} 
Its graphical counterpart is shown in Figure \ref{fig:graph-symbol}. 

\medskip

We would like to cast the data $t,c_{\pm1}$ into a more standard algebraic form, namely that of a Frobenius algebra. 
We recall that a {\em Frobenius algebra} (in a monoidal category) is a unital associative algebra and a counital coassociative coalgebra such that the coproduct is a map of bimodules.

We start by introducing a product:
Let 
\begin{equation} \label{eq:A-product-def}
    \mu  = (t \otimes \id_A) \circ ( \id_{A\otimes A} \otimes  c_{-1})
    ~:~A\otimes A \to A \ .
\end{equation}
In graphical notation, this reads
\begin{equation}
    \mu = \raisebox{-1cm}{
    \begin{tikzpicture}
        \node[name=t1, draw, minimum width=1.5cm] at (0cm,1.5cm) {$t$};
        \coordinate (t1l) at ($(t1.south west)!0.2!(t1.south east)$);
        \coordinate (t1r) at ($(t1.south west)!0.8!(t1.south east)$);
        \coordinate (in1) at ($(t1l)+(0cm,-1cm)$);
        \coordinate (in2) at ($(t1.south)+(0cm,-1cm)$);
        \node[name=cp1, circle, fill, inner sep=0.05cm, draw] at (1cm,0.5cm) {};
        \draw (in1) -- (t1l);
        \draw (in2) -- (t1.south);
        \draw (t1r) .. controls +(down:0.2cm) .. (cp1);
        \draw (cp1) .. controls +(0.3cm,0.3cm) and (1.5cm,1.3cm) .. (1.5cm,2cm);
    \end{tikzpicture}}\;.
\end{equation}
For $\mu$ we will use the graphical shorthand listed in Figure \ref{fig:graph-symbol}.
Conversely, using the non-degeneracy of $c_{\pm 1}$ we can now write the map $t$ as 
\begin{equation}\label{eq:t-via-mu}
    t = b \circ (\mu \otimes \id).
\end{equation}

\begin{lemma}
    The map $\mu :A\otimes A \to A$ 
	is associative.
\end{lemma}

\begin{proof} 
	Using non-degeneracy of $c_{\pm1}$, we can rewrite the 
    cyclicity property \eqref{eq:alg.cyclic} as an identity of morphisms $A^{\otimes 3} \to {\bf 1}_\mathcal{S}$:
\begin{equation}\label{eq:cyclic-only-in}
    \raisebox{-3em}{
    \begin{tikzpicture}
        \coordinate (su) at (0cm,0.5cm);
        \coordinate (sd) at (0cm,-0.5cm);
        \node[name=t, draw, minimum width=1.5cm] at (0cm,2cm) {$t$};
        \coordinate (t1) at ($(t.south west)!0.2!(t.south east)$);
        \coordinate (t2) at (t.south);
        \coordinate (t3) at ($(t.south west)!0.8!(t.south east)$);
        \node[name=N0, draw, naka] at (-1.5cm,0cm) {$s(e_0)$};
        \node[name=N1, draw, naka] at (0cm,0cm) {$s(e_1)$};
        \node[name=N2, draw, naka] at (1.5cm,0cm) {$s(e_2)$};
        \foreach \N in {0,1,2} \coordinate (in\N) at ($(N\N)+(0cm,-1cm)$);
        \foreach \N/\M in {
            N0.north/t1,
            N1.north/t2,
            N2.north/t3,
            in0/N0.south,
            in1/N1.south,
            in2/N2.south}
            \draw (\N) .. controls ($(\N)+(su)$) and ($(\M)+(sd)$) .. (\M);
    \end{tikzpicture}}
    =
    \raisebox{-3em}{
    \begin{tikzpicture}
        \coordinate (su) at (0cm,0.5cm);
        \coordinate (sd) at (0cm,-0.5cm);
        \node[name=t, draw, minimum width=1.5cm] at (0cm,2cm) {$t$};
        \coordinate (t1) at ($(t.south west)!0.2!(t.south east)$);
        \coordinate (t2) at (t.south);
        \coordinate (t3) at ($(t.south west)!0.8!(t.south east)$);
        \node[name=N0, draw, naka] at (-1.5cm,0cm) {$-s(e_0)$};
        \node[name=N1, draw, naka] at (0cm,0cm) {$s(e_1)$};
        \node[name=N2, draw, naka] at (1.5cm,0cm) {$s(e_2)$};
        \foreach \N in {0,1,2} \coordinate (in\N) at ($(N\N)+(0cm,-1cm)$);
        \foreach \N/\M in {
            N0.north/t3,
            N1.north/t1,
            N2.north/t2,
            in0/N0.south,
            in1/N1.south,
            in2/N2.south}
            \draw (\N) .. controls ($(\N)+(su)$) and ($(\M)+(sd)$) .. (\M);
    \end{tikzpicture}}
\end{equation}
The same can be done for the Pachner moves \eqref{eq:alg.pachner2-2} and \eqref{eq:alg.pachner3-1}; we omit the details.
We then compute:
\begin{align}\label{eq:assoc-proof-calc}
    \raisebox{-0.8cm}{
    \begin{tikzpicture}
        \foreach \N in {1,2,3} \coordinate (in\N) at ($(0cm,0cm)+\N*(1cm,0cm)$);
        \node[name=p1, product] at (2.5cm,0.5cm){};
        \node[name=p2, product] at (2cm,1cm){};
        \coordinate (out) at (2cm,1.75cm);
        \draw (in1) -- (p2);
        \draw (in2) -- (p1);
        \draw (in3) -- (p1);
        \draw (p1) -- (p2);
        \draw (p2) -- (out);
    \end{tikzpicture}}
    &=
    \raisebox{-0.8cm}{\begin{tikzpicture}
        \coordinate (su) at (0cm,0.5cm);
        \coordinate (sd) at (0cm,-0.5cm);
        \node[name=t1, draw, minimum width=1.5cm] at (0cm,2cm) {$t$};
        \node[name=t2, draw, minimum width=1.5cm] at (2.5cm,2cm) {$t$};
        \foreach \N in {1,2} {
            \coordinate (t\N1) at ($(t\N.south west)!0.33!(t\N.south)$);
            \coordinate (t\N2) at (t\N.south);
            \coordinate (t\N3) at ($(t\N.south east)!0.33!(t\N.south)$);
        }
        \node[name=c1, draw] at ($0.5*(t1)+0.5*(t2)+(0cm,-1.5cm)$) {$c_{-1}$};
        \node[name=c2, draw] at ($(t2)+(1cm,-1.5cm)$) {$c_{-1}$};
        \foreach \N in {1,2} {
            \coordinate (c\N1) at ($(c\N.north west)!0.33!(c\N.north east)$);
            \coordinate (c\N2) at ($(c\N.north west)!0.66!(c\N.north east)$);
        }
        \foreach \N in {1,2,3} \coordinate (in\N) at ($(t1\N)+(0cm,-1.5cm)$);
        \coordinate (out) at ($(t2)+(1.5cm,0.5cm)$);
        \foreach \N/\M in {
            in1/t21,
            in2/t11,
            in3/t12,
            c11/t13,
            c12/t22,
            c21/t23,
            c22/out} \draw (\N) .. controls ($(\N)+(su)$) and ($(\M)+(sd)$) .. (\M);
        \end{tikzpicture}}
        \stackrel{\eqref{eq:cyclic-only-in}}{=}
    \raisebox{-0.8cm}{\begin{tikzpicture}
        \coordinate (su) at (0cm,0.5cm);
        \coordinate (sd) at (0cm,-0.5cm);
        \node[name=t1, draw, minimum width=1.5cm] at (0cm,2cm) {$t$};
        \node[name=t2, draw, minimum width=1.5cm] at (2.5cm,2cm) {$t$};
        \foreach \N in {1,2} {
            \coordinate (t\N1) at ($(t\N.south west)!0.33!(t\N.south)$);
            \coordinate (t\N2) at (t\N.south);
            \coordinate (t\N3) at ($(t\N.south east)!0.33!(t\N.south)$);
        }
        \node[name=c1, draw] at ($0.5*(t1)+0.5*(t2)+(0cm,-1.5cm)$) {$c_{1}$};
        \node[name=c2, draw] at ($(t2)+(1cm,-1.5cm)$) {$c_{-1}$};
        \foreach \N in {1,2} {
            \coordinate (c\N1) at ($(c\N.north west)!0.33!(c\N.north east)$);
            \coordinate (c\N2) at ($(c\N.north west)!0.66!(c\N.north east)$);
        }
        \foreach \N in {1,2,3} \coordinate (in\N) at ($(t1\N)+(0cm,-1.5cm)$);
        \coordinate (out) at ($(t2)+(1.5cm,0.5cm)$);
        \node[name=N1, naka] at ($(t11)+(0cm,-1cm)$) {};
        \foreach \N/\M in {
            in1/N1.south,
            N1.north/t23,
            in2/t12,
            in3/t13,
            c11/t11,
            c12/t21,
            c21/t22,
            c22/out} \draw (\N) .. controls ($(\N)+(su)$) and ($(\M)+(sd)$) .. (\M);
    \end{tikzpicture}}
    \\&\stackrel{\eqref{eq:alg.pachner2-2}}{=}
    \raisebox{-0.8cm}{\begin{tikzpicture}
        \coordinate (su) at (0cm,0.5cm);
        \coordinate (sd) at (0cm,-0.5cm);
        \node[name=t1, draw, minimum width=1.5cm] at (0cm,2cm) {$t$};
        \node[name=t2, draw, minimum width=1.5cm] at (2.5cm,2cm) {$t$};
        \foreach \N in {1,2} {
            \coordinate (t\N1) at ($(t\N.south west)!0.33!(t\N.south)$);
            \coordinate (t\N2) at (t\N.south);
            \coordinate (t\N3) at ($(t\N.south east)!0.33!(t\N.south)$);
        }
        \node[name=c1, draw] at ($0.5*(t1)+0.5*(t2)+(0cm,-1.5cm)$) {$c_{1}$};
        \node[name=c2, draw] at ($(t2)+(1cm,-1.5cm)$) {$c_{1}$};
        \foreach \N in {1,2} {
            \coordinate (c\N1) at ($(c\N.north west)!0.33!(c\N.north east)$);
            \coordinate (c\N2) at ($(c\N.north west)!0.66!(c\N.north east)$);
        }
        \foreach \N in {1,2,3} \coordinate (in\N) at ($(t1\N)+(0cm,-1.5cm)$);
        \coordinate (in3) at ($(t22)+(0cm,-1.5cm)$);
        \coordinate (out) at ($(t2)+(1.5cm,0.5cm)$);
        \node[name=N1, naka] at ($(t11)+(0cm,-1cm)$) {};
        \node[name=N2, naka] at ($(t12)+(0cm,-1cm)$) {};
        \node[name=N3, naka] at ($(t22)+(0cm,-1cm)$) {};
        \foreach \N/\M in {
            in1/N1.south,
            N1.north/t12,
            in2/N2.south,
            N2.north/t13,
            in3/N3.south,
            N3.north/t22,
            c11/t11,
            c12/t21,
            c21/t23,
            c22/out} \draw (\N) .. controls ($(\N)+(su)$) and ($(\M)+(sd)$) .. (\M);
    \end{tikzpicture}}
    \stackrel{\eqref{eq:alg.leaf}}{=}
    \raisebox{-0.8cm}{\begin{tikzpicture}
        \coordinate (su) at (0cm,0.5cm);
        \coordinate (sd) at (0cm,-0.5cm);
        \node[name=t1, draw, minimum width=1.5cm] at (0cm,2cm) {$t$};
        \node[name=t2, draw, minimum width=1.5cm] at (2.5cm,2cm) {$t$};
        \foreach \N in {1,2} {
            \coordinate (t\N1) at ($(t\N.south west)!0.33!(t\N.south)$);
            \coordinate (t\N2) at (t\N.south);
            \coordinate (t\N3) at ($(t\N.south east)!0.33!(t\N.south)$);
        }
        \node[name=c1, draw] at ($0.5*(t1)+0.5*(t2)+(0cm,-1.5cm)$) {$c_{1}$};
        \node[name=c2, draw] at ($(t2)+(1cm,-1.5cm)$) {$c_{-1}$};
        \foreach \N in {1,2} {
            \coordinate (c\N1) at ($(c\N.north west)!0.33!(c\N.north east)$);
            \coordinate (c\N2) at ($(c\N.north west)!0.66!(c\N.north east)$);
        }
        \foreach \N in {1,2,3} \coordinate (in\N) at ($(t1\N)+(0cm,-1.5cm)$);
        \coordinate (in3) at ($(t22)+(0cm,-1.5cm)$);
        \coordinate (out) at ($(t2)+(1.5cm,0.5cm)$);
        \foreach \N/\M in {
            in1/t12,
            in2/t13,
            in3/t22,
            c11/t11,
            c12/t21,
            c21/t23,
            c22/out} \draw (\N) .. controls ($(\N)+(su)$) and ($(\M)+(sd)$) .. (\M);
    \end{tikzpicture}}
    \nonumber\\&\stackrel{\eqref{eq:cyclic-only-in}}{=}
    \raisebox{-0.8cm}{\begin{tikzpicture}
        \coordinate (su) at (0cm,0.5cm);
        \coordinate (sd) at (0cm,-0.5cm);
        \node[name=t1, draw, minimum width=1.5cm] at (0cm,2cm) {$t$};
        \node[name=t2, draw, minimum width=1.5cm] at (2.5cm,2cm) {$t$};
        \foreach \N in {1,2} {
            \coordinate (t\N1) at ($(t\N.south west)!0.33!(t\N.south)$);
            \coordinate (t\N2) at (t\N.south);
            \coordinate (t\N3) at ($(t\N.south east)!0.33!(t\N.south)$);
        }
        \node[name=c1, draw] at ($0.5*(t1)+0.5*(t2)+(0cm,-1.5cm)$) {$c_{-1}$};
        \node[name=c2, draw] at ($(t2)+(1cm,-1.5cm)$) {$c_{-1}$};
        \foreach \N in {1,2} {
            \coordinate (c\N1) at ($(c\N.north west)!0.33!(c\N.north east)$);
            \coordinate (c\N2) at ($(c\N.north west)!0.66!(c\N.north east)$);
        }
        \foreach \N in {1,2,3} \coordinate (in\N) at ($(t1\N)+(0cm,-1.5cm)$);
        \coordinate (in3) at ($(t22)+(0cm,-1.5cm)$);
        \coordinate (out) at ($(t2)+(1.5cm,0.5cm)$);
        \foreach \N/\M in {
            in1/t11,
            in2/t12,
            in3/t22,
            c11/t13,
            c12/t21,
            c21/t23,
            c22/out} \draw (\N) .. controls ($(\N)+(su)$) and ($(\M)+(sd)$) .. (\M);
    \end{tikzpicture}}
    =
    \raisebox{-0.8cm}{
    \begin{tikzpicture}
        \foreach \N in {1,2,3} \coordinate (in\N) at ($(0cm,0cm)+\N*(1cm,0cm)$);
        \node[name=p1, product] at (1.5cm,0.5cm){};
        \node[name=p2, product] at (2cm,1cm){};
        \coordinate (out) at (2cm,1.75cm);
        \draw (in1) -- (p1);
        \draw (in2) -- (p1);
        \draw (in3) -- (p2);
        \draw (p1) -- (p2);
        \draw (p2) -- (out);
    \end{tikzpicture}}
\nonumber
\end{align}
\end{proof}

We already have a non-degenerate pairing on the algebra $A$, namely $b$. 

\begin{lemma}\label{lem:alg.invariance}
    The pairing $b$ is invariant with respect to the product $\mu $, i.e.
    \begin{equation}
        b \circ (\mu \otimes \id_A) = b \circ (\id_A \otimes  \mu ) \ .
    \end{equation}
\end{lemma}

\begin{proof} By direct computation:
\begin{align}
    \raisebox{-0.8cm}{
    \begin{tikzpicture}
        \foreach \N in {1,2,3} \coordinate (in\N) at ($(0cm,0cm)+\N*(1cm,0cm)$);
        \node[name=p1, product] at (2.5cm,0.75cm){};
        \node[name=p2, product] at (2cm,1.5cm){};
        \draw (in1) -- (p2);
        \draw (in2) -- (p1);
        \draw (in3) -- (p1);
        \draw (p1) -- (p2);
    \end{tikzpicture}}
    &=
    \raisebox{-0.8cm}{\begin{tikzpicture}
        \coordinate (su) at (0cm,0.5cm);
        \coordinate (sd) at (0cm,-0.5cm);
        \node[name=t, draw, minimum width=1.5cm] at (0cm,2cm) {$t$};
        \coordinate (t1) at ($(t.south west)!0.33!(t.south)$);
        \coordinate (t2) at (t.south);
        \coordinate (t3) at ($(t.south east)!0.33!(t.south)$);
        \node[name=b, draw, minimum width=0.8cm] at (1.5cm,2cm) {$b$};
        \coordinate (b1) at ($(b.south west)!0.33!(b.south east)$);
        \coordinate (b2) at ($(b.south west)!0.66!(b.south east)$);
        \node[name=c, draw] at ($0.5*(t)+0.5*(b)+(0cm,-1.5cm)$) {$c_{-1}$};
        \coordinate (c1) at ($(c.north west)!0.33!(c.north east)$);
        \coordinate (c2) at ($(c.north west)!0.66!(c.north east)$);
        \foreach \N in {1,2,3} \coordinate (in\N) at ($(t\N)+(-0.5cm,-1.5cm)$);
        \foreach \N/\M in {
            in1/b1,
            in2/t1,
            in3/t2,
            c1/t3,
            c2/b2} \draw (\N) .. controls ($(\N)+(su)$) and ($(\M)+(sd)$) .. (\M);
    \end{tikzpicture}}
    \stackrel{\eqref{eq:cyclic-only-in}}{=}
    \raisebox{-0.8cm}{\begin{tikzpicture}
        \coordinate (su) at (0cm,0.5cm);
        \coordinate (sd) at (0cm,-0.5cm);
        \node[name=t, draw, minimum width=1.5cm] at (0cm,2cm) {$t$};
        \coordinate (t1) at ($(t.south west)!0.33!(t.south)$);
        \coordinate (t2) at (t.south);
        \coordinate (t3) at ($(t.south east)!0.33!(t.south)$);
        \node[name=b, draw, minimum width=0.8cm] at (-1.5cm,2cm) {$b$};
        \coordinate (b1) at ($(b.south west)!0.33!(b.south east)$);
        \coordinate (b2) at ($(b.south west)!0.66!(b.south east)$);
        \node[name=c, draw] at ($0.5*(t)+0.5*(b)+(0cm,-1.5cm)$) {$c_{1}$};
        \coordinate (c1) at ($(c.north west)!0.33!(c.north east)$);
        \coordinate (c2) at ($(c.north west)!0.66!(c.north east)$);
        \foreach \N in {1,2,3} \coordinate (in\N) at ($(t\N)+(0,-1.5cm)$);
        \coordinate (in1) at ($(b1)+(0cm,-1.5cm)$);
        \foreach \N/\M in {
            in1/b1,
            in2/t2,
            in3/t3,
            c1/t1,
            c2/b2} \draw (\N) .. controls ($(\N)+(su)$) and ($(\M)+(sd)$) .. (\M);
    \end{tikzpicture}}
    \\&\stackrel{\text{nondeg.}}{=}
    \raisebox{-0.8cm}{\begin{tikzpicture}
        \node[name=t, draw, minimum width=1.5cm] at (0cm,2cm) {$t$};
        \coordinate (t1) at ($(t.south west)!0.33!(t.south)$);
        \coordinate (t2) at (t.south);
        \coordinate (t3) at ($(t.south east)!0.33!(t.south)$);
        \foreach \N in {1,2,3} {
            \coordinate (in\N) at ($(t\N)+(0,-1.5cm)$);
            \draw (in\N) -- (t\N);
        }
    \end{tikzpicture}}
    \stackrel{\text{nondeg.}}{=}
    \raisebox{-0.8cm}{\begin{tikzpicture}
        \coordinate (su) at (0cm,0.5cm);
        \coordinate (sd) at (0cm,-0.5cm);
        \node[name=t, draw, minimum width=1.5cm] at (0cm,2cm) {$t$};
        \coordinate (t1) at ($(t.south west)!0.33!(t.south)$);
        \coordinate (t2) at (t.south);
        \coordinate (t3) at ($(t.south east)!0.33!(t.south)$);
        \node[name=b, draw, minimum width=0.8cm] at (1.5cm,2cm) {$b$};
        \coordinate (b1) at ($(b.south west)!0.33!(b.south east)$);
        \coordinate (b2) at ($(b.south west)!0.66!(b.south east)$);
        \node[name=c, draw] at ($0.5*(t)+0.5*(b)+(0cm,-1.5cm)$) {$c_{-1}$};
        \coordinate (c1) at ($(c.north west)!0.33!(c.north east)$);
        \coordinate (c2) at ($(c.north west)!0.66!(c.north east)$);
        \foreach \N in {1,2,3} \coordinate (in\N) at ($(t\N)+(0cm,-1.5cm)$);
        \coordinate (in3) at ($(b2)+(0cm,-1.5cm)$);
        \foreach \N/\M in {
            in1/t1,
            in2/t2,
            in3/b2,
            c1/t3,
            c2/b1} \draw (\N) .. controls ($(\N)+(su)$) and ($(\M)+(sd)$) .. (\M);
    \end{tikzpicture}}=
    \raisebox{-0.8cm}{
    \begin{tikzpicture}
        \foreach \N in {1,2,3} \coordinate (in\N) at ($(0cm,0cm)+\N*(1cm,0cm)$);
        \node[name=p1, product] at (1.5cm,0.75cm){};
        \node[name=p2, product] at (2cm,1.5cm){};
        \draw (in1) -- (p1);
        \draw (in2) -- (p1);
        \draw (in3) -- (p2);
        \draw (p1) -- (p2);
    \end{tikzpicture}}
\nonumber
\end{align}
\end{proof}

To proceed, we need to make our second assumption:
\begin{quote}
{\bf Assumption 2:} The algebra $(A,\mu)$ has a unit $\eta : \mathbf{1}_\mathcal{S} \to A$.
\end{quote}

The graphical notation we use for the unit is listed in Figure \ref{fig:graph-symbol}.
By Lemma \ref{lem:alg.invariance} and Assumption 2, $A$ together with the non-degenerate pairing $b$ is a Frobenius algebra.

One may now define a coalgebra structure on the Frobenius algebra $A$ in the standard way, so that the coproduct is a map of $A$-$A$-bimodules. Explicitly, the coproduct $\Delta :A\to A\otimes A$ and counit $\varepsilon : A \to \mathbf{1}$ are given by
\begin{equation}\label{eq:coproduct-counit}
    \Delta = 
    \raisebox{-0.5cm}{\begin{tikzpicture}
        \coordinate (su) at (0cm,0.5cm);
        \coordinate (sd) at (0cm,-0.5cm);
        \coordinate (sdl) at (-0.5cm,-0.5cm);
        \coordinate (sdr) at (0.5cm,-0.5cm);
        \coordinate (sul) at (-0.5cm,0.5cm);
        \coordinate (sur) at (0.5cm,0.5cm);
        \coordinate (in) at (0.5cm,0cm);
        \coordinate (out1) at (1cm,1.5cm);
        \coordinate (out2) at (2cm,1.5cm);
        \node[name=c, product] at (1.5cm,0.5cm){};
        \node[name=p, product] at (1cm,1cm){};
        \draw (p) -- (c);
        \draw (p) -- (out1);
        \foreach \N/\M/\S/\T in {in/p/su/sdl,c/out2/sur/sd}  
            \draw (\N) .. controls ($(\N)+(\S)$) and ($(\M)+(\T)$) .. (\M);
    \end{tikzpicture}}\quad ,  \qquad
    \varepsilon =
    \raisebox{-0.5cm}{\begin{tikzpicture}
        \coordinate (su) at (0cm,0.3cm);
        \coordinate (sd) at (0cm,-0.3cm);
        \coordinate (sdl) at (-0.3cm,-0.3cm);
        \coordinate (sdr) at (0.3cm,-0.3cm);
        \coordinate (sul) at (-0.3cm,0.3cm);
        \coordinate (sur) at (0.3cm,0.3cm);
        \coordinate (in) at (0cm,0cm);
        \node[name=b, product] at (0.5cm,1cm){};
        \node[name=u, naka] at (0.8cm,0.3cm){};
        \foreach \N/\M/\S/\T in {in/b/su/sdl,u/b/su/sdr}  
            \draw (\N) .. controls ($(\N)+(\S)$) and ($(\M)+(\T)$) .. (\M);
    \end{tikzpicture}}\quad .
\end{equation}
The asymmetry in this definition is only apparent, since
\begin{equation}\label{eq:coproduct-counit-2nd}
    \begin{split}
    \Delta &\stackrel{\text{nondeg.}}{=}
    \raisebox{-0.5cm}{\begin{tikzpicture}
        \coordinate (su) at (0cm,0.5cm);
        \coordinate (sd) at (0cm,-0.5cm);
        \coordinate (sdl) at (-0.5cm,-0.5cm);
        \coordinate (sdr) at (0.5cm,-0.5cm);
        \coordinate (sul) at (-0.5cm,0.5cm);
        \coordinate (sur) at (0.5cm,0.5cm);
        \coordinate (in) at (0.5cm,0cm);
        \coordinate (out1) at (-1cm,1.5cm);
        \coordinate (out2) at (2cm,1.5cm);
        \node[name=c1, product] at (1.5cm,0.5cm){};
        \node[name=c2, product] at (-0.5cm,0.5cm){};
        \node[name=b, product] at (0.5cm,1.5cm){};
        \node[name=p, product] at (1cm,1cm){};
        \foreach \N/\M/\S/\T in {
            in/p/su/sdl,
            c1/p/sul/sdr,
            p/b/sul/sdr,
            c2/b/sur/sdl,
            c2/out1/sul/sd,
            c1/out2/sur/sd}  
            \draw (\N) .. controls ($(\N)+(\S)$) and ($(\M)+(\T)$) .. (\M);
    \end{tikzpicture}}
    \stackrel{\text{invariance}}{=}
    \raisebox{-0.5cm}{\begin{tikzpicture}
        \coordinate (su) at (0cm,0.5cm);
        \coordinate (sd) at (0cm,-0.5cm);
        \coordinate (sdl) at (-0.5cm,-0.5cm);
        \coordinate (sdr) at (0.5cm,-0.5cm);
        \coordinate (sul) at (-0.5cm,0.5cm);
        \coordinate (sur) at (0.5cm,0.5cm);
        \coordinate (in) at (0.5cm,0cm);
        \coordinate (out1) at (-1cm,1.5cm);
        \coordinate (out2) at (2cm,1.5cm);
        \node[name=c1, product] at (1.5cm,0.5cm){};
        \node[name=c2, product] at (-0.5cm,0.5cm){};
        \node[name=b, product] at (0.5cm,1.5cm){};
        \node[name=p, product] at (0cm,1cm){};
        \foreach \N/\M/\S/\T in {
            in/p/su/sdr,
            c1/b/sul/sdr,
            p/b/sur/sdl,
            c2/p/sur/sdl,
            c2/out1/sul/sd,
            c1/out2/sur/sd}  
            \draw (\N) .. controls ($(\N)+(\S)$) and ($(\M)+(\T)$) .. (\M);
    \end{tikzpicture}}
    \stackrel{\text{nondeg.}}{=}
    \raisebox{-0.5cm}{\begin{tikzpicture}
        \coordinate (su) at (0cm,0.5cm);
        \coordinate (sd) at (0cm,-0.5cm);
        \coordinate (sdl) at (-0.5cm,-0.5cm);
        \coordinate (sdr) at (0.5cm,-0.5cm);
        \coordinate (sul) at (-0.5cm,0.5cm);
        \coordinate (sur) at (0.5cm,0.5cm);
        \coordinate (in) at (0.5cm,0cm);
        \coordinate (out1) at (-1cm,1.5cm);
        \coordinate (out2) at (0cm,1.5cm);
        \node[name=c2, product] at (-0.5cm,0.5cm){};
        \node[name=p, product] at (0cm,1cm){};
        \foreach \N/\M/\S/\T in {
            in/p/su/sdr,
            c2/p/sur/sdl,
            c2/out1/sul/sd,
            p/out2/su/sd}
            \draw (\N) .. controls ($(\N)+(\S)$) and ($(\M)+(\T)$) .. (\M);
    \end{tikzpicture}}\ ,\\
    \varepsilon &\stackrel{\text{unit}}{=}
    \raisebox{-0.5cm}{\begin{tikzpicture}
        \coordinate (su) at (0cm,0.3cm);
        \coordinate (sd) at (0cm,-0.3cm);
        \coordinate (sdl) at (-0.3cm,-0.3cm);
        \coordinate (sdr) at (0.3cm,-0.3cm);
        \coordinate (sul) at (-0.3cm,0.3cm);
        \coordinate (sur) at (0.3cm,0.3cm);
        \coordinate (in) at (0.5cm,-0.5cm);
        \node[name=b, product] at (0.5cm,1cm){};
        \node[name=p, product] at (0cm,0.5cm){};
        \node[name=u1, naka] at (0.8cm,0.3cm){};
        \node[name=u2, naka] at (-0.3cm,-0.2cm){};
        \foreach \N/\M/\S/\T in {
            in/p/su/sdr,
            u1/b/su/sdr,
            u2/p/su/sdl,
            p/b/sur/sdl}  
            \draw (\N) .. controls ($(\N)+(\S)$) and ($(\M)+(\T)$) .. (\M);
    \end{tikzpicture}}
    \stackrel{\text{inv.}}{=}
    \raisebox{-0.5cm}{\begin{tikzpicture}
        \begin{scope}[xscale=-1] 
        \coordinate (su) at (0cm,0.3cm);
        \coordinate (sd) at (0cm,-0.3cm);
        \coordinate (sdl) at (-0.3cm,-0.3cm);
        \coordinate (sdr) at (0.3cm,-0.3cm);
        \coordinate (sul) at (-0.3cm,0.3cm);
        \coordinate (sur) at (0.3cm,0.3cm);
        \coordinate (in) at (0.5cm,-0.5cm);
        \node[name=b, product] at (0.5cm,1cm){};
        \node[name=p, product] at (0cm,0.5cm){};
        \node[name=u1, naka] at (0.8cm,0.3cm){};
        \node[name=u2, naka] at (-0.3cm,-0.2cm){};
        \end{scope}
        \foreach \N/\M/\S/\T in {
            in/p/su/sdr,
            u1/b/su/sdr,
            u2/p/su/sdl,
            p/b/sur/sdl}  
            \draw (\N) .. controls ($(\N)+(\S)$) and ($(\M)+(\T)$) .. (\M);
    \end{tikzpicture}}
    \stackrel{\text{unit}}{=}
    \raisebox{-0.5cm}{\begin{tikzpicture}
        \begin{scope}[xscale=-1] 
        \coordinate (su) at (0cm,0.3cm);
        \coordinate (sd) at (0cm,-0.3cm);
        \coordinate (sdl) at (-0.3cm,-0.3cm);
        \coordinate (sdr) at (0.3cm,-0.3cm);
        \coordinate (sul) at (-0.3cm,0.3cm);
        \coordinate (sur) at (0.3cm,0.3cm);
        \coordinate (in) at (0cm,0cm);
        \node[name=b, product] at (0.5cm,1cm){};
        \node[name=u, naka] at (0.8cm,0.3cm){};
        \end{scope}
        \foreach \N/\M/\S/\T in {in/b/su/sdl,u/b/su/sdr}  
            \draw (\N) .. controls ($(\N)+(\S)$) and ($(\M)+(\T)$) .. (\M);
    \end{tikzpicture}}\ .
    \end{split}
\end{equation}
It is now trivial to see that $\varepsilon$ is indeed a counit. Coassociativity is easily checked by combining the two expressions for $\Delta$:
\begin{equation}
    (\Delta \otimes \id_A)\circ \Delta =
    \raisebox{-0.5cm}{\begin{tikzpicture}
        \coordinate (su) at (0cm,0.5cm);
        \coordinate (sd) at (0cm,-0.5cm);
        \coordinate (sdl) at (-0.5cm,-0.5cm);
        \coordinate (sdr) at (0.5cm,-0.5cm);
        \coordinate (sul) at (-0.5cm,0.5cm);
        \coordinate (sur) at (0.5cm,0.5cm);
        \coordinate (in) at (0.5cm,0cm);
        \coordinate (out1) at (-0.5cm,2cm);
        \coordinate (out2) at (0.5cm,2cm);
        \coordinate (out3) at (2cm,2cm);
        \node[name=c1, product] at (1.5cm,0.5cm){};
        \node[name=c2, product] at (0cm,1cm){};
        \node[name=p2, product] at (0.5cm,1.5cm){};
        \node[name=p1, product] at (1cm,1cm){};
        \foreach \N/\M/\S/\T in {
            in/p1/su/sdl,
            c1/p1/sul/sdr,
            p1/p2/sul/sdr,
            c2/p2/sur/sdl,
            c2/out1/sul/sd,
            p2/out2/su/sd,
            c1/out3/sur/sd}  
            \draw (\N) .. controls ($(\N)+(\S)$) and ($(\M)+(\T)$) .. (\M);
    \end{tikzpicture}}
    = 
    \raisebox{-0.5cm}{\begin{tikzpicture}
        \begin{scope}[xscale=-1] 
        \coordinate (su) at (0cm,0.5cm);
        \coordinate (sd) at (0cm,-0.5cm);
        \coordinate (sdl) at (-0.5cm,-0.5cm);
        \coordinate (sdr) at (0.5cm,-0.5cm);
        \coordinate (sul) at (-0.5cm,0.5cm);
        \coordinate (sur) at (0.5cm,0.5cm);
        \coordinate (in) at (0.5cm,0cm);
        \coordinate (out1) at (-0.5cm,2cm);
        \coordinate (out2) at (0.5cm,2cm);
        \coordinate (out3) at (2cm,2cm);
        \node[name=c1, product] at (1.5cm,0.5cm){};
        \node[name=c2, product] at (0cm,1cm){};
        \node[name=p2, product] at (0.5cm,1.5cm){};
        \node[name=p1, product] at (1cm,1cm){};
        \end{scope}
        \foreach \N/\M/\S/\T in {
            in/p1/su/sdl,
            c1/p1/sul/sdr,
            p1/p2/sul/sdr,
            c2/p2/sur/sdl,
            c2/out1/sul/sd,
            p2/out2/su/sd,
            c1/out3/sur/sd}  
            \draw (\N) .. controls ($(\N)+(\S)$) and ($(\M)+(\T)$) .. (\M);
    \end{tikzpicture}}
    = (\id_A \otimes \Delta) \circ \Delta \ .
\end{equation}
The Frobenius property, which states that $\Delta$ is a bimodule map, namely
\begin{equation}
    \raisebox{-0.5cm}{\begin{tikzpicture}
        \coordinate (su) at (0cm,0.5cm);
        \coordinate (sd) at (0cm,-0.5cm);
        \coordinate (sdl) at (-0.5cm,-0.5cm);
        \coordinate (sdr) at (0.5cm,-0.5cm);
        \coordinate (sul) at (-0.5cm,0.5cm);
        \coordinate (sur) at (0.5cm,0.5cm);
        \coordinate (in1) at (0cm,0cm);
        \coordinate (in2) at (1cm,0cm);
        \coordinate (out1) at (0.5cm,1.5cm);
        \coordinate (out2) at (1.5cm,1.5cm);
        \node[name=p, product] at (0.5cm,1cm){};
        \node[name=cp, product] at (1cm,0.5cm){};
        \foreach \N/\M/\S/\T in {
            in2/cp/su/sd,
            in1/p/su/sdl,
            cp/p/sul/sdr,
            cp/out2/sur/sd,
            p/out1/su/sd}
            \draw (\N) .. controls ($(\N)+(\S)$) and ($(\M)+(\T)$) .. (\M);
    \end{tikzpicture}}
    =
    \raisebox{-0.5cm}{\begin{tikzpicture}
        \coordinate (su) at (0cm,0.5cm);
        \coordinate (sd) at (0cm,-0.5cm);
        \coordinate (sdl) at (-0.5cm,-0.5cm);
        \coordinate (sdr) at (0.5cm,-0.5cm);
        \coordinate (sul) at (-0.5cm,0.5cm);
        \coordinate (sur) at (0.5cm,0.5cm);
        \coordinate (in1) at (0cm,0cm);
        \coordinate (in2) at (1cm,0cm);
        \coordinate (out1) at (0cm,1.5cm);
        \coordinate (out2) at (1cm,1.5cm);
        \node[name=p, product] at (0.5cm,0.5cm){};
        \node[name=cp, product] at (0.5cm,1cm){};
        \foreach \N/\M/\S/\T in {
            in1/p/sur/sdl,
            in2/p/sul/sdr,
            p/cp/su/sd,
            cp/out1/sul/sdr,
            cp/out2/sur/sdl}
            \draw (\N) .. controls ($(\N)+(\S)$) and ($(\M)+(\T)$) .. (\M);
    \end{tikzpicture}}
    =
    \raisebox{-0.5cm}{\begin{tikzpicture}
        \begin{scope}[xscale=-1] 
        \coordinate (su) at (0cm,0.5cm);
        \coordinate (sd) at (0cm,-0.5cm);
        \coordinate (sdl) at (-0.5cm,-0.5cm);
        \coordinate (sdr) at (0.5cm,-0.5cm);
        \coordinate (sul) at (-0.5cm,0.5cm);
        \coordinate (sur) at (0.5cm,0.5cm);
        \coordinate (in1) at (0cm,0cm);
        \coordinate (in2) at (1cm,0cm);
        \coordinate (out1) at (0.5cm,1.5cm);
        \coordinate (out2) at (1.5cm,1.5cm);
        \node[name=p, product] at (0.5cm,1cm){};
        \node[name=cp, product] at (1cm,0.5cm){};
        \end{scope}
        \foreach \N/\M/\S/\T in {
            in2/cp/su/sd,
            in1/p/su/sdl,
            cp/p/sul/sdr,
            cp/out2/sur/sd,
            p/out1/su/sd}
            \draw (\N) .. controls ($(\N)+(\S)$) and ($(\M)+(\T)$) .. (\M);
    \end{tikzpicture}}\ ,
\end{equation}
is equally straightforward to check. We omit the details. Finally, note that
\begin{equation}
	b = \varepsilon \circ \mu
	\qquad , \qquad
	c_{-1} = \Delta \circ \eta \ .
\end{equation}
From hereon we consider $A$ as a Frobenius algebra with structure morphisms $\mu,\eta,\Delta,\varepsilon$ as described above. By definition, the morphism $N$ defined in \eqref{eq:Nak-definition} is the Nakayama automorphism of $A$, see e.g.\ \cite{fuchsstigner2009}. For completeness we state

\begin{proposition}\label{prop:Nak-Frob-auto}
The Nakayama automorphism is a unital algebra automorphism and a counital coalgebra automorphism of a Frobenius algebra.
\end{proposition}

\begin{proof}
That $N \circ \eta = \eta$ and $\varepsilon \circ N = \varepsilon$ is straightforward. Compatibility with the product follows from
\begin{equation}
\begin{split}
N \circ \mu &=
    \raisebox{-0.5cm}{\begin{tikzpicture}
        \coordinate (su) at (0cm,0.5cm);
        \coordinate (sd) at (0cm,-0.5cm);
        \coordinate (sdl) at (-0.5cm,-0.5cm);
        \coordinate (sdr) at (0.5cm,-0.5cm);
        \coordinate (sul) at (-0.5cm,0.5cm);
        \coordinate (sur) at (0.5cm,0.5cm);
        \coordinate (in1) at (0cm,-0.5cm);
        \coordinate (in2) at (1cm,-0.5cm);
        \coordinate (out) at (2cm,2cm);
        \node[name=p, product] at (0.5cm,0.5cm){};
        \node[name=b, product] at (1cm,1.5cm){};
        \node[name=c, product] at (1.5cm,0.5cm){};
        \foreach \N/\M/\S/\T in {
            in1/p/su/sdl,
            in2/p/su/sdr,
            p/b/su/sdl,
            c/b/sur/sdr,
            c/out/sul/sd}
            \draw (\N) .. controls ($(\N)+(\S)$) and ($(\M)+(\T)$) .. (\M);
    \end{tikzpicture}}
    \stackrel{\text{ass.}}{=}
    \raisebox{-0.5cm}{\begin{tikzpicture}
        \coordinate (su) at (0cm,0.5cm);
        \coordinate (sd) at (0cm,-0.5cm);
        \coordinate (sdl) at (-0.5cm,-0.5cm);
        \coordinate (sdr) at (0.5cm,-0.5cm);
        \coordinate (sul) at (-0.5cm,0.5cm);
        \coordinate (sur) at (0.5cm,0.5cm);
        \coordinate (in1) at (0cm,-0.5cm);
        \coordinate (in2) at (0.5cm,-0.5cm);
        \coordinate (out) at (2cm,2cm);
        \node[name=p, product] at (1.5cm,1cm){};
        \node[name=b, product] at (1cm,1.5cm){};
        \node[name=c, product] at (1cm,0cm){};
        \foreach \N/\M/\S/\T in {
            in1/b/su/sdl,
            in2/p/su/sdl,
            p/b/sul/sdr,
            c/p/sur/sdr,
            c/out/sul/sd}
            \draw (\N) .. controls ($(\N)+(\S)$) and ($(\M)+(\T)$) .. (\M);
    \end{tikzpicture}}
    \stackrel{\text{Frob.}}{=}
    \raisebox{-0.5cm}{\begin{tikzpicture}
        \coordinate (su) at (0cm,0.5cm);
        \coordinate (sd) at (0cm,-0.5cm);
        \coordinate (sdl) at (-0.5cm,-0.5cm);
        \coordinate (sdr) at (0.5cm,-0.5cm);
        \coordinate (sul) at (-0.5cm,0.5cm);
        \coordinate (sur) at (0.5cm,0.5cm);
        \coordinate (in1) at (0cm,-0.5cm);
        \coordinate (in2) at (0.5cm,-0.5cm);
        \coordinate (out) at (2cm,2cm);
        \node[name=cp, product] at (1.5cm,0.5cm){};
        \node[name=b1, product] at (1cm,1.5cm){};
        \node[name=b2, product] at (1cm,1cm){};
        \node[name=c, product] at (1cm,0cm){};
        \foreach \N/\M/\S/\T in {
            in1/b1/su/sdl,
            in2/b2/su/sdl,
            cp/b1/sur/sdr,
            cp/b2/sul/sdr,
            c/cp/sur/sdl,
            c/out/sul/sd}
            \draw (\N) .. controls ($(\N)+(\S)$) and ($(\M)+(\T)$) .. (\M);
    \end{tikzpicture}}
    \stackrel{\text{coass.}}{=}
    \raisebox{-0.5cm}{\begin{tikzpicture}
        \coordinate (su) at (0cm,0.5cm);
        \coordinate (sd) at (0cm,-0.5cm);
        \coordinate (sdl) at (-0.5cm,-0.5cm);
        \coordinate (sdr) at (0.5cm,-0.5cm);
        \coordinate (sul) at (-0.5cm,0.5cm);
        \coordinate (sur) at (0.5cm,0.5cm);
        \coordinate (in1) at (0cm,-0.5cm);
        \coordinate (in2) at (0.5cm,-0.5cm);
        \coordinate (out) at (2cm,2cm);
        \node[name=cp, product] at (1.2cm,0.3cm){};
        \node[name=b1, product] at (1cm,1.5cm){};
        \node[name=b2, product] at (1cm,1cm){};
        \node[name=c, product] at (1.5cm,0cm){};
        \foreach \N/\M/\S/\T in {
            in1/b1/su/sdl,
            in2/b2/su/sdl,
            cp/b2/sur/sdr,
            cp/out/sul/sd,
            c/cp/sul/sdr,
            c/b1/sur/sdr}
            \draw (\N) .. controls ($(\N)+(\S)$) and ($(\M)+(\T)$) .. (\M);
    \end{tikzpicture}}
    \\&\stackrel{\text{def.\ of $\Delta$}}{=}
    \raisebox{-0.5cm}{\begin{tikzpicture}
        \coordinate (su) at (0cm,0.5cm);
        \coordinate (sd) at (0cm,-0.5cm);
        \coordinate (sdl) at (-0.5cm,-0.5cm);
        \coordinate (sdr) at (0.5cm,-0.5cm);
        \coordinate (sul) at (-0.5cm,0.5cm);
        \coordinate (sur) at (0.5cm,0.5cm);
        \coordinate (in1) at (0cm,0cm);
        \coordinate (in2) at (0.5cm,0cm);
        \coordinate (out) at (1.5cm,2.2cm);
        \node[name=p, product] at (1.5cm,1.5cm){};
        \node[name=b1, product] at (1cm,1.5cm){};
        \node[name=b2, product] at (1cm,2cm){};
        \node[name=c1, product] at (1.5cm,0.5cm){};
        \node[name=c2, product] at (2cm,1cm){};
        \foreach \N/\M/\S/\T in {
            in1/b1/su/sdl,
            in2/b2/su/sdl,
            p/out/su/sd,
            c2/p/sul/sdr,
            c2/b2/sur/sdr,
            c1/p/sul/sdl,
            c1/b1/sur/sdr}
            \draw (\N) .. controls ($(\N)+(\S)$) and ($(\M)+(\T)$) .. (\M);
    \end{tikzpicture}}
    \stackrel{\text{deform}}{=}
    \raisebox{-0.5cm}{\begin{tikzpicture}
        \coordinate (su) at (0cm,0.3cm);
        \coordinate (sd) at (0cm,-0.3cm);
        \coordinate (sdl) at (-0.3cm,-0.3cm);
        \coordinate (sdr) at (0.3cm,-0.3cm);
        \coordinate (sul) at (-0.3cm,0.3cm);
        \coordinate (sur) at (0.3cm,0.3cm);
        \coordinate (in1) at (0.5cm,0cm);
        \coordinate (in2) at (1.5cm,0cm);
        \coordinate (out) at (1.5cm,2.2cm);
        \node[name=p, product] at (1.5cm,1.5cm){};
        \node[name=b1, product] at (0.5cm,1cm){};
        \node[name=b2, product] at (1.5cm,1cm){};
        \node[name=c1, product] at (1cm,0.5cm){};
        \node[name=c2, product] at (2cm,0.5cm){};
        \foreach \N/\M/\S/\T in {
            in1/b1/su/sdl,
            in2/b2/su/sdl,
            p/out/su/sd,
            c2/p/sul/sdr,
            c2/b2/sur/sdr,
            c1/p/sul/sdl,
            c1/b1/sur/sdr}
            \draw (\N) .. controls ($(\N)+(\S)$) and ($(\M)+(\T)$) .. (\M);
    \end{tikzpicture}}
=  \mu \circ (N \otimes N) \ .
\end{split}
\end{equation}
To see compatibility with the coproduct, first note that
\begin{equation}
     \raisebox{-0.5cm}{\begin{tikzpicture}
        \coordinate (su) at (0cm,0.3cm);
        \coordinate (sd) at (0cm,-0.3cm);
        \coordinate (sdl) at (-0.3cm,-0.3cm);
        \coordinate (sdr) at (0.3cm,-0.3cm);
        \coordinate (sul) at (-0.3cm,0.3cm);
        \coordinate (sur) at (0.3cm,0.3cm);
        \coordinate (out1) at (-0.5cm,1.5cm);
        \coordinate (out2) at (0.5cm,1.5cm);
        \node[name=c, product] at (0cm,0cm){};
        \node[name=N, draw] at (-0.5cm,1cm){$N^{-1}$};
        \foreach \N/\M/\S/\T in {
            c/N.south/sul/sd,
            c/out2/sur/sd,
            N.north/out1/su/sd}
            \draw (\N) .. controls ($(\N)+(\S)$) and ($(\M)+(\T)$) .. (\M);
    \end{tikzpicture}}
    =
    \raisebox{-0.5cm}{\begin{tikzpicture}
        \coordinate (su) at (0cm,0.3cm);
        \coordinate (sd) at (0cm,-0.3cm);
        \coordinate (sdl) at (-0.3cm,-0.3cm);
        \coordinate (sdr) at (0.3cm,-0.3cm);
        \coordinate (sul) at (-0.3cm,0.3cm);
        \coordinate (sur) at (0.3cm,0.3cm);
        \coordinate (out1) at (-1cm,1.5cm);
        \coordinate (out2) at (1cm,1.5cm);
        \node[name=c1, product] at (-0.5cm,0cm){};
        \node[name=c2, product] at (0.5cm,0cm){};
        \node[name=b, product] at (0cm,1cm){};
        \foreach \N/\M/\S/\T in {
            c1/out1/sur/sd,
            c1/b/sul/sdl,
            c2/b/sul/sdr,
            c2/out2/sur/sd} 
            \draw (\N) .. controls ($(\N)+(\S)$) and ($(\M)+(\T)$) .. (\M);
    \end{tikzpicture}}
    =
    \raisebox{-0.5cm}{\begin{tikzpicture}
        \coordinate (su) at (0cm,0.5cm);
        \coordinate (sd) at (0cm,-0.5cm);
        \coordinate (sdl) at (-0.5cm,-0.5cm);
        \coordinate (sdr) at (0.5cm,-0.5cm);
        \coordinate (sul) at (-0.5cm,0.5cm);
        \coordinate (sur) at (0.5cm,0.5cm);
        \coordinate (out1) at (-0.5cm,1.5cm);
        \coordinate (out2) at (0.5cm,1.5cm);
        \node[name=c, product] at (0cm,0cm){};
        \foreach \N/\M/\S/\T in {
            c/out1/sur/sd,
            c/out2/sul/sd}
            \draw (\N) .. controls ($(\N)+(\S)$) and ($(\M)+(\T)$) .. (\M);
    \end{tikzpicture}}
    =
    \raisebox{-0.5cm}{\begin{tikzpicture}
        \begin{scope}[xscale=-1] 
        \coordinate (su) at (0cm,0.3cm);
        \coordinate (sd) at (0cm,-0.3cm);
        \coordinate (sdl) at (-0.3cm,-0.3cm);
        \coordinate (sdr) at (0.3cm,-0.3cm);
        \coordinate (sul) at (-0.3cm,0.3cm);
        \coordinate (sur) at (0.3cm,0.3cm);
        \coordinate (out1) at (-1cm,1.5cm);
        \coordinate (out2) at (1cm,1.5cm);
        \node[name=c1, product] at (-0.5cm,0cm){};
        \node[name=c2, product] at (0.5cm,0cm){};
        \node[name=b, product] at (0cm,1cm){};
        \end{scope}
        \foreach \N/\M/\S/\T in {
            c1/out1/sur/sd,
            c1/b/sul/sdl,
            c2/b/sul/sdr,
            c2/out2/sur/sd} 
            \draw (\N) .. controls ($(\N)+(\S)$) and ($(\M)+(\T)$) .. (\M);
    \end{tikzpicture}}
    =
    \raisebox{-0.5cm}{\begin{tikzpicture}
        \begin{scope}[xscale=-1] 
        \coordinate (su) at (0cm,0.3cm);
        \coordinate (sd) at (0cm,-0.3cm);
        \coordinate (sdl) at (-0.3cm,-0.3cm);
        \coordinate (sdr) at (0.3cm,-0.3cm);
        \coordinate (sul) at (-0.3cm,0.3cm);
        \coordinate (sur) at (0.3cm,0.3cm);
        \coordinate (out1) at (-0.5cm,1.5cm);
        \coordinate (out2) at (0.5cm,1.5cm);
        \node[name=c, product] at (0cm,0cm){};
        \node[name=N, draw] at (-0.5cm,1cm){$N$};
        \end{scope}
        \foreach \N/\M/\S/\T in {
            c/N.south/sul/sd,
            c/out2/sur/sd,
            N.north/out1/su/sd}
            \draw (\N) .. controls ($(\N)+(\S)$) and ($(\M)+(\T)$) .. (\M);
    \end{tikzpicture}}\ .
\end{equation}
When combining this with the definition of the coproduct in terms of the product and copairing in \eqref{eq:coproduct-counit}, the compatibility of $N$ with the coproduct follows from the already established result 
	that $N$ is an algebra homomorphism.
\end{proof}

The following identity will be used frequently in the calculations below:

\begin{lemma} \label{lem:product-past-coproduct}
Let $A$ be a Frobenius algebra and $N$ its Nakayama automorphism. Then
\begin{equation}
\big[ \id \otimes (\mu \circ \sigma_{A,A}) \big] \circ \big[ \Delta \otimes \id \big] 
= \big[ \mu \otimes \id \big] \circ \big[ \id \otimes \sigma_{A,A} \big] \circ \big[ \Delta \otimes N \big] \ .
\end{equation}
Graphically, this reads
\begin{equation}
\raisebox{-0.5cm}{\begin{tikzpicture}
        \coordinate (su) at (0cm,0.5cm);
        \coordinate (sd) at (0cm,-0.5cm);
        \coordinate (sdl) at (-0.5cm,-0.5cm);
        \coordinate (sdr) at (0.5cm,-0.5cm);
        \coordinate (sul) at (-0.5cm,0.5cm);
        \coordinate (sur) at (0.5cm,0.5cm);
        \coordinate (out1) at (-0.5cm,2cm);
        \coordinate (out2) at (0.5cm,2cm);
        \coordinate (in1) at (0cm,0cm);
        \coordinate (in2) at (0.5cm,0cm);
        \node[name=p, product] at (0.5cm,1.5cm){};
        \node[name=cp, product] at (0cm,0.5cm){};
        \foreach \N/\M/\S/\T in {
            in1/cp/su/sd,
            in2/p/su/sdl,
            cp/out1/sul/sd,
            cp/p/sur/sdr,
            p/out2/su/sd}
            \draw (\N) .. controls ($(\N)+(\S)$) and ($(\M)+(\T)$) .. (\M);
    \end{tikzpicture}}
    =
    \raisebox{-0.5cm}{\begin{tikzpicture}
        \coordinate (su) at (0cm,0.5cm);
        \coordinate (sd) at (0cm,-0.5cm);
        \coordinate (sdl) at (-0.5cm,-0.5cm);
        \coordinate (sdr) at (0.5cm,-0.5cm);
        \coordinate (sul) at (-0.5cm,0.5cm);
        \coordinate (sur) at (0.5cm,0.5cm);
        \coordinate (out1) at (-0.5cm,2cm);
        \coordinate (out2) at (0.5cm,2cm);
        \coordinate (in1) at (0cm,0cm);
        \coordinate (in2) at (0.5cm,0cm);
        \node[name=p, product] at (-0.5cm,1.5cm){};
        \node[name=cp, product] at (0cm,0.5cm){};
        \node[name=N, naka] at (0.5cm,0.5cm){};
        \foreach \N/\M/\S/\T in {
            in1/cp/su/sd,
            in2/N/su/sd,
            cp/p/sul/sdl,
            N/p/su/sdr,
            p/out1/su/sd,
            cp/out2/sur/sd}
            \draw (\N) .. controls ($(\N)+(\S)$) and ($(\M)+(\T)$) .. (\M);
    \end{tikzpicture}}\quad .
    \label{eq:inner-mult-past-delta}
\end{equation}
\end{lemma}

\begin{proof}
By direct calculation:
\begin{equation*}
     \raisebox{-0.5cm}{\begin{tikzpicture}
        \coordinate (su) at (0cm,0.5cm);
        \coordinate (sd) at (0cm,-0.5cm);
        \coordinate (sdl) at (-0.5cm,-0.5cm);
        \coordinate (sdr) at (0.5cm,-0.5cm);
        \coordinate (sul) at (-0.5cm,0.5cm);
        \coordinate (sur) at (0.5cm,0.5cm);
        \coordinate (out1) at (-0.5cm,2cm);
        \coordinate (out2) at (0.5cm,2cm);
        \coordinate (in1) at (0cm,0cm);
        \coordinate (in2) at (0.5cm,0cm);
        \node[name=p, product] at (0.5cm,1.5cm){};
        \node[name=cp, product] at (0cm,0.5cm){};
        \foreach \N/\M/\S/\T in {
            in1/cp/su/sd,
            in2/p/su/sdl,
            cp/out1/sul/sd,
            cp/p/sur/sdr,
            p/out2/su/sd}
            \draw (\N) .. controls ($(\N)+(\S)$) and ($(\M)+(\T)$) .. (\M);
    \end{tikzpicture}}
    \stackrel{\text{Frob.}}{=}
    \raisebox{-0.5cm}{\begin{tikzpicture}
        \coordinate (su) at (0cm,0.5cm);
        \coordinate (sd) at (0cm,-0.5cm);
        \coordinate (sdl) at (-0.5cm,-0.5cm);
        \coordinate (sdr) at (0.5cm,-0.5cm);
        \coordinate (sul) at (-0.5cm,0.5cm);
        \coordinate (sur) at (0.5cm,0.5cm);
        \coordinate (out1) at (-0.5cm,2cm);
        \coordinate (out2) at (0.5cm,2cm);
        \coordinate (in1) at (0cm,0cm);
        \coordinate (in2) at (0.5cm,0cm);
        \node[name=cp2, product] at (0.5cm,1cm){};
        \node[name=cp1, product] at (0cm,0.5cm){};
        \node[name=b, product] at (0cm,1.5cm){};
        \foreach \N/\M/\S/\T in {
            in1/cp1/su/sd,
            in2/b/su/sdl,
            cp1/out1/sul/sd,
            cp1/cp2/sur/sdl,
            cp2/b/sul/sdr,
            cp2/out2/su/sd}
            \draw (\N) .. controls ($(\N)+(\S)$) and ($(\M)+(\T)$) .. (\M);
    \end{tikzpicture}}
    \stackrel{\text{coass.}}{=}
    \raisebox{-0.5cm}{\begin{tikzpicture}
        \coordinate (su) at (0cm,0.5cm);
        \coordinate (sd) at (0cm,-0.5cm);
        \coordinate (sdl) at (-0.5cm,-0.5cm);
        \coordinate (sdr) at (0.5cm,-0.5cm);
        \coordinate (sul) at (-0.5cm,0.5cm);
        \coordinate (sur) at (0.5cm,0.5cm);
        \coordinate (out1) at (-0.5cm,2cm);
        \coordinate (out2) at (0.5cm,2cm);
        \coordinate (in1) at (0cm,0cm);
        \coordinate (in2) at (0.5cm,0cm);
        \node[name=cp2, product] at (-0.5cm,1cm){};
        \node[name=cp1, product] at (0cm,0.5cm){};
        \node[name=b, product] at (-0.2cm,1.7cm){};
        \foreach \N/\M/\S/\T in {
            in1/cp1/su/sd,
            in2/b/su/sdl,
            cp1/cp2/sul/sdr,
            cp1/out2/sur/sd,
            cp2/b/sur/sdr,
            cp2/out1/su/sd}
            \draw (\N) .. controls ($(\N)+(\S)$) and ($(\M)+(\T)$) .. (\M);
    \end{tikzpicture}}
    \stackrel{\text{Frob.}}{=}
    \raisebox{-0.5cm}{\begin{tikzpicture}
        \coordinate (su) at (0cm,0.3cm);
        \coordinate (sd) at (0cm,-0.3cm);
        \coordinate (sdl) at (-0.3cm,-0.3cm);
        \coordinate (sdr) at (0.3cm,-0.3cm);
        \coordinate (sul) at (-0.3cm,0.3cm);
        \coordinate (sur) at (0.3cm,0.3cm);
        \coordinate (out1) at (-0.5cm,2cm);
        \coordinate (out2) at (0.5cm,2cm);
        \coordinate (in1) at (0cm,0cm);
        \coordinate (in2) at (0.5cm,0cm);
        \node[name=p, product] at (-0.5cm,1.5cm){};
        \node[name=cp, product] at (0cm,0.5cm){};
        \node[name=b, product] at (0cm,1.8cm){};
        \node[name=c, product] at (-0.1cm,0.8cm){};
        \foreach \N/\M/\S/\T in {
            in1/cp/su/sd,
            in2/b/su/sdl,
            cp/p/sul/sdl,
            cp/out2/sur/sd,
            c/p/sul/sdr,
            c/b/sur/sdr,
            p/out1/su/sd}
            \draw (\N) .. controls ($(\N)+(\S)$) and ($(\M)+(\T)$) .. (\M);
    \end{tikzpicture}}
    \stackrel{\text{deform}}{=}
    \raisebox{-0.5cm}{\begin{tikzpicture}
        \coordinate (su) at (0cm,0.3cm);
        \coordinate (sd) at (0cm,-0.3cm);
        \coordinate (sdl) at (-0.3cm,-0.3cm);
        \coordinate (sdr) at (0.3cm,-0.3cm);
        \coordinate (sul) at (-0.3cm,0.3cm);
        \coordinate (sur) at (0.3cm,0.3cm);
        \coordinate (out1) at (-0.5cm,2cm);
        \coordinate (out2) at (0.5cm,2cm);
        \coordinate (in1) at (0cm,0cm);
        \coordinate (in2) at (0.5cm,0cm);
        \node[name=p, product] at (-0.5cm,1.5cm){};
        \node[name=cp, product] at (0cm,0.5cm){};
        \node[name=b, product] at (0.8cm,1.3cm){};
        \node[name=c, product] at (1cm,0.5cm){};
        \foreach \N/\M/\S/\T in {
            in1/cp/su/sd,
            in2/b/su/sdl,
            cp/p/sul/sdl,
            cp/out2/sur/sd,
            c/p/sul/sdr,
            c/b/sur/sdr,
            p/out1/su/sd}
            \draw (\N) .. controls ($(\N)+(\S)$) and ($(\M)+(\T)$) .. (\M);
    \end{tikzpicture}}
    =
    \raisebox{-0.5cm}{\begin{tikzpicture}
        \coordinate (su) at (0cm,0.5cm);
        \coordinate (sd) at (0cm,-0.5cm);
        \coordinate (sdl) at (-0.5cm,-0.5cm);
        \coordinate (sdr) at (0.5cm,-0.5cm);
        \coordinate (sul) at (-0.5cm,0.5cm);
        \coordinate (sur) at (0.5cm,0.5cm);
        \coordinate (out1) at (-0.5cm,2cm);
        \coordinate (out2) at (0.5cm,2cm);
        \coordinate (in1) at (0cm,0cm);
        \coordinate (in2) at (0.5cm,0cm);
        \node[name=p, product] at (-0.5cm,1.5cm){};
        \node[name=cp, product] at (0cm,0.5cm){};
        \node[name=N, naka] at (0.5cm,0.5cm){};
        \foreach \N/\M/\S/\T in {
            in1/cp/su/sd,
            in2/N/su/sd,
            cp/p/sul/sdl,
            N/p/su/sdr,
            p/out1/su/sd,
            cp/out2/sur/sd}
            \draw (\N) .. controls ($(\N)+(\S)$) and ($(\M)+(\T)$) .. (\M);
    \end{tikzpicture}}
\end{equation*}
\end{proof}

Proposition \ref{prop:Nak-Frob-auto} and Lemma \ref{lem:product-past-coproduct} hold in general. For the Frobenius algebra $A$ constructed above from $t,c_{\pm1}$ we have in addition:

\begin{lemma}
The Nakayama automorphism of $A$ is an involution.
\end{lemma}

\begin{proof}
    We use the unit property, nondegeneracy and Lemma \ref{lem:naka.t}:
    \begin{equation}
        N^2 = (t\otimes \id_A)\circ (N^2 \otimes \eta \otimes c_{-1}) \stackrel{\ref{lem:naka.t}}{=} (t\otimes \id_A)\circ (\id_A \otimes \eta \otimes c_{-1}) = \id_A\ .
    \end{equation}
\end{proof}
We call a Frobenius algebra {\em $\Delta$-separable} if $\mu \circ \Delta  \circ \eta = \eta$, 
i.e. $\Delta\circ \eta$ is a separability idempotent. 
We have:

\begin{lemma}
    $A$ is $\Delta $-separable.
\end{lemma}

\begin{proof}
The statement follows from the identities
    \begin{align}
    \raisebox{-0.5cm}{\begin{tikzpicture}
        \coordinate (su) at (0cm,0.5cm);
        \coordinate (sd) at (0cm,-0.5cm);
        \coordinate (sdl) at (-0.5cm,-0.5cm);
        \coordinate (sdr) at (0.5cm,-0.5cm);
        \coordinate (sul) at (-0.5cm,0.5cm);
        \coordinate (sur) at (0.5cm,0.5cm);
        \node[name=p, product] at (0cm,0cm){};
        \node[name=c, product] at (0cm,-1cm){};
        \coordinate (out) at (0cm,0.5cm);
        \foreach \N/\M/\S/\T in {
            c/p/sul/sdl,
            c/p/sur/sdr,
            p/out/su/sd}
            \draw (\N) .. controls ($(\N)+(\S)$) and ($(\M)+(\T)$) .. (\M);
    \end{tikzpicture}}
    &\stackrel{\text{Frob.}}{=}
    \raisebox{-0.5cm}{\begin{tikzpicture}
        \coordinate (su) at (0cm,0.5cm);
        \coordinate (sd) at (0cm,-0.5cm);
        \coordinate (sdl) at (-0.5cm,-0.5cm);
        \coordinate (sdr) at (0.5cm,-0.5cm);
        \coordinate (sul) at (-0.5cm,0.5cm);
        \coordinate (sur) at (0.5cm,0.5cm);
        \node[name=p, product] at (-0.5cm,-0.5cm){};
        \node[name=c1, product] at (0cm,-1cm){};
        \node[name=c2, product] at (-1cm,-1cm){};
        \node[name=b, product] at (0cm,0cm){};
        \coordinate (out) at (-1.5cm,0.5cm);
        \foreach \N/\M/\S/\T in {
            c1/p/sul/sdr,
            c1/b/sur/sdr,
            p/b/sur/sdl,
            c2/p/sur/sdl,
            c2/out/sul/sd}
            \draw (\N) .. controls ($(\N)+(\S)$) and ($(\M)+(\T)$) .. (\M);
    \end{tikzpicture}}
    \stackrel{\text{nondeg.}}{=}
    \raisebox{-1cm}{\begin{tikzpicture}
        \coordinate (su) at (0cm,0.5cm);
        \coordinate (sd) at (0cm,-0.5cm);
        \coordinate (sdl) at (-0.5cm,-0.5cm);
        \coordinate (sdr) at (0.5cm,-0.5cm);
        \coordinate (sul) at (-0.5cm,0.5cm);
        \coordinate (sur) at (0.5cm,0.5cm);
        \node[name=p, product] at (-0.5cm,-0.5cm){};
        \node[name=c1, product] at (1cm,-1.5cm){};
        \node[name=c2, product] at (-1cm,-1cm){};
        \node[name=c3, product] at (0cm,-1.5cm){};
        \node[name=c4, product] at (2cm,-1.5cm){};
        \node[name=b1, product] at (0cm,0cm){};
        \node[name=b2, product] at (1cm,0cm){};
        \node[name=b3, product] at (2cm,0cm){};
        \coordinate (out) at (-1.5cm,0.5cm);
        \foreach \N/\M/\S/\T in {
            c1/b3/sul/sdl,
            c1/b1/sur/sdr,
            p/b1/sur/sdl,
            c2/p/sur/sdl,
            c2/out/sul/sd,
            c3/p/sur/sdr,
            c3/b2/sul/sdr,
            c4/b3/sul/sdr,
            c4/b2/sur/sdl}
            \draw (\N) .. controls ($(\N)+(\S)$) and ($(\M)+(\T)$) .. (\M);
    \end{tikzpicture}}
    \\
    &\stackrel{\text{unit}}{=}
    \raisebox{-1cm}{\begin{tikzpicture}
        \coordinate (su) at (0cm,0.5cm);
        \coordinate (sd) at (0cm,-0.5cm);
        \coordinate (sdl) at (-0.5cm,-0.5cm);
        \coordinate (sdr) at (0.5cm,-0.5cm);
        \coordinate (sul) at (-0.5cm,0.5cm);
        \coordinate (sur) at (0.5cm,0.5cm);
        \node[name=p1, product] at (-0.5cm,-0.5cm){};
        \node[name=c1, product] at (1.75cm,-2.5cm){};
        \node[name=c2, product] at (-1cm,-1cm){};
        \node[name=c3, product] at (0.5cm,-2.5cm){};
        \node[name=c4, product] at (3cm,-2.5cm){};
        \node[name=b1, product] at (0cm,0cm){};
        \node[name=b2, product] at (2cm,0cm){};
        \node[name=b3, product] at (4cm,0cm){};
        \node[name=p2, product] at ($(b2)+(-0.5cm,-0.5cm)$){};
        \node[name=p3, product] at ($(b3)+(-0.5cm,-0.5cm)$){};
        \node[name=u1, naka] at ($(p2)+(-0.3cm,-0.3cm)$){};
        \node[name=u2, naka] at ($(p3)+(-0.3cm,-0.3cm)$){};
        \coordinate (out) at (-1.5cm,0.5cm);
        \foreach \N/\M/\S/\T in {
            c1/p3/sul/sdr,
            c1/b1/sur/sdr,
            p1/b1/sur/sdl,
            p2/b2/sur/sdl,
            p3/b3/sur/sdl,
            c2/p1/sur/sdl,
            c2/out/sul/sd,
            c3/p1/sur/sdr,
            c3/b2/sul/sdr,
            c4/b3/sul/sdr,
            c4/p2/sur/sdr,
            u1/p2/sur/sdl,
            u2/p3/sur/sdl}
            \draw (\N) .. controls ($(\N)+(\S)$) and ($(\M)+(\T)$) .. (\M);
    \end{tikzpicture}}
    \nonumber\\ 
    &\stackrel{\text{\eqref{eq:t-via-mu}}}{=}
    \raisebox{-1cm}{\begin{tikzpicture}
        \coordinate (su) at (0cm,0.5cm);
        \coordinate (sd) at (0cm,-0.5cm);
        \coordinate (sdl) at (-0.5cm,-0.5cm);
        \coordinate (sdr) at (0.5cm,-0.5cm);
        \coordinate (sul) at (-0.5cm,0.5cm);
        \coordinate (sur) at (0.5cm,0.5cm);
        \node[name=c0, product] at (-1cm,-1cm){};
        \node[name=c1, draw, minimum width=0.8cm, minimum height=1.4em] at (0.5cm,-2.5cm){$c_{1}$};
        \node[name=c2, draw, minimum width=0.8cm, minimum height=1.4em] at (1.75cm,-2.5cm){$c_{-1}$};
        \node[name=c3, draw, minimum width=0.8cm, minimum height=1.4em] at (3cm,-2.5cm){$c_{1}$};
        \foreach \N in {1,2,3} {
            \coordinate (c\N1) at ($(c\N.north west)!0.25!(c\N.north east)$);
            \coordinate (c\N2) at ($(c\N.north west)!0.75!(c\N.north east)$);
        }
        \node[name=t1, draw, minimum width=1.5cm, minimum height=1.4em] at (0cm,0cm){$t$};
        \node[name=t2, draw, minimum width=1.5cm, minimum height=1.4em] at (2cm,0cm){$t$};
        \node[name=t3, draw, minimum width=1.5cm, minimum height=1.4em] at (4cm,0cm){$t$};
        \foreach \N in {1,2,3} {
            \coordinate (t\N1) at ($(t\N.south west)!0.33!(t\N.south)$);
            \coordinate (t\N2) at (t\N.south);
            \coordinate (t\N3) at ($(t\N.south east)!0.33!(t\N.south)$);
        }
        \node[name=u1, naka] at ($(t21)+(0cm,-0.8cm)$){};
        \node[name=u2, naka] at ($(t31)+(0cm,-0.8cm)$){};
        \coordinate (out) at (-1.5cm,0.5cm);
        \foreach \N/\M/\S/\T in {
            c0/t11/sur/sd,
            c0/out/sul/sd}
            \draw (\N) .. controls ($(\N)+0.5*(\S)$) and ($(\M)+0.5*(\T)$) .. (\M);
        \foreach \N/\M in {
            u1/t21,
            u2/t31,
            c11/t12,
            c12/t23,
            c21/t32,
            c22/t13,
            c31/t22,
            c32/t33}
             \draw (\N) .. controls ($(\N)+(su)$) and ($(\M)+(sd)$) .. (\M);
    \end{tikzpicture}}
        \nonumber\\ 
&\stackrel{\text{\eqref{eq:alg.pachner3-1}}}{=}
    \raisebox{-0.5cm}{\begin{tikzpicture}
        \coordinate (su) at (0cm,0.5cm);
        \coordinate (sd) at (0cm,-0.5cm);
        \coordinate (sdl) at (-0.5cm,-0.5cm);
        \coordinate (sdr) at (0.5cm,-0.5cm);
        \coordinate (sul) at (-0.5cm,0.5cm);
        \coordinate (sur) at (0.5cm,0.5cm);
        \node[name=c0, product] at (-1cm,-1cm){};
        \node[name=t1, draw, minimum width=1.5cm, minimum height=1.4em] at (0cm,0cm){$t$};
        \foreach \N in {1} {
            \coordinate (t\N1) at ($(t\N.south west)!0.33!(t\N.south)$);
            \coordinate (t\N2) at (t\N.south);
            \coordinate (t\N3) at ($(t\N.south east)!0.33!(t\N.south)$);
        }
        \node[name=u1, naka] at ($(t12)+(0cm,-0.8cm)$){};
        \node[name=u2, naka] at ($(t13)+(0cm,-0.8cm)$){};
        \coordinate (out) at (-1.5cm,0.5cm);
        \foreach \N/\M/\S/\T in {
            c0/t11/sur/sd,
            c0/out/sul/sd}
            \draw (\N) .. controls ($(\N)+0.5*(\S)$) and ($(\M)+0.5*(\T)$) .. (\M);
        \foreach \N/\M in {
            u1/t12,
            u2/t13}
             \draw (\N) .. controls ($(\N)+(su)$) and ($(\M)+(sd)$) .. (\M);
    \end{tikzpicture}}
    =
    \raisebox{-0.5cm}{\begin{tikzpicture}
        \coordinate (su) at (0cm,0.5cm);
        \coordinate (sd) at (0cm,-0.5cm);
        \coordinate (sdl) at (-0.5cm,-0.5cm);
        \coordinate (sdr) at (0.5cm,-0.5cm);
        \coordinate (sul) at (-0.5cm,0.5cm);
        \coordinate (sur) at (0.5cm,0.5cm);
        \node[name=c0, product] at (-1cm,-1cm){};
        \node[name=p, product] at (-0.5cm,-0.5cm){};
        \node[name=b, product] at (0cm,0cm){};
        \node[name=u1, naka] at (0cm,-1cm){};
        \node[name=u2, naka] at (1cm,-1cm){};
        \coordinate (out) at (-1.5cm,0.5cm);
        \foreach \N/\M/\S/\T in {
            c0/p/sur/sdl,
            c0/out/sul/sd,
            u1/p/sul/sdr,
            u2/b/sul/sdr,
            p/b/sur/sdl}
            \draw (\N) .. controls ($(\N)+0.5*(\S)$) and ($(\M)+0.5*(\T)$) .. (\M);
    \end{tikzpicture}}
    = 
    \raisebox{-0.5cm}{\begin{tikzpicture}
        \coordinate (su) at (0cm,0.5cm);
        \coordinate (sd) at (0cm,-0.5cm);
        \coordinate (sdl) at (-0.5cm,-0.5cm);
        \coordinate (sdr) at (0.5cm,-0.5cm);
        \coordinate (sul) at (-0.5cm,0.5cm);
        \coordinate (sur) at (0.5cm,0.5cm);
        \node[name=u1, naka] at (0cm,-1.5cm){};
        \coordinate (out) at (0cm,0cm);
        \foreach \N/\M/\S/\T in {
            u1/out/su/sd}
            \draw (\N) .. controls ($(\N)+0.5*(\S)$) and ($(\M)+0.5*(\T)$) .. (\M);
    \end{tikzpicture}}
    \quad .
    \nonumber
\end{align}
\end{proof}

We have now arrived at the desired algebra structure encoding $t,c_{\pm1}$ and their properties. As described above, under Assumptions 1 and 2, the data $t,c_{\pm1}$, subject to  relations 1--5 in Section \ref{sec:localmoves}, give rise to a $\Delta$-separable Frobenius algebra whose Nakayama automorphism is an involution. The following result shows that the converse holds as well.

\begin{proposition}    \label{lem:moves_from_alg}
Let $A$ be a $\Delta $-separable Frobenius algebra whose Nakayama automorphism is an involution. Set 
\begin{equation} \label{eq:tc-via-A}
 t =\varepsilon \circ \mu  \circ (\mu \otimes \id_A) ~, \quad
 c_{-1} = \Delta \circ \eta ~, \quad
 c_{1} = \sigma_{A,A} \circ \Delta \circ \eta \ .
\end{equation} 
Then $t,c_{\pm 1}$ fulfil relations 1--5 in Section \ref{sec:localmoves}.
\end{proposition}

\begin{proof}
\noindent
{\em Relation (1):} 
Immediate from \eqref{eq:tc-via-A}.

\medskip\noindent
{\em Relation (2):}
Follows from $(N \otimes \id) \circ c_{\nu} = c_{-\nu}$ together with the fact that $N$ is an automorphism of Frobenius algebras (Proposition \ref{prop:Nak-Frob-auto}).

\medskip\noindent
{\em Relation (3):} 
Applying the pairing $b$ to each leg shows that \eqref{eq:alg.cyclic} is equivalent to
\begin{equation}
	t \circ (N_{-s(e_0)} \otimes N_{-s(e_1)} \otimes N_{-s(e_2)}) 
	= t \circ c_{A,A\otimes A} \circ  (N_{s(e_0)} \otimes N_{-s(e_1)} \otimes N_{-s(e_2)}) \ .
\end{equation}
Cancelling the Nakayama automorphisms 
gives the following reformulation of \eqref{eq:alg.cyclic}:
\begin{equation}
	t \circ (N \otimes \id_{A \otimes A})
	= t \circ c_{A,A\otimes A} \ .
\end{equation}
To see that this equality holds, 
first substitute $t = b \circ (\id_A \otimes \mu)$ and then use $b \circ (N \otimes \id_A) = b \circ c_{A,A}$. This last identity follows by first composing the definition of $N$ in \eqref{eq:Nak-definition-pic} with $b$ to get $b \circ (\id_A \otimes N) = b \circ c_{A,A}$ and then noting that $b \circ (\id_A \otimes N) = b \circ (N \otimes \id_A)$.

\medskip\noindent
{\em Relation (4):} 
Recall the calculation in \eqref{eq:assoc-proof-calc} which was used to establish associativity. Remove the equal sign labelled by Equation \eqref{eq:alg.pachner2-2} and instead use associativity of $\mu$ to equate the first and last expression. Since we have already established Relations 2 and 3, that is, Equations \eqref{eq:alg.leaf} and \eqref{eq:alg.cyclic}, this reformulation of the calculation in \eqref{eq:assoc-proof-calc} shows that the equality labelled by \eqref{eq:alg.pachner2-2} in \eqref{eq:assoc-proof-calc} holds. This equality proves a special case of relation (4), i.e.\ of \eqref{eq:alg.pachner2-2}: use $b$ to turn the first three out-going legs into in-going legs and set $s_A=s_B=s_C=s_D=s=1$.
The remaining cases are established by composing with Nakayama automorphisms as appropriate.

\medskip\noindent
{\em Relation (5):} We have to show the identity \eqref{eq:alg.pachner3-1}. By composing with Nakayama automorphisms as appropriate, we may assume $s_A=s_B=s_C=-1$. Using $b$ to turn all out-going legs into in-going ones and substituting the definitions of $t$ and $c_{\pm1}$, we see that \eqref{eq:alg.pachner3-1} is equivalent to
\begin{equation}\label{eq:Frob-to-pachner31-aux1}
   \raisebox{-2cm}{\begin{tikzpicture}
        \coordinate (su) at (0cm,0.5cm);
        \coordinate (sd) at (0cm,-0.5cm);
        \coordinate (sdl) at (-0.5cm,-0.5cm);
        \coordinate (sdr) at (0.5cm,-0.5cm);
        \coordinate (sul) at (-0.5cm,0.5cm);
        \coordinate (sur) at (0.5cm,0.5cm);
        \node[name=p1, product] at (-0.5cm,-0.5cm){};
        \node[name=c1, product] at (0.5cm,-3cm){};
        \node[name=c2, product] at (2cm,-3cm){};
        \node[name=c3, product] at (4.5cm,-3cm){};
        \node[name=b1, product] at (0cm,0cm){};
        \node[name=b2, product] at (2cm,0cm){};
        \node[name=b3, product] at (4cm,0cm){};
        \node[name=p2, product] at ($(b2)+(-0.5cm,-0.5cm)$){};
        \node[name=p3, product] at ($(b3)+(-0.5cm,-0.5cm)$){};
        \node[name=N1, naka] at ($(c1)+(-0.5cm,0.7cm)$){$-s_{12}$};
        \node[name=N2, naka] at ($(c2)+(0.5cm,0.7cm)$){$-s_{31}$};
        \node[name=N3, naka] at ($(c3)+(-0.5cm,0.7cm)$){$-s_{23}$};
        \foreach \N in {1,2,3} \coordinate (in\N) at ($(p\N)+(-0.5cm,-3cm)$);
        \coordinate (in3) at ($(in3)+(0.2cm,0cm)$);
        \coordinate (ic2) at ($(c2)+(-0.5cm,0.7cm)$);
        \foreach \N in {1,2,3} \coordinate (ip\N) at ($(p\N)+(-0.5cm,-1cm)$);
        \coordinate (ip3) at ($(ip3)+(0.2cm,0cm)$);
        \node[name=C1, ellipse, dashed, minimum height=1.8cm, minimum width=4cm, draw, rotate=-70] at ($(N1)+(-0.25cm,0.5cm)$){};
        \node[name=C2, circle, dashed, minimum size=1.8cm, draw] at ($(N2)+(-0.25cm,0cm)$){};
        \node[name=C3, circle, dashed, minimum size=1.5cm, draw] at ($(b3)+(-0.25cm,-0.3cm)$){};
        \node[left] at (C1.south east){\small{(1)}};
        \node[below] at (C2.south){\small{(2)}};
        \node[right] at (C3.east){\small{(3)}};
        \foreach \N/\M/\S/\T in {
            ic2/p3/su/sdr,
            N2.north/b1/su/sdr,
            p1/b1/sur/sdl,
            p2/b2/sur/sdl,
            p3/b3/sur/sdl,
            N1.north/p1/su/sdr,
            c1/b2/sur/sdr,
            c3/b3/sur/sdr,
            N3.north/p2/su/sdr,
            in1/ip1/su/sd,
            ip1/p1/su/sdl,
            in2/ip2/su/sd,
            ip2/p2/su/sdl,
            in3/ip3/su/sd}
            \draw (\N) .. controls ($(\N)+(\S)$) and ($(\M)+(\T)$) .. (\M);
        \foreach \N/\M/\S/\T in {
            c2/ic2/sul/sd,
            c1/N1.south/sul/sd,
            c2/N2.south/sur/sd,
            c3/N3.south/sul/sd,
            ip3/p3/su/sdl}
            \draw (\N) .. controls ($(\N)+0.5*(\S)$) and ($(\M)+0.5*(\T)$) .. (\M);
    \end{tikzpicture}}
	= \varepsilon \circ \mu \circ (\mu \otimes \id_A) \circ (\id_A \otimes N_{s_{12}} \otimes N_{-s_{31}}) \ .
\end{equation}
To prove this identity, start from the left hand side. Inside the dashed circle 1, move $N_{-s_{12}}$ past the copairing $c_{-1}$ and convert product and copairing to a coproduct by substituting \eqref{eq:coproduct-counit}. In dashed circle 2, remove the braiding by replacing $N_{-s_{31}}$ by $N_{s_{31}}$. Then one can use the duality properties to cancel $c_{-1}$ against $b$. In dashed circle 3, apply associativity. Deforming the resulting string diagram slightly gives the first equality in:
\begin{align}
\text{lhs.\ of \eqref{eq:Frob-to-pachner31-aux1}}
&\overset{(1)}=
    \raisebox{-1cm}{\begin{tikzpicture}
        \coordinate (su) at (0cm,0.5cm);
        \coordinate (sd) at (0cm,-0.5cm);
        \coordinate (sdl) at (-0.5cm,-0.5cm);
        \coordinate (sdr) at (0.5cm,-0.5cm);
        \coordinate (sul) at (-0.5cm,0.5cm);
        \coordinate (sur) at (0.5cm,0.5cm);
        \coordinate (in1) at (0cm,-0.2cm);
        \coordinate (in2) at (1.5cm,-0.2cm);
        \coordinate (in3) at (4cm,-0.2cm);
        \coordinate (iin3) at (4cm,3.5cm);
        \node[name=cp1, product] at (0cm,0.5cm){};
        \node[name=N1, naka] at (0.5cm,1.25cm){$-s_{12}$};
        \node[name=N2, naka] at (0cm,2.5cm){$s_{31}$};
        \node[name=N3, naka] at (2.5cm,0.75cm){$-s_{23}$};
        \node[name=b1, product] at (1.5cm,2.5cm){};
        \node[name=b2, product] at (3.5cm,4.5cm){};
        \coordinate (ib1) at ($(b1)+(0cm,-0.5cm)$);
        \node[name=p1, product] at (2cm,1.5cm){};
        \node[name=p2, product] at (3.5cm,3.5cm){};
        \node[name=c1, product] at (3cm,0cm){};
         \foreach \N/\M/\S/\T in {
            in1/cp1/su/sd,
            cp1/N2.south/sul/sd,
            p1/ib1/sul/sdr,
            in2/p1/su/sdl,
            c1/p2/sur/sdr,
            N2.north/p2/su/sdl,
            p2/b2/su/sdr,
            in3/iin3/su/sd,
            iin3/b2/su/sdl}
            \draw (\N) .. controls ($(\N)+(\S)$) and ($(\M)+(\T)$) .. (\M);
        \foreach \N/\M/\S/\T in {
            cp1/N1.south/sur/sd,
            N1.north/ib1/su/sdl,
            ib1/b1/sur/sdr,
            ib1/b1/sul/sdl,
            N3.north/p1/su/sdr,
            c1/N3.south/sul/sd}
            \draw (\N) .. controls ($(\N)+0.5*(\S)$) and ($(\M)+0.5*(\T)$) .. (\M);
    \end{tikzpicture}}
\overset{(2)}=
    \raisebox{-1cm}{\begin{tikzpicture}
        \coordinate (su) at (0cm,0.5cm);
        \coordinate (sd) at (0cm,-0.5cm);
        \coordinate (sdl) at (-0.5cm,-0.5cm);
        \coordinate (sdr) at (0.5cm,-0.5cm);
        \coordinate (sul) at (-0.5cm,0.5cm);
        \coordinate (sur) at (0.5cm,0.5cm);
        \coordinate (in1) at (0cm,0cm);
        \coordinate (in2) at (1.5cm,0cm);
        \coordinate (in3) at (2.5cm,0cm);
        \node[name=cp1, product] at (0cm,0.5cm){};
        \node[name=N12, naka] at (0.5cm,1.25cm){$s_{12}$};
        \node[name=p1, product] at (1cm,2cm){};
        \node[name=N23, naka] at (1.5cm,2.75cm){$-s_{23}$}; 
        \node[name=N31, naka] at (0cm,2.5cm){$s_{31}$};
        \node[name=p2, product] at (1.5cm,3.5cm){};
        \node[name=b, product] at (2cm,4cm){};
        \node[name=Nr, naka] at (2.5cm,2.75cm){};
        \foreach \N/\M/\S/\T in {
            in1/cp1/su/sd,
            in2/p1/su/sdr,
            in3/Nr.south/su/sd,
            Nr.north/b/su/sdr,
            cp1/N31.south/sul/sd, 
        p2/b/sur/sdl}
            \draw (\N) .. controls ($(\N)+(\S)$) and ($(\M)+(\T)$) .. (\M);
        \foreach \N/\M/\S/\T in {
            cp1/N12.south/sur/sd,
            N12.north/p1/su/sdl,
            p1/N23.south/su/sd,
            N23.north/p2/su/sd,
            N31.north/p2/su/sdl}
            \draw (\N) .. controls ($(\N)+0.5*(\S)$) and ($(\M)+0.5*(\T)$) .. (\M);
    \end{tikzpicture}}
\\&\overset{(3)}=
    \raisebox{-1cm}{\begin{tikzpicture}
        \coordinate (su) at (0cm,0.5cm);
        \coordinate (sd) at (0cm,-0.5cm);
        \coordinate (sdl) at (-0.5cm,-0.5cm);
        \coordinate (sdr) at (0.5cm,-0.5cm);
        \coordinate (sul) at (-0.5cm,0.5cm);
        \coordinate (sur) at (0.5cm,0.5cm);
        \coordinate (in1) at (0cm,0cm);
        \coordinate (in2) at (1cm,0cm);
        \coordinate (in3) at (2.5cm,0cm);
        \coordinate (iin3) at (2.5cm,3cm);
        \node[name=N12, naka] at (1cm,0.75cm) {$s_{12}$};
        \node[name=p1, product] at (0.5cm,1.5cm) {};
        \node[name=cp1, product] at (0.5cm,2cm){};
        \coordinate (icp1) at (-0.5cm,3cm);
        \node[name=NBig, naka] at (1cm,3cm) {$-s_{12}s_{23}s_{31}$};
        \node[name=p2, product] at (0.5cm,4cm) {};
        \node[name=b, product] at (1cm,4.5cm){};
        \node[name=N31, naka] at (2.5cm,0.75cm){$-s_{31}$};
        \foreach \N/\M/\S/\T in {
            in2/N12.south/su/sd,
            in1/p1/su/sdl,
            p1/cp1/su/sd,
            p2/b/sur/sdl,
            in3/N31.south/su/sd,
            N31.north/iin3/su/sd,
            iin3/b/su/sdr}
            \draw (\N) .. controls ($(\N)+(\S)$) and ($(\M)+(\T)$) .. (\M);
        \foreach \N/\M/\S/\T in {
            N12.north/p1/su/sdr,
            cp1/NBig/sur/sd,
            cp1/icp1/sul/sd,
            icp1/p2/su/sdl,
            NBig/p2/su/sdr}
            \draw (\N) .. controls ($(\N)+0.5*(\S)$) and ($(\M)+0.5*(\T)$) .. (\M);
    \end{tikzpicture}}
~\overset{(4)}=~
\text{rhs.\ of \eqref{eq:Frob-to-pachner31-aux1}} \ .
\nonumber
\end{align}
In the second equality, $b \circ c_{A,A} = b \circ (N \otimes \id_A)$ is used twice, and after one use of associativity, a pairing has been cancelled against a copairing. Step 3 is associativity and the fact that $N$ is an algebra automorphism. Equality 4 uses that $s_{12}s_{23}s_{31}=-1$ and $\Delta$-separability of $A$.

\end{proof}

With the tools assembled so far, we can prove the main result of this paper. To state the result, we need a little bit more notation. Let $B$ be the number of boundary components of the given spin surface. We would like to think of the morphism assigned to this spin surface as a ``correlator'', that is, we prefer to write it as a morphism $A^{\otimes 3B} \to  \mathbf{1}_\mathcal{S}$ rather than the other way around as is the case for $T_\text{triang}$ in \eqref{eq:Ttriang-def}. We use the map $b$ to achieve this and define
    \begin{equation}\label{eq:TA-def}
        T_A(\Sigma) := (b^{\otimes 3B}) \circ \tau \circ \big( \id_{A^{\otimes 3B}} \otimes  T_\text{triang}(\mathcal{C},\tilde{\varphi},\Sigma) \big)
  \ ,
    \end{equation}
where $\tau$ is a permutation $(A^{\otimes 3B} \otimes A^{\otimes 3B}) \to (A^{\otimes 2})^{\otimes 3B}$, which connects the $i$'th factor of $A$ in the first (resp.\ second) copy of $A^{\otimes 3B}$ in the source object to the first (resp.\ second) copy of $A$ in the $i$'th factor of $A^{\otimes 2}$ in the target object.

\begin{theorem}\label{thm:TA-triang-indep}
Let $A$ be a Frobenius algebra in a symmetric 
	strict
monoidal category $\mathcal{S}$, such that $A$ is $\Delta$-separable and has a Nakayama automorphism which is an involution. 
Then $T_A(\Sigma)$ is independent of the choice of spin triangulation of the spin surface $\Sigma$ and $T_A(\Sigma) = T_A(\Sigma')$ for isomorphic spin surfaces $\Sigma$ and $\Sigma'$.
\end{theorem}

\begin{proof}
Let $\Lambda= ((\mathcal{C},f_i,d^1_0,d^2_0), (\varphi, \tilde\chi_\sigma) ,(\Sigma,\tilde\varphi_i))$ be a spin triangulation of $\Sigma$. By Proposition \ref{lem:moves_from_alg}, $t,c_{\pm1}$ as defined via $A$ satisfy relations 1--5 in Section \ref{sec:localmoves}. By Proposition \ref{prop:invariance}, this implies that $T_\text{triang}(\Lambda)$ is independent of the choice of spin triangulation. 
Given an isomorphism $\tilde f : \Sigma \to \Sigma'$ of spin surfaces, we obtain a spin triangulation 
$f_*\Lambda := ((\mathcal{C},f_i,d^1_0,d^2_0), (f \circ \varphi, \tilde f \circ\tilde\chi_\sigma) ,(\Sigma',\tilde\varphi_i'))$ of $\Sigma'$ which produces the same edge signs as $\Lambda$. Hence, $T_\text{triang}$ gives the same morphism in $\mathcal{S}$ for $\Lambda$ and $f_*\Lambda$.
\end{proof}

\begin{remark}\label{rem:compare-Lurie}
At this point we can make some connections to Lurie's description of topological field theories as fully dualisable objects \cite{lurie2009classification} (see also \cite{schommer2009classification,davidovich2011state} for discussions of the two-dimensional case). In the symmetric monoidal bicategory of algebras, bimodules, and bimodule morphisms (over some algebraically closed field), the fully dualisable objects are finite-dimensional semisimple algebras $A$. The dual is $A^\text{op}$, the algebra with opposite product. Write $A^e = A \otimes A^{op}$. Then the dualising bimodules are ${}_{A^e}A_{\mathbf1}$ and ${}_{\mathbf1}A_{A^e}$. The homotopy SO(2) action is given by the Serre-automorphism, i.e.\ by tensoring with the $A$-$A$-bimodule $A^*$. To pass from framed to spin TFTs, we need the Serre-automorphism to be an involution. 

Let now $A$ be a Frobenius algebra as in Theorem \ref{thm:TA-triang-indep}. Since $A$ is separable, it is semi-simple. Since $A^* \cong A_N$ as bimodules (where $A_N$ is the $A$-$A$-bimodule $A$ with right action twisted by the Nakayama automorphism), and $A_N \otimes_A A_N \cong A_{N^2}$, we see that the Serre-automorphism is indeed an involution.
\end{remark}

\subsection{Behaviour of the morphisms under gluing of spin surfaces}

Let $\Sigma$ be a spin triangulated surface, and $(i,j,\varepsilon)$ be spin gluing data
such that $(i,j)$ is simplicial gluing data (see Section \ref{sec:smooth-triang})
Consider the morphism
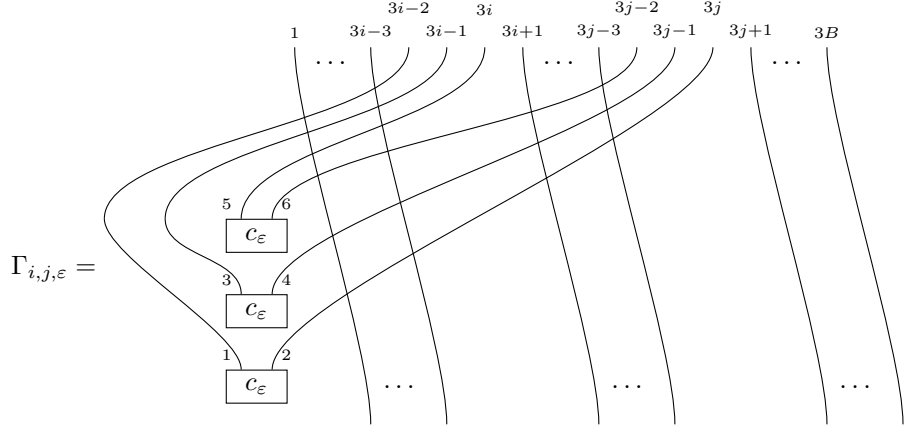
\begin{equation}
	\Gamma_{i,j,\varepsilon} =
    \raisebox{-2cm}{\begin{tikzpicture}
        \coordinate (su) at (0cm,1cm);
        \coordinate (sd) at (0cm,-1cm);
        \coordinate (out1) at (0.5cm,5cm);
        \foreach \N in {2,3,4,5,6} \coordinate (out\N) at ($(0.5cm,5cm)+\N*(0.5cm,0cm)$);
        \foreach \N in {7,8,9,10,11} \coordinate (out\N) at ($(1cm,5cm)+\N*(0.5cm,0cm)$);
        \coordinate (out12) at (7.5cm,5cm);
        \foreach \N in {1,...,12} \coordinate (in\N) at ($(out\N)+(1cm,-5cm)$);
        \node[name=c1, draw, minimum width=0.8cm] at (0cm,0.5cm){$c_{\varepsilon}$};
        \node[name=c2, draw, minimum width=0.8cm] at (0cm,1.5cm){$c_{\varepsilon}$};
        \node[name=c3, draw, minimum width=0.8cm] at (0cm,2.5cm){$c_{\varepsilon}$};
        \foreach \N in {1,2,3} {
            \coordinate (c\N1) at ($(c\N.north west)!0.25!(c\N.north east)$);
            \coordinate (c\N2) at ($(c\N.north west)!0.75!(c\N.north east)$);
        }
        \coordinate (i1) at ($(c11)+(-1.8cm,2cm)$);
        \coordinate (i2) at ($(c21)+(-1cm,1cm)$);
        \foreach \N/\M in {
            i1/out3,
            c12/out10,
            i2/out4,
            c22/out9,
            c31/out5,
            c32/out8,
            in1/out1,
            in2/out2,
            in6/out6,
            in7/out7,
            in11/out11,
            in12/out12}
            \draw (\N) .. controls ($(\N)+(su)$) and ($(\M)+(sd)$) .. (\M);
        \foreach \N/\M in {
            c11/i1,
            c21/i2}
            \draw (\N) .. controls ($(\N)+0.5*(su)$) and ($(\M)+0.5*(sd)$) .. (\M);
        \node[] at ($0.5*(in1)+0.5*(in2)+(-0.1cm,0.5cm)$){$\dots$};
        \node[] at ($0.5*(in6)+0.5*(in7)+(-0.1cm,0.5cm)$){$\dots$};
        \node[] at ($0.5*(in11)+0.5*(in12)+(-0.1cm,0.5cm)$){$\dots$};
        \node[] at ($0.5*(out1)+0.5*(out2)+(0cm,-0.2cm)$){$\dots$};
        \node[] at ($0.5*(out6)+0.5*(out7)+(0cm,-0.2cm)$){$\dots$};
        \node[] at ($0.5*(out11)+0.5*(out12)+(0cm,-0.2cm)$){$\dots$};
        \foreach \N/\M in {1/1,2/3,3/5} \node[above left] at (c\N1) {\tiny{$\M$}};
        \foreach \N/\M in {1/2,2/4,3/6} \node[above right] at (c\N2) {\tiny{$\M$}};
        \foreach \N/\M in {1/1,2/3i{-}3,4/3i{-}1,6/3i{+}1,7/3j{-}3,9/3j{-}1,11/3j{+}1,12/3B}
            \node[above] at (out\N) {\tiny{$\M$}};
        \foreach \N/\M in {3/3i{-}2,5/3i,8/3j{-}2,10/3j,12/}
            \node[above=0.3cm] at (out\N) {\tiny{$\M$}};
        \end{tikzpicture}}
    \ .
\end{equation}
from $A^{3(B-2)}$ to $A^{3B}$ with $B$ the number of boundary components of $\Sigma$.
In formulas, this reads
\begin{equation} \label{eq:Gamma-ijkeps}
	\Gamma_{i,j,\varepsilon} =  \tau \circ \big( c_{\varepsilon}^{\otimes 3} \otimes \id_A^{\otimes 3(B-2)} \big) \ ,
\end{equation}
where $\tau:A^{\otimes 3B}\to A^{\otimes 3B}$ represents the permutation that connects
\begin{itemize}
    \item the first output to the $(3i-2)$'th input,
    \item the second output to the $(3j)$'th input,
    \item the third output to the $(3i-1)$'th input,
    \item the fourth output to the $(3j-1)$'th input,
    \item the fifth output to the $(3i)$'th input and
    \item the sixth output to the $(3j-2)$'th input.
\end{itemize}
$\tau$ keeps the order of the remaining tensor factors fixed. 

We now come to an important property of $T_A$: the behaviour under gluing.

\begin{proposition}\label{prop:T-gluing}
Let $\Sigma$ be a spin triangulated surface, and $(i,j,\varepsilon)$ be spin gluing data. Then 
\begin{equation}
    T(\Sigma_{i\#j}^\varepsilon) = T(\Sigma) \circ \Gamma_{i,j,\varepsilon} \ .
\end{equation}
\end{proposition}

\begin{proof}
The connectivity represented by the map $\tau$ follows from the composition map \eqref{eq:combinatorial.gluing.map}. Lemma \ref{lem:index.gluing} gives the spin signs of a glued edge as
	$s=\varepsilon \,s_i\,s_j$. Since
    \begin{equation}
        (N\otimes \id_A) \circ c_{\varepsilon} = c_{-\varepsilon} = (\id_A \otimes N) \circ c_{\varepsilon} \ ,
    \end{equation}
    precomposition with $c_{\varepsilon}$ gives indeed the same result as evaluating the glued surface.
\end{proof}

\subsection{Evaluation of the TFT on the cylinder}\label{sec:TFT.cylinder}
Let $C$ be a spin cylinder, i.e. a spin surface such that $\underline{C} \cong S^1\times [0,1]$. Figure \ref{fig:cylinder.triangulation} gives an explicit triangulation of $\underline C$ with markings and labels. 

\begin{figure}[tb]
    \centering
    \begin{tikzpicture}
        \begin{scope}[decoration={
    markings,
    mark=at position 0.5 with {\arrow{stealth}}}
    ] 
       \node[name=A, rectangle, minimum width=5cm, minimum height=3cm]{};
       \draw[green!40, line width=3pt] ($(A.south west)!0.33!(A.north west)$)+(0pt,1pt) -- ($0.66*(A.south east)+0.33*(A.north east)+(0pt,1pt)$);
       \draw[green!40, line width=3pt] (A.south west)+(0pt,1pt) -- ($(A.south east)+(0pt,1pt)$);
       \draw[green!40, line width=3pt] ($0.33*(A.south west)+0.66*(A.north west)+(0pt,1pt)$) -- ($0.33*(A.south east)+0.66*(A.north east)+(0pt,1pt)$);
       \draw[green!40, line width=3pt] (A.north west)+(1pt,1pt) -- ($0.33*(A.south east)+0.66*(A.north east)+(0pt,1pt)$);
       \draw[green!40, line width=3pt] ($0.33*(A.south west)+0.66*(A.north west)+(1pt,1pt)$) -- ($0.66*(A.south east)+0.33*(A.north east)+(1pt,1pt)$);
       \draw[green!40, line width=3pt] ($0.66*(A.south west)+0.33*(A.north west)+(1pt,1pt)$) -- ($(A.south east)+(1pt,1pt)$);
       \draw[postaction=decorate] ($(A.south west)!0.33!(A.north west)$) -- ($(A.south east)!0.33!(A.north east)$);
       \draw[postaction=decorate] (A.north west) -- (A.north east);
       \draw[postaction=decorate] (A.south west) -- (A.south east);
       \draw[postaction=decorate] ($(A.south west)!0.66!(A.north west)$) -- ($(A.south east)!0.66!(A.north east)$);
       \draw[postaction=decorate] (A.north west) -- ($(A.south east)!0.66!(A.north east)$);
       \draw[postaction=decorate] ($(A.south west)!0.66!(A.north west)$) -- ($(A.south east)!0.33!(A.north east)$);
       \draw[postaction=decorate] ($(A.south west)!0.33!(A.north west)$) -- (A.south east);
\end{scope}
        \begin{scope}[decoration={
    markings,
    mark=at position 0.166 with {\arrow{stealth}};,
    mark=at position 0.5 with {\arrow{stealth}};,
    mark=at position 0.833 with {\arrow{stealth}}}
]
       \draw[postaction=decorate] (A.south west) -- (A.north west);
       \draw[postaction=decorate] (A.north east) -- (A.south east);
   \end{scope}
   \node[left] at ($(A.south west)!0.166!(A.north west)$) {\tiny{$0$}} ;
   \node[left] at ($(A.south west)!0.5!(A.north west)$) {\tiny{$1$}} ;
   \node[left] at ($(A.south west)!0.833!(A.north west)$) {\tiny{$2$}} ;
   \node[right] at ($(A.south east)!0.166!(A.north east)$) {\tiny{$2$}} ;
   \node[right] at ($(A.south east)!0.5!(A.north east)$) {\tiny{$1$}} ;
   \node[right] at ($(A.south east)!0.833!(A.north east)$) {\tiny{$0$}} ;
   \node[right] at ($(A.south west)!0.166!(A.north west)$) {\tiny{$e_1$}} ;
   \node[right] at ($(A.south west)!0.5!(A.north west)$) {\tiny{$e_2$}} ;
   \node[right] at ($(A.south west)!0.833!(A.north west)$) {\tiny{$e_3$}} ;
   \node[left] at ($(A.south east)!0.166!(A.north east)$) {\tiny{$e_6$}} ;
   \node[left] at ($(A.south east)!0.5!(A.north east)$) {\tiny{$e_5$}} ;
   \node[left] at ($(A.south east)!0.833!(A.north east)$) {\tiny{$e_4$}} ;
   \node[below] at ($0.5*(A.south west)+0.5*(A.south east)$) {\tiny{$e_7$}} ;
   \node[below] at ($0.5*(A.south west)+0.5*(A.south east)+0.5*0.33*(A.north west)-0.5*0.33*(A.south west)$) {\tiny{$e_8$}} ;
   \node[below] at ($0.5*(A.south west)+0.5*(A.south east)+0.5*0.66*(A.north west)-0.5*0.66*(A.south west)$) {\tiny{$e_9$}} ;
   \node[below] at ($0.5*(A.south west)+0.5*(A.south east)+0.5*(A.north west)-0.5*(A.south west)$) {\tiny{$e_{10}$}} ;
   \node[below] at ($0.5*(A.south west)+0.5*(A.south east)+0.5*1.33*(A.north west)-0.5*1.33*(A.south west)$) {\tiny{$e_{11}$}} ;
   \node[below] at ($0.5*(A.south west)+0.5*(A.south east)+0.5*1.67*(A.north west)-0.5*1.67*(A.south west)$) {\tiny{$e_{12}$}} ;
   \node[below] at ($0.5*(A.north west)+0.5*(A.north east)$) {\tiny{$e_7$}} ;
   \node[] at ($0.166*0.2*(A.north east)+0.833*0.2*(A.south east)+0.166*0.8*(A.north west)+0.833*0.8*(A.south west)$) {\small{$\sigma_1$}} ;
   \node[] at ($0.166*0.8*(A.north east)+0.833*0.8*(A.south east)+0.166*0.2*(A.north west)+0.833*0.2*(A.south west)$) {\small{$\sigma_2$}} ;
   \node[] at ($0.5*0.2*(A.north east)+0.5*0.2*(A.south east)+0.5*0.8*(A.north west)+0.5*0.8*(A.south west)$) {\small{$\sigma_3$}} ;
   \node[] at ($0.5*0.8*(A.north east)+0.5*0.8*(A.south east)+0.5*0.2*(A.north west)+0.5*0.2*(A.south west)$) {\small{$\sigma_4$}} ;
   \node[] at ($0.833*0.2*(A.north east)+0.166*0.2*(A.south east)+0.833*0.8*(A.north west)+0.166*0.8*(A.south west)$) {\small{$\sigma_5$}} ;
   \node[] at ($0.833*0.8*(A.north east)+0.166*0.8*(A.south east)+0.833*0.2*(A.north west)+0.166*0.2*(A.south west)$) {\small{$\sigma_6$}} ;
   \node[below left] at (A.south west) {\tiny{$v_1$}} ;
   \node[left] at ($(A.south west)!0.33!(A.north west)$) {\tiny{$v_2$}} ;
   \node[left] at ($(A.south west)!0.66!(A.north west)$) {\tiny{$v_3$}} ;
   \node[above left] at (A.north west) {\tiny{$v_1$}} ;
   \node[below right] at (A.south east) {\tiny{$v_4$}} ;
   \node[right] at ($(A.south east)!0.33!(A.north east)$) {\tiny{$v_6$}} ;
   \node[right] at ($(A.south east)!0.66!(A.north east)$) {\tiny{$v_5$}} ;
   \node[above right] at (A.north east) {\tiny{$v_4$}} ;
    \end{tikzpicture}
    \caption{A triangulation of the cylinder together with markings and labels. Triangles are labeled by $\sigma_1,\dots,\sigma_6$, edges are labeled by $e_1,\dots,e_{12}$ and vertices are labeled by $v_1,\dots,v_6$. The bottom and top edge -- both labeled by $e_7$ -- are identified. The marked edge $d^2_0(\sigma_i)$ for each triangle $\sigma_i$ is indicated by a fat green line. The positions of the boundary edges on the respective boundary are given by small numbers $0,1,2$. The boundaries are ordered from left to right.}
    \label{fig:cylinder.triangulation}
\end{figure}
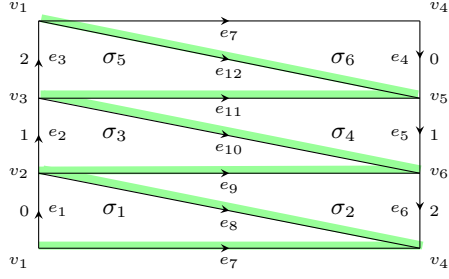

By the construction in Section \ref{sse:spin_reconstruct}, an admissible edge sign configuration determines a spin structure. We proceed to find the admissible edge sign configurations.
Let $s_i:=s(e_i)$ be the edge signs, which are yet to be determined. 
By applying Lemma \ref{lem:index.marking}\,(1) for the faces $\sigma_1,\dots,\sigma_6$, we can choose some of the edge signs to our liking. Specifically,
we will set 
\begin{equation}\label{eq:cylinder.signs.convention}
    s_4=s_5=s_6=s_8=s_{10}=s_{12} = 1 \ .
\end{equation}
Next we evaluate Lemma \ref{lem:vertex.rule.boundary}, the boundary vertex rule at $v_1,\dots,v_6$.
For NS-NS-boundary conditions we obtain:
\begin{align}
    v_1&:\; s_1s_3s_7=-1 ~,~~ &
        v_4&:\; s_7=-1 \ , \label{eq:cylinder.signs.ns}\\
    v_2&:\; s_1s_2s_9=1 ~,~~ &
	    v_5&:\;s_{11}=1 \ ,  \nonumber\\
    v_3&:\; s_2s_3s_{11}=1~,~~ &
	    v_6&:\;s_{9}=1 \ . \nonumber
\end{align}
For R-R-boundary conditions we obtain:
\begin{align}
    v_1&:\; s_1s_3s_7=1~,~~ &
        v_4&:\; s_7=1\ , \label{eq:cylinder.signs.r}\\
    v_2&:\; s_1s_2s_9=1~,~~ &
	    v_5&:\; s_{11}=1 \ ,  \nonumber\\
    v_3&:\; s_2s_3s_{11}=1~,~~ &
	    v_6&:\; s_{9}=1 \ .  \nonumber
\end{align}
Let $\varepsilon\in \{\pm 1\}$. All solutions to these sets of equations are then given by
\begin{equation}\label{eq:cylinder.signs.solution}
    s_1=s_2=s_3= \varepsilon \quad , \qquad
    s_9=s_{11}=1 \ ,
\end{equation}
as well as $s_7=-1$ for NS-NS-boundary conditions and $s_7=1$ for R-R-boundary conditions. Let us denote the resulting spin cylinders by $C^\varepsilon_{NS}$ and $C^\varepsilon_{R}$.

The calculation of the corresponding morphism $T_A(C^\pm_{NS/R})$ is given in Appendix \ref{sse:cylinder_calculation}. To express the resulting morphisms, it is helpful to define the maps
\begin{equation}\label{eq:iota-pi-PNSR}
\iota^{13} = \raisebox{-3em}{
\begin{tikzpicture}
    \coordinate (in) at (0cm,0.5cm);
    \node[name=cp1,circle,fill,inner sep=0.05cm] at (0cm,1cm) {};
    \node[name=cp2,circle,fill,inner sep=0.05cm] at (0.5cm,1.5cm) {};
    \coordinate (i1) at (-1cm,1.5cm);
    \coordinate (i2) at (0cm,2cm);
    \coordinate (i3) at (1cm,2cm);
    \coordinate (i4) at (0.5cm,2cm);
    \coordinate (i5) at (-0.5cm,2.5cm);
    \coordinate (out1) at (-0.5cm,3cm);
    \coordinate (out2) at (0cm,3cm);
    \coordinate (out3) at (0.5cm,3cm);
    \draw (in) -- (cp1);
    \draw (cp1) -- (cp2);
    \draw (cp1) .. controls (i1) and (i4) .. (out3);
    \draw (cp2) .. controls (i2) .. (out2);
    \draw (cp2) .. controls (i3) and (i5) .. (out1);
\end{tikzpicture}}
~,~~
\pi^{31} =  \raisebox{-3em}{
\begin{tikzpicture}
    \coordinate (in) at (0cm,-0.5cm);
    \node[name=cp1,circle,fill,inner sep=0.05cm] at (0cm,-1cm) {};
    \node[name=cp2,circle,fill,inner sep=0.05cm] at (0.5cm,-1.5cm) {};
    \coordinate (i1) at (-1cm,-1.5cm);
    \coordinate (i2) at (0cm,-2cm);
    \coordinate (i3) at (1cm,-2cm);
    \coordinate (i4) at (0.5cm,-2cm);
    \coordinate (i5) at (-0.5cm,-2.5cm);
    \coordinate (out1) at (-0.5cm,-3cm);
    \coordinate (out2) at (0cm,-3cm);
    \coordinate (out3) at (0.5cm,-3cm);
    \draw (in) -- (cp1);
    \draw (cp1) -- (cp2);
    \draw (cp1) .. controls (i1) and (i4) .. (out3);
    \draw (cp2) .. controls (i2) .. (out2);
    \draw (cp2) .. controls (i3) and (i5) .. (out1);
\end{tikzpicture}}
~,~~
P^{NS} = \raisebox{-3em}{
\begin{tikzpicture}
    \coordinate (in) at (0cm,0cm);
    \node[name=cp,circle,fill,inner sep=0.05cm] at (0cm,0.5cm) {};
    \node[name=p,circle,fill,inner sep=0.05cm] at (0cm,2cm) {};
    \node[name=N,circle,draw,inner sep=0.05cm] at (-0.3cm,1cm) {};
    \coordinate (id1r) at (0.3cm,1cm);
    \coordinate (id2l) at (-0.3cm,1.5cm);
    \coordinate (id2r) at (0.3cm,1.5cm);
    \coordinate (out) at (0cm,2.5cm);
    \coordinate (sd) at (0cm,-0.2cm);
    \coordinate (su) at (0cm,0.2cm);
    \draw (in) -- (cp);
    \draw (cp) .. controls ($(N)+(sd)$) .. (N);
    \draw (cp) .. controls ($(id1r)+(sd)$) .. (id1r);
    \draw (N) .. controls ($(N)+(su)$) and ($(id2r)+(sd)$) .. (id2r);
    \draw (id1r) .. controls ($(id1r)+(su)$) and ($(id2l)+(sd)$) .. (id2l);
    \draw (id2l) .. controls ($(id2l)+(su)$) .. (p);
    \draw (id2r) .. controls ($(id2r)+(su)$) .. (p);
    \draw (p) -- (out);
\end{tikzpicture}}
~,~~
P^{R} = \raisebox{-3em}{
\begin{tikzpicture}
    \coordinate (in) at (0cm,0cm);
    \node[name=cp,circle,fill,inner sep=0.05cm] at (0cm,0.5cm) {};
    \node[name=p,circle,fill,inner sep=0.05cm] at (0cm,2cm) {};
    \coordinate (id1l) at (-0.3cm,1cm);
    \coordinate (id1r) at (0.3cm,1cm);
    \coordinate (id2l) at (-0.3cm,1.5cm);
    \coordinate (id2r) at (0.3cm,1.5cm);
    \coordinate (out) at (0cm,2.5cm);
    \coordinate (sd) at (0cm,-0.2cm);
    \coordinate (su) at (0cm,0.2cm);
    \draw (in) -- (cp);
    \draw (cp) .. controls ($(id1l)+(sd)$) .. (id1l);
    \draw (cp) .. controls ($(id1r)+(sd)$) .. (id1r);
    \draw (id1l) .. controls ($(id1l)+(su)$) and ($(id2r)+(sd)$) .. (id2r);
    \draw (id1r) .. controls ($(id1r)+(su)$) and ($(id2l)+(sd)$) .. (id2l);
    \draw (id2l) .. controls ($(id2l)+(su)$) .. (p);
    \draw (id2r) .. controls ($(id2r)+(su)$) .. (p);
    \draw (p) -- (out);
\end{tikzpicture}}
 \ .
\end{equation}
In terms of these, the cylinder morphisms read
\begin{equation} \label{eq:TFT-cylinder-values}
    T_A(C_{\delta}^{\varepsilon}) = b \circ ( P^{\delta}\otimes \id_A) \circ (N_{-\varepsilon} \otimes \id_A) \circ (\pi^{31}\otimes \pi^{31}) \ ,
\end{equation}
where $\delta  \in \{NS,R\}$.

\subsection{Cylinder projections and state spaces}\label{sec:cylinder-projections}

In the following, we will write
\begin{equation} \label{eq:q_nu-def}
	q_\nu = \mu \circ \sigma_{A,A} \circ (N_{-\nu} \otimes \id_A) \circ \Delta \ ,
\end{equation}
so that $P^{NS} = q_{+}$ and $P^R = q_{-}$. 

\begin{lemma}
For $\nu \in \{\pm1\}$,
    \begin{enumerate}
\item
$q_\nu \circ q_\nu = q_\nu$, i.e.\ $P^{NS}$ and $P^R$ are idempotents.
\item
$q_\nu \circ N = N \circ q_\nu$, i.e.\ $P^{NS}$ and $P^R$ commute with the Nakayama automorphism.
\item
	$P^{NS} \circ N = P^{NS}$, or, equivalently,
	$q_\nu \circ N = q_\nu \circ N_\nu$ and $q_\nu = q_\nu \circ N_{-\nu}$.
\end{enumerate}
\label{lem:PNS/R-idemp}
\end{lemma}

\begin{proof}
For part 1 one computes
\begin{equation}\label{eq:PNS/R-idemp-aux1}
    q_{-\nu} \circ q_{-\nu} =
     \raisebox{-2cm}{\begin{tikzpicture}
        \coordinate (su) at (0cm,0.5cm);
        \coordinate (sd) at (0cm,-0.5cm);
        \coordinate (sdl) at (-0.5cm,-0.5cm);
        \coordinate (sdr) at (0.5cm,-0.5cm);
        \coordinate (sul) at (-0.5cm,0.5cm);
        \coordinate (sur) at (0.5cm,0.5cm);
        \node[name=p1, product] at (0.3cm,4.2cm){};
        \node[name=cp1, product] at (0cm,0.5cm){};
        \node[name=n1, naka] at (-0.5cm,1cm){$\nu$};
        \coordinate (in) at (0cm,0cm);
        \coordinate (i1) at (0.5cm,2cm);
        \coordinate (i2) at (-0.5cm,1.8cm);
        \node[name=p2, product] at (0cm,4.5cm){};
        \node[name=cp2, product] at (0cm,2.5cm){};
        \node[name=n2, naka] at (-0.5cm,3cm){$\nu$};
        \coordinate (i3) at ($(p2)+(-0.5cm,-0.7cm)$);
        \coordinate (out) at (0cm,5cm);
        \foreach \N/\M/\S/\T in {
            in/cp1/su/sd,
            cp1/i2/sur/sd,
            i2/cp2/su/sd,
            n1/i1/su/sd,
            i1/p1/su/sdl,
            n2/p1/su/sdr,
            p1/p2/sul/sdr,
            cp2/i3/sur/sd,
            p2/out/su/sd}
            \draw (\N) .. controls ($(\N)+(\S)$) and ($(\M)+(\T)$) .. (\M);
        \foreach \N/\M/\S/\T in {
            i3/p2/su/sdl,
            cp1/n1/sul/sd,
            cp2/n2/sul/sd}
            \draw (\N) .. controls ($(\N)+0.5*(\S)$) and ($(\M)+0.5*(\T)$) .. (\M);
    \end{tikzpicture}}
    =
    \raisebox{-2cm}{\begin{tikzpicture}
        \coordinate (su) at (0cm,0.5cm);
        \coordinate (sd) at (0cm,-0.5cm);
        \coordinate (sdl) at (-0.5cm,-0.5cm);
        \coordinate (sdr) at (0.5cm,-0.5cm);
        \coordinate (sul) at (-0.5cm,0.5cm);
        \coordinate (sur) at (0.5cm,0.5cm);
        \node[name=p1, product] at (0cm,2cm){};
        \node[name=cp1, product] at (0cm,0.5cm){};
        \node[name=n1, naka] at (-0.5cm,1cm){$\nu$};
        \coordinate (in) at (0cm,0cm);
        \coordinate (i1) at (0.5cm,2cm);
        \coordinate (i2) at (-0.5cm,1.8cm);
        \node[name=p2, product] at (0cm,4.5cm){};
        \node[name=cp2, product] at (0cm,2.5cm){};
        \node[name=n2, naka] at (-0.5cm,3cm){$\nu$};
        \node[name=n3, naka] at (-0.5cm,1.5cm){$\nu$};
        \coordinate (i3) at ($(p2)+(-0.5cm,-0.7cm)$);
        \coordinate (out) at (0cm,5cm);
        \foreach \N/\M/\S/\T in {
            in/cp1/su/sd,
            cp1/p1/sur/sdr,
            n1/n3/su/sd,
            p1/cp2/su/sd,
            cp2/i3/sur/sd,
            n2/p2/su/sdr,
            p2/out/su/sd}
            \draw (\N) .. controls ($(\N)+(\S)$) and ($(\M)+(\T)$) .. (\M);
        \foreach \N/\M/\S/\T in {
            i3/p2/su/sdl,
            cp1/n1/sul/sd,
            cp2/n2/sul/sd,
            n3/p1/su/sdl}
            \draw (\N) .. controls ($(\N)+0.5*(\S)$) and ($(\M)+0.5*(\T)$) .. (\M);
    \end{tikzpicture}}
    = q_{-\nu} \ .
\end{equation}
Part 2 follows since $N$ is an automorphism of Frobenius algebras (Proposition \ref{prop:Nak-Frob-auto}). For part 3 first note that
\begin{equation}
    \raisebox{-1.5cm}{\begin{tikzpicture}
    \coordinate (in) at (0cm,-0.5cm);
    \node[name=cp,circle,fill,inner sep=0.05cm] at (0cm,0.5cm) {};
    \node[name=p,circle,fill,inner sep=0.05cm] at (0cm,2cm) {};
    \node[name=N,circle,draw,inner sep=0.05cm] at (-0.3cm,1cm) {};
    \coordinate (id1r) at (0.3cm,1cm);
    \coordinate (id2l) at (-0.3cm,1.5cm);
    \coordinate (id2r) at (0.3cm,1.5cm);
    \coordinate (out) at (0cm,3cm);
    \coordinate (sd) at (0cm,-0.2cm);
    \coordinate (su) at (0cm,0.2cm);
    \draw (in) -- (cp);
    \draw (cp) .. controls ($(N)+(sd)$) .. (N);
    \draw (cp) .. controls ($(id1r)+(sd)$) .. (id1r);
    \draw (N) .. controls ($(N)+(su)$) and ($(id2r)+(sd)$) .. (id2r);
    \draw (id1r) .. controls ($(id1r)+(su)$) and ($(id2l)+(sd)$) .. (id2l);
    \draw (id2l) .. controls ($(id2l)+(su)$) .. (p);
    \draw (id2r) .. controls ($(id2r)+(su)$) .. (p);
    \draw (p) -- (out);
    \end{tikzpicture}}
   \overset{\text{\eqref{eq:Nak-definition-pic}}}=
   \raisebox{-1.5cm}{\begin{tikzpicture}
        \coordinate (su) at (0cm,0.5cm);
        \coordinate (sd) at (0cm,-0.5cm);
        \coordinate (sdl) at (-0.5cm,-0.5cm);
        \coordinate (sdr) at (0.5cm,-0.5cm);
        \coordinate (sul) at (-0.5cm,0.5cm);
        \coordinate (sur) at (0.5cm,0.5cm);
        \node[name=p, product] at (0.5cm,3cm){};
        \node[name=cp, product] at (0cm,0.5cm){};
        \coordinate (in) at (0cm,0cm);
        \node[name=b, product] at (-0.5cm,2cm){};
        \node[name=c, product] at (0cm,1cm){};
        \coordinate (i1) at (0.5cm,1.5cm);
        \coordinate (out) at (0.5cm,3.5cm);
        \foreach \N/\M/\S/\T in {
            in/cp/su/sd,
            cp/b/sul/sdl,
            cp/i1/sur/sd,
            i1/p/su/sdl,
            c/b/sur/sdr,
            c/p/sul/sdr,
            p/out/su/sd}
            \draw (\N) .. controls ($(\N)+(\S)$) and ($(\M)+(\T)$) .. (\M);
    \end{tikzpicture}}
   \overset{\text{deform}}=
    \raisebox{-1.5cm}{\begin{tikzpicture}
        \coordinate (su) at (0cm,0.5cm);
        \coordinate (sd) at (0cm,-0.5cm);
        \coordinate (sdl) at (-0.5cm,-0.5cm);
        \coordinate (sdr) at (0.5cm,-0.5cm);
        \coordinate (sul) at (-0.5cm,0.5cm);
        \coordinate (sur) at (0.5cm,0.5cm);
        \node[name=p, product] at (0.5cm,3cm){};
        \node[name=cp, product] at (0cm,0.5cm){};
        \coordinate (in) at (0cm,0cm);
        \node[name=b, product] at (1cm,3cm){};
        \node[name=c, product] at (1cm,2cm){};
        \coordinate (i1) at (0.5cm,1.5cm);
        \coordinate (out) at (0.5cm,3.5cm);
        \foreach \N/\M/\S/\T in {
            in/cp/su/sd,
            cp/b/sul/sdl,
            cp/i1/sur/sd,
            i1/p/su/sdl,
            c/b/sur/sdr,
            c/p/sul/sdr,
            p/out/su/sd}
            \draw (\N) .. controls ($(\N)+(\S)$) and ($(\M)+(\T)$) .. (\M);
    \end{tikzpicture}}
   \underset{\text{coassoc.}}{\overset{\text{\eqref{eq:coproduct-counit}}}=}
    \raisebox{-1.5cm}{\begin{tikzpicture}
        \coordinate (su) at (0cm,0.5cm);
        \coordinate (sd) at (0cm,-0.5cm);
        \coordinate (sdl) at (-0.5cm,-0.5cm);
        \coordinate (sdr) at (0.5cm,-0.5cm);
        \coordinate (sul) at (-0.5cm,0.5cm);
        \coordinate (sur) at (0.5cm,0.5cm);
        \node[name=cp1, product] at (0cm,1.5cm){};
        \node[name=cp2, product] at (0cm,0.5cm){};
        \coordinate (in) at (0cm,0cm);
        \node[name=b, product] at (0.5cm,3cm){};
        \coordinate (out) at (0cm,3.5cm);
        \foreach \N/\M/\S/\T in {
            in/cp1/su/sd,
            cp1/b/sul/sdl,
            cp1/cp2/su/sd,
            cp2/b/sur/sdr,
            cp2/out/su/sd}
            \draw (\N) .. controls ($(\N)+(\S)$) and ($(\M)+(\T)$) .. (\M);
    \end{tikzpicture}}\ .
\label{eq:PRN=PR-aux1}
\end{equation}
On the other hand, using that $N=N^{-1}$ and inserting the explicit expression \eqref{eq:Nak-inv} for $N^{-1}$ we get
\begin{equation}\label{eq:PRN=PR-aux2}
    \raisebox{-1.5cm}{\begin{tikzpicture}
    \begin{scope}[xscale=-1]
    \coordinate (in) at (0cm,-0.5cm);
    \node[name=cp,circle,fill,inner sep=0.05cm] at (0cm,0.5cm) {};
    \node[name=p,circle,fill,inner sep=0.05cm] at (0cm,2cm) {};
    \node[name=N,circle,draw,inner sep=0.05cm] at (-0.3cm,1cm) {};
    \coordinate (id1r) at (0.3cm,1cm);
    \coordinate (id2l) at (-0.3cm,1.5cm);
    \coordinate (id2r) at (0.3cm,1.5cm);
    \coordinate (out) at (0cm,3cm);
    \coordinate (sd) at (0cm,-0.2cm);
    \coordinate (su) at (0cm,0.2cm);
    \end{scope}
    \draw (in) -- (cp);
    \draw (cp) .. controls ($(N)+(sd)$) .. (N);
    \draw (cp) .. controls ($(id1r)+(sd)$) .. (id1r);
    \draw (N) .. controls ($(N)+(su)$) and ($(id2r)+(sd)$) .. (id2r);
    \draw (id1r) .. controls ($(id1r)+(su)$) and ($(id2l)+(sd)$) .. (id2l);
    \draw (id2l) .. controls ($(id2l)+(su)$) .. (p);
    \draw (id2r) .. controls ($(id2r)+(su)$) .. (p);
    \draw (p) -- (out);
    \end{tikzpicture}}
\overset{\text{\eqref{eq:Nak-inv}}}=
   \raisebox{-1.5cm}{\begin{tikzpicture}
       \begin{scope}[xscale=-1]
        \coordinate (su) at (0cm,0.5cm);
        \coordinate (sd) at (0cm,-0.5cm);
        \coordinate (sdl) at (-0.5cm,-0.5cm);
        \coordinate (sdr) at (0.5cm,-0.5cm);
        \coordinate (sul) at (-0.5cm,0.5cm);
        \coordinate (sur) at (0.5cm,0.5cm);
        \node[name=p, product] at (0.5cm,3cm){};
        \node[name=cp, product] at (0cm,0.5cm){};
        \coordinate (in) at (0cm,0cm);
        \node[name=b, product] at (-0.5cm,2cm){};
        \node[name=c, product] at (0cm,1cm){};
        \coordinate (i1) at (0.5cm,1.5cm);
        \coordinate (out) at (0.5cm,3.5cm);
        \end{scope}
        \foreach \N/\M/\S/\T in {
            in/cp/su/sd,
            cp/b/sul/sdl,
            cp/i1/sur/sd,
            i1/p/su/sdl,
            c/b/sur/sdr,
            c/p/sul/sdr,
            p/out/su/sd}
            \draw (\N) .. controls ($(\N)+(\S)$) and ($(\M)+(\T)$) .. (\M);
    \end{tikzpicture}}
   \overset{\text{deform}}=
    \raisebox{-1.5cm}{\begin{tikzpicture}
       \begin{scope}[xscale=-1]
        \coordinate (su) at (0cm,0.5cm);
        \coordinate (sd) at (0cm,-0.5cm);
        \coordinate (sdl) at (-0.5cm,-0.5cm);
        \coordinate (sdr) at (0.5cm,-0.5cm);
        \coordinate (sul) at (-0.5cm,0.5cm);
        \coordinate (sur) at (0.5cm,0.5cm);
        \node[name=p, product] at (0.5cm,3cm){};
        \node[name=cp, product] at (0cm,0.5cm){};
        \coordinate (in) at (0cm,0cm);
        \node[name=b, product] at (1cm,3cm){};
        \node[name=c, product] at (1cm,2cm){};
        \coordinate (i1) at (0.5cm,1.5cm);
        \coordinate (out) at (0.5cm,3.5cm);
        \end{scope}
        \foreach \N/\M/\S/\T in {
            in/cp/su/sd,
            cp/b/sul/sdl,
            cp/i1/sur/sd,
            i1/p/su/sdl,
            c/b/sur/sdr,
            c/p/sul/sdr,
            p/out/su/sd}
            \draw (\N) .. controls ($(\N)+(\S)$) and ($(\M)+(\T)$) .. (\M);
    \end{tikzpicture}} 
   \overset{\text{\eqref{eq:coproduct-counit-2nd}}}=
    \raisebox{-1.5cm}{\begin{tikzpicture}
       \begin{scope}[xscale=-1]
        \coordinate (su) at (0cm,0.5cm);
        \coordinate (sd) at (0cm,-0.5cm);
        \coordinate (sdl) at (-0.5cm,-0.5cm);
        \coordinate (sdr) at (0.5cm,-0.5cm);
        \coordinate (sul) at (-0.5cm,0.5cm);
        \coordinate (sur) at (0.5cm,0.5cm);
        \node[name=cp1, product] at (0cm,0.5cm){};
        \node[name=cp2, product] at (0cm,1.5cm){};
        \coordinate (in) at (0cm,0cm);
        \node[name=b, product] at (0.5cm,3cm){};
        \coordinate (out) at (0cm,3.5cm);
        \end{scope}
        \foreach \N/\M/\S/\T in {
            in/cp1/su/sd,
            cp1/b/sul/sdl,
            cp1/cp2/su/sd,
            cp2/b/sur/sdr,
            cp2/out/su/sd}
            \draw (\N) .. controls ($(\N)+(\S)$) and ($(\M)+(\T)$) .. (\M);
    \end{tikzpicture}}\ ,
\end{equation}
which is equal to  \eqref{eq:PRN=PR-aux1}. This implies $P^{NS} \circ N = P^{NS}$.
\end{proof}

Next we turn the cylinder morphisms into endomorphisms of $A^{\otimes 3}$ via the ``copairing'' $\Gamma$ we defined in \eqref{eq:Gamma-ijkeps}. First observe that
\begin{equation} \label{eq:pi31-through-Gamma}
	(\pi^{31} \otimes \id_{A^{\otimes 3}}) \circ \Gamma_{1,2,-} = (\id_A \otimes \iota^{13}) \circ \Delta \circ \eta \ .
\end{equation}
Using this, one arrives at
\begin{equation}\label{eq:f-eps-NS/R-def}
	f^\varepsilon_{NS/R} := (T_A(C_{NS/R}^\varepsilon) \otimes \id_{A^{\otimes 3}}) \circ \Gamma_{2,3,-}
	= \iota^{13} \circ P^{NS/R} \circ N_{-\varepsilon} \circ \pi^{31} \ .
\end{equation}
Since $\pi^{31}\circ \iota^{13} = \id_A$, together with Lemma \ref{lem:PNS/R-idemp} it follows that 
	$f^-_{NS/R}$ are idempotents. 

To speak about state spaces, we need a further assumption:
\begin{quote}
{\bf Assumption 3:} The idempotents $P^{NS/R}$ are split.
\end{quote}
Let us denote the image of $P^{NS/R}$ by $Z^{NS/R}(A)$ and write $\iota^{NS/R} : Z^{NS/R}(A) \to A$ and $\pi^{NS/R} : A \to Z^{NS/R}(A)$ for the embedding and restriction maps. 
	We will call $Z^{NS}(A)$ (resp.\ $Z^R(A)$) the $NS$-{\em type} (resp.\ $R$-{\em type}) {\em state space} of the spin TFT associated to $A$.
The morphism $T_A(\Sigma)$ assigned to a spin surface $\Sigma$ factors through the state spaces in the following sense:

\begin{proposition}\label{prop:state-space-projectors-omit}
Let $\Sigma$ be a spin surface and let $\delta_i \in \{NS,R\}$ be the type of the $i$'th of the $B$ boundary components of $\Sigma$. Then
    \begin{equation}
        T_A(\Sigma) = T_A(\Sigma) \circ
	\bigotimes_{i=1}^B  \big( \iota^{13} \circ \iota^{\delta_i} \circ \pi^{\delta_i} \circ \pi^{31} \big)
  \ .
    \end{equation}
\end{proposition}

\begin{figure}[tb]
    \centering
    \begin{tikzpicture}
        \node[name=A, minimum width=3cm, minimum height=4cm] at (0cm,0cm){};
        \foreach \N in {1,2,3,4} \coordinate (Av\N) at ($(A.south west)+\N*0.333*(A.south east)-\N*0.333*(A.south west)-0.333*(A.south east)+0.333*(A.south west)$);
        \coordinate (Av5) at ($(A.east)!0.8!(A.west)$);
        \coordinate (Av6) at ($(A.east)!0.4!(A.west)$);
        \coordinate (Av7) at ($($(A.east)!0.15!(A.west)$)+0.1*(A.south)$);
        \coordinate (Av8) at ($0.2*(A.east)+0.4*(A.north)+0.7*(A.west)$);
        \coordinate (Av9) at ($0.4*(A.east)+0.5*(A.north)+0.3*(A.west)$);
        \coordinate (Av10) at ($0.7*(A.east)+0.3*(A.north)+0.3*(A.west)$);
        \coordinate (Av11) at (A.west);
        \coordinate (Av12) at (A.east);
        \coordinate (Ad1) at ($0.5*(Av11)+0.5*(Av8)$);
        \coordinate (Ad2) at ($(Av8)+0.2*(A.west)+0.05*(A.north)$);
        \coordinate (Ad3) at ($(Av8)+0.05*(A.east)+0.2*(A.north)$);
        \coordinate (Ad4) at ($(Av9)+0.05*(A.east)+0.2*(A.north)$);
        \coordinate (Ad5) at  ($(Av9)+0.05*(A.west)+0.2*(A.north)$);
        \coordinate (Ad6) at ($(Av10)+0.1*(A.east)+0.1*(A.north)$);
        \coordinate (Ad7) at ($(Av7)+0.15*(A.north)$);
        \foreach \N/\M in {
            1/4,
            1/5,
            2/5,
            2/6,
            3/6,
            3/7,
            4/7,
            2/8,
            5/8,
            2/9,
            8/9,
            6/9,
            3/10,
            6/10,
            7/10,
            9/10,
            5/11,
            7/12} \draw (Av\N) -- (Av\M);
        \foreach \N/\M in {
            A.south west/A.north west,
            A.south east/A.north east,
            Av5/Ad1,
            Av8/Ad2,
            Av8/Ad3,
            Av9/Ad4,
            Av9/Ad5,
            Av10/Ad6,
            Av7/Ad7}
            \draw[dashed] (\N) -- (\M);
            \node[] at ($(A.west)+(0cm,0.1cm)$) {$=$};
            \node[] at ($(A.east)+(0cm,0.1cm)$) {$=$};
        \node[name=B, minimum width=3cm, minimum height=4cm] at (6cm,0cm){};
        \foreach \N in {1,...,12} \coordinate (Bv\N) at ($(Av\N)+(6cm,0cm)$);
        \foreach \N in {1,...,4} \coordinate (Bn\N) at (Bv\N);
        \foreach \N in {1,...,4} \coordinate (Bv\N) at ($(Bv\N)+(0cm,1cm)$);
        \foreach \N in {1,...,7} \coordinate (Bd\N) at ($(Ad\N)+(6cm,0cm)$);
        \foreach \N/\M in {
            1/4,
            1/5,
            2/5,
            2/6,
            3/6,
            3/7,
            4/7,
            2/8,
            5/8,
            2/9,
            8/9,
            6/9,
            3/10,
            6/10,
            7/10,
            9/10,
            5/11,
            7/12} \draw (Bv\N) -- (Bv\M);
        \foreach \N/\M in {
            Bn1/Bn4,
            Bn1/Bv1,
            Bn2/Bv2,
            Bn3/Bv3,
            Bn4/Bv4,
            Bn1/Bv2,
            Bn2/Bv3,
            Bn3/Bv4} \draw (\N) -- (\M);
        \foreach \N/\M in {
            Bv1/B.north west,
            Bv4/B.north east,
            Bv5/Bd1,
            Bv8/Bd2,
            Bv8/Bd3,
            Bv9/Bd4,
            Bv9/Bd5,
            Bv10/Bd6,
            Bv7/Bd7}
            \draw[dashed] (\N) -- (\M);
            \node[] at ($(A.west)+(6cm,0.1cm)$) {$=$};
            \node[] at ($(A.east)+(6cm,0.1cm)$) {$=$};
        \draw[->] ($(A.east)+(1cm,0cm)$) -- ($(B.west)+(-1cm,0cm)$);
    \end{tikzpicture}
    \caption{Adding the cylinder triangulation from Figure \ref{fig:cylinder.triangulation} at a boundary component by pushing the existing triangulation inwards.}
    \label{fig:push-triang-inwards}
\end{figure}
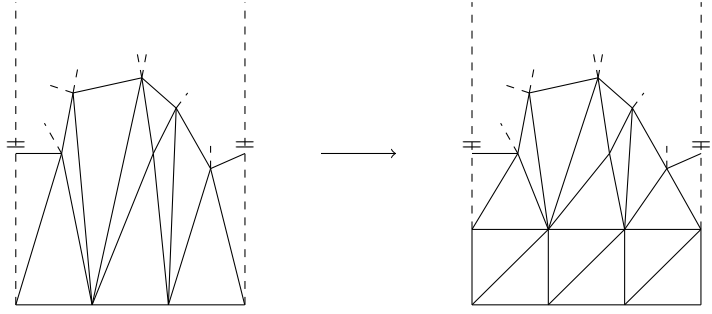

\begin{proof}
Choose a spin triangulation $\Lambda$ of $\Sigma$. Consider the triangulation close to the $i$'th boundary component of $\Sigma$. We may assume that none of the boundary edges are marked edges of the adjacent triangle, and that the edge sign is 1 for all boundary edges. From this build a new spin triangulation $\Lambda'$ by ``pushing the triangulation $\Lambda$ slightly inward'' and adding a cylinder as in Figure \ref{fig:cylinder.triangulation} with $\varepsilon=-1$. This is illustrated in Figure \ref{fig:push-triang-inwards}. That $\varepsilon=-1$ occurs and not $\varepsilon=1$ makes a difference only for R-type boundary conditions (Lemma \ref{lem:PNS/R-idemp}). For these it can be checked by direct calculation or by observing that $P^{R}$ is an idempotent, but $P^{R} \circ N$ is in general not.

The edge signs of the initial spin triangulation $\Lambda$ are not affected by the above procedure. By Theorem \ref{thm:TA-triang-indep}, the spin triangulations $\Lambda$ and $\Lambda'$ produce the same morphism $T_A(\Sigma)$. Comparing the morphisms produced by these two triangulations, this amounts to the identity
    \begin{equation}
        T_A(\Sigma) = \big( T_A(C_{\delta_i}^-) \otimes T_A(\Sigma)  \big) \circ \Gamma_{2,i,-}
  \ .
    \end{equation}
Substituting $f_{\delta_i}^-$ from \eqref{eq:f-eps-NS/R-def}, on sees that the right hand side is equal to $T_A(\Sigma) \circ (\id \otimes \cdots \otimes f_{\delta_i}^- \otimes \cdots \otimes \id)$, where $f_{\delta_i}^-$ is inserted on the $i$'th tensor factor  $A^{\otimes 3}$. Repeating this for each $i$ and using $f_{\delta_i}^- =  \iota^{13} \circ \iota^{\delta_i} \circ \pi^{\delta_i} \circ \pi^{31}$ proves the claim.
\end{proof}

\begin{remark}\label{rem:N-involution-on-Z}
	Let $\delta \in \{NS,R\}$.
The Nakayama automorphism of $A$ induces an involution $\pi^\delta \circ N \circ \iota^\delta$ on $Z^{\delta}(A)$, which by abuse of notation we still call $N$, or $N|_{Z^{\delta}}$ if we want to be more specific. That $N|_{Z^{\delta}}$ is an involution follows from $\pi^{\delta} \circ \iota^{\delta} = \id_{Z^{\delta}(A)}$ and $N \circ N = \id_A$. By Lemma \ref{lem:PNS/R-idemp}, actually $N|_{Z^{NS}} = \id_{Z^{NS}}$. As pointed out in \cite[Ex.\,1 in Sect.\,2.6]{moore2006d},
geometrically, this amounts to the observation that $C_{NS}^+ \cong C_{NS}^-$ as spin surfaces (via a Dehn-twist).
\end{remark}

\begin{remark}\label{rem:BT-BCP-relation} (i)
In \cite{barrett2013spin}, state sum models on spin surfaces are addressed independently of our work. Let us briefly point out some similarities in the algebraic structures considered in \cite{barrett2013spin} and here. In \cite{barrett2013spin}, the algebraic datum is a separable Frobenius algebra on a symmetrically braided vector space, subject to further conditions, see \cite[Def.\,4.1]{barrett2013spin}. In particular, it is imposed that the Nakayama automorphism is an involution. Only closed surfaces are considered in \cite{barrett2013spin}, so that the question of glueing and of state spaces does not arise. Nonetheless, in evaluating the model on higher genus surfaces, the projectors $P^{NS}$ and $P^R$ onto the NS and R state spaces appear, see \cite[Lem.\,4.5]{barrett2013spin}. 

\smallskip\noindent
(ii) The projectors $P^{NS}$ and $P^R$ also appear in \cite{Brunner:2013ota}, see Section 3.2 there, where they are called $\pi^{\text{(c,c)}}_A$ and $\pi_A^{\text{RR}}$, respectively. In \cite{Brunner:2013ota}, the authors are not concerned with spin TFTs, but instead consider ``generalised twisted sectors'' in orbifolds of 2d TFTs. Accordingly, in \cite{Brunner:2013ota} no restriction on the Nakayama automorphism is imposed.
The generalised twisted sectors are described as  the images of $\pi^{\text{(c,c)}}_A$ and $\pi_A^{\text{RR}}$. In the case that the TFT arises as the topological twist of an $N=(2,2)$ supersymmetric field theory, the construction of \cite{Brunner:2013ota} recovers the $(c,c)$-ring and the $R$-ground states of the orbifolded theory, hence the names.
\end{remark}

\subsection{Pair of pants and multiplication}\label{sec:pair-of-pants}
We proceed to evaluate the TFT on the genus $0$ surface with $3$ boundaries, $\Sigma^{0,3}$. 
	The results are collected in Lemmas \ref{lem:pairofpants.spinstructures} and \ref{lem:pair-of-pants-TFT-value} below. The computations going into the proofs 
are slightly tedious and therefore collected in Appendix \ref{app:pairpants}. The general procedure, however, is as in Section \ref{sec:TFT.cylinder}: We fix a triangulation for the surface, determine admissible edge sign configurations and then proceed to calculate the morphism $T_A(\Sigma^{0,3})$. 

\begin{lemma}
    The surface $\Sigma^{0,3}$ admits spin structures only if the spin structures on the boundary are $NS$-$NS$-$NS$ or $NS$-$R$-$R$ (in any order). If so, 
	then up to isomorphism of spin structures,
there are exactly $4$ spin structures.
    \label{lem:pairofpants.spinstructures}
\end{lemma}

We parametrise the corresponding spin surfaces by numbers $\varepsilon_1,\varepsilon_2\in \{1,-1\}$ and denote them by 
$\Sigma^{0,3}_{\delta_1,\delta_2,\delta_3,\varepsilon_1,\varepsilon_2}$ with $\delta_i \in \{NS,R\}$. 
The definition of these spin structures in terms of admissible edge signs is given in Appendix \ref{app:pairpants}.
One can check that up to diffeomorphisms of spin surfaces, there is just one spin surface with boundaries of type $NS$-$NS$-$NS$ and two such surfaces with boundary types $NS$-$R$-$R$.

\begin{lemma}\label{lem:pair-of-pants-TFT-value}
    The values of the TFT on the above spin surfaces are
    \begin{align}\label{eq:3-holed-sphere-morphs}
	T_A\left( \Sigma^{0,3}_{\delta_1,\delta_2,\delta_3,\varepsilon_1,\varepsilon_2} \right) = 
    b\circ \left(\id_A\otimes \mu\right) 
    &\circ \big(P^{\delta_1} \otimes P^{\delta_2} \otimes  P^{\delta_3}\big) 
    \\ & 
    \circ \big(N_{\varepsilon_1}\otimes N_{\varepsilon_2} \otimes \id_A \big) 
    \circ (\pi^{31})^{\otimes 3}\ . \nonumber
    \end{align}
\end{lemma}

	We want to use the morphisms \eqref{eq:3-holed-sphere-morphs} to define an algebra structure on the state space. To do so, it is convenient to add another assumption to our list:
\begin{quote}
	{\bf Assumption 4:} The symmetric 
		strict	
	monoidal category $\mathcal{S}$ is additive.
\end{quote}
Recall the definition of the $NS$- and $R$-type state spaces $Z^{NS/R}(A)$ from the previous section. Under the above assumption, we can now define the {\em total state space}
\begin{equation}\label{eq:Z-total}
	Z(A) := Z^{NS}(A) \oplus Z^R(A)
\end{equation}
of the spin TFT. 
 In the following, we will define an associative, $\mathbb{Z}_2$-graded product on $Z(A)$ and investigate some of its properties. In particular, we will show that the product agrees with the morphisms \eqref{eq:3-holed-sphere-morphs}.

\medskip

To start with, we need to know how the projectors $P^{NS}$ and $P^{R}$ interact with the structure maps of $A$.
Recall the definition of $q_\pm$ from \eqref{eq:q_nu-def}.

\begin{lemma}\label{lem:forget-one-projector}
Let $\nu_1,\nu_2,\nu_3\in \{1,-1\}$ such that $\nu_1\nu_2\nu_3=1$. Then, 
    \begin{align}\label{eq:projectors-product}
        q_{\nu_1}\circ \mu \circ \left( q_{\nu_2}\otimes q_{\nu_3} \right) &= q_{\nu_1} \circ \mu \circ \left( q_{\nu_2} \otimes \id_A \right)  \\
        &= q_{\nu_1} \circ \mu \circ \left( \id_A \otimes q_{\nu_3} \right) \nonumber \\
        &= \mu \circ \left( q_{\nu_2} \otimes q_{\nu_3} \right) \ . \nonumber
    \end{align}
The unit and counit of $A$ satisfy
    \begin{equation}
P^{NS}\circ\eta=\eta 
\quad , \qquad
\varepsilon\circ P^{NS}=\varepsilon 
\ .
        \label{eq:projector-unit-counit}
    \end{equation}
\end{lemma}
\begin{proof}
The identities \eqref{eq:projector-unit-counit}
	are a consequence of $b \circ \sigma_{A,A} \circ (N \otimes \id_A) = b$ and $\Delta$-separability. 
    We show the first equality in equation \eqref{eq:projectors-product}:
    \begin{align}
            &q_{\nu_1}\circ \mu \circ \left( q_{\nu_2}\otimes q_{\nu_3} \right) \stackrel{\text{Lem.\,\ref{lem:PNS/R-idemp}}}{=} q_{\nu_1}\circ \mu \circ \left( q_{\nu_2}\otimes
			(q_{\nu_3}\circ N_{-\nu_3}) \right) \\
            &=\raisebox{-2cm}{\begin{tikzpicture}
            \coordinate (sd) at (0cm,-0.2cm);
            \coordinate (su) at (0cm,0.2cm);
            \coordinate (in1) at (0cm,0cm);
            \coordinate (id1l1) at (-0.3cm,1cm);
            \coordinate (id1r1) at (0.3cm,1cm);
            \coordinate (id2l1) at (-0.3cm,1.5cm);
            \coordinate (id2r1) at (0.3cm,1.5cm);
            \coordinate (out1) at (0cm,2.5cm);
            \node[name=cp1,circle,fill,inner sep=0.05cm] at (0cm,0.5cm) {};
            \node[name=p1,circle,fill,inner sep=0.05cm] at (0cm,2cm) {};
            \node[name=N1,circle,draw,inner sep=0.05cm] at ($(id1l1)+(0cm,-0.1cm)$) {\tiny{$-\nu_2$}};
            \draw (in1) -- (cp1);
            \draw (cp1) .. controls ($(id1r1)+(sd)$) .. (id1r1);
            \draw (cp1) .. controls ($(N1.south)+0.2*(sd)$) .. (N1.south);
            \draw (N1.north) .. controls ($(N1.north)+(su)$) and ($(id2r1)+(sd)$) .. (id2r1);
            \draw (id1r1) .. controls ($(id1r1)+(su)$) and ($(id2l1)+(sd)$) .. (id2l1);
            \draw (id2l1) .. controls ($(id2l1)+(su)$) .. (p1);
            \draw (id2r1) .. controls ($(id2r1)+(su)$) .. (p1);
            \coordinate (sd3) at (1.5cm,0cm); 
            \coordinate (in3) at ($(in1)+(sd3)$);
            \node[name=cp3,circle,fill,inner sep=0.05cm] at ($(cp1)+(sd3)$) {};
            \node[name=p3,circle,fill,inner sep=0.05cm] at ($(p1)+(sd3)$) {};
            \node[name=N3,circle,draw,inner sep=0.05cm] at ($(id1r1)+(sd3)+(0cm,-0.1cm)$) {\tiny{$-\nu_3$}};
            \coordinate (id1l3) at ($(id1l1)+(sd3)$);
            \coordinate (id2l3) at ($(id2l1)+(sd3)$);
            \coordinate (id2r3) at ($(id2r1)+(sd3)$);
            \coordinate (out3) at ($(out1)+(sd3)$);
            \draw (in3) -- (cp3);
            \draw (cp3) .. controls ($(id1l3)+(sd)$) .. (id1l3);
            \draw (cp3) .. controls ($(N3.south)+0.2*(sd)$) .. (N3.south);
            \draw (N3.north) .. controls ($(N3.north)+(su)$) and ($(id2l3)+(sd)$) .. (id2l3);
            \draw (id1l3) .. controls ($(id1l3)+(su)$) and ($(id2r3)+(sd)$) .. (id2r3);
            \draw (id2l3) .. controls ($(id2l3)+(su)$) .. (p3);
            \draw (id2r3) .. controls ($(id2r3)+(su)$) .. (p3);
            \coordinate (sd2) at (0.75cm,3cm); 
            \coordinate (in2) at ($(in1)+(sd2)$);
            \node[name=pc,circle,fill,inner sep=0.05cm] at (in2) {};
            \node[name=cp2,circle,fill,inner sep=0.05cm] at ($(cp1)+(sd2)$) {};
            \node[name=p2,circle,fill,inner sep=0.05cm] at ($(p1)+(sd2)$) {};
            \node[name=N2,circle,draw,inner sep=0.05cm] at ($(N1)+(sd2)$) {\tiny{$-\nu_1$}};
            \coordinate (id1r2) at ($(id1r1)+(sd2)$);
            \coordinate (id2l2) at ($(id2l1)+(sd2)$);
            \coordinate (id2r2) at ($(id2r1)+(sd2)$);
            \coordinate (out2) at ($(out1)+(sd2)$);
            \draw (pc) -- (cp2);
            \draw (cp2) .. controls ($(id1r2)+(sd)$) .. (id1r2);
            \draw (cp2) .. controls ($(N2.south)+0.2*(sd)$) .. (N2.south);
            \draw (N2.north) .. controls ($(N2.north)+(su)$) and ($(id2r2)+(sd)$) .. (id2r2);
            \draw (id1r2) .. controls ($(id1r2)+(su)$) and ($(id2l2)+(sd)$) .. (id2l2);
            \draw (id2l2) .. controls ($(id2l2)+(su)$) .. (p2);
            \draw (id2r2) .. controls ($(id2r2)+(su)$) .. (p2);
            \draw (p2) -- (out2);
            \draw (p1) .. controls ($(p1)+(su)$) .. (pc);
            \draw (p3) .. controls ($(p3)+(su)$) .. (pc);
        \end{tikzpicture}}
\overset{\text{assoc.}}=
       \raisebox{-2cm}{\begin{tikzpicture}
            \coordinate (sd) at (0cm,-0.2cm);
            \coordinate (su) at (0cm,0.2cm);
            \coordinate (in1) at (0cm,0cm);
            \coordinate (id1l1) at (-0.3cm,1cm);
            \coordinate (id1r1) at (0.3cm,1cm);
            \coordinate (id2l1) at (-0.3cm,1.5cm);
            \coordinate (out1) at (0cm,2.5cm);
            \node[name=cp1,circle,fill,inner sep=0.05cm] at (0cm,0.5cm) {};
            \node[name=p1,circle,fill,inner sep=0.05cm] at (0cm,2cm) {};
            \node[name=N1,circle,draw,inner sep=0.05cm] at (-0.3cm,0.9cm) {\tiny{$-\nu_2$}};
            \node[name=id2r1,circle,fill,inner sep=0.05cm] at (0.3cm,1.5cm) {};
            \draw (in1) -- (cp1);
            \draw (cp1) .. controls ($(id1r1)+(sd)$) .. (id1r1);
            \draw (cp1) .. controls ($(N1.south)+0.2*(sd)$) .. (N1.south);
            \draw (N1.north) .. controls ($(N1.north)+(su)$) and ($(id2r1)+(sd)$) .. (id2r1);
            \draw (id1r1) .. controls ($(id1r1)+(su)$) and ($(id2l1)+(sd)$) .. (id2l1);
            \draw (id2l1) .. controls ($(id2l1)+(su)$) .. (p1);
            \draw (id2r1) .. controls ($(id2r1)+(su)$) .. (p1);
            \coordinate (sd3) at (1.5cm,0cm); 
            \coordinate (in3) at ($(in1)+(sd3)$);
            \node[name=cp3,circle,fill,inner sep=0.05cm] at ($(cp1)+(sd3)$) {};
            \coordinate (p3) at ($(p1)+(sd3)$);
            \node[name=N3,circle,draw,inner sep=0.05cm] at ($(id1r1)+(sd3)+(0cm,-0.1cm)$) {\tiny{$-\nu_3$}};
            \coordinate (id1l3) at ($(id1l1)+(sd3)$);
            \coordinate (id1r3) at ($(id1r1)+(sd3)$);
            \coordinate (id2l3) at ($(id2l1)+(sd3)$);
            \coordinate (id2r3) at ($(id2r1)+(sd3)$);
            \coordinate (out3) at ($(out1)+(sd3)$);
            \draw (in3) -- (cp3);
            \draw (cp3) .. controls ($(id1l3)+(sd)$) .. (id1l3);
            \draw (cp3) .. controls ($(N3.south)+0.2*(sd)$) .. (N3.south);
            \draw (N3.north) .. controls ($(N3.north)+(su)$) .. (id2r1);
            \draw (id1l3) .. controls ($(id1l3)+(su)$) and ($(p3)+(sd)$) .. (p3);
            \coordinate (sd2) at (0.75cm,3cm); 
            \coordinate (in2) at ($(in1)+(sd2)$);
            \node[name=pc,circle,fill,inner sep=0.05cm] at (in2) {};
            \node[name=cp2,circle,fill,inner sep=0.05cm] at ($(cp1)+(sd2)$) {};
            \node[name=p2,circle,fill,inner sep=0.05cm] at ($(p1)+(sd2)$) {};
            \node[name=N2,circle,draw,inner sep=0.05cm] at ($(N1)+(sd2)$) {\tiny{$-\nu_1$}};
            \coordinate (id1r2) at ($(id1r1)+(sd2)$);
            \coordinate (id2l2) at ($(id2l1)+(sd2)$);
            \coordinate (id2r2) at ($(id2r1)+(sd2)$);
            \coordinate (out2) at ($(out1)+(sd2)$);
            \draw (pc) -- (cp2);
            \draw (cp2) .. controls ($(id1r2)+(sd)$) .. (id1r2);
            \draw (cp2) .. controls ($(N2.south)+0.2*(sd)$) .. (N2.south);
            \draw (N2.north) .. controls ($(N2.north)+(su)$) and ($(id2r2)+(sd)$) .. (id2r2);
            \draw (id1r2) .. controls ($(id1r2)+(su)$) and ($(id2l2)+(sd)$) .. (id2l2);
            \draw (id2l2) .. controls ($(id2l2)+(su)$) .. (p2);
            \draw (id2r2) .. controls ($(id2r2)+(su)$) .. (p2);
            \draw (p2) -- (out2);
            \draw (p1) .. controls ($(p1)+(su)$) .. (pc);
            \draw (p3) .. controls ($(p3)+(su)$) .. (pc);
        \end{tikzpicture}}
\nonumber\\
&\underset{\text{assoc.}}{\overset{\text{\eqref{eq:inner-mult-past-delta}}}=}
        \raisebox{-2cm}{\begin{tikzpicture}
            \coordinate (sd) at (0cm,-0.2cm);
            \coordinate (su) at (0cm,0.2cm);
            \node[name=pc,circle,fill,inner sep=0.05cm] at (in2) {};
            \node[name=pc2, circle, fill, inner sep=0.05cm] at ($(in2)+(0cm,0.25cm)$){};
            \coordinate (ipc2) at ($(pc2)+(-0.5cm,-0.5cm)$);
            \coordinate (in1) at (0cm,0cm);
            \coordinate (id1l1) at (-0.3cm,1cm);
            \coordinate (id1r1) at (0.3cm,1cm);
            \coordinate (id2l1) at (-0.3cm,1.5cm);
            \coordinate (id2r1) at (0.3cm,1.5cm);
            \coordinate (out1) at (0cm,2.5cm);
            \node[name=cp1,circle,fill,inner sep=0.05cm] at (0cm,0.5cm) {};
            \node[name=p1,circle,fill,inner sep=0.05cm] at (0cm,2cm) {};
            \node[name=N1,circle,draw,inner sep=0.05cm] at (-0.3cm,0.9cm) {\tiny{$-\nu_2$}};
            \draw (in1) -- (cp1);
            \draw (cp1) .. controls ($(id1r1)+(sd)$) .. (id1r1);
            \draw (cp1) .. controls ($(N1.south)+0.2*(sd)$) .. (N1.south);
            \draw (N1.north) .. controls ($(N1.north)+(su)$) and ($(id2r1)+(sd)$) .. (id2r1);
            \draw (id1r1) .. controls ($(id1r1)+(su)$) and ($(id2l1)+(sd)$) .. (id2l1);
            \draw (id2l1) .. controls ($(id2l1)+(su)$) .. (p1);
            \draw (id2r1) .. controls ($(id2r1)+(su)$) .. (p1);
            \coordinate (sd3) at (1.5cm,0cm); 
            \coordinate (in3) at ($(in1)+(sd3)$);
            \node[name=cp3,circle,fill,inner sep=0.05cm] at ($(cp1)+(sd3)$) {};
            \coordinate (p3) at ($(p1)+(sd3)$);
            \node[name=N3,circle,draw,inner sep=0.05cm] at ($(id1r1)+(sd3)+(0.5cm,0.5cm)$) {\tiny{$-\nu_2\nu_3$}};
            \coordinate (id1l3) at ($(id1l1)+(sd3)$);
            \coordinate (id1r3) at ($(id1r1)+(sd3)$);
            \coordinate (id2l3) at ($(id2l1)+(sd3)$);
            \coordinate (id2r3) at ($(id2r1)+(sd3)$);
            \coordinate (out3) at ($(out1)+(sd3)$);
            \draw (in3) -- (cp3);
            \draw (cp3) .. controls ($(id1l3)+(sd)$) .. (id1l3);
            \draw (cp3) .. controls ($(N3.south)+0.2*(sd)$) .. (N3.south);
            \draw (N3.north) .. controls ($(N3.north)+(su)$) and ($(ipc2)+(sd)$) .. ($(ipc2)$);
            \draw (ipc2) .. controls ($(ipc2)+(su)$) .. (pc2);
            \draw (id1l3) .. controls ($(id1l3)+(su)$) and ($(p3)+(sd)$) .. (p3);
            \coordinate (sd2) at (0.75cm,3cm); 
            \coordinate (in2) at ($(in1)+(sd2)$);
            \node[name=cp2,circle,fill,inner sep=0.05cm] at ($(cp1)+(sd2)$) {};
            \node[name=p2,circle,fill,inner sep=0.05cm] at ($(p1)+(sd2)$) {};
            \node[name=N2,circle,draw,inner sep=0.05cm] at ($(N1)+(sd2)$) {\tiny{$-\nu_1$}};
            \coordinate (id1r2) at ($(id1r1)+(sd2)$);
            \coordinate (id2l2) at ($(id2l1)+(sd2)$);
            \coordinate (id2r2) at ($(id2r1)+(sd2)$);
            \coordinate (out2) at ($(out1)+(sd2)$);
            \draw (pc) -- (pc2);
            \draw (pc2) -- (cp2);
            \draw (cp2) .. controls ($(id1r2)+(sd)$) .. (id1r2);
            \draw (cp2) .. controls ($(N2.south)+0.2*(sd)$) .. (N2.south);
            \draw (N2.north) .. controls ($(N2.north)+(su)$) and ($(id2r2)+(sd)$) .. (id2r2);
            \draw (id1r2) .. controls ($(id1r2)+(su)$) and ($(id2l2)+(sd)$) .. (id2l2);
            \draw (id2l2) .. controls ($(id2l2)+(su)$) .. (p2);
            \draw (id2r2) .. controls ($(id2r2)+(su)$) .. (p2);
            \draw (p2) -- (out2);
            \draw (p1) .. controls ($(p1)+(su)$) .. (pc);
            \draw (p3) .. controls ($(p3)+(su)$) .. (pc);
        \end{tikzpicture}}
\underset{\text{assoc.}}{\overset{\text{Frob.}}=}
        \raisebox{-2cm}{\begin{tikzpicture}
            \coordinate (sd) at (0cm,-0.2cm);
            \coordinate (su) at (0cm,0.2cm);
            \node[name=pc,circle,fill,inner sep=0.05cm] at (in2) {};
            \coordinate (in1) at (0cm,0cm);
            \coordinate (id1l1) at (-0.3cm,1cm);
            \coordinate (id1r1) at (0.3cm,1cm);
            \coordinate (id2l1) at (-0.3cm,1.5cm);
            \coordinate (id2r1) at (0.3cm,1.5cm);
            \coordinate (out1) at (0cm,2.5cm);
            \node[name=cp1,circle,fill,inner sep=0.05cm] at (0cm,0.5cm) {};
            \node[name=p1,circle,fill,inner sep=0.05cm] at (0cm,2cm) {};
            \node[name=N1,circle,draw,inner sep=0.05cm] at (-0.3cm,0.9cm) {\tiny{$-\nu_2$}};
            \draw (in1) -- (cp1);
            \draw (cp1) .. controls ($(id1r1)+(sd)$) .. (id1r1);
            \draw (cp1) .. controls ($(N1.south)+0.2*(sd)$) .. (N1.south);
            \draw (N1.north) .. controls ($(N1.north)+(su)$) and ($(id2r1)+(sd)$) .. (id2r1);
            \draw (id1r1) .. controls ($(id1r1)+(su)$) and ($(id2l1)+(sd)$) .. (id2l1);
            \draw (id2l1) .. controls ($(id2l1)+(su)$) .. (p1);
            \draw (id2r1) .. controls ($(id2r1)+(su)$) .. (p1);
            \coordinate (sd3) at (1.5cm,0cm); 
            \coordinate (in3) at ($(in1)+(sd3)$);
            \node[name=cp3,circle,fill,inner sep=0.05cm] at ($(cp1)+(sd3)$) {};
            \node[name=p3,circle,fill,inner sep=0.05cm] at ($(p1)+(sd3)$) {};
            \node[name=N3,circle,draw,inner sep=0.05cm] at ($(id1r1)+(sd3)+(0.2cm,0.2cm)$) {\tiny{$\nu_1\nu_2\nu_3$}};
            \coordinate (id1l3) at ($(id1l1)+(sd3)$);
            \coordinate (id1r3) at ($(id1r1)+(sd3)$);
            \coordinate (id2l3) at ($(id2l1)+(sd3)$);
            \coordinate (id2r3) at ($(id2r1)+(sd3)$);
            \coordinate (out3) at ($(out1)+(sd3)$);
            \draw (in3) -- (cp3);
            \draw (cp3) .. controls ($(id1l3)+(sd)$) .. (id1l3);
            \draw (cp3) .. controls ($(N3.south)+0.2*(sd)$) .. (N3.south);
            \draw (N3.north) .. controls ($(N3.north)+(su)$)  .. (p3);
            \draw (id1l3) .. controls ($(id1l3)+(su)$) .. (p3);
            \coordinate (sd2) at (0.75cm,3cm); 
            \coordinate (in2) at ($(in1)+(sd2)$);
            \node[name=cp2,circle,fill,inner sep=0.05cm] at ($(cp1)+(sd2)$) {};
            \node[name=p2,circle,fill,inner sep=0.05cm] at ($(p1)+(sd2)$) {};
            \node[name=N2,circle,draw,inner sep=0.05cm] at ($(N1)+(sd2)$) {\tiny{$-\nu_1$}};
            \coordinate (id1r2) at ($(id1r1)+(sd2)$);
            \coordinate (id2l2) at ($(id2l1)+(sd2)$);
            \coordinate (id2r2) at ($(id2r1)+(sd2)$);
            \coordinate (out2) at ($(out1)+(sd2)$);
            \draw (pc) -- (cp2);
            \draw (cp2) .. controls ($(id1r2)+(sd)$) .. (id1r2);
            \draw (cp2) .. controls ($(N2.south)+0.2*(sd)$) .. (N2.south);
            \draw (N2.north) .. controls ($(N2.north)+(su)$) and ($(id2r2)+(sd)$) .. (id2r2);
            \draw (id1r2) .. controls ($(id1r2)+(su)$) and ($(id2l2)+(sd)$) .. (id2l2);
            \draw (id2l2) .. controls ($(id2l2)+(su)$) .. (p2);
            \draw (id2r2) .. controls ($(id2r2)+(su)$) .. (p2);
            \draw (p2) -- (out2);
            \draw (p1) .. controls ($(p1)+(su)$) .. (pc);
            \draw (p3) .. controls ($(p3)+(su)$) .. (pc);
        \end{tikzpicture}}
        \nonumber\\
        &=q_{\nu_1} \circ \mu \circ \left( q_{\nu_2} \otimes \id_A \right) \ . \nonumber
    \end{align}
In the last step $\Delta$-separability and the assumption $\nu_1\nu_2\nu_3=1$ are used. The other cases are analogous.
\end{proof}

\begin{lemma}\label{lem:P_NS/R-and-braiding}
Let $\nu \in \{\pm1\}$. We have $\mu \circ \sigma_{A,A} \circ (q_\nu \otimes \id_A) = \mu \circ (q_\nu \otimes N_\nu)$.
\end{lemma}

\begin{proof}
By direct calculation:
\begin{equation}
    \raisebox{-1.5cm}{\begin{tikzpicture}
        \coordinate (su) at (0cm,0.5cm);
        \coordinate (sd) at (0cm,-0.5cm);
        \coordinate (sdl) at (-0.5cm,-0.5cm);
        \coordinate (sdr) at (0.5cm,-0.5cm);
        \coordinate (sul) at (-0.5cm,0.5cm);
        \coordinate (sur) at (0.5cm,0.5cm);
        \node[name=p1, product] at (0cm,2.5cm){};
        \node[name=p2, product] at (0.5cm,3.5cm){};
        \node[name=cp, product] at (0cm,0.5cm){};
        \node[name=N, naka] at (-0.5cm,1.25cm){$-\nu$};
        \coordinate (i1) at (0.5cm,1.25cm);
        \coordinate (i2) at (1cm,2.5cm);
        \coordinate (in1) at (0cm,0cm);
        \coordinate (in2) at (1cm,0cm);
        \coordinate (out) at (0.5cm,4cm);
        \foreach \N/\M/\S/\T in {
            in1/cp/su/sd,
            in2/i2/su/sd,
            i2/p2/su/sdl,
            i1/p1/su/sdl,
            N/p1/su/sdr,
            p1/p2/su/sdr,
            p2/out/su/sd}
            \draw (\N) .. controls ($(\N)+(\S)$) and ($(\M)+(\T)$) .. (\M);
        \foreach \N/\M/\S/\T in {
            cp/N/sul/sd,
            cp/i1/sur/sd}
            \draw (\N) .. controls ($(\N)+0.5*(\S)$) and ($(\M)+0.5*(\T)$) .. (\M);
    \end{tikzpicture}}
    =
    \raisebox{-1.5cm}{\begin{tikzpicture}
        \coordinate (su) at (0cm,0.5cm);
        \coordinate (sd) at (0cm,-0.5cm);
        \coordinate (sdl) at (-0.5cm,-0.5cm);
        \coordinate (sdr) at (0.5cm,-0.5cm);
        \coordinate (sul) at (-0.5cm,0.5cm);
        \coordinate (sur) at (0.5cm,0.5cm);
        \node[name=p1, product] at (0cm,2.5cm){};
        \node[name=p2, product] at (0.5cm,1.25cm){};
        \node[name=cp, product] at (0cm,0.5cm){};
        \node[name=N, naka] at (-0.5cm,1.25cm){$-\nu$};
        \coordinate (in1) at (0cm,-0.5cm);
        \coordinate (in2) at (1cm,-0.5cm);
        \coordinate (out) at (0cm,3.5cm);
        \foreach \N/\M/\S/\T in {
            in1/cp/su/sd,
            in2/p2/su/sdl,
            p2/p1/su/sdl,
            N/p1/su/sdr,
            p1/out/su/sd}
            \draw (\N) .. controls ($(\N)+(\S)$) and ($(\M)+(\T)$) .. (\M);
        \foreach \N/\M/\S/\T in {
            cp/N/sul/sd,
            cp/p2/sur/sdr}
            \draw (\N) .. controls ($(\N)+0.5*(\S)$) and ($(\M)+0.5*(\T)$) .. (\M);
    \end{tikzpicture}}
    \stackrel{\text{\eqref{eq:inner-mult-past-delta}}}{=}
    \raisebox{-1.5cm}{\begin{tikzpicture}
        \coordinate (su) at (0cm,0.5cm);
        \coordinate (sd) at (0cm,-0.5cm);
        \coordinate (sdl) at (-0.5cm,-0.5cm);
        \coordinate (sdr) at (0.5cm,-0.5cm);
        \coordinate (sul) at (-0.5cm,0.5cm);
        \coordinate (sur) at (0.5cm,0.5cm);
        \node[name=p1, product] at (0cm,2.5cm){};
        \node[name=p2, product] at (-0.5cm,0.5cm){};
        \node[name=cp, product] at (0cm,-0.5cm){};
        \node[name=N1, naka] at (-0.5cm,1.25cm){$-\nu$};
        \node[name=N2, naka] at (1cm,-0.5cm){};
        \coordinate (i1) at (0.5cm,1.25cm);
        \coordinate (in1) at (0cm,-1cm);
        \coordinate (in2) at (1cm,-1cm);
        \coordinate (out) at (0cm,3cm);
        \foreach \N/\M/\S/\T in {
            in1/cp/su/sd,
            in2/N2/su/sd,
            N2/p2/su/sdr,
            p2/N1/su/sd,
            N1/p1/su/sdr,
            p1/out/su/sd,
            i1/p1/su/sdl}
            \draw (\N) .. controls ($(\N)+(\S)$) and ($(\M)+(\T)$) .. (\M);
        \foreach \N/\M/\S/\T in {
            cp/i1/sur/sd,
            cp/p2/sul/sdl}
            \draw (\N) .. controls ($(\N)+0.5*(\S)$) and ($(\M)+0.5*(\T)$) .. (\M);
    \end{tikzpicture}}
    =\raisebox{-1.5cm}{\begin{tikzpicture}
        \coordinate (su) at (0cm,0.5cm);
        \coordinate (sd) at (0cm,-0.5cm);
        \coordinate (sdl) at (-0.5cm,-0.5cm);
        \coordinate (sdr) at (0.5cm,-0.5cm);
        \coordinate (sul) at (-0.5cm,0.5cm);
        \coordinate (sur) at (0.5cm,0.5cm);
        \node[name=p1, product] at (0cm,2.5cm){};
        \node[name=p2, product] at (0.5cm,3.5cm){};
        \node[name=cp, product] at (0cm,0.5cm){};
        \node[name=N1, naka] at (-0.5cm,1.25cm){$-\nu$};
        \node[name=N2, naka] at (1cm,2.5cm){$\nu$};
        \coordinate (i1) at (0.5cm,1.25cm);
        \coordinate (in1) at (0cm,0cm);
        \coordinate (in2) at (1cm,0cm);
        \coordinate (out) at (0.5cm,4cm);
        \foreach \N/\M/\S/\T in {
            in1/cp/su/sd,
            in2/N2/su/sd,
            N2/p2/su/sdr,
            i1/p1/su/sdl,
            N1/p1/su/sdr,
            p1/p2/su/sdl,
            p2/out/su/sd}
            \draw (\N) .. controls ($(\N)+(\S)$) and ($(\M)+(\T)$) .. (\M);
        \foreach \N/\M/\S/\T in {
            cp/N1/sul/sd,
            cp/i1/sur/sd}
            \draw (\N) .. controls ($(\N)+0.5*(\S)$) and ($(\M)+0.5*(\T)$) .. (\M);
    \end{tikzpicture}}
\end{equation}
\end{proof}

Recall that in a symmetric monoidal category, a {\em centre of an algebra} $A$ is an object $C$ together with a morphism $c : C \to A$ such that 
\begin{enumerate}
\item $\mu \circ \sigma_{A,A} \circ (c \otimes \id) = \mu \circ (c \otimes \id)$, and
\item the following universal property holds: for all $d : D \to A$ such that $\mu \circ \sigma_{A,A} \circ (d \otimes \id) = \mu \circ (d \otimes \id)$, there is a unique morphism $f : D \to C$ such that $d = c \circ f$. 
\end{enumerate}
It is easy to see that $c$ is necessarily a monomorphism. Since a centre $(C,c)$ is unique up to unique isomorphism, we will speak of ``the'' centre.

\begin{lemma}\label{lem:ZNS-is-centre-of-A}
$Z^{NS}(A)$, together with $\iota^{NS} : Z^{NS}(A) \to A$, is the centre of $A$.
\end{lemma}

\begin{proof}
That $\iota^{NS}$ satisfies the property 1 follows from $\iota^{NS} = p^{NS} \circ \iota^{NS}$ and Lemma \ref{lem:P_NS/R-and-braiding}. To check property 2, let $d : D \to A$ satisfy $\mu \circ \sigma_{A,A} \circ (d \otimes \id) = \mu \circ (d \otimes \id)$. Then
\begin{equation*}
	P^{NS} \circ d =
     \raisebox{-1.5cm}{\begin{tikzpicture}
        \coordinate (su) at (0cm,0.5cm);
        \coordinate (sd) at (0cm,-0.5cm);
        \coordinate (sdl) at (-0.5cm,-0.5cm);
        \coordinate (sdr) at (0.5cm,-0.5cm);
        \coordinate (sul) at (-0.5cm,0.5cm);
        \coordinate (sur) at (0.5cm,0.5cm);
        \node[name=p1, product] at (0cm,2.5cm){};
        \node[name=p2, product] at (0.5cm,1.25cm){};
        \node[name=cp, product] at (0cm,0.5cm){};
        \node[name=N, naka] at (-0.5cm,1.25cm){};
        \coordinate (in2) at (1cm,-0.5cm);
        \coordinate (out) at (0cm,3cm);
        \node[name=d, draw, minimum height=1.4em] at (1cm,0.25cm) {$d$};
        \foreach \N/\M/\S/\T in {
            in2/d.south/su/sd,
            d.north/p2/su/sdr,
            p2/p1/su/sdl,
            N/p1/su/sdr,
            p1/out/su/sd}
            \draw (\N) .. controls ($(\N)+(\S)$) and ($(\M)+(\T)$) .. (\M);
        \foreach \N/\M/\S/\T in {
            cp/N/sul/sd,
            cp/p2/sur/sdl}
            \draw (\N) .. controls ($(\N)+0.5*(\S)$) and ($(\M)+0.5*(\T)$) .. (\M);
    \end{tikzpicture}}
    \stackrel{\text{def.\ of $d$}}{=}
    \raisebox{-1.5cm}{\begin{tikzpicture}
        \coordinate (su) at (0cm,0.5cm);
        \coordinate (sd) at (0cm,-0.5cm);
        \coordinate (sdl) at (-0.5cm,-0.5cm);
        \coordinate (sdr) at (0.5cm,-0.5cm);
        \coordinate (sul) at (-0.5cm,0.5cm);
        \coordinate (sur) at (0.5cm,0.5cm);
        \node[name=p1, product] at (0cm,2.5cm){};
        \node[name=p2, product] at (0.5cm,1.25cm){};
        \node[name=cp, product] at (0cm,0.5cm){};
        \node[name=N, naka] at (-0.5cm,1.25cm){};
        \coordinate (in2) at (1cm,-0.5cm);
        \coordinate (out) at (0cm,3cm);
        \node[name=d, draw, minimum height=1.4em] at (1cm,0.25cm) {$d$};
        \foreach \N/\M/\S/\T in {
            in2/d.south/su/sd,
            d.north/p2/su/sdl,
            p2/p1/su/sdl,
            N/p1/su/sdr,
            p1/out/su/sd}
            \draw (\N) .. controls ($(\N)+(\S)$) and ($(\M)+(\T)$) .. (\M);
        \foreach \N/\M/\S/\T in {
            cp/N/sul/sd,
            cp/p2/sur/sdr}
            \draw (\N) .. controls ($(\N)+0.5*(\S)$) and ($(\M)+0.5*(\T)$) .. (\M);
    \end{tikzpicture}}
    \stackrel{\text{deform}}{=}
    \raisebox{-1.5cm}{\begin{tikzpicture}
        \coordinate (su) at (0cm,0.5cm);
        \coordinate (sd) at (0cm,-0.5cm);
        \coordinate (sdl) at (-0.5cm,-0.5cm);
        \coordinate (sdr) at (0.5cm,-0.5cm);
        \coordinate (sul) at (-0.5cm,0.5cm);
        \coordinate (sur) at (0.5cm,0.5cm);
        \node[name=p1, product] at (0cm,2.5cm){};
        \coordinate (out) at (0cm,3cm);
        \begin{scope}[xscale=-1]
        \node[name=p2, product] at (0.5cm,1.25cm){};
        \node[name=cp, product] at (0cm,0.5cm){};
        \node[name=N, naka] at (-0.5cm,1.25cm){};
        \coordinate (in2) at (1cm,-0.5cm);
        \node[name=d, draw, minimum height=1.4em] at (1cm,0.25cm) {$d$};
        \end{scope}
        \foreach \N/\M/\S/\T in {
            in2/d.south/su/sd,
            d.north/p2/su/sdl,
            p2/p1/su/sdl,
            N/p1/su/sdr,
            p1/out/su/sd}
            \draw (\N) .. controls ($(\N)+(\S)$) and ($(\M)+(\T)$) .. (\M);
        \foreach \N/\M/\S/\T in {
            cp/N/sul/sd,
            cp/p2/sur/sdr}
            \draw (\N) .. controls ($(\N)+0.5*(\S)$) and ($(\M)+0.5*(\T)$) .. (\M);
    \end{tikzpicture}}
    \overset{N=N^{-1}}=
    \raisebox{-1.5cm}{\begin{tikzpicture}
        \coordinate (su) at (0cm,0.5cm);
        \coordinate (sd) at (0cm,-0.5cm);
        \coordinate (sdl) at (-0.5cm,-0.5cm);
        \coordinate (sdr) at (0.5cm,-0.5cm);
        \coordinate (sul) at (-0.5cm,0.5cm);
        \coordinate (sur) at (0.5cm,0.5cm);
        \node[name=p1, product] at (0cm,2.5cm){};
        \coordinate (out) at (0cm,3cm);
        \begin{scope}[xscale=-1]
        \node[name=p2, product] at (0.5cm,1.25cm){};
        \node[name=cp, product] at (0cm,0.5cm){};
        \coordinate (i1) at (-0.5cm,1.25cm);
        \coordinate (in2) at (1cm,-0.5cm);
        \node[name=d, draw, minimum height=1.4em] at (1cm,0.25cm) {$d$};
        \end{scope}
        \foreach \N/\M/\S/\T in {
            in2/d.south/su/sd,
            d.north/p2/su/sdl,
            p2/p1/su/sdl,
            i1/p1/su/sdr,
            p1/out/su/sd}
            \draw (\N) .. controls ($(\N)+(\S)$) and ($(\M)+(\T)$) .. (\M);
        \foreach \N/\M/\S/\T in {
            cp/i1/sur/sd,
            cp/p2/sul/sdr}
            \draw (\N) .. controls ($(\N)+0.5*(\S)$) and ($(\M)+0.5*(\T)$) .. (\M);
    \end{tikzpicture}}
    = d \ .
\end{equation*}
In property 2 we can hence choose $f = \pi^{NS} \circ d$, since $\iota^{NS} \circ f = P^{NS} \circ d = d$. This shows existence. For uniqueness, note that $d = \iota^{NS} \circ f$ and $d = \iota^{NS} \circ f'$ implies $f=f'$ since $\iota^{NS}$ is mono.
\end{proof}

\begin{lemma}
\begin{enumerate}
\item $\mu \circ (q_\nu \otimes \id) \circ \Delta = \mu \circ (q_{-\nu} \otimes \id) \circ \Delta$ , where $\nu \in \{\pm 1\}$ ,
\item $\mu \circ (P^{NS} \otimes P^{NS}) \circ \Delta \circ \eta = \mu \circ (P^R \otimes P^{R}) \circ \Delta \circ \eta$ .
\end{enumerate}
\label{lem:eulerchar-NS-R}
\end{lemma}

\begin{proof}
We prove part 1 by direct calculation:
\begin{equation*}
	\mu \circ (q_\nu \otimes \id) \circ \Delta 
	\underset{\text{coassoc.}}{\overset{\text{assoc.}}=}
    \raisebox{-1.5cm}{\begin{tikzpicture}
        \coordinate (su) at (0cm,0.5cm);
        \coordinate (sd) at (0cm,-0.5cm);
        \coordinate (sdl) at (-0.5cm,-0.5cm);
        \coordinate (sdr) at (0.5cm,-0.5cm);
        \coordinate (sul) at (-0.5cm,0.5cm);
        \coordinate (sur) at (0.5cm,0.5cm);
        \coordinate (in) at (0cm,0cm);
        \node[name=cp1, product] at (0cm,0.5cm){};
        \node[name=cp2, product] at (0cm,1cm){};
        \node[name=N, naka] at (-0.8cm,1.1cm){$-\nu$};
        \node[name=p1, product] at (0.5cm,3cm){};
        \node[name=p2, product] at (0cm,3.5cm){};
        \coordinate (out) at (0cm,4cm);
        \coordinate (i1) at (-0.5cm,1.5cm);
        \coordinate (i2) at (-0.2cm,2.25cm);
        \coordinate (i3) at (-0.5cm,3cm);
        \foreach \N/\M/\S/\T in {
            in/cp1/su/sd,
            N.north/p1/su/sdl,
            cp1/cp2/su/sd,
            cp2/p1/sur/sdr,
            p1/p2/sul/sdr,
            p2/out/su/sd}
            \draw (\N) .. controls ($(\N)+(\S)$) and ($(\M)+(\T)$) .. (\M);
        \foreach \N/\M/\S/\T in {
            cp2/i1/sul/sd,
            i1/i2/su/sd,
            i2/i3/su/sd,
            i3/p2/su/sdl,
            cp1/N.south/sul/sd}
            \draw (\N) .. controls ($(\N)+0.5*(\S)$) and ($(\M)+0.5*(\T)$) .. (\M);
    \end{tikzpicture}}
	\overset{\text{\eqref{eq:inner-mult-past-delta}}}= 
    \raisebox{-1.5cm}{\begin{tikzpicture}
        \coordinate (su) at (0cm,0.5cm);
        \coordinate (sd) at (0cm,-0.5cm);
        \coordinate (sdl) at (-0.5cm,-0.5cm);
        \coordinate (sdr) at (0.5cm,-0.5cm);
        \coordinate (sul) at (-0.5cm,0.5cm);
        \coordinate (sur) at (0.5cm,0.5cm);
        \coordinate (in) at (0cm,0cm);
        \node[name=cp1, product] at (0cm,0.5cm){};
        \node[name=cp2, product] at (0cm,1.5cm){};
        \node[name=N1, naka] at (-0.8cm,1.1cm){$-\nu$};
        \node[name=N2, naka] at (-0.8cm,1.5cm){};
        \node[name=p1, product] at (-0.5cm,2.5cm){};
        \node[name=p2, product] at (0cm,3.5cm){};
        \coordinate (out) at (0cm,4cm);
        \foreach \N/\M/\S/\T in {
            in/cp1/su/sd,
            N2.north/p1/su/sdr,
            cp1/cp2/su/sd,
            cp2/p2/sur/sdr,
            cp2/p1/sul/sdl,
            p1/p2/su/sdl,
            p2/out/su/sd}
            \draw (\N) .. controls ($(\N)+(\S)$) and ($(\M)+(\T)$) .. (\M);
        \foreach \N/\M/\S/\T in {
            N1.north/N2.south/su/sd,
            cp1/N1.south/sul/sd}
            \draw (\N) .. controls ($(\N)+0.5*(\S)$) and ($(\M)+0.5*(\T)$) .. (\M);
    \end{tikzpicture}}
	\underset{\text{coassoc.}}= \mu \circ (q_{-\nu} \otimes \id) \circ \Delta \ .
\end{equation*}
For part 2 first note that 
\begin{equation} \label{eq:eulerchar-NS-R-aux1}
(q_\nu \otimes q_\nu) \circ \Delta \circ \eta = (q_\nu \otimes \id) \circ \Delta \circ \eta \ .
\end{equation}
This follows from Lemma \ref{lem:forget-one-projector}: replace $\eta = q_+ \circ \eta$ and omit one of the $q_\nu$ idempotents. Using this, we compute
\begin{align}
&\mu \circ (q_+ \otimes q_+) \circ \Delta \circ \eta
\overset{\text{\eqref{eq:eulerchar-NS-R-aux1}}}=
\mu \circ (q_+ \otimes \id) \circ \Delta \circ \eta
\\
&\overset{\text{part 1}}=
\mu \circ (q_- \otimes \id) \circ \Delta \circ \eta
\overset{\text{\eqref{eq:eulerchar-NS-R-aux1}}}=
\mu \circ (q_- \otimes q_-) \circ \Delta \circ \eta \ . \nonumber
\end{align}
\end{proof}

Let us drop the $A$ from the state spaces for brevity: $Z = Z^{NS} \oplus Z^R$. It will be convenient to define $Z_+ := Z^{NS}$ and $Z_- := Z^{R}$, so that $Z_\nu$ is the image of the idempotent $q_\nu$. Define the embedding and projection maps
\begin{equation}
\iota^Z_\nu : Z_\nu \to Z ~,~~
\pi^Z_\nu : Z \to Z_\nu \quad , \qquad
\iota^A_\nu : Z_\nu \to A~,~~
\pi^A_\nu : A \to Z_\nu \ .
\end{equation}
Comparing to the notation in Section \ref{sec:cylinder-projections} we have $\pi^{NS} = \pi^A_+$ and $\pi^{R} = \pi^A_-$, and analogously for $\iota^{NS/R}$ and $\iota^A_\pm$. We abbreviate, for $\nu \in \{\pm1\}$,
\begin{equation}
e_\nu = \big[ Z \xrightarrow{\pi^Z_\nu} Z_\nu  \xrightarrow{\iota^A_\nu} A \big] \quad , \qquad
f_\nu = \big[  A  \xrightarrow{\pi^A_\nu} Z_\nu \xrightarrow{\iota^Z_\nu} Z  \big] \ .
\end{equation}
One quickly checks that $e_\nu \circ f_\nu = q_\nu$ and that $\sum_{\nu =\pm1} f_\nu \circ e_\nu = \id_Z$.
Finally, define the morphisms
\begin{align} \label{eq:Z-Frobalg}
\mu_Z &=  \!\! \sum_{\alpha,\beta = \pm 1} \!\! \big[ 
	Z \otimes Z \xrightarrow{e_\alpha \otimes e_\beta} 
	A \otimes A \xrightarrow{\mu} A
	\xrightarrow{f_{\alpha\beta}} Z \big] ~,
&
\eta_Z &= \big[ \mathbf{1} \xrightarrow{\eta} A \xrightarrow{f_+} Z \big]
\\
\Delta_Z &= \!\! \sum_{\alpha,\beta = \pm 1} \!\! \big[ 
	Z 
	\xrightarrow{e_{\alpha\beta}} 
	A
	\xrightarrow{\Delta}
	A \otimes A 
	\xrightarrow{f_\alpha \otimes f_\beta} 
	Z \otimes Z 
	\big]~,
&
\varepsilon_Z &= \big[ Z \xrightarrow{e_+} A \xrightarrow{\varepsilon} \mathbf{1} \big]
\nonumber 
\end{align}
Recall the involution $N|_{Z^{NS/R}}$ from Remark \ref{rem:N-involution-on-Z}. We obtain an involution $N^Z$ on $Z$, which in the present notation reads
\begin{equation} \label{eq:NZ-def-via-e-f}
	N^Z = \sum_{\nu = \pm} \big[ Z \xrightarrow{e_\nu} A \xrightarrow{N} A \xrightarrow{f_\nu} Z \big] \ .
\end{equation}
As for the Nakayama automorphism, we set $N^Z_{+1} = \id_Z$ and $N^Z_{-1} = N^Z$.

\begin{proposition}
\begin{enumerate}
\item
The morphisms \eqref{eq:Z-Frobalg} define the structure of a $\mathbb{Z}_2$-graded Frobenius algebra on $Z$ with graded components $Z_+$ and $Z_-$ (here, $\mathbb{Z}_2$ is written multiplicatively). 
\item
$N^Z$ is the Nakayama automorphism of $Z$. It satisfies $(N^Z)^2=\id_Z$ and $N^Z \circ \iota^Z_+ = \iota^Z_+$.
\item
The product of $Z$ is graded commutative in the sense that
\begin{align}\label{eq:Z-is-graded-comm}
	\mu_Z \circ \sigma_{Z,Z} \circ (\iota^Z_\alpha \otimes \iota^Z_\beta) 
	&= \mu_Z \circ (N^Z_{\gamma} \otimes \id_Z) \circ  (\iota^Z_\alpha \otimes \iota^Z_\beta) \\
	&= \mu_Z \circ (\id_Z \otimes N^Z_{\gamma}) \circ  (\iota^Z_\alpha \otimes \iota^Z_\beta) \ ,
	\nonumber
\end{align}
where $\alpha,\beta \in \{\pm1\}$, and $\gamma=-1$ if $\alpha=\beta=-1$ and $\gamma=1$ else. 
\end{enumerate}
\label{prop:Z-is-graded-Frobalg}
\end{proposition}

\begin{proof}
{\em Part 1:}
That $Z$ is a Frobenius algebra follows in a straightforward way from Lemma \ref{lem:forget-one-projector}. We illustrate this for the associativity relation:
\begin{align*}
\mu_Z \circ (\id_Z \otimes \mu_Z)
&\stackrel{\text{(1)}}= 
	\sum_{\alpha,\beta,\gamma \in \pm1} f_{\alpha\beta\gamma} \circ \mu \circ 
	(id_A \otimes q_{\beta\gamma}) \circ (\id_A \otimes \mu) \circ (e_\alpha \otimes e_\beta \otimes e_\gamma)
\\
&\stackrel{\text{(2)}}= 
	\sum_{\alpha,\beta,\gamma \in \pm1} f_{\alpha\beta\gamma} \circ \mu \circ 
	(\id_A \otimes \mu) \circ (e_\alpha \otimes e_\beta \otimes e_\gamma)
\\
&\stackrel{\text{(3)}}= 
	\sum_{\alpha,\beta,\gamma \in \pm1} f_{\alpha\beta\gamma} \circ \mu \circ 
	(\mu \otimes \id_A) \circ (e_\alpha \otimes e_\beta \otimes e_\gamma) 
\\
&\stackrel{\text{(4)}}= \mu_Z \circ (\mu_Z \otimes \id_Z) \ . 
\end{align*}
In step 1 we substituted the definition of $\mu_Z$ and used $e_{\beta\gamma} \circ f_{\beta\gamma} = q_{\beta\gamma}$. For step 2 note that $f_\nu = f_\nu \circ q_\nu$ and $e_\nu = q_\nu \circ e_\nu$. Therefore, each of the two products can be surrounded by three idempotents $q$, and by Lemma \ref{lem:forget-one-projector} we can omit one of these. We choose to omit $q_{\beta\gamma}$. Step 3 is just associativity of $\mu$, and in step 4 one carries out steps 1 and 2 backwards.

That $Z$ is graded is clear from the definition of the structure maps in \eqref{eq:Z-Frobalg}.

\medskip\noindent
{\em Part 2:}
That $N^Z$ is the Nakayama automorphism of $Z$ follows from substituting \eqref{eq:Z-Frobalg} into the definition \eqref{eq:Nak-definition-pic} of the Nakayama automorphism and using Lemma \ref{lem:forget-one-projector}. That $N^Z$ is an involution is immediate from  \eqref{eq:NZ-def-via-e-f}, together with Lemma \ref{lem:PNS/R-idemp}\,(2).
Furthermore, from Remark \ref{rem:N-involution-on-Z} one concludes that  $N^Z \circ \iota^Z_+ = \iota^Z_+$.

\medskip\noindent
{\em Part 3:}
To establish graded commutativity, first compute
\begin{align}
	\mu_Z \circ \sigma_{Z,Z} \circ (\iota^Z_\alpha \otimes \iota^Z_\beta) 
        &\stackrel{\text{def.\,$\mu_Z$}}= 
	f_{\alpha\beta} \circ \mu \circ \sigma_{A,A} \circ (\iota^A_\alpha \otimes \iota^A_\beta) 
\\
        &\stackrel{\iota^A_\alpha = q_\alpha \iota^A_\alpha}= 
	f_{\alpha\beta} \circ \mu \circ \sigma_{A,A} \circ (q_\alpha \otimes q_\beta) \circ (\iota^A_\alpha \otimes \iota^A_\beta) 
\nonumber\\
        &\stackrel{\text{Lem.\,\ref{lem:P_NS/R-and-braiding}}}= 
	f_{\alpha\beta} \circ \mu \circ (\id_A \otimes N_\alpha) \circ (\iota^A_\alpha \otimes \iota^A_\beta) 
\nonumber\\
        &\stackrel{\text{def.\,$\mu_Z$,$N^Z$}}= 
	\mu_Z \circ (\id_Z \otimes N^Z_\alpha) \circ (\iota^Z_\alpha \otimes \iota^Z_\beta) 
\nonumber\\
        &\stackrel{\text{$N^Z$ autom.}}= 
	N^Z_\alpha \circ \mu_Z \circ (N^Z_\alpha \otimes \id_Z) \circ (\iota^Z_\alpha \otimes \iota^Z_\beta) \ .
\nonumber
\end{align}
Now use the last two lines to compare to \eqref{eq:Z-is-graded-comm} in all four cases $\alpha,\beta = \pm$. For example, consider the last line with $\alpha = \beta = -1$: The multiplication $\mu_Z$ has image in $Z_{\alpha\beta} = Z_{+}$. By part 2, $N^Z$ is the identity on $Z_{+}$, and so the last line reads $\mu_Z \circ (N^Z \otimes \id_Z) \circ (\iota^Z_{-} \otimes \iota^Z_{-})$, as required.
\end{proof}

\medskip

We still need to relate the product defined in \eqref{eq:Z-Frobalg} to the amplitude the spin TFT assigns to three-holed spheres in \eqref{eq:3-holed-sphere-morphs}. This is done in the following proposition whose proof is immediate from Lemma \ref{lem:pair-of-pants-TFT-value}.

\begin{proposition}
We have
\begin{align}
        &T_A(\Sigma^{0,3}_{\delta_1,\delta_2,\delta_3,\varepsilon_1,\varepsilon_2}) \circ
	\bigotimes_{i=1}^3  \big( \iota^{13} \circ \iota^{\delta_i}  \big)
	\\
	&= \varepsilon_Z \circ
	\mu_Z \circ (\mu_Z \otimes \id_Z) \circ (N^Z_{\varepsilon_1} \otimes N^Z_{\varepsilon_2} \otimes \id_Z) \circ
	(\iota^{\delta_1} \otimes\iota^{\delta_2} \otimes\iota^{\delta_3}) \ .
	\nonumber
\end{align}
\end{proposition}

\begin{remark} \label{rem:MS-euler-char-matches}
Let us compare the properties of the total state space $Z$ to the algebraic description of spin TFTs given by Moore and Segal in \cite[Sect.\,2.6]{moore2006d}. They consider $\mathcal{S} = \mathbf{SVect}$ and require $N^Z$ to be the parity involution. As discussed in the introduction, by Remark \ref{rem:N-involution-on-Z}, $Z_+$ must then be purely even and graded commutativity \eqref{eq:Z-is-graded-comm} translates into ordinary commutativity of $Z$, considered as an algebra in $\mathbf{Vect}$: $\mu_Z \circ c_{Z,Z}^{\mathbf{Vect}} = \mu_Z$. In addition, Moore and Segal require that the NS- and R-type ``Euler characters'' agree, $\chi_{ns} = \chi_r$. In our notation, this can be expressed as follows. Write $p_\nu = \iota^Z_\nu \circ \pi^Z_\nu : Z \to Z$ for the projection onto the subobject $Z_\nu$ of $Z$. Then
\begin{align}
\chi_{ns} &= \mu^Z \circ (p_+ \otimes p_+) \circ \Delta^Z \circ \eta^Z \ ,\\
\chi_{r} &= \mu^Z \circ (p_- \otimes p_-) \circ \Delta^Z \circ \eta^Z \ . \nonumber
\end{align}
Substituting the definitions, we can rewrite $\chi_{ns} = \mu \circ (P^{NS} \otimes P^{NS}) \circ \Delta \circ \eta$ and $\chi_r = \mu \circ (P^R \otimes P^{R}) \circ \Delta \circ \eta$. These two are indeed equal by Lemma \ref{lem:eulerchar-NS-R}\,(2).
\end{remark}

\subsection{Value of the TFT on spin tori}\label{sec:TFT-spin-torus}

Let $\delta \in \{ NS,R\}$ and consider the cylinder $C^\varepsilon_\delta$ as defined in Section \ref{sec:TFT.cylinder}. Define the spin torus $T^\varepsilon_\delta$ by gluing together the two boundary components
\begin{equation}\label{eq:Teps-delta-def}
	T^\varepsilon_\delta := (C^\varepsilon_\delta)_{1 \# 2}^- \ ,
\end{equation}
where the marking and edge signs of the glued spin triangulated surface are determined via Lemma \ref{lem:index.gluing}. The choice of ``$-$'' over ``$+$'' in the above gluing is a convention; it can be absorbed into the value of $\varepsilon$. 

As it stands, the above gluing is actually ill-defined as it does not produce a spin triangulated surface. We should instead first subdivide the triangulation before gluing the boundaries. However, this does not make a difference to the value of the TFT, and we take the liberty to work with \eqref{eq:Teps-delta-def}, even though it is not a spin triangulated surface. 

The cell decomposition of the surface $\underline T^\varepsilon_\delta$ is as in Figure \ref{fig:cylinder.triangulation}, except that the left vertical circle (consisting of edges $e_1,e_2,e_3$) is replaced by the right one (consisting of edges $e_4,e_5,e_6$). The new edge signs are
\begin{equation}
	s(e_4) = - s_3 s_4 = - \varepsilon~,~~
	s(e_5) = - s_2 s_5 = - \varepsilon~,~~
	s(e_6) = - s_1 s_6 = - \varepsilon \ .
\end{equation}
The remaining edge signs are as in Section \ref{sec:TFT.cylinder}. Let $\nu=+1$ if $\delta = NS$ and $\nu=-1$  if $\delta = R$. Then
\begin{equation}
	s(e_7) = -\nu~~,~~~
	s(e_8) = 
	s(e_9) = 
	s(e_{10}) = 
	s(e_{11}) = 
	s(e_{12}) = 1\ .
\end{equation}

As an application of Lemma \ref{lem:path_lifting_inner_generic} we can now investigate the lifting properties of the two simple closed curves $\gamma_h$, which we define to run horizontally through $\sigma_3$ and $\sigma_4$, and $\gamma_v$, which we define to run vertically through all triangles $\sigma_1,\dots,\sigma_6$.

\medskip\noindent 
{\em For $\gamma_h$:} The product of the edge signs is $s(e_{10})s(e_5)=-\varepsilon$, all $\mu_{i,i+1}$ are $1$,  $k_3=2$, $k_4=0$, $\eta_3=-1$, $\eta_4=+1$. The left hand side of \eqref{eq:lemma_path_lifting_inner_generic} reads
$-\varepsilon e^{\pi i(0+1)} = \varepsilon$. Thus $\gamma_h$ has a closed lift iff $\varepsilon=1$. 

\medskip\noindent 
{\em For $\gamma_v$:} The product of the edge signs is $-\nu$, all $\mu_{i,i+1}$ are $1$,  $k_1,\dots,k_6=0$, $\eta_1=\eta_3=\eta_5=1$, $\eta_2=\eta_4=\eta_6=-1$. Thus the left hand side of \eqref{eq:lemma_path_lifting_inner_generic} reads
$-\nu e^{\pi i(-3+3)} = - \nu$. As we already knew from the type of the boundary components, for $\delta = NS$ we get $-1$, so that the curve does not have a closed lift, and for $\delta=R$ we get $+1$, so that the curve does have a closed lift.

\medskip

Let us now compute $T_A(T^\varepsilon_\delta)$. According to Proposition \ref{prop:T-gluing}, $T_A(T^\varepsilon_\delta) = T_A(C^\varepsilon_\delta) \circ \Gamma_{1,2,-}$. Combining \eqref{eq:TFT-cylinder-values} and \eqref{eq:pi31-through-Gamma}, we find
\begin{align} \label{eq:spin-torus-value}
T_A(T^\varepsilon_\delta) 
&= b \circ ( P^{\delta}\otimes \id_A) \circ (N_{-\varepsilon} \otimes \id_A) \circ (\pi^{31}\otimes \pi^{31}) \circ \Gamma_{1,2,-}
\\
&= \varepsilon \circ \mu \circ \big[( P^{\delta} \circ N_{-\varepsilon}) \otimes \id_A \big]  \circ \Delta \circ \eta \ .
\nonumber
\end{align}
Since by Lemma \ref{lem:PNS/R-idemp}, $P^{NS} \circ N = P^{NS}$, we see $T_A(T^+_{NS}) = T_A(T^-_{NS})$. From Lemma \ref{lem:eulerchar-NS-R} we learn $T_A(T^-_{NS}) = T_A(T^-_{R})$. 
Thus, on a spin torus, $T_A$ can take at most two different values,
\begin{equation}
	T_A(T^+_{NS}) = T_A(T^-_{NS}) = T_A(T^-_{R})
	\qquad \text{and} \qquad T_A(T^+_{R}) \ ,
\end{equation}
in agreement with the action of spin diffeomorphisms.

\subsection{Admissible edge signs and the condition $N \ast \id=0$}\label{sec:N*id=0}

Suppose we are given a marked triangulated surface $\underline\Lambda = ((\mathcal{C},f_i,d^1_0,d^2_0), (\varphi, \chi_\sigma) ,(\Sigma,\varphi_i))$ together with a choice of edge signs $s$. For this data we can define a morphism $T_A'(\underline\Lambda;s)$ in the same way as $T_A$. Of course, $T_A'$ will no longer be independent of the choice of triangulation as this relies on the assumption that the edge signs come from a spin structure.

Recall the definition of admissible edge signs from Section \ref{sse:spin_reconstruct}. In terms of punctured spin triangulated surface from Section \ref{sec:lifting-prop}, the admissibility condition is satisfied iff the spin structure extends to the vertices. The resulting spin structure was denoted by $S(\underline\Lambda,s)$ in \eqref{eq:S(Sig,s)-def}.
In case the spin structure extends, by definition and by Theorem \ref{thm:TA-triang-indep}, we have $T_A'(\underline\Lambda;s) = T_A(S(\underline\Lambda,s))$. 

It turns out that there is a simple property of the algebra $A$ which implies that $T_A'(\underline\Lambda;s)$ is zero unless the edge signs are admissible. 
	(Recall that by Assumption 4, $\mathcal{S}$ is additive, so that it makes sense to say that a morphism in $\mathcal{S}$ is zero.)
We proceed to describe this in more detail.

A Frobenius algebra $A$ allows to define the {\em convolution product} on the set $\mathrm{End}(A)$. Namely, for $f,g \in \mathrm{End}(A)$ we set $f \ast g = \mu \circ (f\otimes g) \circ \Delta$. This defines an associative product with unit $\eta \circ \varepsilon$. Clearly, $A$ is $\Delta$-separable iff $\id_A \ast \id_A = \id_A$. Furthermore, if $A$ is $\Delta$-separable, every automorphism $\alpha$ of the algebra (or coalgebra) $A$ satisfies $\alpha \ast \alpha=  \alpha$.

\begin{lemma}\label{lem:N*id-NS-R-zero1}
Let $A$ be a $\Delta$-separable Frobenius algebra such that $N \circ N=\id_A$ and $N \ast \id=0$. Then $P^R \circ P^{NS} = 0 = P^{NS} \circ P^R$.
\end{lemma}

\begin{proof}
Analogous to the proof of part 1 of Lemma \ref{lem:PNS/R-idemp}: instead of \eqref{eq:PNS/R-idemp-aux1} one computes $q_\varepsilon \circ q_\nu = q_\varepsilon \circ \mu \circ ((N_{-\varepsilon} \circ N_{-\nu}) \otimes \id_A) \circ \Delta$. Due to the assumption $N \ast \id=0$, this is zero unless $\varepsilon = \nu$.
\end{proof}

Note that if the Nakayama automorphism $N$ is an involution, as required in the above lemma, the conditions $N \ast \id = 0$ and $\id \ast N=0$ are equivalent since, for example, the first condition implies $0 = N \circ (N\ast\id) = N^2 \ast N = \id \ast N$.

\medskip
Suppose now that $A$ is a Frobenius algebra such that $N \circ N=\id_A$.
Then $\pi_+ = \frac12(\id_A + N)$ is an idempotent, namely the projection onto the eigenspace for eigenvalue $+1$ of $N$. Let $A_+$ be its image (assuming $\pi_+$ splits). We observe the following implication of the $N\ast \id=0$ condition:

\begin{lemma}\label{lem:N*id-NS-R-zero2}
Let $A$ be a Frobenius algebra such that $N \circ N=\id_A$ and let $A_+$ be the image of $\pi_+$.
The structure maps of $A$ turn $A_+$ into a Frobenius algebra. If in addition $A$ is $\Delta$-separable and $N \ast \id=0$, then $A_+$ satisfies $\mu_+ \circ \Delta_+ = \frac12 \,\id_{A_+}$.
\end{lemma}

\begin{proof}
First check the two-out-of-three result $\pi_+ \circ \mu \circ (\pi_+ \otimes \pi_+) = \mu \circ (\pi_+ \otimes \pi_+) = \pi_+ \circ \mu \circ (\id \otimes \pi_+) = \pi_+ \circ \mu \circ (\pi_+ \otimes \id)$. This implies that $A_+$ is a Frobenius algebra by a reasoning analogous to Proposition \ref{prop:Z-is-graded-Frobalg}. If $A$ is $\Delta$-separable and $N \ast \id=0$ holds, we have
\begin{align}
\mu \circ (\pi_+ \otimes \pi_+) \circ \Delta &= \tfrac14\big( \mu \circ (\id \otimes \id) \circ \Delta + \mu \circ (N \otimes N) \circ \Delta \big) 
\\
&= \tfrac14\big( \id_A + N \big) \circ \mu \circ \Delta \ .
\nonumber
\end{align}

\end{proof}

Next we show how the $N\ast\id=0$ condition imposes admissibility of edge signs.

\begin{proposition}\label{prop:N*id=0-admissible}
Let $A$ be a $\Delta$-separable Frobenius algebra $A$ such that $N \circ N=\id_A$ and $N \ast \id=0$. Let $\underline\Lambda$ be a marked triangulated surface together with some choice edge signs $s$. Let $B$, $\iota^{13}$, $\delta_i$ be as in Proposition \ref{prop:state-space-projectors-omit}. If the edge signs are not admissible, then
\begin{equation}
T_A'(\underline\Lambda;s) \circ \bigotimes_{i=1}^B  \big( \iota^{13} \circ P^{\delta_i} \big) = 0
  \ .
\end{equation}  
\end{proposition}

\begin{proof} If the edge signs are not admissible, there is a vertex $v$ around which the extendibility conditions in Corollary \ref{lem:vertex.rule.inner} and Lemma \ref{lem:vertex.rule.boundary} are not satisfied. 

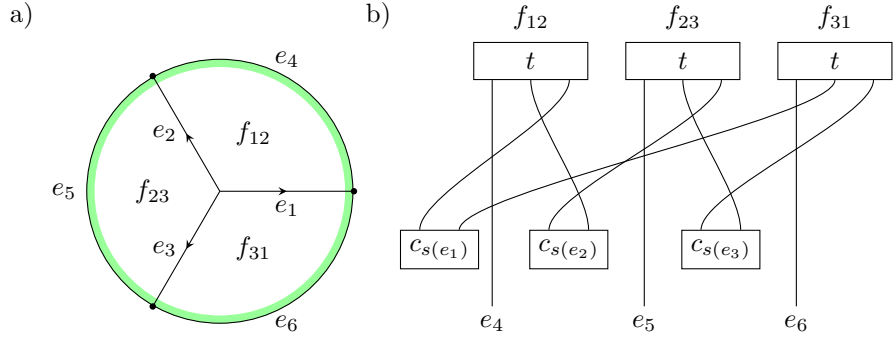
\begin{figure}[tb]
    \centering
    \raisebox{12em}{a)}
    \begin{tikzpicture}
        \node[circle, green!40, line width=3pt, draw, minimum size=3.4cm] at (0cm,0cm){};
        \node[circle, draw, minimum size=3.5cm] at (0cm,0cm){};
        \node[name=D1, regular polygon, regular polygon sides=3, minimum size=3.5cm, rotate=30, dotted, red] at (0cm,0cm){};
        \node[name=D2, regular polygon, regular polygon sides=3, minimum size=3.5cm, rotate=-30, dotted, red] at (0cm,0cm){};
        \foreach \N in {1,2,3} \node[circle, fill, inner sep=0.03cm] at (D1.corner \N){};
        \foreach \N/\M/\K in {1/4/above,2/5/left,3/6/below} \node[\K] at (D2.corner \N){$e_{\M}$};
        \begin{scope}[decoration={markings, mark=at position 0.5 with {\arrow{stealth}}}]
            \foreach \N in {1,2,3} \draw[postaction=decorate] (D1.center) -- (D1.corner \N);
        \end{scope}
        \foreach \N/\M in {1/23,2/31,3/12} \node at (D1.side \N) {$f_{\M}$};
        \foreach \N/\M/\K in {3/1/below,1/2/left,2/3/left} \node[\K] at (D2.side \N) {$e_{\M}$};
    \end{tikzpicture}
    \raisebox{12em}{b)}
    \begin{tikzpicture}
        \coordinate (su) at (0cm,0.5cm);
        \coordinate (sd) at (0cm,-0.5cm);
        \node[name=c1, draw, minimum width=0.8cm, minimum height=1.4em] at (0.5cm,-2.5cm){$c_{s(e_2)}$};
        \node[name=c2, draw, minimum width=0.8cm, minimum height=1.4em] at (-1.2cm,-2.5cm){$c_{s(e_1)}$};
        \node[name=c3, draw, minimum width=0.8cm, minimum height=1.4em] at (2.5cm,-2.5cm){$c_{s(e_3)}$};
        \foreach \N in {1,2,3} {
            \coordinate (c\N1) at ($(c\N.north west)!0.25!(c\N.north east)$);
            \coordinate (c\N2) at ($(c\N.north west)!0.75!(c\N.north east)$);
        }
        \node[name=t1, draw, minimum width=1.5cm, minimum height=1.4em] at (0cm,0cm){$t$};
        \node[name=t2, draw, minimum width=1.5cm, minimum height=1.4em] at (2cm,0cm){$t$};
        \node[name=t3, draw, minimum width=1.5cm, minimum height=1.4em] at (4cm,0cm){$t$};
        \foreach \N in {1,2,3} {
            \coordinate (t\N1) at ($(t\N.south west)!0.33!(t\N.south)$);
            \coordinate (t\N2) at (t\N.south);
            \coordinate (t\N3) at ($(t\N.south east)!0.33!(t\N.south)$);
        }
        \foreach \N in {1,2,3} \coordinate (in\N) at ($(t\N1)+(0cm,-3cm)$);
        \foreach \N/\M in {1/12,2/23,3/31} \node[above=0.3cm] at (t\N) {$f_{\M}$};
        \foreach \N/\M in {1/4,2/5,3/6} \node[below] at (in\N) {$e_{\M}$};
        \foreach \N/\M in {
            in1/t11,
            in2/t21,
            in3/t31,
            c12/t12,
            c11/t23,
            c22/t32,
            c21/t13,
            c32/t22,
            c31/t33}
             \draw (\N) .. controls ($(\N)+(su)$) and ($(\M)+(sd)$) .. (\M);
    \end{tikzpicture}
    \caption{a) Marking of triangles sharing the inner vertex $v$ as used in evaluating the admissibility condition for the edge signs. b) Corresponding part of the diagram computed from the triangulation.}
    \label{fig:check-sign-rule-inner}
\end{figure}

\smallskip\noindent
{\em $v$ is inner:} 
	The edge signs define a spin structure on $\Sigma$ minus its vertices (see Section \ref{sse:spin_reconstruct}). This spin structure is left invariant by Pachner 2-2 moves. We may thus use these moves to achieve that $v$ is shared by exactly three triangles. Using moves on the marking, we may further assume that all edges point away from $v$ and that the marked edge of each triangle does not touch $v$, see Figure \ref{fig:check-sign-rule-inner}. The sign rule from Corollary \ref{lem:vertex.rule.inner} now gives $s(e_1) s(e_2) s(e_3) = (-1)^{0+3+1}=1$, where $e_i$, $i=1,2,3$ are the three edges containing $v$. Figure \ref{fig:check-sign-rule-inner} also shows the corresponding part of the diagram computed from this marked triangulation with edge signs. It is straightforward to check that the morphism obtained from the diagram can be rewritten as 
	(abbreviate $s_i := s(e_i)$)
\begin{equation}
    \raisebox{-0.75cm}{\begin{tikzpicture}
        \coordinate (su) at (0cm,0.5cm);
        \coordinate (sd) at (0cm,-0.5cm);
        \coordinate (sdl) at (-0.5cm,-0.5cm);
        \coordinate (sdr) at (0.5cm,-0.5cm);
        \coordinate (sul) at (-0.5cm,0.5cm);
        \coordinate (sur) at (0.5cm,0.5cm);
        \foreach \N in {1,2,3} \coordinate (out\N) at ($(-2cm,1cm)+\N*(2cm,0cm)$);
        \foreach \N in {1,2,3} \node[name=p\N, product] at ($(-2cm,0.5cm)+\N*(2cm,0cm)$){};
        \foreach \N in {1,2,3} \node[name=c\N, product] at ($(-2cm,-3cm)+\N*(1.5cm,0cm)$){};
        \foreach \N/\M in {1/2,2/3,3/1} \node[name=N\N, naka] at ($(-3cm,-1cm)+\N*(2cm,0cm)$){$-s_\M$};
        \foreach \N in {1,2,3} \coordinate (i\N) at ($(-2cm,-1cm)+\N*(2cm,0cm)$);
        \node[name=N4, naka] at ($(N1)+(0cm,-1cm)$){};
        \foreach \N/\M/\S/\T in {
            c1/N4.south/sul/sd,
            N4.north/N1.south/su/sd,
            N1.north/p1/su/sdl,
            c1/i2/sur/sd,
            i2/p2/su/sdr,
            c2/i1/sul/sd,
            i1/p1/su/sdr,
            c2/N3.south/sur/sd,
            N3.north/p3/su/sdl,
            c3/N2.south/sur/sd,
            N2.north/p2/su/sdl,
            c3/i3/sul/sd,
            i3/p3/su/sdr,
            p1/out1/su/sd,
            p2/out2/su/sd,
            p3/out3/su/sd}
            \draw (\N) .. controls ($(\N)+(\S)$) and ($(\M)+(\T)$) .. (\M);
    \end{tikzpicture}}
    =
    \raisebox{-1.5cm}{\begin{tikzpicture}
        \coordinate (su) at (0cm,0.5cm);
        \coordinate (sd) at (0cm,-0.5cm);
        \coordinate (sdl) at (-0.5cm,-0.5cm);
        \coordinate (sdr) at (0.5cm,-0.5cm);
        \coordinate (sul) at (-0.5cm,0.5cm);
        \coordinate (sur) at (0.5cm,0.5cm);
        \foreach \N in {1,2,3} \coordinate (out\N) at ($(-1cm,0.75cm)+\N*(1cm,0cm)$);
        \node[name=N1, naka] at (0cm,0cm){$s_1$};
        \node[name=N2, naka] at (1cm,0cm){$-s_3$};
        \coordinate (i1) at (2cm,0cm);
        \node[name=N3, naka] at (0.5cm,-3cm){$s_1s_2s_3$};
        \node[name=cp1, product] at (0.5cm,-1cm){};
        \node[name=cp2, product] at (1cm,-1.5cm){};
        \node[name=b, product] at (1cm,-4cm){};
        \node[name=p, product] at (1cm,-2cm){};
        \foreach \N/\M/\S/\T in {
            N1.north/out1/su/sd,
            N2.north/out2/su/sd,
            cp1/N1.south/sul/sd,
            cp1/i1/sur/sd,
            i1/out3/su/sd,
            cp2/cp1/sul/sdr,
            cp2/N2.south/sur/sd,
            p/cp2/su/sd,
            b/p/sur/sdr}
            \draw (\N) .. controls ($(\N)+(\S)$) and ($(\M)+(\T)$) .. (\M);
        \foreach \N/\M/\S/\T in {
            N3.north/p/su/sdl,
            b/N3.south/sul/sd}
            \draw (\N) .. controls ($(\N)+0.5*(\S)$) and ($(\M)+0.5*(\T)$) .. (\M);
    \end{tikzpicture}}\ .
\end{equation}
As $A$ satisfies $N \ast \id=0$, this is zero unless $s(e_1)s(e_2)s(e_3)=1$, as required by the sign rule.

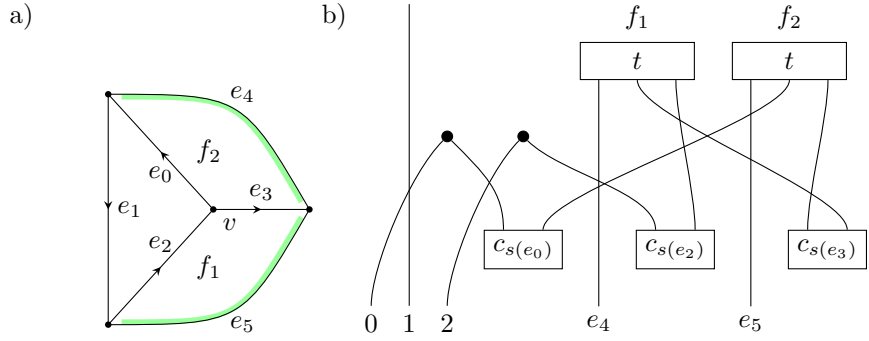
\begin{figure}[tb]
\centering
    \raisebox{12em}{a)}
    \begin{tikzpicture}
        \node[name=D1, regular polygon, regular polygon sides=3, minimum size=3.5cm, rotate=30, dotted, red] at (0cm,0cm){};
        \node[name=D2, regular polygon, regular polygon sides=3, minimum size=3.5cm, rotate=-30, dotted, red] at (0cm,0cm){};
        \foreach \N in {1,2,3} \node[circle, fill, inner sep=0.03cm] at (D1.corner \N){};
        \foreach \N/\M in {1/4,3/5} \node[] at (D2.corner \N){$e_{\M}$};
        \coordinate (nc) at  ($(D1.center)+(0.5cm,0cm)$);
        \node[circle, fill, inner sep=0.03cm] at (nc){};
        \begin{scope}[decoration={markings, mark=at position 0.5 with {\arrow{stealth}}}]
            \foreach \N in {1,3} \draw[postaction=decorate] (nc) -- (D1.corner \N);
            \draw[postaction=decorate] (D1.corner 1) -- (D1.corner 2);
            \draw[postaction=decorate] (D1.corner 2) -- (nc);
        \end{scope}
        \foreach \N/\M in {2/1,3/2} \node at (D1.side \N) {$f_{\M}$};
        \node[right] at (D1.side 1) {$e_1$};
        \node[below=2pt] at ($(nc)!0.5!(D1.corner 1)$) {$e_0$};
        \node[above=2pt] at ($(nc)!0.5!(D1.corner 2)$) {$e_2$};
        \node[above] at ($(nc)!0.5!(D1.corner 3)$) {$e_3$};
        \node[below right] at (nc) {$v$};
        \draw[green!40, line width=2pt] ($(D1.corner 1)+(0pt,-1pt)+0.1*(D2.corner 1)-0.1*(D1.corner 1)$) .. controls ($(D2.corner 1)+(-1pt,-1pt)$) .. ($(D1.corner 3)+(-0.5pt,-1.5pt)+0.1*(D2.corner 1)-0.1*(D1.corner 3)$);
        \draw[green!40, line width=2pt] ($(D1.corner 2)+(0pt,1pt)+0.1*(D2.corner 3)-0.1*(D1.corner 2)$) .. controls ($(D2.corner 3)+(-1pt,1pt)$) .. ($(D1.corner 3)+(-0.5pt,1.5pt)+0.1*(D2.corner 3)-0.1*(D1.corner 3)$);
        \foreach \N/\M/\K in {D1.corner 1/D1.corner 3/D2.corner 1, D1.corner 2/D1.corner 3/D2.corner 3} \draw (\N) .. controls (\K) .. (\M);
    \end{tikzpicture}
    \raisebox{12em}{b)}
    \begin{tikzpicture}
        \coordinate (su) at (0cm,0.5cm);
        \coordinate (sd) at (0cm,-0.5cm);
        \coordinate (sdl) at (-0.5cm,-0.5cm);
        \coordinate (sdr) at (0.5cm,-0.5cm);
        \coordinate (sul) at (-0.5cm,0.5cm);
        \coordinate (sur) at (0.5cm,0.5cm);
        \node[name=c1, draw, minimum width=0.8cm, minimum height=1.4em] at (2.5cm,-2.5cm){$c_{s(e_2)}$};
        \node[name=c2, draw, minimum width=0.8cm, minimum height=1.4em] at (0.5cm,-2.5cm){$c_{s(e_0)}$};
        \node[name=c3, draw, minimum width=0.8cm, minimum height=1.4em] at (4.5cm,-2.5cm){$c_{s(e_3)}$};
        \foreach \N in {1,2,3} {
            \coordinate (c\N1) at ($(c\N.north west)!0.25!(c\N.north east)$);
            \coordinate (c\N2) at ($(c\N.north west)!0.75!(c\N.north east)$);
        }
        \node[name=t2, draw, minimum width=1.5cm, minimum height=1.4em] at (2cm,0cm){$t$};
        \node[name=t3, draw, minimum width=1.5cm, minimum height=1.4em] at (4cm,0cm){$t$};
        \foreach \N in {2,3} {
            \coordinate (t\N1) at ($(t\N.south west)!0.33!(t\N.south)$);
            \coordinate (t\N2) at (t\N.south);
            \coordinate (t\N3) at ($(t\N.south east)!0.33!(t\N.south)$);
        }
        \foreach \N in {2,3} \coordinate (in\N) at ($(t\N1)+(0cm,-3cm)$);
        \foreach \N/\M in {2/1,3/2} \node[above=0.3cm] at (t\N) {$f_{\M}$};
        \foreach \N/\M in {2/4,3/5} \node[below] at (in\N) {$e_{\M}$};
        \foreach \N in {0,1,2} {
            \coordinate (inA\N) at ($(-1.5cm,-3cm)+\N*(0.5cm,0cm)+(0em,-0.7em)$);
            \node[below] at (inA\N) {$\N$};
        }
        \coordinate (out) at ($(inA1)+(0cm,4cm)$);
        \node[name=b1, product] at (0.5cm,-1cm){};
        \node[name=b2, product] at (-0.5cm,-1cm){};
        \foreach \N/\M in {
            in2/t21,
            in3/t31,
            c12/t23,
            c22/t32,
            c32/t22,
            c31/t33,
            inA1/out}
             \draw (\N) .. controls ($(\N)+(su)$) and ($(\M)+(sd)$) .. (\M);
        \foreach \N/\M/\S/\T in {
            c11/b1/su/sdr,
            inA2/b1/su/sdl,
            c21/b2/su/sdr,
            inA0/b2/su/sdl}
            \draw (\N) .. controls ($(\N)+(\S)$) and ($(\M)+(\T)$) .. (\M);
    \end{tikzpicture}
    \caption{a) Marking of triangles sharing the vertex $v$ in the case of $v$ being a boundary vertex. b) Corresponding part of the diagram computed from the triangulation. The lines starting at ``1'' and ``$e_4$'', ``$e_5$'' connect to other parts of the diagram.}
    \label{fig:check-sign-rule-bnd}
\end{figure}

\smallskip\noindent
{\em $v$ is on the boundary:} Consider the case that $v = \varphi_i(1)$, i.e.\ $v$ is the image of the point 1 under the boundary parametrisation. We again start by Pachner 2-2 moves so that $v$ is shared by exactly two triangles. Using moves on the marking, we arrive at the setup in Figure \ref{fig:check-sign-rule-bnd}. In the sign rule from Lemma \ref{lem:vertex.rule.boundary} we have $D_R=0$, $D_{NS}=1$, $K=2$, so that the extension condition is $s(e_0) s(e_2) s(e_3) = (-1)^{1+2+1}=1$ for an NS-type boundary component, and $s(e_0) s(e_2) s(e_3) = (-1)^{0+2+1}=-1$ for R-type.
The morphism from the diagram can be rewritten as
	(abbreviate $s_i := s(e_i)$)
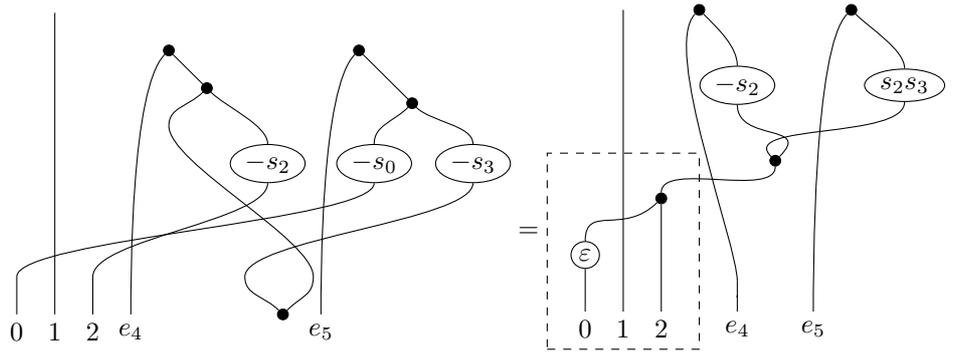
\begin{equation*}
    \raisebox{-1.5cm}{\begin{tikzpicture}
        \coordinate (su) at (0cm,0.5cm);
        \coordinate (sd) at (0cm,-0.5cm);
        \coordinate (sdl) at (-0.5cm,-0.5cm);
        \coordinate (sdr) at (0.5cm,-0.5cm);
        \coordinate (sul) at (-0.5cm,0.5cm);
        \coordinate (sur) at (0.5cm,0.5cm);
        \foreach \N in {0,1,2} {
            \coordinate (in\N) at ($\N*(0.5cm,0cm)$);
            \node[below] at (in\N) {$\N$};
        }
        \coordinate (in4) at (1.5cm,0cm);
        \node[below] at (in4) {$e_4$};
        \coordinate (in5) at (4cm,0cm);
        \node[below] at (in5) {$e_5$};
        \node[name=c, product] at (3.5cm,0cm){};
        \node[name=b1, product] at (2cm,3.5cm){};
        \node[name=b2, product] at (4.5cm,3.5cm){};
        \node[name=p1, product] at ($(b1)+(0.5cm,-0.5cm)$){};
        \node[name=p2, product] at ($(b2)+(0.7cm,-0.7cm)$){};
        \node[name=N1, naka] at (3.3cm,2cm) {$-s_2$};
        \node[name=N2, naka] at (4.7cm,2cm) {$-s_0$};
        \node[name=N3, naka] at (6cm,2cm) {$-s_3$};
        \coordinate (i0) at ($(in0)+(0cm,0.5cm)$);
        \coordinate (i2) at ($(in2)+(0cm,0.5cm)$);
        \coordinate (i3) at ($(c)+(-0.5cm,0.5cm)$);
        \coordinate (i4) at ($(c)+(0.4cm,0.5cm)$);
        \coordinate (i5) at ($(p1)+(-0.5cm,-0.5cm)$);
        \coordinate (out) at (0.5cm,4cm);
        \foreach \N/\M/\S/\T in {
            in1/out/su/sd,
            in0/i0/su/sd,
            i0/N2.south/su/sd,
            in2/i2/su/sd,
            i2/N1.south/su/sd,
            in4/b1/su/sdl,
            in5/b2/su/sdl,
            i3/N3.south/su/sd,
            i4/i5/su/sd,
            p1/b1/sul/sdr,
            p2/b2/sul/sdr}
            \draw (\N) .. controls ($(\N)+(\S)$) and ($(\M)+(\T)$) .. (\M);
        \foreach \N/\M/\S/\T in {
            c/i3/sul/sd,
            c/i4/sur/sd,
            i5/p1/su/sdl,
            N1.north/p1/su/sdr,
            N2.north/p2/su/sdl,
            N3.north/p2/su/sdr}
            \draw (\N) .. controls ($(\N)+0.5*(\S)$) and ($(\M)+0.5*(\T)$) .. (\M);
    \end{tikzpicture}}
    = 
    \raisebox{-1.5cm}{\begin{tikzpicture}
        \coordinate (su) at (0cm,0.5cm);
        \coordinate (sd) at (0cm,-0.5cm);
        \coordinate (sdl) at (-0.5cm,-0.5cm);
        \coordinate (sdr) at (0.5cm,-0.5cm);
        \coordinate (sul) at (-0.5cm,0.5cm);
        \coordinate (sur) at (0.5cm,0.5cm);
        \foreach \N in {0,1,2} {
            \coordinate (in\N) at ($\N*(0.5cm,0cm)$);
            \node[below] at (in\N) {$\N$};
        }
        \coordinate (in4) at (2cm,0cm);
        \node[below] at (in4) {$e_4$};
        \coordinate (in5) at (3cm,0cm);
        \node[below] at (in5) {$e_5$};
        \node[name=p, product] at (1cm,1.5cm){};
        \node[name=cp, product] at ($(p)+(1.5cm,0.5cm)$){};
        \node[name=b1, product] at ($(cp)+(-1cm,2cm)$){};
        \node[name=b2, product] at ($(cp)+(1cm,2cm)$){};
        \node[name=N1, naka] at (0cm,0.75cm) {$\varepsilon$};
        \node[name=N2, naka] at ($(cp)+(-0.5cm,1cm)$) {$-s_2$};
        \node[name=N3, naka] at ($(cp)+(1.7cm,1cm)$) {$s_2s_3$};
        \coordinate (i1) at ($(in4)+(0cm,0.4cm)$);
        \coordinate (out) at (0.5cm,4cm);
        \foreach \N/\M/\S/\T in {
            in1/out/su/sd,
            in0/N1.south/su/sd,
            N1.north/p/su/sdl,
            in2/p/su/sd,
            p/cp/su/sd,
            cp/N2.south/sur/sd,
            cp/N3.south/sul/sd,
            i1/b1/su/sdl,
            in5/b2/su/sdl,
            in4/i1/su/sd}
            \draw (\N) .. controls ($(\N)+(\S)$) and ($(\M)+(\T)$) .. (\M);
        \foreach \N/\M/\S/\T in {
            N2.north/b1/su/sdr,
            N3.north/b2/su/sdr}
            \draw (\N) .. controls ($(\N)+0.5*(\S)$) and ($(\M)+0.5*(\T)$) .. (\M);
        \node[dashed, draw, minimum width=2cm, minimum height=2.6cm] at (0.5cm,0.8cm){};
    \end{tikzpicture}}\ ,
\end{equation*}
where $\varepsilon =  -s(e_0)s(e_2)s(e_3)$.
Composing
	the boxed part of
this morphism with $ \iota^{13} \circ P^{\delta_i}$ gives
\begin{equation}
   \raisebox{-3cm}{\begin{tikzpicture}
        \coordinate (su) at (0cm,0.5cm);
        \coordinate (sd) at (0cm,-0.5cm);
        \coordinate (sdl) at (-0.5cm,-0.5cm);
        \coordinate (sdr) at (0.5cm,-0.5cm);
        \coordinate (sul) at (-0.5cm,0.5cm);
        \coordinate (sur) at (0.5cm,0.5cm);
        \coordinate (in) at (0cm,0cm);
        \node[name=cp1, product] at (0cm,0.5cm){};
        \node[name=N1, naka] at (-0.35cm,1.1cm){$\nu$};
        \coordinate (i1) at ($(N1.north)+(0.7cm,0cm)$);
        \coordinate (i0) at ($(N1.south)+(0.7cm,0cm)$);
        \node[name=p1, product] at (0cm,2cm){};
        \node[name=cp2, product] at (0cm,2.5cm){};
        \coordinate (i2) at (0.5cm,3.2cm);
        \node[name=cp3, product] at (-0.5cm,3cm){};
        \coordinate (i3) at (-0.8cm,3.5cm);
        \coordinate (i4) at (0.4cm,4cm);
        \foreach \N in {0,1,2} {
            \coordinate (ie\N) at ($(-0.75cm,5cm)+\N*(0.75cm,0cm)$);
            \coordinate (out\N) at ($(-0.75cm,7.3cm)+\N*(0.75cm,0cm)$);
            \node[above left] at (ie\N) {\small{$\N$}};
        }
        \node[name=N2, naka] at (-0.75cm,5.8cm) {$\varepsilon$};
        \node[name=p2, product] at (0.75cm,6.8cm){};
        \foreach \N/\M/\S/\T in {
            in/cp1/su/sd,
            N1.north/p1/su/sdr,
            i1/p1/su/sdl,
            p1/cp2/su/sd,
            cp2/cp3/sul/sdr,
            i2/ie0/su/sd,
            i3/ie2/su/sd,
            i4/ie1/su/sd,
            ie0/N2.south/su/sd,
            ie1/out1/su/sd,
            ie2/p2/su/sd,
            N2.north/p2/su/sdl,
            p2/out2/su/sd}
            \draw (\N) .. controls ($(\N)+(\S)$) and ($(\M)+(\T)$) .. (\M);
        \foreach \N/\M/\S/\T in {
            i0/i1/su/sd,
            cp1/N1.south/sul/sd,
            cp1/i0/sur/sd,
            cp2/i2/sur/sd,
            cp3/i3/sul/sd,
            cp3/i4/sur/sd}
            \draw (\N) .. controls ($(\N)+0.5*(\S)$) and ($(\M)+0.5*(\T)$) .. (\M);
        \draw[dashed] ($(ie0)+(-1cm,0cm)$) -- ($(ie2)+(1cm,0cm)$);
    \end{tikzpicture}}
	= \sigma_{A,A} \circ \Delta \circ q_{-\varepsilon} \circ q_{-\nu} \ ,
\end{equation}
	where $\nu = -1$ for $\delta_i=NS$ and $\nu=+1$ for $\delta_i=R$, 
and $q_{-\varepsilon}, q_{-\nu}$ are as in \eqref{eq:q_nu-def}. By Lemma \ref{lem:N*id-NS-R-zero1}, this is zero unless $\varepsilon=\nu$, i.e.\ unless $s(e_0)s(e_2)s(e_3)=1$ (NS-type) or $s(e_0)s(e_2)s(e_3)=-1$ (R-type), in agreement with the extendibility condition.

If $v \neq \varphi_i(1)$, then the sign rule from Lemma \ref{lem:vertex.rule.boundary} gives $s(e_0) s(e_2) s(e_3) =-1$ for both, NS- and R-type. The morphism obtained from the triangulation then is, in the example $v = \varphi_i(e^{2 \pi i/3})$,
\begin{equation}
    \raisebox{-3cm}{\begin{tikzpicture}
        \coordinate (su) at (0cm,0.5cm);
        \coordinate (sd) at (0cm,-0.5cm);
        \coordinate (sdl) at (-0.5cm,-0.5cm);
        \coordinate (sdr) at (0.5cm,-0.5cm);
        \coordinate (sul) at (-0.5cm,0.5cm);
        \coordinate (sur) at (0.5cm,0.5cm);
        \coordinate (in) at (0cm,0cm);
        \node[name=cp1, product] at (0cm,0.5cm){};
        \node[name=N1, naka] at (-0.35cm,1.1cm){$\nu$};
        \coordinate (i1) at ($(N1.north)+(0.7cm,0cm)$);
        \coordinate (i0) at ($(N1.south)+(0.7cm,0cm)$);
        \node[name=p1, product] at (0cm,2cm){};
        \node[name=cp2, product] at (0cm,2.5cm){};
        \coordinate (i2) at (0.5cm,3.2cm);
        \node[name=cp3, product] at (-0.5cm,3cm){};
        \coordinate (i3) at (-0.8cm,3.5cm);
        \coordinate (i4) at (0.4cm,4cm);
        \foreach \N in {0,1,2} {
            \coordinate (ie\N) at ($(-0.75cm,5cm)+\N*(0.75cm,0cm)$);
            \coordinate (ix\N) at ($(-0.75cm,6cm)+\N*(0.75cm,0cm)$);
            \coordinate (out\N) at ($(-0.75cm,8.3cm)+\N*(0.75cm,0cm)$);
        }
        \node[above left] at (ie0) {\small{$0$}};
        \node[above right] at (ie1) {\small{$1$}};
        \node[above right] at (ie2) {\small{$2$}};
        \node[name=N2, naka] at (-0.75cm,6.8cm) {$\varepsilon$};
        \node[name=p2, product] at (0.75cm,7.8cm){};
        \foreach \N/\M/\S/\T in {
            in/cp1/su/sd,
            N1.north/p1/su/sdr,
            i1/p1/su/sdl,
            p1/cp2/su/sd,
            cp2/cp3/sul/sdr,
            i2/ie0/su/sd,
            i3/ie2/su/sd,
            i4/ie1/su/sd,
            ix0/N2.south/su/sd,
            ix1/out1/su/sd,
            ix2/p2/su/sd,
            N2.north/p2/su/sdl,
            p2/out2/su/sd}
            \draw (\N) .. controls ($(\N)+(\S)$) and ($(\M)+(\T)$) .. (\M);
        \foreach \N/\M/\S/\T in {
            ie0/ix2/su/sd,
            ie1/ix0/su/sd,
            ie2/ix1/su/sd,
            i0/i1/su/sd,
            cp1/N1.south/sul/sd,
            cp1/i0/sur/sd,
            cp2/i2/sur/sd,
            cp3/i3/sul/sd,
            cp3/i4/sur/sd}
            \draw (\N) .. controls ($(\N)+0.5*(\S)$) and ($(\M)+0.5*(\T)$) .. (\M);
        \draw[dashed] ($(ie0)+(-1cm,0cm)$) -- ($(ie2)+(1cm,0cm)$);
    \end{tikzpicture}} \ ,
\end{equation}
where $\varepsilon$ and $\nu$ are as above. After disentangling the lines and using coassociativity once, one identifies a morphism $\mu \circ (N_\varepsilon \otimes \id_A) \circ \Delta$ contained in the diagram. Thus we obtain zero unless $\varepsilon=1$, i.e.\ unless $s(e_0)s(e_2)s(e_3)=-1$, as required. The computation for $v = \varphi_i(e^{4 \pi i/3})$ is analogous.
\end{proof}

\begin{remark}\label{rem:spin-to-oriented}
Lemma \ref{lem:N*id-NS-R-zero2} and Proposition \ref{prop:N*id=0-admissible} have an interesting interpretation when considering the local statistical model implementing spin structures mentioned in Section \ref{sec:outlook}: think of the edge signs as statistical variables and of $T'_A$ as the statistical weight of a given edge sign configuration. When summing over all edge sign configurations, we see that only admissible edge signs (i.e.\ those corresponding to a spin structure, see Section \ref{sse:spin_reconstruct}) receive a non-zero weight. On the other hand, summing over the edge signs and assigning a weight of $\frac12$ to each edge amounts to inserting a projector $\pi_+$ on each edge, effectively replacing the algebra $A$ by $A_+$. The latter defines an oriented TFT (by construction the Nakayama automorphism of $A_+$ is the identity), provided one compensates the factor of $\frac12$ in Lemma \ref{lem:N*id-NS-R-zero2} by assigning a weight of $2$ to each vertex (to achieve invariance under the 3-1 Pachner move). In this sense, the sum over spin structures projects the spin TFT defined by $A$ to the oriented TFT defined by $A_+$.
\end{remark}

\subsection{Examples}
\begin{example}
Let $A=k^{1|1} \in \mathbf{SVect}(k)$ for $\text{char}(k)\neq 2$. Define the product 
\begin{equation}
    \mu \,:\, A\otimes A \to A ~~ , \quad
     \begin{pmatrix}x_0\\x_1\end{pmatrix}\otimes\begin{pmatrix}y_0\\y_1\end{pmatrix} \mapsto \begin{pmatrix} x_0y_0+x_1y_1\\ x_0y_1+x_1y_0 \end{pmatrix}
\end{equation}
and the unit/counit
\begin{equation}
\eta : k \to A ~~,~~~
    k\ni \lambda \mapsto \lambda\cdot \begin{pmatrix}1\\0\end{pmatrix}
    \qquad , \qquad
    \varepsilon : A \to k ~~,~~~
    \begin{pmatrix}x_0\\x_1\end{pmatrix}  \mapsto 2x_0 \ .
\end{equation}
It is straightforward to verify that these maps turn $A$ into a Frobenius algebra.
The pairing $b :A\otimes A \to k$ is then given by
\begin{equation}
    b(\begin{pmatrix}x_0\\x_1\end{pmatrix} \otimes \begin{pmatrix}y_0\\y_1\end{pmatrix}) = 2(x_0y_0+x_1y_1).
\end{equation}
	Since $\text{char}(k)\neq2$, $b$ is nondegenerate with
	copairing
\begin{equation}
    c_{-1} \,:\, k\to A\otimes A ~~,~~~ 1\mapsto
    \frac{1}{2} \left( \begin{pmatrix}1\\0\end{pmatrix} \otimes \begin{pmatrix}1\\0\end{pmatrix} + \begin{pmatrix}0\\1\end{pmatrix} \otimes\begin{pmatrix}0\\1\end{pmatrix}  \right) \ .
\end{equation}
From this one computes the Nakayama automorphism of $A$ to be
\begin{equation}
    N \,:\, A\to A ~~,~~~
    \begin{pmatrix}x_0\\x_1\end{pmatrix} \mapsto \begin{pmatrix}x_0\\-x_1\end{pmatrix} \quad .
\end{equation}
Thus $N^2=\text{id}_A$. 
The coproduct can be computed from the copairing as
\begin{align}
    \Delta\begin{pmatrix}x_0\\x_1\end{pmatrix} &=
    \frac{1}{2} \left( \begin{pmatrix}x_0\\x_1\end{pmatrix} \otimes \begin{pmatrix}1\\0\end{pmatrix} + \begin{pmatrix}x_1\\x_0\end{pmatrix} \otimes\begin{pmatrix}0\\1\end{pmatrix}  \right) 
\\
    &=
    \frac{1}{2} \left(  \begin{pmatrix}1\\0\end{pmatrix} \otimes\begin{pmatrix}x_0\\x_1\end{pmatrix} + \begin{pmatrix}0\\1\end{pmatrix}\otimes\begin{pmatrix}x_1\\x_0\end{pmatrix}   \right) \quad .
    \nonumber
\end{align}
From this it is immediate that $A$ is $\Delta$-separable, and that the condition $N \star \id=0$ is satisfied. One also easily computes the idempotents $P^{NS/R}$ to be
\begin{align}
	 P^{NS}\begin{pmatrix}x_0\\x_1\end{pmatrix} =\begin{pmatrix}x_0\\0\end{pmatrix} \quad , \quad 
	 P^{R}\begin{pmatrix}x_0\\x_1\end{pmatrix} =\begin{pmatrix}0\\x_1\end{pmatrix} \quad .
\end{align}
Therefore, the state spaces are given by $Z_{NS}=k^{1|0}$ and $Z_R=k^{0|1}$, and from the formulas for the structure maps in \eqref{eq:Z-Frobalg} we see that in fact $Z = A$ as Frobenius algebras.

Evaluating the TFT for $A$ on $T_{NS/R}^\varepsilon$ according to \eqref{eq:spin-torus-value} gives
\begin{equation}
	T_A(T_{NS}^\pm) = T_A(T_{R}^-) = 1 ~~ , \quad
	T_A(T_{R}^+) = -1 \ .
\end{equation}

In general one verifies that for a closed spin surface $\Sigma$ of genus $g$ one has $T_A(\Sigma)=2^{1-g}\text{Arf}(\Sigma)$, where $\text{Arf}(\Sigma)\in \{1,-1\}$ is the Arf invariant \cite{johnson1980} of $\Sigma$, cf.\ \cite[Sect.\,2.6]{moore2006d} and \cite[Thm.\,4.1]{barrett2013spin}; we omit the details.

\end{example}

\begin{example}
Let $(A,\mu,\eta,\varepsilon)$ be a symmetric Frobenius algebra 
	in an additive symmetric strict	monoidal category $\mathcal{S}$.
Let $x\in \text{Hom}\left( \mathbf{1}, A \right)$ be invertible with respect to the algebra product $\mu$. We denote its inverse by $x^{-1}\in \text{Hom}\left( \mathbf{1}, A \right)$. Let 
    \begin{equation}
            \varepsilon_x:= \varepsilon \circ \mu \circ \left( x \otimes \id_A  \right).
    \end{equation}
    Then $A_x=(A,\mu,\eta,\varepsilon_x)$ is again a Frobenius algebra,
    see e.g. \cite[Lemma 19]{fuchsstigner2009}.
    In the following we draw the original Frobenius algebra morphisms as in Figure \ref{fig:graph-symbol}. 
The coproduct and Nakayama automorphism of $A_x$ are given by
    \begin{equation}
        \Delta_x = \raisebox{-2.5em}{ 
        \begin{tikzpicture}
            \coordinate (in) at (0cm,0cm);
            \node[name=cp,circle,fill,inner sep=0.05cm] at (0cm,0.5cm){};
            \node[name=x,draw, inner sep=0.05cm] at (0.2cm,1.2cm) {\small{$x^{-1}$}};
            \node[name=p,circle,fill,inner sep=0.05cm] at (1cm,2cm){};
            \coordinate (out1) at (-0.5cm,2.5cm);
            \coordinate (out2) at (1cm,2.5cm);
            \coordinate (i1) at (-0.5cm,1cm);
            \coordinate (su) at (0cm,0.3cm);
            \coordinate (sd) at (0cm,-0.3cm);
            \draw (in) -- (cp);
            \draw (cp) .. controls (i1) and ($(out1)+2*(sd)$) .. (out1);
            \draw (x.north) .. controls ($(x.north) + (su)$) .. (p);
            \draw (cp) .. controls ($(p)+3*(sd)$) .. (p);
            \draw (p) -- (out2);
        \end{tikzpicture}}\quad  , \qquad
        N_x = \raisebox{-2.5em}{
        \begin{tikzpicture}
            \coordinate (in) at (0cm,0cm);
            \coordinate (out) at (0cm,2.5cm);
            \node[name=p1,circle,fill,inner sep=0.05cm] at (0cm,1cm){};
            \node[name=p2,circle,fill,inner sep=0.05cm] at (0cm,2cm){};
            \node[name=x,draw] at (-0.5cm,0.3cm) {\small{$x$}};
            \node[name=xinv,draw, inner sep=0.05cm] at (0.7cm,1.2cm) {\small{$x^{-1}$}};
            \coordinate (su) at (0cm,0.2cm);
            \draw (in) -- (p1);
            \draw (p1) -- (p2);
            \draw (p2) -- (out);
            \draw (x.north) .. controls ($(x.north)+(su)$) .. (p1);
            \draw (xinv.north) .. controls ($(xinv.north)+(su)$) .. (p2);
        \end{tikzpicture}} \quad .
    \end{equation}
    The Nakayama automorphism $N_x$ is thus the inner automorphism generated by $x$. 
	It satisfies $N_x^2=\id_A$ iff $x^2 := \mu \circ (x \otimes x)$ is central in $A$, i.e.\ if $\mu \circ \sigma_{A,A} \circ (x^2 \otimes \id) = \mu \circ (x^2 \otimes \id)$.
	By definition, $A_x$ is $\Delta$-separable iff $\mu \circ (\id \otimes \mu) \circ (\id \otimes x^{-1} \otimes \id) \circ \Delta = \id$ holds in $A$.
    If $N_x^2=\id_A$ then the condition $N_x \star_{A_x} \id=0$ is satisfied in $A_x$ iff $\id \star \id=0$ in $A$.

It turns out that in this example, the TFT does not actually depend on the spin structure. Namely, let $R_x := \mu \circ (\id_A \otimes x)$ be right multiplication by $x$. One quickly checks that $P^{NS} = R_x^{-1} \circ P^R \circ R_x$, so that the NS- and R-state spaces are isomorphic. This identity furthermore implies that $P^R \circ N = P^R$ (in addition to $P^{NS} \circ N = P^{NS}$, which holds by Lemma \ref{lem:PNS/R-idemp}). The latter observation implies independence of the spin structure on closed surfaces, cf.\ expression \eqref{eq:spin-torus-value} for the torus.

From the point of view of fully extended TFTs this is not too surprising, since a fully dualisable object in the symmetric monoidal bicategory of algebras does not involve the pairing as a piece of data. Hence, if there exists a symmetric pairing on $A$, the resulting TFT will be independent of the spin structure.

Still, the next example shows that such TFTs can nonetheless be interesting.
    \label{xmp:frobenius.generic}
\end{example}

\begin{example}
    \label{xmp:matrix.alg.in.vect}
Let $A=M_n(k)$ be the algebra of $n{\times}n$ matrices for some integer $n>0$ and a field $k$. Let $E_{ij}$ be the $n{\times}n$ matrix with zero entries everywhere but in place $(i,j)$, where it has entry 1. It satisfies $E_{ij}E_{kl} = \delta_{j,k} E_{il}$ and consequently $\mathrm{tr}(E_{ij}E_{kl}) = \delta_{i,l} \delta_{j,k}$. Thus, the trace pairing on $A$ is non-degenerate (independent of the characteristic of $k$) and we can use it to turn $A$ into a (symmetric) Frobenius algebra. Concretely, the counit and coproduct are
\begin{equation}
        \varepsilon(M) = \mathrm{tr}(M) ~~,\quad
	\Delta(M) = \sum_{i,j=1}^n (ME_{ij}) \otimes E_{ji} = \sum_{i,j=1}^n E_{ij} \otimes (E_{ji}M) \ .
\end{equation}
Now choose $X \in GL_n(k)$ such that, for some $\lambda \in k^\times$,
\begin{enumerate}
\item $X^2 = \lambda \mathbf{1}$, i.e.\ $X^{-1} = \lambda^{-1} X$, and
\item $\mathrm{tr}(X)=\lambda$ .
\end{enumerate}
From Example \ref{xmp:frobenius.generic} we obtain a new Frobenius algebra $A_X$ by twisting the counit with $X$. Condition (1) shows $(N_X)^2 = \id_A$ (since $X^2$ is central), and condition (2) implies that $A_X$ is $\Delta$-separable:
\begin{equation}
\mu(\Delta_X(M)) = \sum_{i,j=1}^n E_{ij} X^{-1} E_{ji}M = \mathrm{tr}(X^{-1}) \cdot M = \lambda^{-1} \mathrm{tr}(X) \cdot M \ .
\end{equation}
Thus, $A_X$ is an example of a $\Delta$-separable Frobenius algebra whose Nakayama automorphism is an involution. 
The projectors $P^{NS/R}$ are straightforward to compute:
\begin{equation}
	P^{NS}(M) = \lambda^{-1} \mathrm{tr}(MX) \cdot \mathbf{1} ~~,\quad 
	P^{R}(M) = \lambda^{-1} \mathrm{tr}(M) \cdot X \ .
\end{equation}
Thus, the state spaces $Z^{NS/R}$ are one-dimensional and given by $Z^{NS} = k \, \mathbf{1}$, $Z^R = k\, X$.

From Example \ref{xmp:frobenius.generic} we know that the TFT for $A_X$ is independent of the spin structure. For example, evaluating the TFT on $T_{NS/R}^\varepsilon$ according to \eqref{eq:spin-torus-value} gives $T_{A_X}(T_{NS}^\pm) = T_{A_X}(T_{R}^\pm) = 1$. 

\medskip

Next we consider the condition $N_X \ast_{A_X} \id = 0$ in $A_X$. From  Example \ref{xmp:frobenius.generic} we know that this is equivalent to $\id \ast \id=0$ in $A$. But $(\id \ast \id)(M) = \mu(\Delta(M)) = n \cdot M$. We see that $N_X \ast_{A_X} \id = 0$ iff
\begin{equation}
	n = 0 \quad \text{in the field} ~~ k \ .
\end{equation}
A simple example would be to take $k$ of characteristic 3, and $n=3$, $\lambda=1$, $X = \mathrm{diag}(1,1,-1)$.

We learn that it is possible to satisfy the $N \ast \id = 0$ even if the TFT is independent of the spin structure. 
In terms of the statistical model from Remark \ref{rem:spin-to-oriented} this means that all admissible edge sign configurations receive the same weight (and non-admissible ones receive weight zero). 

\end{example}

\appendix
\section{Evaluation of the TFT on the cylinder}
\label{sse:cylinder_calculation}

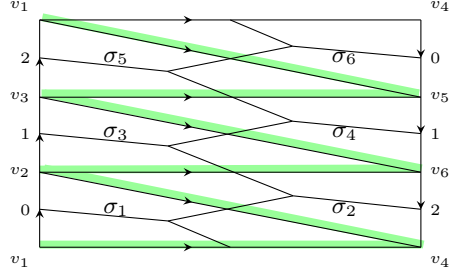
\begin{figure}[tb]
    \centering
    \begin{tikzpicture}
        \begin{scope}[decoration={
    markings,
    mark=at position 0.4 with {\arrow{stealth}}}
    ] 
       \node[name=A, rectangle, minimum width=5cm, minimum height=3cm]{};
       \draw[green!40, line width=3pt] ($(A.south west)!0.33!(A.north west)$)+(0pt,1pt) -- ($0.66*(A.south east)+0.33*(A.north east)+(0pt,1pt)$);
       \draw[green!40, line width=3pt] (A.south west)+(0pt,1pt) -- ($(A.south east)+(0pt,1pt)$);
       \draw[green!40, line width=3pt] ($0.33*(A.south west)+0.66*(A.north west)+(0pt,1pt)$) -- ($0.33*(A.south east)+0.66*(A.north east)+(0pt,1pt)$);
       \draw[green!40, line width=3pt] (A.north west)+(1pt,1pt) -- ($0.33*(A.south east)+0.66*(A.north east)+(0pt,1pt)$);
       \draw[green!40, line width=3pt] ($0.33*(A.south west)+0.66*(A.north west)+(1pt,1pt)$) -- ($0.66*(A.south east)+0.33*(A.north east)+(1pt,1pt)$);
       \draw[green!40, line width=3pt] ($0.66*(A.south west)+0.33*(A.north west)+(1pt,1pt)$) -- ($(A.south east)+(1pt,1pt)$);
       \draw[postaction=decorate] ($(A.south west)!0.33!(A.north west)$) -- ($(A.south east)!0.33!(A.north east)$);
       \draw[postaction=decorate] (A.north west) -- (A.north east);
       \draw[postaction=decorate] (A.south west) -- (A.south east);
       \draw[postaction=decorate] ($(A.south west)!0.66!(A.north west)$) -- ($(A.south east)!0.66!(A.north east)$);
       \draw[postaction=decorate] (A.north west) -- ($(A.south east)!0.66!(A.north east)$);
       \draw[postaction=decorate] ($(A.south west)!0.66!(A.north west)$) -- ($(A.south east)!0.33!(A.north east)$);
       \draw[postaction=decorate] ($(A.south west)!0.33!(A.north west)$) -- (A.south east);
\end{scope}
        \begin{scope}[decoration={
    markings,
    mark=at position 0.166 with {\arrow{stealth}};,
    mark=at position 0.5 with {\arrow{stealth}};,
    mark=at position 0.833 with {\arrow{stealth}}}
]
       \draw[postaction=decorate] (A.south west) -- (A.north west);
       \draw[postaction=decorate] (A.north east) -- (A.south east);
   \end{scope}
   \node[left] at ($(A.south west)!0.166!(A.north west)$) {\tiny{$0$}} ;
   \node[left] at ($(A.south west)!0.5!(A.north west)$) {\tiny{$1$}} ;
   \node[left] at ($(A.south west)!0.833!(A.north west)$) {\tiny{$2$}} ;
   \node[right] at ($(A.south east)!0.166!(A.north east)$) {\tiny{$2$}} ;
   \node[right] at ($(A.south east)!0.5!(A.north east)$) {\tiny{$1$}} ;
   \node[right] at ($(A.south east)!0.833!(A.north east)$) {\tiny{$0$}} ;
   \node[] at ($0.166*0.2*(A.north east)+0.833*0.2*(A.south east)+0.166*0.8*(A.north west)+0.833*0.8*(A.south west)$) {\small{$\sigma_1$}} ;
   \node[] at ($0.166*0.8*(A.north east)+0.833*0.8*(A.south east)+0.166*0.2*(A.north west)+0.833*0.2*(A.south west)$) {\small{$\sigma_2$}} ;
   \node[] at ($0.5*0.2*(A.north east)+0.5*0.2*(A.south east)+0.5*0.8*(A.north west)+0.5*0.8*(A.south west)$) {\small{$\sigma_3$}} ;
   \node[] at ($0.5*0.8*(A.north east)+0.5*0.8*(A.south east)+0.5*0.2*(A.north west)+0.5*0.2*(A.south west)$) {\small{$\sigma_4$}} ;
   \node[] at ($0.833*0.2*(A.north east)+0.166*0.2*(A.south east)+0.833*0.8*(A.north west)+0.166*0.8*(A.south west)$) {\small{$\sigma_5$}} ;
   \node[] at ($0.833*0.8*(A.north east)+0.166*0.8*(A.south east)+0.833*0.2*(A.north west)+0.166*0.2*(A.south west)$) {\small{$\sigma_6$}} ;
   \node[below left] at (A.south west) {\tiny{$v_1$}} ;
   \node[left] at ($(A.south west)!0.33!(A.north west)$) {\tiny{$v_2$}} ;
   \node[left] at ($(A.south west)!0.66!(A.north west)$) {\tiny{$v_3$}} ;
   \node[above left] at (A.north west) {\tiny{$v_1$}} ;
   \node[below right] at (A.south east) {\tiny{$v_4$}} ;
   \node[right] at ($(A.south east)!0.33!(A.north east)$) {\tiny{$v_6$}} ;
   \node[right] at ($(A.south east)!0.66!(A.north east)$) {\tiny{$v_5$}} ;
   \node[above right] at (A.north east) {\tiny{$v_4$}} ;
   \coordinate (sigma5) at ($0.33*(A.north west)+0.33*0.66*(A.north west)+0.33*0.33*(A.south west)+0.33*0.66*(A.north east)+0.33*0.33*(A.south east)$);
   \coordinate (sigma3) at ($0.33*0.66*(A.north west)+0.33*0.33*(A.south west)+0.33*0.33*(A.north west)+0.33*0.66*(A.south west)+0.33*0.33*(A.north east)+0.33*0.66*(A.south east)$);
   \coordinate (sigma1) at ($0.33*(A.south west)+0.33*0.66*(A.south west)+0.33*0.33*(A.north west)+0.33*(A.south east)$);
   \coordinate (sigma6) at ($0.33*(A.north east)+0.33*0.66*(A.north east)+0.33*0.33*(A.south east)+0.33*(A.north west)$);
   \coordinate (sigma4) at ($0.33*0.66*(A.north east)+0.33*0.33*(A.south east)+0.33*0.33*(A.north east)+0.33*0.66*(A.south east)+0.33*0.66*(A.north west)+0.33*0.33*(A.south west)$);
   \coordinate (sigma2) at ($0.33*(A.south east)+0.33*0.66*(A.south east)+0.33*0.33*(A.north east)+0.33*0.66*(A.south west)+0.33*0.33*(A.north west)$);
   \draw ($(A.south west)!0.833!(A.north west)$) -- (sigma5);
   \draw ($(A.south west)!0.5!(A.north west)$) -- (sigma3);
   \draw ($(A.north west)!0.833!(A.south west)$) -- (sigma1);
   \draw ($(A.south east)!0.833!(A.north east)$) -- (sigma6);
   \draw ($(A.south east)!0.5!(A.north east)$) -- (sigma4);
   \draw ($(A.north east)!0.833!(A.south east)$) -- (sigma2);
   \draw (sigma5) -- (sigma6);
   \draw (sigma3) -- (sigma4);
   \draw (sigma1) -- (sigma2);
   \draw (sigma5) -- (sigma4);
   \draw (sigma3) -- (sigma2);
   \draw (sigma1) -- ($(A.south east)!0.5!(A.south west)$);
   \draw (sigma6) -- ($(A.north east)!0.5!(A.north west)$);
    \end{tikzpicture}
    \caption{The triangulation in Figure \ref{fig:cylinder.triangulation} together with its dual.}
    \label{fig:cylinder.triangulation.dual}
\end{figure}

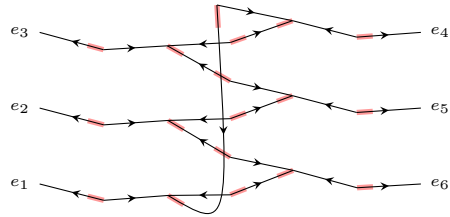
\begin{figure}[tb]
    \centering
    \begin{tikzpicture}
       \node[name=A, rectangle, minimum width=5cm, minimum height=3cm,draw=none]{}; 
   \node[left] at ($(A.south west)!0.166!(A.north west)$) {\tiny{$e_1$}} ;
   \node[left] at ($(A.south west)!0.5!(A.north west)$) {\tiny{$e_2$}} ;
   \node[left] at ($(A.south west)!0.833!(A.north west)$) {\tiny{$e_3$}} ;
   \node[right] at ($(A.south east)!0.166!(A.north east)$) {\tiny{$e_6$}} ;
   \node[right] at ($(A.south east)!0.5!(A.north east)$) {\tiny{$e_5$}} ;
   \node[right] at ($(A.south east)!0.833!(A.north east)$) {\tiny{$e_4$}} ;
   \coordinate (sigma5) at ($0.33*(A.north west)+0.33*0.66*(A.north west)+0.33*0.33*(A.south west)+0.33*0.66*(A.north east)+0.33*0.33*(A.south east)$);
   \coordinate (sigma3) at ($0.33*0.66*(A.north west)+0.33*0.33*(A.south west)+0.33*0.33*(A.north west)+0.33*0.66*(A.south west)+0.33*0.33*(A.north east)+0.33*0.66*(A.south east)$);
   \coordinate (sigma1) at ($0.33*(A.south west)+0.33*0.66*(A.south west)+0.33*0.33*(A.north west)+0.33*(A.south east)$);
   \coordinate (sigma6) at ($0.33*(A.north east)+0.33*0.66*(A.north east)+0.33*0.33*(A.south east)+0.33*(A.north west)$);
   \coordinate (sigma4) at ($0.33*0.66*(A.north east)+0.33*0.33*(A.south east)+0.33*0.33*(A.north east)+0.33*0.66*(A.south east)+0.33*0.66*(A.north west)+0.33*0.33*(A.south west)$);
   \coordinate (sigma2) at ($0.33*(A.south east)+0.33*0.66*(A.south east)+0.33*0.33*(A.north east)+0.33*0.66*(A.south west)+0.33*0.33*(A.north west)$);
   \coordinate (e7d) at ($0.5*(A.south west)+0.5*(A.south east)+(0,-4pt)$);
   \coordinate (e8) at ($0.5*(A.south west)+0.5*(A.south east)+0.5*0.33*(A.north west)-0.5*0.33*(A.south west)+(0,-4pt)$);
   \coordinate (e9) at ($0.5*(A.south west)+0.5*(A.south east)+0.5*0.66*(A.north west)-0.5*0.66*(A.south west)+(0,-4pt)$);
   \coordinate (e10) at ($0.5*(A.south west)+0.5*(A.south east)+0.5*(A.north west)-0.5*(A.south west)+(0,-4pt)$);
   \coordinate (e11) at ($0.5*(A.south west)+0.5*(A.south east)+0.5*1.33*(A.north west)-0.5*1.33*(A.south west)+(0,-4pt)$);
   \coordinate (e12) at ($0.5*(A.south west)+0.5*(A.south east)+0.5*1.67*(A.north west)-0.5*1.67*(A.south west)+(0,-4pt)$);
   \coordinate (e7u) at ($0.5*(A.north west)+0.5*(A.north east)+(-5pt,-4pt)$);
   \coordinate (el3) at ($(A.south west)!0.166!(A.north west)$) ;
   \coordinate (el2) at ($(A.south west)!0.5!(A.north west)$) ;
   \coordinate (el1) at ($(A.south west)!0.833!(A.north west)$) ;
   \coordinate (el6) at ($(A.south east)!0.166!(A.north east)$) ;
   \coordinate (el5) at ($(A.south east)!0.5!(A.north east)$) ;
   \coordinate (el4) at ($(A.south east)!0.833!(A.north east)$) ;
   \foreach \N/\M in {1/5,2/3,3/1,4/6,5/4,6/2} 
   \coordinate (e\N) at ($0.5*(el\N)+0.5*(sigma\M)+(0,-4pt)$);
   \foreach \N in {1,...,6} \draw[red!40, line width=2pt] (e\N) -- ($(e\N)!0.25!(el\N)$);
    \draw[red!40, line width=2pt] (e7u) -- ($(e7u)!0.1!(e7d)$);
    \draw[red!40, line width=2pt] (e12) -- ($(e12)!0.25!(sigma6)$);
    \draw[red!40, line width=2pt] (e11) -- ($(e11)!0.25!(sigma5)$);
    \draw[red!40, line width=2pt] (e10) -- ($(e10)!0.25!(sigma4)$);
    \draw[red!40, line width=2pt] (e9) -- ($(e9)!0.25!(sigma3)$);
    \draw[red!40, line width=2pt] (e8) -- ($(e8)!0.25!(sigma2)$);
    \draw[red!40, line width=2pt] (sigma5) -- ($(sigma5)!0.25!(e11)$);
    \draw[red!40, line width=2pt] (sigma6) -- ($(sigma6)!0.25!(e12)$);
    \draw[red!40, line width=2pt] (sigma4) -- ($(sigma4)!0.25!(e10)$);
    \draw[red!40, line width=2pt] (sigma3) -- ($(sigma3)!0.25!(e9)$);
    \draw[red!40, line width=2pt] (sigma2) -- ($(sigma2)!0.25!(e8)$);
    \draw[red!40, line width=2pt] (sigma1) -- ($(sigma1)!0.25!(e7d)$);

  \begin{scope}[decoration={
    markings,
    mark=at position 0.5 with {\arrow{stealth}}}
    ] 

    \foreach \N/\M in {1/5,2/3,3/1,4/6,5/4,6/2} {
        \draw[postaction=decorate] (e\N) -- (sigma\M);
        \draw[postaction=decorate] (e\N) -- (el\N);
    }
   \draw[postaction=decorate] (e12) -- (sigma5);
   \draw[postaction=decorate] (e12) -- (sigma6);
   \draw[postaction=decorate] (e11) -- (sigma5);
   \draw[postaction=decorate] (e11) -- (sigma4);
   \draw[postaction=decorate] (e10) -- (sigma4);
   \draw[postaction=decorate] (e10) -- (sigma3);
   \draw[postaction=decorate] (e9) -- (sigma3);
   \draw[postaction=decorate] (e9) -- (sigma2);
   \draw[postaction=decorate] (e8) -- (sigma2);
   \draw[postaction=decorate] (e8) -- (sigma1);
   \draw[postaction=decorate] (e7u) -- (sigma6);
   \draw[postaction=decorate] (e7u) .. controls (e7d) .. (sigma1);
   \end{scope}
       \end{tikzpicture}
    \caption{The resulting graph $\Gamma(\mathcal{C})$. 
	We give the polarisation by marking the first leg (i.e.\ the leg number 0, see \eqref{eq:morph-and-edge-order-convention}) of each vertex in red. The remaining legs are labelled counterclockwise for edges in $\text{in}(v)$ and clockwise for edges in $\text{out}(v)$, see again \eqref{eq:morph-and-edge-order-convention}.}
\label{fig:cylinder.triangulation.graph}
\end{figure}

In this appendix we give some details of how to calculate the morphism $T_A(C^\pm_{NS/R})$ defined in Section \ref{sec:TFT.cylinder}. We start with the triangulation of the cylinder given in Figure \ref{fig:cylinder.triangulation}. The dual triangulation is depicted in Figure \ref{fig:cylinder.triangulation.dual}, and the corresponding graph $\Gamma(\mathcal{C})$ in Figure \ref{fig:cylinder.triangulation.graph}.
We label the graph in Figure \ref{fig:cylinder.triangulation.graph} according to the construction in Section \ref{sec:prelim-graph} and then turn it into a correlator as in Equation \eqref{eq:TA-def}. This gives the morphism $T_A$ as a string diagram
 in $\mathcal{S}$:
\begin{equation}\label{eq:cyl-calc-aux1}
    T_A(\Sigma)=\raisebox{-3cm}{
    \begin{tikzpicture}
        \foreach \N/\Na in {1/8,2/9,3/10,4/11,5/12,6/7} {
        \node[name=t\N, rectangle, draw, minimum width=1.5cm, minimum height=0.5cm] at ($2*\N*(1cm,0cm)$) {$t(\sigma_\N)$};
        \node[name=c\Na, rectangle, draw, minimum width=0.5cm, minimum height=0.5cm] at ($(1cm,-2cm)+2*\N*(1cm,0cm)$) {$c_{s_{\Na}}$};
    }
    \foreach \N in {1,...,6}{
        \coordinate (t\N1) at ($(t\N.south west)!0.33!(t\N.south)$);
        \coordinate (t\N2) at (t\N.south);
        \coordinate (t\N3) at ($(t\N.south east)!0.33!(t\N.south)$);
    }
    \foreach \N in {7,...,12}{
        \coordinate (c\N1) at ($(c\N.north west)!0.5!(c\N.north)$);
        \coordinate (c\N2) at ($(c\N.north east)!0.5!(c\N.north)$);
    }
    \coordinate (su) at (0cm,0.5cm);
    \coordinate (sd) at (0cm,-0.5cm);
    \foreach \N in {1,2,3} \node[name=N\N, ellipse, draw, inner sep=0.05cm] at ($(2.5cm,-4cm)+\N*(1.1cm,0cm)$) {$-s_\N$};
    \foreach \N in {4,5,6} \node[name=N\N, ellipse, draw, inner sep=0.05cm] at ($(5cm,-4cm)+\N*(1.1cm,0cm)$) {$-s_\N$};
    \foreach \N in {1,...,6} \coordinate (in\N) at ($(N\N)+(0cm,-0.7cm)$);
    \foreach \N/\M in { 
        c81/t21,
        c82/t12,
        c91/t31,
        c92/t23,
        c101/t41,
        c102/t32,
        c111/t51,
        c112/t43,
        c121/t61,
        c122/t52,
        c71/t11,
        c72/t63} \draw (\N) .. controls ($(\N)+(su)$) and ($(\M)+(sd)$) .. (\M);
    \foreach \N in {1,...,6} \coordinate (h\N) at ($(t\N2)+(0cm,-2cm)$);
    \coordinate (h1) at ($(h1)+(0.5cm,0cm)$);
    \coordinate (h2) at ($(h2)+(0cm,0.3cm)$);
    \coordinate (h3) at ($(h3)+(0.5cm,0cm)$);
    \coordinate (h5) at ($(h5)+(0.5cm,0cm)$);
    \foreach \N/\M/\X in {
        N1.north/t13/h1,
        N2.north/t33/h3,
        N3.north/t53/h5,
        N4.north/t62/h6,
        N5.north/t42/h4,
        N6.north/t22/h2}{
            \draw (\N) .. controls ($(\N)+(su)$) and ($(\X)+2*(sd)$) .. (\X);
            \draw (\X) .. controls ($(\X)+(su)$) and ($(\M)+(sd)$) .. (\M);
        }
    \foreach \N in {1,...,6} \draw (in\N) -- (N\N.south);
\end{tikzpicture}}
\end{equation}
The $\sigma_i$ in $t(\sigma_i)$ is just a reference to which triangle the map comes from in order to make it easier for the reader to verify; the map is in all cases the same map $t:A^{\otimes 3} \to \mathbf{1}$.

Next we replace the morphisms $t$ and $c_{\pm 1}$ by structure maps of the Frobenius algebra as in Proposition \ref{lem:moves_from_alg}. After a tedious but straightforward calculation one arrives at
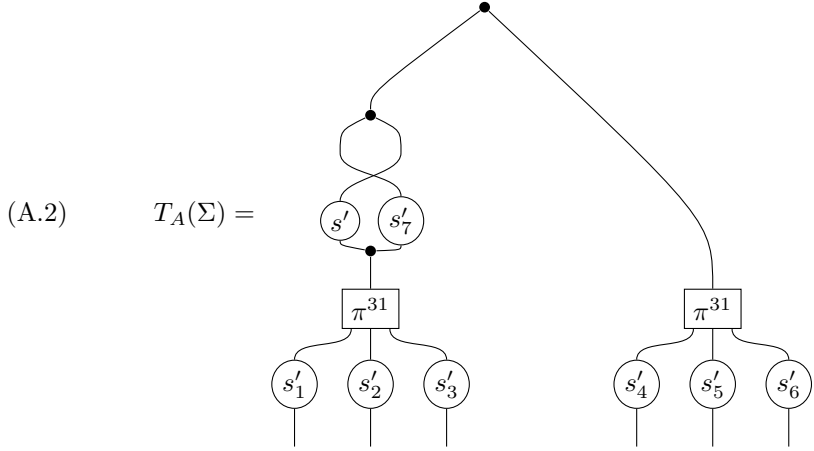
\begin{equation}
    T_A(\Sigma)=\raisebox{-3cm}{
        \begin{tikzpicture}
            \coordinate (su) at (0cm,0.3cm);
            \coordinate (sd) at (0cm,-0.3cm);
            \node[name=top,circle,fill,inner sep=0.05cm] at (3.5cm,0cm) {};
            \foreach \N in {1,2,3} \node[name=N\N, ellipse, draw, inner sep=0.05cm] at ($(0cm,-5cm)+\N*(1cm,0cm)$) {$s_\N'$};
            \foreach \N in {4,5,6} \node[name=N\N, ellipse, draw, inner sep=0.05cm] at ($(1.5cm,-5cm)+\N*(1cm,0cm)$) {$s_\N'$};
            \node[name=proj31L, draw] at ($(N2)+(0cm,1cm)$) {$\pi^{31}$};
            \node[name=proj31R, draw] at ($(N5)+(0cm,1cm)$) {$\pi^{31}$};
            \foreach \N in {1,...,6} \draw ($(N\N.south)+(0cm,-0.5cm)$) -- (N\N);
            \draw (N1.north) .. controls ($(N1.north)+(su)$) and ($($(proj31L.south west)!0.33!(proj31L.south)$)+(sd)$) .. ($(proj31L.south west)!0.33!(proj31L.south)$);
            \draw (N2.north) -- (proj31L.south);
            \draw (N3.north) .. controls ($(N3.north)+(su)$) and ($($(proj31L.south east)!0.33!(proj31L.south)$)+(sd)$) .. ($(proj31L.south east)!0.33!(proj31L.south)$);
            \draw (N4.north) .. controls ($(N4.north)+(su)$) and ($($(proj31R.south west)!0.33!(proj31R.south)$)+(sd)$) .. ($(proj31R.south west)!0.33!(proj31R.south)$);
            \draw (N5.north) -- (proj31R.south);
            \draw (N6.north) .. controls ($(N6.north)+(su)$) and ($($(proj31R.south east)!0.33!(proj31R.south)$)+(sd)$) .. ($(proj31R.south east)!0.33!(proj31R.south)$);
            \begin{scope}[shift={(proj31L.north)}]
            \coordinate (in1) at (0cm,0cm);
            \coordinate (id1l1) at (-0.4cm,1cm);
            \coordinate (id1r1) at (0.4cm,1cm);
            \coordinate (id2l1) at (-0.4cm,1.8cm);
            \coordinate (id2r1) at (0.4cm,1.8cm);
            \coordinate (out1) at (0cm,2.5cm);
            \node[name=cp1,circle,fill,inner sep=0.05cm] at (0cm,0.5cm) {};
            \node[name=p1,circle,fill,inner sep=0.05cm] at (0cm,2.3cm) {};
            \node[name=NL,ellipse,draw,inner sep=0.05cm] at ($(id1l1)+(0cm,-0.1cm)$) {{$s'$}};
            \node[name=NR,ellipse,draw,inner sep=0.05cm] at ($(id1r1)+(0cm,-0.1cm)$) {{$s'_7$}};
            \end{scope}
            \draw (in1) -- (cp1);
            \draw (cp1) .. controls ($(NR.south)+0.2*(sd)$) .. (NR.south);
            \draw (cp1) .. controls ($(NL.south)+0.2*(sd)$) .. (NL.south);
            \draw (NL.north) .. controls ($(NL.north)+(su)$) and ($(id2r1)+(sd)$) .. (id2r1);
            \draw (NR.north) .. controls ($(NR.north)+(su)$) and ($(id2l1)+(sd)$) .. (id2l1);
            \draw (id2l1) .. controls ($(id2l1)+(su)$) .. (p1);
            \draw (id2r1) .. controls ($(id2r1)+(su)$) .. (p1);
            \draw (p1) .. controls ($(p1)+(su)$) .. (top);
            \draw (proj31R.north) .. controls ($(proj31R.north)+3*(su)$) .. (top);
        \end{tikzpicture}
    }
\end{equation}
with the signs $s'$ and $s_1',\dots,s_7'$ given by
\begin{align}
    s_1'&=-s_1~,~~ &
    s_4'&=s_4\ , \\
    s_2'&=-s_2s_8s_9~,~~ &
    s_5'&=s_5s_{11}s_{12} \ ,  \nonumber\\
    s_3'&=-s_3s_8s_9s_{10}s_{11}~,~~ &
    s_6'&=s_6s_9s_{10}s_{11}s_{12} \ ,  \nonumber\\
    s_7'&=s_7~,~~ &s'&=s_8s_9s_{10}s_{11}s_{12} \ .\nonumber
\end{align}
Evaluating this for the signs $s_i$ given in \eqref{eq:cylinder.signs.convention} and \eqref{eq:cylinder.signs.solution} then yields \eqref{eq:TFT-cylinder-values}.

\section{Evaluation of the TFT on the pair of pants}\label{app:pairpants}

\begin{figure}[tb]
    \centering
    \begin{tikzpicture}
    \node[name=B1, regular polygon, regular polygon sides=3, minimum size=15cm] at (0,0) {};
    \foreach \N/\M/\X in {1/3/right,2/2/left,3/1/above} {
        \coordinate (v\N) at (B1.corner \M);
        \node[\X] at (v\N) {$v_\N$};
    }
    \node[below=0.5cm] at (B1.side 2) {$B_1$};
    \node[name=B2, regular polygon, regular polygon sides=3, minimum size=3cm] at ($0.33*(v1)+0.66*(v2)+0.33*(v3)$) {$B_2$};
    \foreach \N/\M/\X in {4/2/below,5/3/below,6/1/above left} {
        \coordinate (v\N) at (B2.corner \M);
        \node[\X] at (v\N) {$v_\N$};
    }
    \node[name=B3, regular polygon, regular polygon sides=3, minimum size=3cm] at ($0.66*(v1)+0.33*(v2)+0.33*(v3)$) {$B_3$}; 
    \foreach \N/\M/\X in {7/2/below,8/3/below,9/1/above right} {
        \coordinate (v\N) at (B3.corner \M);
        \node[\X] at (v\N) {$v_\N$};
    }
    \begin{scope}[green!40, line width=2pt]
        \foreach \N/\M/\X/\Y in { 
            1/2/0pt/-2pt,
            1/2/0pt/2pt,
            2/3/1pt/-1pt,
            3/1/-1pt/-1pt,
            4/5/0pt/-2pt,
            4/5/0pt/2pt,
            5/6/1pt/1pt,
            6/4/-1pt/1pt,
            7/8/0pt/-2pt,
            7/8/0pt/2pt,
            8/9/1pt/1pt,
            9/7/-1pt/1pt,
            5/7/0pt/-2pt,
            6/9/0pt/2pt
        }{
            \coordinate (x1) at ($(v\N)+(\X,\Y)$);
            \coordinate (x2) at ($(v\M)+(\X,\Y)$);
            \draw ($(x1)!0.05!(x2)$) -- ($(x1)!0.95!(x2)$);
        }
    \end{scope}
    \begin{scope}[decoration={markings,mark=at position 0.5 with {\arrow{stealth}}}]
        \foreach \N/\M/\X/\E in { 
            1/2/above/1,
            2/3/left/2,
            3/1/right/3,
            4/5/above/4,
            5/6/left/5,
            6/4/right/6,
            7/8/above/7,
            8/9/left/8,
            9/7/right/9,
            2/7/below/10,
            2/4/above/12,
            2/6/left/13,
            1/7/below/14,
            1/8/left/15,
            1/9/right/16,
            5/9/above left/17,
            6/9/below/18,
            3/6/right/19,
            3/9/left/20,
            5/7/above/21}
        {
            \draw[postaction=decorate] (v\N) -- (v\M); 
            \node[\X] at ($(v\N)!0.5!(v\M)$) {$e_{\E}$};
        }
        \draw[postaction=decorate] (v2) -- (v5); \node[left=5pt] at ($(v2)!0.5!(v5)$) {$e_{11}$};
    \end{scope}
    \foreach \N/\M/\K/\T in { 
        1/2/7/1,
        2/7/5/2,
        2/5/4/3,
        2/4/6/4,
        2/6/3/5,
        1/8/7/6,
        1/9/8/7,
        1/3/9/8,
        5/7/9/9,
        5/6/9/10,
        3/6/9/11
    }
        \node at ($0.33*(v\N)+0.33*(v\M)+0.33*(v\K)$) {$\sigma_{\T}$};
    \end{tikzpicture}
    \caption{A triangulation of the genus 0 surface with 3 boundaries with markings and labels. Boundaries are labelled by $B_1$, $B_2$, $B_3$ and correspondingly the edges $e_1$, \dots,$e_9$ are boundary edges.}
    \label{fig:pairofpants.triang}
\end{figure}
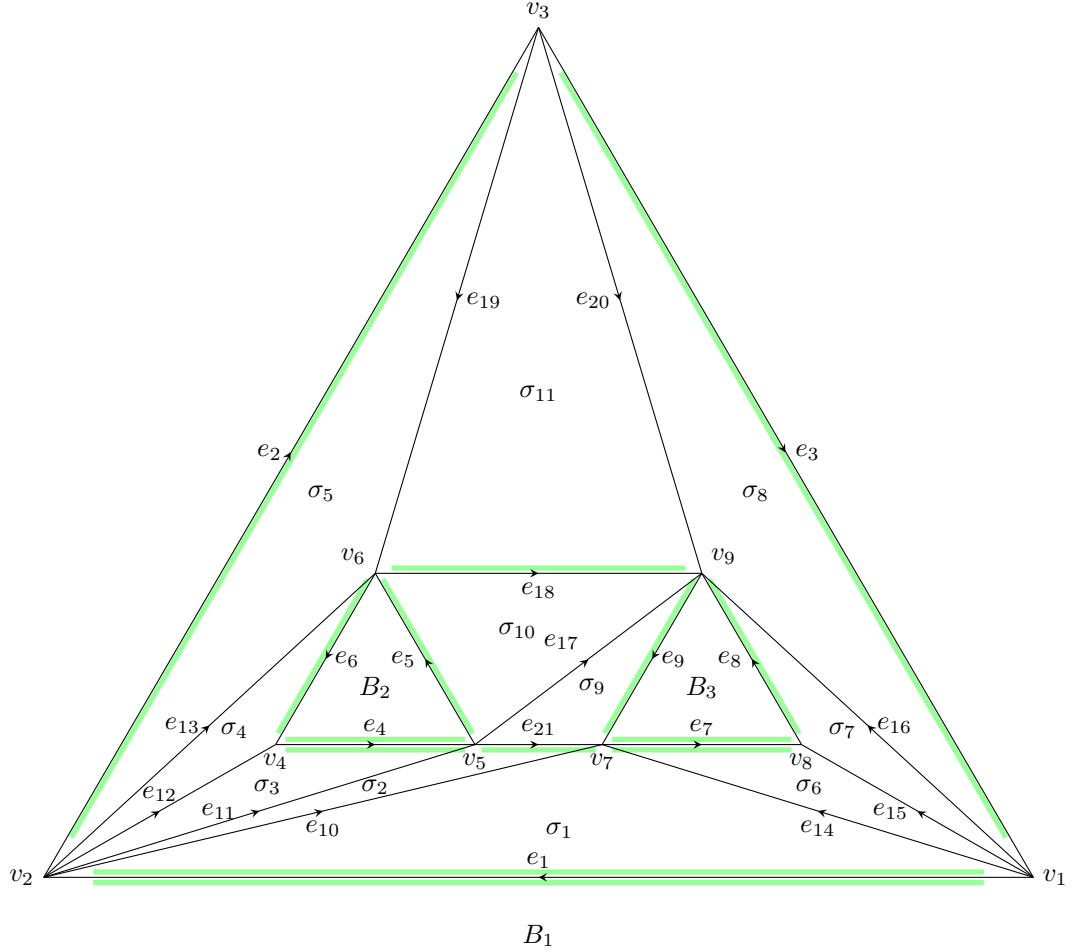

In this appendix we compute the value of the TFT on the surface $\Sigma^{0,3}$. We demand that the $i$'th boundary component $B_i$ is of type $\delta_i$, where $i=1,2,3$ and $\delta_i \in \{NS,R\}$. Our starting point is the triangulation and marking given in Figure \ref{fig:pairofpants.triang}. We determine the possible spin structures with the given boundary types by computing all admissible edge signs (see Section \ref{sse:spin_reconstruct}).

To reduce the number of parameters, use Lemma \ref{lem:index.marking}(1) to set an edge sign to $1$ for each of the triangles $\sigma_1$, \dots, $\sigma_{11}$:

\begin{center}
\begin{tabular}[]{lccccccccccc}
    triangle & $\sigma_1$ & $\sigma_2$ &$\sigma_3$ &$\sigma_4$ &$\sigma_5$ &$\sigma_6$ &$\sigma_7$ &$\sigma_8$ &$\sigma_9$ &$\sigma_{10}$ &$\sigma_{11}$ \\ \hline
    edge fixed& $e_{10}$ &$e_{11}$ &$e_4$ &$e_6$ &$e_{13}$ &$e_7$ &$e_8$ &$e_{16}$ &$e_9$ &$e_5$ &$e_{18}$ 
\end{tabular}
\end{center}
Let $s_i:= s(e_i)$. We thus have 
\begin{equation}\label{eq:pants-s=1-choice}
    s_4 = s_5 = s_6 = s_7 = s_8 = s_9 = s_{10} = s_{11} = s_{13} = s_{16} = s_{18} = 1 \ .
\end{equation}
We now have to evaluate the vertex rules at vertices $v_1$, \dots, $v_9$. These depend on the spin structure on the boundaries. Let 
\begin{equation}
    \nu_i = 
    \begin{cases}
        1 & \text{if $B_i$ is of $NS$-type} \\
        -1 & \text{if $B_i$ is of $R$-type}
    \end{cases}
\end{equation}
	for $i=1,2,3$. 
The conditions at the vertices can then be evaluated to
\begin{equation}
\begin{tabular}{lll}
        $v_1:s_{1}s_{3}s_{14}s_{15}= -\nu_1$ , & 
            $v_2:s_{1}s_{2}s_{12}=-1$ , &
            $v_3:s_{2}s_{3}s_{19}s_{20}=-1$ ,  \\ 
        $v_4:s_{12}=\nu_2$ , &
            $v_5:s_{17}s_{21}=-1$ , &
            $v_6:s_{19}=-1$ , \\
        $v_7:s_{14}s_{21}=-\nu_3$ , &
            $v_8:s_{15}=-1$ , &
            $v_9:s_{17}s_{20}=-1$  .
    \end{tabular} 
\end{equation}
From these equations it follows that 
\begin{equation}
    \nu_1 \, \nu_2 \, \nu_3=1 \ . 
    \label{eq:pop.boundary.spin.restriction}
\end{equation}
If this is the case, let $\alpha_1, \alpha_2 \in \{1,-1\}$. Then all solutions to these equations are given by 
\begin{equation}\label{eq:pants-s_via_alpha-sol}
    \begin{tabular}[]{r|cccccccccc}
        $i$ & $1$ & $2$ & $3$ & $12$ & $14$ & $15$ & $17$ & $19$ & $20$ & $21$ \\ \hline
        $s_i$ & $\alpha_1$ & $-\nu_2\alpha_1$ & $\nu_1\alpha_1\alpha_2$ & $\nu_2$ & $\alpha_2$ & $-1$ & $\nu_3\alpha_2$ & $-1$ & $-\nu_3\alpha_2$ & $-\nu_3\alpha_2$
    \end{tabular} \quad .
\end{equation}

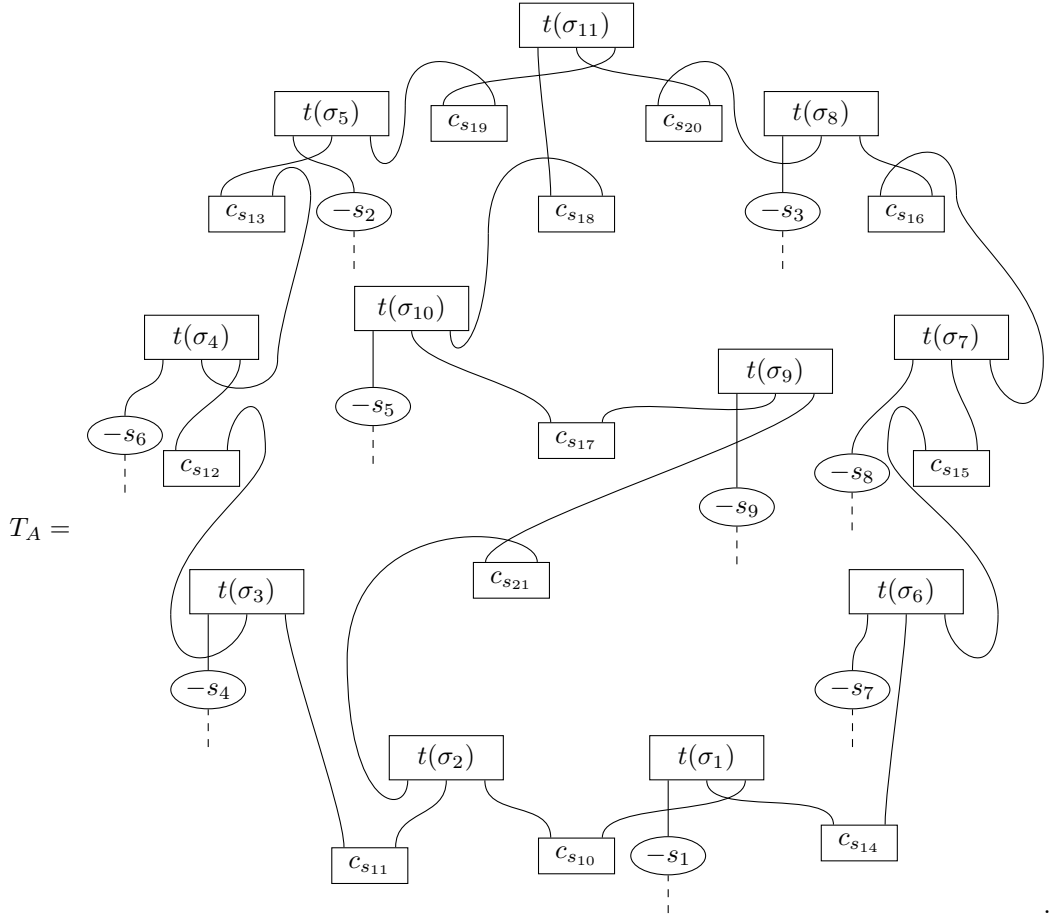
\begin{figure}[tb]
\begin{equation*}
    T_A =\raisebox{-5cm}{ 
    \begin{tikzpicture}
    \node[name=C1, regular polygon, regular polygon sides=9, minimum size=10cm, red, dashed] at (0,0) {};
    \foreach \N/\M in {1/11,2/5,3/4,4/3,5/2,6/1,7/6,8/7,9/8} \node[name=t\M, draw, minimum width=1.5cm] at (C1.corner \N) {$t(\sigma_{\M})$};
    \node[name=t10, draw, minimum width=1.5cm] at ($(t3)!0.5!(t11)$) {$t(\sigma_{10})$};
    \node[name=t9, draw, minimum width=1.5cm] at ($(t1)!0.6!(t8)$) {$t(\sigma_{9})$};
    \node[name=C2, regular polygon, regular polygon sides=9, minimum size=10cm, red, dashed, rotate=20] at (0,0) {};
    \coordinate (sh19) at (0.3cm,-1cm);    
    \coordinate (sh13) at (0cm,0cm);    
    \coordinate (sh12) at (0cm,0cm);    
    \coordinate (sh11) at (0.5cm,-2.3cm);    
    \coordinate (sh10) at (0cm,-1cm);    
    \coordinate (sh14) at (0.5cm,-2cm);    
    \coordinate (sh15) at (0cm,0cm);    
    \coordinate (sh16) at (0cm,0cm);    
    \coordinate (sh20) at (-0.3cm,-1cm);    
    \foreach \N/\M in {1/19,2/13,3/12,4/11,5/10,6/14,7/15,8/16,9/20} \node[name=c\M, draw, minimum width=1cm] at ($(C2.corner \N)+(sh\M)$) {$c_{s_{\M}}$};
    \node[name=c17, draw, minimum width=1cm] at ($(C2.center)+(0cm,-0.5cm)$) {$c_{s_{17}}$};
    \node[name=c18, draw, minimum width=1cm] at ($(C2.center)!0.5!(t11)$) {$c_{s_{18}}$};
    \node[name=c21, draw, minimum width=1cm] at ($(C2.center)!0.5!(t2)$) {$c_{s_{21}}$};
    \foreach \N in {1,...,9} \coordinate (c\N) at (0,0);
    \coordinate (su) at (0cm,0.5cm);
    \coordinate (sd) at (0cm,-0.5cm);
    \foreach \N in {1,...,11} {
        \coordinate (tx\N1) at ($(t\N.south west)!0.33!(t\N.south)$);
        \coordinate (tx\N2) at (t\N.south);
        \coordinate (tx\N3) at ($(t\N.south east)!0.33!(t\N.south)$);
    }
    \foreach \N in {1,...,21} {
        \coordinate (cx\N1) at ($(c\N.north west)!0.33!(c\N.north)$);
        \coordinate (cx\N2) at ($(c\N.north east)!0.33!(c\N.north)$);
        }
    \foreach \C/\T in {
        101/23,
        102/13,
        111/33,
        112/22,
        121/43,
        131/52,
        141/12,
        142/62,
        152/72,
        162/83,
        171/102,
        172/92,
        181/111,
        191/113,
        202/112,
        211/93}
     \draw (cx\C) .. controls ($(cx\C)+(su)$) and ($(tx\T)+(sd)$) .. (tx\T);
    \coordinate (h192) at ($(cx191)+(-0.5cm,0cm)$);
    \coordinate (h182) at ($(cx182)+(-1.5cm,-0.5cm)$);
    \coordinate (h201) at ($(cx201)+(1cm,0cm)$);
    \coordinate (h212) at ($(cx212)+(-2.5cm,-1.3cm)$);
    \foreach \C/\T/\M in { 
        192/53/2,
        182/103/3,
        201/82/2,
        212/21/4}{
        \draw (cx\C) .. controls ($(cx\C)+(su)$) and ($(h\C)+\M*(su)$) .. (h\C);
        \draw (h\C) .. controls ($(h\C)+\M*(sd)$) and ($(tx\T)+(sd)$) .. (tx\T);
    } 
    \coordinate (hc122) at ($(cx122)+(0.5cm,0cm)$);
    \coordinate (ht32) at ($(tx32)+(-1cm,0cm)$);
    \coordinate (hc132) at ($(cx132)+(0.5cm,0cm)$);
    \coordinate (ht42) at ($(tx42)+(1cm,0cm)$);
    \coordinate (hc161) at ($(cx161)+(1cm,0cm)$);
    \coordinate (ht73) at ($(tx73)+(0.7cm,0cm)$);
    \coordinate (hc151) at ($(cx151)+(-0.5cm,0cm)$);
    \coordinate (ht63) at ($(tx63)+(0.7cm,0cm)$);
    \foreach \C/\T/\M in {
        122/32/2,
        132/42/1,
        151/63/2,
        161/73/2}{
        \draw (cx\C) .. controls ($(cx\C)+(su)$) and ($(hc\C)+\M*(su)$) .. (hc\C);
        \draw (hc\C) .. controls ($(hc\C)+(sd)$) and ($(ht\T)+\M*(su)$) .. (ht\T);
        \draw (ht\T) .. controls ($(ht\T)+\M*(sd)$) and ($(tx\T)+(sd)$) .. (tx\T);
    }
    \foreach \T in {1,...,9} \coordinate (s\T) at (0,0);
    \coordinate (s2) at (0.8cm,0cm);
    \coordinate (s6) at (-0.5cm,0cm);
    \coordinate (s7) at (-0.2cm,0cm);
    \coordinate (s8) at (-0.8cm,-0.5cm);
    \coordinate (s9) at (0cm,-0.5cm);
    \foreach \C/\T in {1/1,5/2,8/3,3/4,10/5,4/6,6/7,7/8,9/9} {
        \node[name=N\T, ellipse, draw, inner sep=0.05cm] at ($(tx\C1)+(0cm,-1cm)+(s\T)$) {$-s_{\T}$};
        \draw (N\T.north) .. controls ($(N\T.north)+(su)$) and ($(tx\C1)+(sd)$) .. (tx\C1);
        \draw[dashed] ($(N\T.south)+(0cm,-0.5cm)$) -- (N\T.south);
    }
    \end{tikzpicture}.}
\end{equation*}
    \caption{The string diagram resulting from the triangulation in Figure \ref{fig:pairofpants.triang}. Here the dotted ingoing lines have to be ordered from $1,\dots,9$ according to the edge $e_i$ they correspond to. As in \eqref{eq:cyl-calc-aux1} we write $t(\sigma_i)$ to indicate which triangle the map comes from, but the map is $t:A^{\otimes 3} \to \mathbf{1}$ in all cases.}
\label{fig:pairofpants.morph}
\end{figure}

The result of translating the triangulation in Figure \ref{fig:pairofpants.triang} into a string diagram as in Section \ref{sec:prelim-graph} and Equation \eqref{eq:TA-def} is shown in Figure \ref{fig:pairofpants.morph}.
We now replace the maps $t$ and $c_{\pm 1}$ as in Proposition \ref{lem:moves_from_alg}. 
After a tedious but straightforward calculation one arrives at
    \begin{equation}\label{eq:3-holed-sphere-morphs-aux1}
    T_A = 
    b\circ \left(\id_A\otimes \mu\right) 
    \circ \big(q_{\nu_1} \otimes q_{\nu_2} \otimes q_{\nu_3} \big) 
    \circ \big(N_{\alpha_1}\otimes \id_A \otimes N_{-\eta_1\alpha_2}\big) 
    \circ (\pi^{31})^{\otimes 3}\ . 
    \end{equation}
Two of the identities used to get this result are worth pointing out: firstly, Lemma \ref{lem:forget-one-projector} has been used to insert an additional $q$ to make the expression more symmetric; secondly, the Nakayama automorphism satisfies $q_\nu \circ N_{-\nu} = q_\nu$ (Lemma \ref{lem:PNS/R-idemp}). 

\medskip

We now turn to the proofs of Lemmas \ref{lem:pairofpants.spinstructures} and \ref{lem:pair-of-pants-TFT-value} from Section \ref{sec:pair-of-pants}.

\begin{proof}[Proof of Lemma \ref{lem:pairofpants.spinstructures}]
A spin structure with boundary types $\delta_1, \delta_2, \delta_3$ exists if and only if there are admissible edge signs on the marked triangulation given in Figure \ref{fig:pairofpants.triang}. The necessary (and sufficient) condition for this stated in \eqref{eq:pop.boundary.spin.restriction} proves the first part of the lemma.

For the second statement we need to check that up to isomorphism there are exactly four spin structures, and that representatives of these are provided by the four sets of admissible edge signs found above. 

All possible spin structures are produced from any one spin structure by composing a boundary parametrisation with a leaf exchange. This gives a transitive action of $(\mathbb{Z}_2)^3$ on the set of spin structures. Since the surface is connected, the only non-trivial automorphism of the spin structure in the interior of $\Sigma^{0,3}$ is leaf exchange. On the boundary, this induces the diagonal $\mathbb{Z}_2$-action. The quotient of $(\mathbb{Z}_2)^3$ by the diagonal $\mathbb{Z}_2$ thus acts transitively and faithfully, showing that there are four spin structures (with parametrised boundary). Finally, since changing $\alpha_1$ and $\alpha_2$ amounts to precomposing two of the three boundaries with a leaf exchange, the four values of $(\alpha_1,\alpha_2)$ precisely give the four possible spin structures.
\end{proof}

\begin{proof}[Proof of Lemma \ref{lem:pair-of-pants-TFT-value}]
Given the condition in \eqref{eq:pop.boundary.spin.restriction}, we get four spin structures parametrised by $(\alpha_1,\alpha_2)$. The value of the TFT on the corresponding spin surface is given in \eqref{eq:3-holed-sphere-morphs-aux1}. Now substitute 
\begin{equation}\label{eq:pants-alpha-via-eps}
	\alpha_1 = \varepsilon_1 \, \varepsilon_2
	\quad , \qquad
	\alpha_2 = - \eta_1\, \varepsilon_2 \ ,
\end{equation}
as well as $\id = N_{\varepsilon_2} \circ N_{\varepsilon_2}$. One can then remove one factor of $N_{\varepsilon_2}$ from each leg by moving it through $b \circ (\mu \otimes \id)$. This results in the expression stated in the lemma.
\end{proof}

The above proof also determines the spin structure of $\Sigma^{0,3}_{\delta_1,\delta_2,\delta_3,\varepsilon_1,\varepsilon_2}$ 
to be those obtained from the marked triangulation equipped with the edge signs \eqref{eq:pants-s=1-choice} and \eqref{eq:pants-s_via_alpha-sol}, where $\alpha_{1,2}$ have been replaced as in \eqref{eq:pants-alpha-via-eps}.


\newcommand\arxiv[2]      {\href{http://arXiv.org/abs/#1}{#2}}
\newcommand\doi[2]        {\href{http://dx.doi.org/#1}{#2}}
\newcommand\httpurl[2]    {\href{http://#1}{#2}}

\providecommand{\href}[2]{#2}\begingroup\raggedright

\end{document}